\documentclass{amsart}
\usepackage{srcltx,hyperref}
\usepackage{amsmath,amssymb,amscd}
\usepackage{mathrsfs}
\usepackage{amsfonts}
\usepackage{enumerate}
\numberwithin{equation}{section}

\usepackage{geometry}
\def\labelenumi{\theenumi}
\def\theenumi{(\alph{enumi})}

\def\p@enumii{\theenumi}

\long\def\forget#1{}

\newcommand{\TimesQQQ}{\times}
\newcommand{\TimesL}{\times}
\newcommand{\TimesBC}{\times}

\usepackage{amsthm}
\theoremstyle{plain}
\newtheorem{thm}{Theorem}[section]
\newtheorem{prop}[thm]{Proposition}
\newtheorem{lemma}[thm]{Lemma}
\newtheorem{cor}[thm]{Corollary}
\newtheorem{con}[thm]{Conjecture}
\theoremstyle{definition}
\newtheorem{defn}[thm]{Definition}
\newtheorem{notn}[thm]{Notation}
\newtheorem{setup}[thm]{Setup}
\newtheorem{point}[thm]{}
\newtheorem{example}[thm]{Example}

\newtheorem{rem}[thm]{Remark}

\newcounter{proofnumber}
\newcounter{technique}
\newcommand{\technique}[1]{\refstepcounter{technique}\noindent{\bfseries Reduction Technique~(\Alph{technique}) #1.}}

\newcommand{\dbl}{{\mathchoice{\mbox{\rm [\hspace{-0.15em}[}}
                              {\mbox{\rm [\hspace{-0.15em}[}}
                              {\mbox{\scriptsize\rm [\hspace{-0.15em}[}}
                              {\mbox{\tiny\rm [\hspace{-0.15em}[}}}}
\newcommand{\dbr}{{\mathchoice{\mbox{\rm ]\hspace{-0.15em}]}}
                              {\mbox{\rm ]\hspace{-0.15em}]}}
                              {\mbox{\scriptsize\rm ]\hspace{-0.15em}]}}
                              {\mbox{\tiny\rm ]\hspace{-0.15em}]}}}}

\usepackage{color}
\newcounter{commentcounter}

\newcommand{\Urscomment}[1]{}{
\immediate\write16{}
\immediate\write16{Warning: There was still a comment . . . }
\immediate\write16{}}

\def\?{\ {\color{magenta}???}\ 
\immediate\write16{}
\immediate\write16{Warning: There was still a question mark . . . }
\immediate\write16{}}

\usepackage[curve,matrix,arrow,cmtip]{xy}

\newdir^{ (}{{}*!/-3pt/\dir^{(}}    
\newdir^{  }{{}*!/-3pt/\dir^{}}    
\newdir_{ (}{{}*!/-3pt/\dir_{(}}    
\newdir_{  }{{}*!/-3pt/\dir_{}}    

\newcommand{\BOne} {{\mathchoice{\hbox{\rm1\kern-2.7pt l\kern.9pt}}
                              {\hbox{\rm1\kern-2.7pt l\kern.9pt}}
                              {\hbox{\scriptsize\rm1\kern-2.3pt l\kern.4pt}}
                              {\hbox{\scriptsize\rm1\kern-2.4pt l\kern.5pt}}}}
\newcommand{\UOne}{\underline{\BOne}}

\newcommand{\BA}{{\mathbb{A}}}

\newcommand{\BC}{{\mathbb{C}}}

\newcommand{\BE}{{\mathbb{E}}}
\newcommand{\BF}{{\mathbb{F}}}
\newcommand{\BG}{{\mathbb{G}}}
\newcommand{\BH}{{\mathbb{H}}}

\newcommand{\BJ}{{\mathbb{J}}}
\newcommand{\BK}{{\mathbb{K}}}
\newcommand{\BL}{{\mathbb{L}}}

\newcommand{\BN}{{\mathbb{N}}}

\newcommand{\BQ}{{\mathbb{Q}}}
\newcommand{\BR}{{\mathbb{R}}}
\newcommand{\BS}{{\mathbb{S}}}
\newcommand{\BT}{{\mathbb{T}}}

\newcommand{\BZ}{{\mathbb{Z}}}

\newcommand{\CB}{{\mathcal{B}}}
\newcommand{\CC}{{\mathcal{C}}}
\newcommand{\mathcalD}{{\mathcal{D}}}
\newcommand{\CE}{{\mathcal{E}}}
\newcommand{\CF}{{\mathcal{F}}}
\newcommand{\CG}{{\mathcal{G}}}
\newcommand{\CH}{{\mathcal{H}}}
\newcommand{\CI}{{\mathcal{I}}}

\newcommand{\CK}{{\mathcal{K}}}
\newcommand{\CL}{{\mathcal{L}}}

\newcommand{\CN}{{\mathcal{N}}}
\newcommand{\CO}{{\mathcal{O}}}

\newcommand{\CS}{{\mathcal{S}}}
\newcommand{\CT}{{\mathcal{T}}}
\newcommand{\CU}{{\mathcal{U}}}

\newcommand{\Fu}{{\mathfrak{u}}}

\DeclareMathOperator{\Ad}{Ad}
\DeclareMathOperator{\Aut}{Aut}

\DeclareMathOperator{\End}{End}

\DeclareMathOperator{\Frob}{Frob}
\DeclareMathOperator{\Gal}{Gal}
\DeclareMathOperator{\GL}{GL}
\DeclareMathOperator{\Gr}{Gr}
\DeclareMathOperator{\Koh}{H}

\DeclareMathOperator{\Hom}{Hom}

\DeclareMathOperator{\Ind}{Ind}

\DeclareMathOperator{\Isom}{Isom}
\DeclareMathOperator{\Lie}{Lie}

\DeclareMathOperator{\PGL}{PGL}

\DeclareMathOperator{\Rep}{Rep}
\DeclareMathOperator{\Res}{Res}

\DeclareMathOperator{\Spec}{Spec}

\DeclareMathOperator{\Tr}{Tr}
\DeclareMathOperator{\Var}{V}

\DeclareMathOperator{\charakt}{char}

\newcommand{\cont}{{\rm c}}

\newcommand{\et}{{\rm\acute{e}t}}

\DeclareMathOperator{\id}{\,id}
\DeclareMathOperator{\im}{im}

\renewcommand{\mod}{\;{\rm mod}\;}

\newcommand{\red}{{\rm red}}

\newcommand{\rig}{{\rm rig}}
\DeclareMathOperator{\rk}{rk}

\newcommand{\un}{{\rm un}}

\newcommand{\open}{^\circ}
\newcommand{\dual}{^{\scriptscriptstyle\vee}}
\newcommand{\mal}{^\times}
\newcommand{\ul}[1]{{\underline{#1}}}
\newcommand{\ol}[1]{{\overline{#1}}}
\newcommand{\wh}[1]{{\widehat{#1}}}
\newcommand{\wt}[1]{{\widetilde{#1}}}

\newcommand{\olB}{{\,{\overline{\!B}}{}}}
\newcommand{\olC}{{\,{\overline{\!C}}{}}}
\newcommand{\olD}{{\,{\overline{\!D}}{}}}

\newcommand{\olG}{{\,{\overline{\!G}}{}}}
\newcommand{\olH}{{\,{\overline{\!H}}{}}}
\newcommand{\olK}{{\,{\overline{\!K}}{}}}
\newcommand{\olP}{{\,{\overline{\!P}}{}}}
\newcommand{\olS}{{\,{\overline{\!S}}{}}}
\newcommand{\olT}{{\overline{T}}{}}

\newcommand{\olV}{{\overline{V}}{}}
\newcommand{\olW}{{\overline{W}}{}}
\newcommand{\olX}{{\,{\overline{\!X}}{}}}
\newcommand{\olZ}{{\,{\overline{\!Z}}{}}}

\def\longto{\longrightarrow}
\def\into{\hookrightarrow}
\def\longinto{\lhook\joinrel\longrightarrow}
\def\onto{\twoheadrightarrow}
\newcommand{\longonto}{\hbox{$\longrightarrow\kern-12pt\to\kern2pt$}}
\def\isoto{\stackrel{}{\mbox{\hspace{1mm}\raisebox{+1.4mm}{$\scriptstyle\sim$}\hspace{-3.5mm}$\longrightarrow$}}}

\newcommand{\dal}{{\mathchoice{\mbox{$\langle\hspace{-0.15em}\langle$}}
                              {\mbox{$\langle\hspace{-0.15em}\langle$}}
                              {\mbox{\scriptsize$\langle\hspace{-0.15em}\langle$}}
                              {\mbox{\tiny$\langle\hspace{-0.15em}\langle$}}}}
\newcommand{\dar}{{\mathchoice{\mbox{$\rangle\hspace{-0.15em}\rangle$}}
                              {\mbox{$\rangle\hspace{-0.15em}\rangle$}}
                              {\mbox{\scriptsize$\rangle\hspace{-0.15em}\rangle$}}
                              {\mbox{\tiny$\rangle\hspace{-0.15em}\rangle$}}}}

\DeclareMathOperator{\FIsoc}{{\mathchoice{{\mbox{\rm $F$-Isoc}}}
                                         {{\mbox{\rm $F$-Isoc}}}
                                         {{\mbox{\rm\scriptsize$F$-Isoc}}}
                                         {{\mbox{\rm\tiny$F$-Isoc}}}}}
           
\DeclareMathOperator{\Isoc}{{\mathchoice{{\mbox{\rm Isoc}}}
                                         {{\mbox{\rm Isoc}}}
                                         {{\mbox{\rm\scriptsize Isoc}}}
                                         {{\mbox{\rm\tiny Isoc}}}}}

\DeclareMathOperator{\FnIsoc}{{\mathchoice{{\mbox{\rm$F^n$-Isoc}}}
                                          {{\mbox{\rm$F^n$-Isoc}}}
                                          {{\mbox{\rm\scriptsize$F^n$-Isoc}}}
                                          {{\mbox{\rm\tiny$F^n$-Isoc}}}}}

\newcommand{\BasePoint}{u}
\newcommand{\BasePtDeg}{e}
\DeclareMathOperator{\FuIsoc}{{\mathchoice{{\mbox{\rm$F^\BasePtDeg$-Isoc}}}
                                          {{\mbox{\rm$F^\BasePtDeg$-Isoc}}}
                                          {{\mbox{\rm\scriptsize$F^\BasePtDeg$-Isoc}}}
                                          {{\mbox{\rm\tiny$F^\BasePtDeg$-Isoc}}}}}

\DeclareMathOperator{\FUR}{{\mathchoice{{\mbox{\rm $F$-UR}}}
                                         {{\mbox{\rm $F$-UR}}}
                                         {{\mbox{\rm\scriptsize$F$-UR}}}
                                         {{\mbox{\rm\tiny$F$-UR}}}}}
\DeclareMathOperator{\FConst}{{\mathchoice{{\mbox{\rm $F$-Const}}}
                                         {{\mbox{\rm $F$-Const}}}
                                         {{\mbox{\rm\scriptsize$F$-Const}}}
                                         {{\mbox{\rm\tiny$F$-Const}}}}}

\DeclareMathOperator{\DGal}{DGal}

\renewcommand{\epsilon}{\varepsilon}
\renewcommand{\phi}{\varphi}

\begin{document}
\title{Crystalline Chebotar\"ev Density Theorems}
\author{Urs Hartl and Ambrus P\'al}
\address{Universit\"at M\"unster, Mathematisches Institut, Einsteinstrasse 62,
48149 M\"un\-ster, Germany}
\email{https:/\hspace{-1mm}/www.uni-muenster.de/Arithm/hartl/}
\address{Eötvös Loránd University, Mathematical Institute, 1117 Budapest, Pázmány Péter sétány 1/C, Hungary}
\email{ambrus.pal@ttk.elte.hu}

\mbox{ } \vspace{-1.3cm}
\maketitle

\vspace{-1cm}

\begin{abstract}
Using the Tannakian formalism, we formulate conjectural analogs of Chebotar\"ev's Density Theorem for $F$-isocrystals over a smooth geometrically irreducible variety defined over a finite field. We prove these analogs for several large classes, including (a) constant $F$-isocrystals, (b) direct sums of isoclinic convergent $F$-isocrystals, (c) semi-simple overconvergent $F$-isocrystals, and (d) semi-simple convergent $F$-isocrystals which have an overconvergent extension. Case (a) is a generalization of the Mordell-Lang Conjecture for tori and enters in the proofs of (b) and (c). For (b) we use the classical Chebotar\"ev Density Theorem, and point counting techniques in $p$-adic Lie groups building on a result of Oesterl\'e. For (c) we give two proofs. One of them uses deep input on the Langlands correspondence by Abe and Lafforgue, and the theory of Frobenius weights of Kedlaya, Abe and Caro. Building on this we formulate and prove the $p$-adic analog of Deligne's Equidistribution Theorem. Then (c) follows by applying real algebraic geometry to maximal compact subgroups in complex algebraic groups, measure theory, and a convergence result on complex hypersurfaces. For (d) we develop the theory of maximal quasi-tori (generalizing maximal tori in non-connected linear algebraic groups) and use D'Addezio's result on Crew's parabolicity conjecture to reduce to (b). These arguments also yield a second proof of (c). Besides of the deep inputs mentioned above and some Tannakian arguments, our main technique is the theory of linear algebraic groups. We include a comparison with the recent article of Cadoret and Tamagawa on the same topic.

\end{abstract}

\setcounter{tocdepth}{2}
\tableofcontents

\vspace{-1cm}

\section{Introduction} \label{SectIntroduction}

The classical Chebotar\"ev Density Theorem has plenty of applications in arithmetic geometry. It was originaly formulated for Galois extensions of number fields. We consider its function field version in the following setup.
Let $\BF_q$ be a finite field with $q$ elements and characteristic $p$. Let $U$ be a smooth, geometrically irreducible, quasi-compact scheme over $\BF_q$ and let $|U|$ be the set of closed points in $U$. Let $\bar\BasePoint\in U(\ol\BF_q)$ be a geometric base point of $U$ and let $\pi_1^\et(U,\bar\BasePoint)$ be the \'etale fundamental group of $U$. It classifies separable algebraic extensions of the function field of $U$ which are \'etale above $U$. In the function field version it replaces the Galois group of extensions of number fields. For every $x\in|U|$ write $\deg(x)=[\BF_x:\BF_q]$, where $\BF_x$ denotes the residue field of $x$, and let $\Frob_x\subset \pi_1^\et(U,\bar\BasePoint)$ be the Frobenius conjugacy class at $x$. The classical Chebotar\"ev Density Theorem \cite[Theorem~7]{Serre63} says that for any finite quotient group $\pi_1^\et(U,\bar\BasePoint)\onto\Gamma$ and conjugation invariant subset $C\subset\Gamma$, the set of points $x\in |U|$ for which $\Frob_x$ maps into $C$ has Dirichlet density $\tfrac{\#C}{\#\Gamma}$. As a consequence, for every subset $S\subset|U|$ of Dirichlet density $1$, the Frobenius conjugacy classes $\Frob_x$ for $x\in S$ are dense in $\pi_1^\et(U,\bar\BasePoint)$. Thus representations of $\pi_1^\et(U,\bar\BasePoint)$ are largely controlled by the images of the $\Frob_x$. A particularly interesting case are representations in finite $\BQ_p$-vector spaces. These are equivalent by Crew's fundamental result \cite{Crew87} to convergent unit-root $F$-isocrystals. We may therefore ask whether the classical Chebotar\"ev Density Theorem can be generalized to larger classes of $F$-isocrystals. Considering the historical context of $F$-isocrystals we refer to such generalizations as ``Crystalline Chebotar\"ev Density Theorems''.

More precisely, let ${\rm Fr}_{q,U}$ denote the absolute $q$-Frobenius on $U$ which is the identity on points and the $q$-power map on the structure sheaf. Let $K$ be a finite totally ramified extension of the fraction field $W(\BF_q)[\tfrac{1}{p}]$ of the ring of Witt vectors. To shorten notation, we write $\olK$ (and not $\ol\BQ_p$) for an algebraic closure of $K$. The $q$-Frobenius ${\rm Fr}_{q,U}$ on $U$ corresponds to the $q$-Frobenius ${\rm Fr}_{q,\BF_q}$ on $\Spec \BF_q$. The latter equals the identity on $\Spec\BF_q$. Thus, we may choose the identity $F:=\text{id}_K$ on the lift $K$ of $\BF_q$ as a lift of ${\rm Fr}_{q,\BF_q}$. We consider the $K$-linear rigid abelian tensor categories $\FIsoc_K(U)$, respectively $\FIsoc^\dagger_K(U)$ of $K$-linear convergent, respectively overconvergent $F$-isocrystals on $U$ and their $\olK$-linear scalar extensions $\FIsoc_\olK(U):=\FIsoc_K(U)\otimes_K\olK$ and $\FIsoc^\dagger_\olK(U):=\FIsoc^\dagger_K(U)\otimes_K\olK$; see Section~\ref{SectPrelim} for details. In this article we only consider $F$-isocrystals with $\olK$-coefficients. If $\CF\in\FIsoc_\olK(U)$ we let $\dal\CF\dar$ denote the strictly full rigid abelian tensor sub-category of $\FIsoc_\olK(U)$ generated by $\CF$ and likewise for $\FIsoc^\dagger_\olK(U)$. We fix a base point $\BasePoint\in U(\BF_{q^\BasePtDeg})$ below $\bar\BasePoint$ with $\BasePtDeg=\deg(\BasePoint)$. Pulling back to $u$ defines a faithful fiber functor $\omega_\BasePoint\colon\CF\mapsto\BasePoint^*\CF$ to $\olK$-vector spaces, which makes $\FIsoc_\olK(U)$ into a neutral Tannakian category and $\dal\CF\dar$ into a neutral Tannakian sub-category of $\FIsoc_\olK(U)$, and likewise for $\FIsoc^\dagger_\olK(U)$; see \cite{Deligne-Milne82}. If $\CF$ is a convergent, respectively an overconvergent $F$-isocrystal on $U$ we consider its \emph{monodromy group}, denoted by $\Gr(\CF/U)$, respectively $\Gr^\dagger(\CF/U)$, and defined as the Tannakian fundamental group of $\dal\CF\dar$, that is, the linear algebraic group $\Aut^\otimes(\omega_\BasePoint|_{\dal\CF\dar})$ over $\olK$ consisting of the tensor automorphisms of $\omega_\BasePoint$. For every closed point $x\in|U|$ the Frobenius $F_\CF$ of $\CF$ furnishes a conjugacy class $\Frob_x(\CF)$ in $\Gr(\CF/U)$, respectively $\Frob^\dagger_x(\CF)$ in $\Gr^\dagger(\CF/U)$; see Definition~\ref{FrobClass}. The crystalline version of Chebotar\"ev's Density Theorem for a convergent $F$-isocrystal $\CF\in\FIsoc_\olK(U)$ is the following

\begin{con}\label{MainConj}
For every subset $S\subset|U|$ of upper Dirichlet density one, the set $\bigcup_{x\in S}\Frob_x(\CF)$ is Zariski-dense in $\Gr(\CF/U)$.
\end{con}

We follow Serre~\cite{Serre63} in the definition of (upper) Dirichlet density; see Definition~\ref{dirichlet_density_def}. The notion of positive upper Dirichlet density is a natural weakening of positive Dirichlet density. We explain in Remark~\ref{RemOtherTopologies} why we do not expect a density statement for any other topology than the Zariski-topology. As a consequence, we always use the words ``closure'' and ``density'' with respect to the Zariski-topology in this article, unless specified otherwise. We prove in Proposition~\ref{reduction3->2} that Conjecture~\ref{MainConj} is a consequence of the following stronger

\begin{con}\label{MainConj3}
For every subset $S\subset|U|$ of positive upper Dirichlet density the Zariski-closure of the set $\bigcup_{x\in S}\Frob_x(\CF)$ contains a connected component of the group $\Gr(\CF/U)$. 
\end{con}

There are the obvious analogs of the two conjectures for an overconvergent $F$-isocrystal $\CF\in\FIsoc^\dagger_\olK(U)$. After Lemma~\ref{Lemma1.10Converse} we will give the proof of the following application of crystalline Chebotar\"ev density. For $\CF\in\FIsoc_\olK(U)$ and $x\in|U|$ let $\Tr(\Frob_x(\CF))$ denote the common trace of all elements of $\Frob_x(\CF)$, considered as endomorphisms of the $\olK$-vector space $\omega_\BasePoint(\CF)$.

\begin{cor}\label{Cor1.9}
Let $S\subset|U|$ be a subset of upper Dirichlet density one, and let $\CF,\CG$ be two convergent (respectively overconvergent) $F$-isocrystals of the same rank on $U$ such that $\Tr(\Frob_x(\CF))=\Tr(\Frob_x(\CG))$ for every $x\in S$. Assume that Conjecture~\ref{MainConj} holds for the direct sum $\CF^{ss}\oplus\CG^{ss}$ of their semi-simplifications. Then $\CF^{ss}$ and $\CG^{ss}$ are isomorphic in $\FIsoc_\olK(U)$ (respectively in $\FIsoc^\dagger_\olK(U)$ ).
\end{cor}

Although currently we are unable to establish Conjectures~\ref{MainConj} and \ref{MainConj3} in general, we can still prove them in many cases. See \ref{PointCadTam} for further known cases. We start by proving Conjecture~\ref{MainConj} for unit-root $F$-isocrystals, because this explains the relation to the classical Chebotar\"ev Density Theorem. The \emph{Newton slopes} of a convergent $F$-isocrystal $\CF\in\FIsoc_\olK(U)$ at a closed point $x\in|U|$ are $\tfrac{1}{\deg(x)}$ times the common eigenvalues of the elements of $\Frob_x(\CF)\subset\Gr(\CF/U,\BasePoint)$ acting on $\omega_\BasePoint(\CF)$; see the explanation after Definition~\ref{FrobClass}. Recall that an $F$-isocrystal $\CF$ is called \emph{isoclinic} if for every $x\in|U|$ all the Newton slopes of $\CF$ at $x$ are equal. They then are independent of $x$. If $\CF$ is isoclinic of slope zero, it is called \emph{unit-root}.

\begin{prop}\label{PropUnitRoot}
Conjecture~\ref{MainConj} holds for convergent unit-root $F$-isocrystals.
\end{prop}
We will prove a more general statement later, but we think that the proof is rather instructive, and it is also a good motivation for our conjectures. Therefore, we decided to present its proof here.

\begin{proof}[Proof of Proposition~\ref{PropUnitRoot}] By a result of R.~Crew~\cite[Theorem~2.1 and Remark~2.2.4]{Crew87} (see Theorem~\ref{ThmCrewsThm} below) the full sub-category of $\FIsoc_\olK(U)$ consisting of unit-root $F$-isocrystals is tensor equivalent to the category of continuous representations of $\pi_1^\et(U,\bar\BasePoint)$ on finite dimensional $\olK$-vector spaces. Let $\rho\colon \pi_1^\et(U,\bar\BasePoint)\to\GL_r(\olK)$ be a representation corresponding to a unit-root $F$-isocrystal $\CF$. Then $\Gr(\CF/U)$ is a closed subgroup of $\GL_{r,\olK}$ and by Corollary~\ref{CorCrewsThm} below $\Gr(\CF/U)$ equals the Zariski-closure of the image of $\rho$. Moreover, for all points $x\in|U|$ the $\Gr(\CF/U)$-conjugacy classes of $\rho(x_*\Frob_x^{-1})$ and $\Frob_x(\CF)$ coincide, where $\Frob_x^{-1}\in\Gal(\ol\BF_x/\BF_x)$ is the geometric Frobenius at $x$ which maps $a\in\ol\BF_x$ to $a^{1/\#\BF_x}$.

To prove Conjecture~\ref{MainConj} let $S\subset|U|$ be a subset of upper Dirichlet density one. By the Chebotar\"ev Density Theorem \cite[Theorem~7]{Serre63} the Frobenius conjugacy classes for the points $x\in S$ are dense in $\pi_1^\et(U,\bar\BasePoint)$ with respect to the pro-finite topology. Since this topology is finer than the restriction of the Zariski-topology from $\Gr(\CF/U)$, the set $\bigcup_{x\in S}\Frob_x(\CF)$ is Zariski-dense in $\Gr(\CF/U)$.
\end{proof}

Note that Conjecture~\ref{MainConj} for convergent unit-root $F$-isocrystals on $U$ is considerably weaker than the classical Chebotar\"ev Density Theorem for $U$; see Remark~\ref{RemOtherTopologies} for more explanations. Let us next describe cases for which we prove the conjecture. In Section~\ref{SectIsoclinic} we use a theorem of Oesterl\'e~\cite{Oesterle} to strengthen Proposition~\ref{PropUnitRoot} to

\begin{thm}\label{ThmIsoclinic} 
Conjectures~\ref{MainConj} and \ref{MainConj3} hold true for direct sums of isoclinic convergent $F$-isocrystals.
\end{thm}

As important ingredient of its proof, we also show the following Theorem~\ref{ThmCheboForConstants}. The proofs of Theorems~\ref{ThmIsoclinic}, respectively \ref{ThmCheboForConstants}, are formulated in an abstract group theoretic way in Theorems~\ref{Thm4.3}, respectively \ref{TheoremMordellLang}, and proven by using $\ell$-adic Lie groups for $\ell=p$ in Theorem~\ref{Thm4.3}, respectively for an auxiliary prime $\ell$ in Theorem~\ref{TheoremMordellLang} that might be different from $p$.

\begin{thm}\label{ThmCheboForConstants} 
Let $\CF$ be a constant $F$-isocrystal on $U$ and $S\subset |U|$ be an infinite subset. Then the Zariski-closure of the set $\bigcup_{x\in S}\Frob_x(\CF)$ contains a connected component of $\Gr(\CF/U)$. In particular, Conjectures~\ref{MainConj} and \ref{MainConj3} hold true for $\CF$.
\end{thm}

We will also look at the analogous problem for \emph{overconvergent} $F$-isocrystals and give two different proofs of the following

\begin{thm}\label{ThmOverconvIntro}
The analogs of Conjectures~\ref{MainConj} and \ref{MainConj3} hold true for semi-simple overconvergent $F$-isocrystals.
\end{thm}

One proof, given in Section~\ref{Section11}, is inspired by the proof of Theorem~\ref{ThmIsoclinic}. But we base change the monodromy groups to the complex numbers and study maximal compact subgroups with methods of real algebraic geometry. And more importantly, we replace Oesterl\'e's result by the $p$-adic analog of Deligne's Equidistribution Theorem~\cite{DeligneWeil2}, which we formulate and prove in Theorem~\ref{ThmEquiDistr} using the theory of Frobenius weights for overconvergent $F$-isocrystals of Kedlaya, Abe and Caro. We reduce to this theorem using the purity part of the companion conjecture via Corollary~\ref{langlands_implies_mixedness}, which relies on the Langlands correspondence of Abe ($p$-adic) and Lafforgue ($\ell$-adic). More detailed outlines of the proof will be given at the end of Section~\ref{SectStrategy} and the beginning of Section~\ref{Section11}.

\bigskip

We now formulate our hardest result on Conjectures~\ref{MainConj} and \ref{MainConj3} for convergent $F$-isocrystals. Recall that by the specialization theorem of Grothendieck and Katz~\cite[Corollary~2.3.2]{Katz79} there is a non-empty open subscheme $f\colon V\into U$ on which the Newton slopes of $\CF$ are constant, and by the slope filtration theorem \cite[Corollary~2.6.3]{Katz79} the restriction $f^*\CF$ to $V$ has a slope filtration with isoclinic subquotients. Thus the semi-simplification $(f^*\CF)^{ss}$ is a direct sum of isoclinic convergent $F$-isocrystals to which Theorem~\ref{ThmIsoclinic} applies. Using this restriction $f^*\CF$, our main result says that Conjectures~\ref{MainConj} and \ref{MainConj3} hold true for a semi-simple $\CF\in\FIsoc_\olK(U)$ provided that every \emph{maximal quasi-torus} of $\Gr(f^*\CF/V)$ also is a maximal quasi-torus of $\Gr(\CF/U)$. (See Definition \ref{DefQuasiTorus} below for the notion of maximal quasi-tori, which is a good generalization of maximal tori in non-connected linear algebraic groups.) More precisely, we prove the following stronger characterization.

\begin{thm}\label{ThmWPinkImpliesChebotarevIntro}
Let $\CF\in\FIsoc_\olK(U)$ be semi-simple and let $f\colon V\into U$ be an open subscheme such that $f^*\CF$ has a slope filtration on $V$ with isoclinic subquotients. Consider the natural inclusion $\beta\colon \Gr(f^*\CF/V)\into\Gr(\CF/U)$ and the following statements:
\begin{enumerate}
\item\label{ThmWPinkImpliesChebotarevIntro_A}
Under $\beta$ every maximal \emph{torus} of $\Gr(f^*\CF/V)$ is also a maximal torus of the group $\Gr(\CF/U)$.
\item\label{ThmWPinkImpliesChebotarevIntro_B}
Under $\beta$ every maximal \emph{quasi-torus} of $\Gr(f^*\CF/V)$ is also a maximal quasi-torus of $\Gr(\CF/U)$.
\item\label{ThmWPinkImpliesChebotarevIntro_C}
Conjectures~\ref{MainConj} and \ref{MainConj3} hold true for any $\CG\oplus \CI$, where $\CG\in\dal\CF\dar$ and $\CI$ is a direct sum of isoclinic convergent $F$-isocrystals on $U$.
\item\label{ThmWPinkImpliesChebotarevIntro_D}
Conjecture~\ref{MainConj} holds true for $\CF$ and for the set $S=|V|$, which has (upper) Dirichlet density $1$.
\end{enumerate}
Then $\ref{ThmWPinkImpliesChebotarevIntro_B}\Rightarrow\ref{ThmWPinkImpliesChebotarevIntro_C}\Rightarrow\ref{ThmWPinkImpliesChebotarevIntro_D}\Rightarrow\ref{ThmWPinkImpliesChebotarevIntro_A}$. If in addition the maximal quasi-tori of $\Gr(f^*\CF/V)$ are abelian, then also $\ref{ThmWPinkImpliesChebotarevIntro_A}\Rightarrow\ref{ThmWPinkImpliesChebotarevIntro_B}$.
\end{thm}

The proof of the most interesting implication $\ref{ThmWPinkImpliesChebotarevIntro_B}\Rightarrow\ref{ThmWPinkImpliesChebotarevIntro_C}$ given on page~\pageref{DefFrobSS} consists of three main steps: by the theory of maximal quasi-tori, which we develop in Subsection~\ref{SubSectQuasiTori}, we first show that it is enough to prove Conjecture~\ref{MainConj3} for $\CF$ only. Next we apply Theorem~\ref{ThmIsoclinic} to the semi-simplification $(f^*\CF)^{ss}$. By assumption \ref{ThmWPinkImpliesChebotarevIntro_B}, respectively by Corollary~\ref{Cor6.10}, the images of a maximal quasi-torus $T\subset\Gr(f^*\CF/V)$ in $\Gr(\CF/U)$, respectively in $\Gr((f^*\CF)^{ss}/V)$, are maximal quasi-tori. In the third step, carried out in Subsection~\ref{SubSectConjMQT}, we show that for the reductive groups $\Gr(\CF/U)$ and $\Gr((f^*\CF)^{ss}/V)$ the density of Conjecture~\ref{MainConj3} can be tested on maximal quasi-tori. Therefore, it can be transferred from $\Gr((f^*\CF)^{ss}/V)$ to $\Gr(\CF/U)$ by a careful analysis of the different conjugacy classes in these two groups.

The arguments for the implication $\ref{ThmWPinkImpliesChebotarevIntro_B}\Rightarrow\ref{ThmWPinkImpliesChebotarevIntro_C}$ also give a second proof of Theorem~\ref{ThmOverconvIntro}, see Remark~\ref{RemPinkForOverconv}. As another result supplementing Theorem~\ref{ThmWPinkImpliesChebotarevIntro}, we show in Proposition~\ref{PropWeaklyPinkQuasiTorus} that the hypothesis~\ref{ThmWPinkImpliesChebotarevIntro_B} is a consequence of the following 

\begin{con}[Parabolicity Conjecture for convergent $F$-isocrystals] \label{ConjPink}
Let $\CF\in\FIsoc_\olK(U)$ and let $f\colon V\into U$ be an open subscheme. Then the natural inclusion $\Gr(f^*\CF/V)\into \Gr(\CF/U)$ identifies $\Gr(f^*\CF/V)$ with a parabolic subgroup of $\Gr(\CF/U)$.
\end{con}

Note that one of the distinguishing features of \emph{convergent} $F$-isocrystals is that their monodromy groups in general shrink when we shrink the underlying scheme $U$. This shows that $F$-isocrystals lack the type of rigidity which we have for $l$-adic and even $p$-adic Galois representations and \emph{overconvergent} $F$-isocrystals, where the monodromy group does not change when one shrinks the scheme. Conjecture~\ref{ConjPink} is the variant for convergent $F$-isocrystals of Crew's parabolicity conjecture. There is also an overconvergent variant:

\begin{thm}\label{ThmDAddezio}
Let $\CF^\dagger\in\FIsoc^\dagger_\olK(U)$ be an overconvergent $F$-isocrystal on $U$ and write $\CF=j_U(\CF^\dagger)\in\FIsoc_\olK(U)$ for the induced convergent $F$-isocrystal. Let $f\colon V\into U$ be a dense open subscheme such that $f^*\CF$ has a slope filtration. Then $\Gr(f^*\CF/V)$ is the stabilizer of the slope filtration in $\Gr(\CF/U)$ and $\Gr^\dagger(\CF^\dagger/U)$. Moreover, when $\CF^\dagger$ is semi-simple then  Conjecture~\ref{ConjPink} holds true for $\CF$ and any dense open subscheme of $U$.
\end{thm}

\begin{proof}
This was recently proven by D'Addezio~\cite[Theorem~1.1.1]{DAddezioParabol} with the ``Moreover'' part added by Cadoret and Tamagawa~\cite[Lemma~9.2.3]{CadTam}.
\end{proof}

As a consequence, Theorems~\ref{ThmWPinkImpliesChebotarevIntro} and \ref{ThmDAddezio} together with Proposition~\ref{PropWeaklyPinkQuasiTorus}\ref{PropWeaklyPinkQuasiTorus_D} imply the following

\begin{cor}
Conjectures~\ref{MainConj} and \ref{MainConj3} hold true for $\CG\oplus\CI$, where $\CG\in\dal\CF\dar$ for a semi-simple $\CF\in\FIsoc_\olK(U)$ having an overconvergent extension $\CF^\dagger$ as in Theorem~\ref{ThmDAddezio}, and $\CI$ is a direct sum of isoclinic convergent $F$-isocrystals on $U$.
\end{cor}

\begin{point}\label{PointWordAboutProofs} {\bfseries Some words about our proofs.}
Throughout, we use the following input from arithmetic geometry which is specific to the theory of $F$-isocrystals:
\renewcommand{\labelenumi}{(\arabic{enumi})}
\renewcommand{\theenumi}{(\alph{enumi})}
\begin{itemize}
\item\label{PointWordAboutProofs_1} The specialization theorem and the slope filtration theorem for convergent $F$-isocrystals of Grothendieck and Katz~\cite{Katz79}.
\item\label{PointWordAboutProofs_3} Crew's~\cite{Crew87} equivalence between convergent unit-root $F$-isocrystals and $p$-adic representations of $\pi_1^\et(U,\bar\BasePoint)$; see Theorem~\ref{ThmCrewsThm}.
\end{itemize}
Moreover, in the proof of Theorem~\ref{ThmOverconvIntro} in Section~\ref{Section11} we use the following deep results:
\begin{itemize}
\setcounter{enumi}{2}
\item\label{PointWordAboutProofs_4} Kedlaya's~\cite{Kedlaya04} full faithfulness theorem for the completion functor $j_U\colon\FIsoc^\dagger_\olK(U) \to \FIsoc_\olK(U)$; see Theorem~\ref{ThmKedlayaFF}.
\item Tsuzuki's extension theorem \cite[Theorem~1.3.1]{TsuzukiDuke} for overconvergent unit-root $F$-isocrystals.
\item\label{PointWordAboutProofs_5} That an overconvergent $F$-isocrystal with finite determinant is pure of weight zero; see Corollary~\ref{langlands_implies_mixedness}. This is a consequence of Abe’s $p$-adic Langlands correspondence~\cite{Abe13} and Lafforgue’s Ramanujan-Peterson conjecture~\cite{Lafforgue02}.
\item\label{PointWordAboutProofs_6} The theory of Frobenius weights for overconvergent $F$-isocrystals by Kedlaya~\cite{Kedlaya06}, Abe and Caro. It is used in the proof of our Equidistribution Theorem~\ref{ThmEquiDistr} by describing the holomorphy region of $L$-functions of overconvergent $F$-isocrystals in Lemma~\ref{LemmaClaimOnL-Fcn}. 
\end{itemize}
Other deep input from arithmetic geometry is
\begin{itemize}
\setcounter{enumi}{5}
\item\label{PointWordAboutProofs_7} Deligne's~\cite{DeligneWeil2} prime number theorem for $U$ and the holomorphy properties of its Zeta-function; see Theorem~\ref{ThmWeilBounds}.
\end{itemize}
The remaining arguments do not use that one deals with $F$-isocrystals. We employ various types of techniques:
\begin{enumerate}
\item\label{PointWordAboutProofs_A} 
easy Tannakian reduction steps, explained in Subsection~\ref{SubSectTannRed}. They allow to reduce to the case that $\CF=\CU\oplus\CC$ is a direct sum with $\CC$ constant and $\CU$ convergent unit-root, respectively overconvergent pure of weight zero. 
\item\label{PointWordAboutProofs_0}
an effective version of the classical Chebotar\"ev density theorem; see Theorem~\ref{useful_che}.
\item\label{PointWordAboutProofs_B}
point counting in $p$-adic Lie groups building on Oesterl\'e's result to prove Theorem~\ref{ThmIsoclinic} in Section~\ref{SectIsoclinic}. 
\item\label{PointWordAboutProofs_C}
measure theory. In Appendix~\ref{AppendixB} we develop the pullback of measures under nice maps between topological spaces and its compatibility with weak convergence. This is used to prove the Equidistribution Theorem~\ref{ThmEquiDistr}.
\item\label{PointWordAboutProofs_D}
complex algebraic groups, maximal compact subgroups, and real algebraic geometry to prove Theorem~\ref{ThmOverconvIntro}.
\item\label{PointWordAboutProofs_E}
a compactness argument proving convergence of complex hypersurfaces in Theorem~\ref{new_oesterle1} to prove Theorem~\ref{ThmOverconvIntro}.
\item\label{PointWordAboutProofs_F}
The large majority of our arguments and difficulties concern \emph{non-connected} (reductive) linear algebraic groups over algebraically closed fields. For this purpose, Section~\ref{SectMaxQuasiTori} develops the theory of maximal quasi-tori and proves transfer statements for density of conjugacy invariant subsets between the full group and a maximal quasi-torus.
\end{enumerate}
A difficult problem arises from the reduction step (1): The monodromy group of a direct sum is a fiber product $\Gr(\CU\oplus\CC/U)=\Gr(\CU/U) \times_{\Gr(\dal\CU\dar \cap \dal\CC\dar/U)} \Gr(\CC/U)$; see Proposition~\ref{PropGroupOfSum}\ref{PropGroupOfSum_B}. We have to prove the density of a conjugation invariant subset $C\subset \Gr(\CU\oplus\CC/U)$, while we know that the projections of $C$ are dense in $\Gr(\CU/U)$ and $\Gr(\CC/U)$ (e.g.~from Proposition~\ref{PropUnitRoot} and Theorem~\ref{ThmCheboForConstants}). This cannot be handled by group theory alone, and we use the techniques~(2), (3) for Theorem~\ref{ThmIsoclinic}, respectively (4), (5), (6) for Theorem~\ref{ThmOverconvIntro}.
\end{point}

\begin{point}\label{PointCadTam} {\bfseries Relation to the article \cite{CadTam} of Cadoret and Tamagawa.}
After a first version of our article was distributed on arXiv, Cadoret and Tamagawa~\cite{CadTam} generalized certain of our results using different techniques. More precisely, in addition to (over)convergent $F$-isocrystals $\CF$, they consider locally constant constructible Weil $\ol\BQ_\ell$-sheaves $\CF$ on $U$ with $\ell\ne p$ in the sense of \cite[(1.1)]{DeligneWeil2}, and locally constant constructible quasi-tame Weil $\ol\BQ_\Fu$-sheaves $\CF$ on $U$ for a non-principal ultra-filter $\Fu$ in the sense of \cite[\S\,3.6.2]{Cadoret19}. In \cite[Theorem~1.5]{CadTam}, they prove the analog of Conjectures~\ref{MainConj} and \ref{MainConj3} for these two types of coefficients $\CF$ and, in addition, our Theorem~\ref{ThmOverconvIntro} without the semi-simplicity hypothesis. Their arguments use the deep theory of companions and its by-product, the existence of weight filtrations, the Kazhdan-Larsen-Varshavsky reconstruction theorem for connected reductive groups~\cite{KLV14}, and go as follows: (1) For \'etale $\ol\BQ_\ell$-sheaves, i.e.~those corresponding to $\ol\BQ_\ell$-representations of $\pi_1^\et(U,\bar\BasePoint)$, the theorem is an easy consequence of the classical Chebotar\"ev Density Theorem as in our Proposition~\ref{PropUnitRoot}. (2) For general semi-simple $\CF$ they use the existence of an \'etale $\ol\BQ_\ell$-companion to reduce to (1). (3) From the semi-simple case they deduce the assertion for those $\CF$ which are direct sums of pure ones. (4) They show that the theorem for arbitrary $\CF$ follows from the case of direct sums of pure ones. Their remaining reduction steps rely on group-theoretic arguments. In particular, they develop the theory of quasi-Cartan subgroups which (as they write) was inspired by our theory of maximal quasi-tori and replaces it.

In contrast, our proofs in this article are ``elementary'' in the sense that they do not use automorphic techniques via the companion conjecture or the formalism of Frobenius weights, except for our proof of Theorem~\ref{ThmOverconvIntro} in Section~\ref{Section11} via equidistribution.

Without the theory of companions, but relying on our Theorem~\ref{ThmIsoclinic} and D'Addezio's Theorem~\ref{ThmDAddezio}, \cite[\S\,9.3]{CadTam} give another ``elementary'' proof of Conjectures~\ref{MainConj} and \ref{MainConj3} for overconvergent $F$-isocrystals, and a proof of Conjectures~\ref{MainConj} and \ref{MainConj3} for convergent $F$-isocrystals $\CF$ in the cases when $\CF$ has an overconvergent extension, or more generally, $\CF$ satisfies (a weakening of) Conjecture~\ref{ConjPink}. This generalizes the implication $\ref{ThmWPinkImpliesChebotarevIntro_B}\Rightarrow\ref{ThmWPinkImpliesChebotarevIntro_C}$ in our Theorem~\ref{ThmWPinkImpliesChebotarevIntro} to the non-semi-simple case by working with quasi-Cartan subgroups instead of maximal quasi-tori.

To summarize, there are currently three proofs of Conjectures~\ref{MainConj} and \ref{MainConj3} for overconvergent $F$-isocrystals:
\begin{itemize}
\item a proof in \cite{CadTam} via the companion conjecture based on automorphic techniques,
\item our proof (in the semi-simple case) of Theorem~\ref{ThmOverconvIntro} in Section~\ref{Section11} via equidistribution, using the purity part of the companion conjecture and the theory of Frobenius weights,
\item an ``elementary'' proof, without using companions, in our Remark~\ref{RemPinkForOverconv} (in the semi-simple case) and in \cite{CadTam} (in the general case) using our Theorem~\ref{ThmIsoclinic} and D'Addezio's Theorem~\ref{ThmDAddezio}.
\end{itemize}
For convergent $F$-isocrystals $\CF\in\FIsoc_\olK(U)$ Conjectures~\ref{MainConj} and \ref{MainConj3} are proven in the following cases:
\begin{itemize}
\item if $\CF$ is constant, by our Theorem~\ref{ThmCheboForConstants}, (In this case, $\CF$ is also overconvergent and Theorem~\ref{ThmCheboForConstants} provides a fourth, $\ell$-adic proof for overconvergent constant $F$-isocrystals, different from the three proofs mentioned above, although our argument could be viewed as the companion conjecture for $\Spec\BF_q$.) 
\item if $\CF$ is a direct sum of isoclinic convergent $F$-isocrystals, by our Theorem~\ref{ThmIsoclinic},
\item if $\CF$ satisfies the Parabolicity Conjecture~\ref{ConjPink} (or a weakening of it, namely assertion \ref{ThmWPinkImpliesChebotarevIntro_B} of Theorem~\ref{ThmWPinkImpliesChebotarevIntro} when $\CF$ is semi-simple, or \cite[Conjecture~9.2.4]{CadTam} for general $\CF$), by reduction to our Theorem~\ref{ThmIsoclinic},
\item if $\CF$ has an overconvergent extension, because it then satisfies the Parabolicity Conjecture~\ref{ConjPink} by Theorem~\ref{ThmDAddezio}.
\item if $\CF$ has a slope filtration on $U$ with isoclinic factors, because it then satisfies \cite[Conjecture~9.2.4]{CadTam}.
\end{itemize}
So for the time being, our Theorem~\ref{ThmIsoclinic} seems to be the necessary stepping stone for crystalline Chebotar\"ev density results on convergent $F$-isocrystals, and for ``elementary'' proofs for overconvergent $F$-isocrystals not using companions.
\end{point}

\begin{point}\label{PointSummary} {\bfseries Summary of the article.}
We describe the content of the individual sections. In Section~\ref{SectPrelim} we give the precise definition of the Frobenius conjugacy class $\Frob_x(\CF)$ and more details about (over)convergent $F$-isocrystals. We also review constant $F$-isocrystals here and explain the relation of $\Gr(\CF/U)$ with Crew's monodromy group \cite{Crew92}. Section~\ref{SectionUnitRootAndConnComp} recalls Crew's theory \cite{Crew87} of unit-root $F$-isocrystals in Subsection~\ref{SubSectionUnit}, and discusses the group of connected components of the monodromy group $\Gr(\CF/U)$ and the reduction of Conjecture~\ref{MainConj} to Conjecture~\ref{MainConj3} in Subsection~\ref{SubSectionComponents}. Also at the beginning of Subsection~\ref{SubSectionComponents}, we collect facts about the relation of the monodromy groups of an overconvergent $F$-isocrystal $\CF^\dagger\in \FIsoc^\dagger_\olK(U)$ and its convergent completion $j_U(\CF^\dagger)\in\FIsoc_\olK(U)$. In Section~\ref{SectDensity} we review the notion of (upper) Dirichlet density and prove an effective version of the classical Chebotar\"ev Density Theorem for function fields. In Sections~\ref{SectStrategy} and \ref{SectIsoclinic} we prove our Chebotar\"ev Density Conjectures for constant $F$-isocrystals (Theorem~\ref{ThmCheboForConstants}), respectively direct sums of isoclinic $F$-isocrystals (Theorem~\ref{ThmIsoclinic}). We also give an overview of the proofs of Theorems~\ref{ThmIsoclinic} and \ref{ThmOverconvIntro}, and describe the abstract Tannakian reduction techniques we use in Subsection~\ref{SubSectTannRed}. As a tool for the two last sections, we develop in Section~\ref{SectMaxQuasiTori} the theory of maximal quasi-tori and study in Subsection~\ref{SubSectConjMQT} the intersections of conjugacy classes with maximal quasi-tori. In Subsection~\ref{SubSectMaxCompact} we collect some results on maximal compact subgroups of complex algebraic groups. In Section~\ref{Section11} we treat the case of overconvergent $F$-isocrystals, and derive Theorem~\ref{ThmOverconvIntro} and the $p$-adic version of Deligne's Equidistribution Theorem in Theorem~\ref{ThmEquiDistr}. In the final Section~\ref{SectPink} we formulate properties of the closed subgroup $\Gr(f^*\CF/V)\subset\Gr(\CF/U)$ which one might expect when one restricts a convergent $F$-isocrystal on $U$ to an open subscheme $f\colon V\into U$ and we prove Theorem~\ref{ThmWPinkImpliesChebotarevIntro} saying that these properties imply our Chebotar\"ev density conjectures. We also give a second, ``elementary'' proof of Theorem~\ref{ThmOverconvIntro} there.
\end{point}

\begin{notn}\label{InitialNotation}
Throughout we denote topological ($\ell$-adic or complex analytic) groups by bold face letters like $\BG,\BH,\BK$ and algebraic groups over an algebraically closed field $L$ of characteristic zero by roman letters like $G,H,P,Q,T$. We do not assume that any of these groups are connected. For such an algebraic group $G$ we identify $G$ with its group $G(L)$ of rational points. Furthermore, we let $G\open$ be its identity component, $R_u G$ the unipotent radical, $r_G\colon G\onto G^\red:=G/R_u G$ the surjection onto the maximal reductive quotient. Note that $G/G\open\isoto G^\red/(G^\red)\open$, because $R_u G=R_u G\open$ is connected in characteristic zero, see Lemma~\ref{LemmaUnipotConnected}\ref{LemmaUnipotConnected_D}. If $\phi\colon G\onto H$ is a surjection of linear algebraic groups, then there is an induced surjection $\phi^\red\colon G^\red\onto H^\red$ between the reductions, because the image of $R_uG$ is a closed connected unipotent normal subgroup, and hence contained in $R_u H$. We say that $G$ is \emph{reductive} if $G\open$ is.

For an algebraic subgroup $H\subset G$ and an element $g\in G$, we let $N_G(H)$ be the normalizer, $Z_G(H)$ the centralizer, $H^g:=Z_{H\open}(g):=\{h\in H\open\colon gh=hg\}$ the fixed points of conjugation by $g$ on $H\open$, and $H^g{}\open:=(H^g)\open$. (Note that we abstain from the more correct notation $H^{\circ g}$, as it leads to $H^{\circ g\circ}$). For a closed subgroup $H\subset G$ and a set $C\subset G$ let ${}^{H\!}C=\bigcup\limits_{g\in H}gCg^{-1}$ be the union of the $H$-conjugacy classes of elements of $C$. Clearly the map $C\mapsto {}^{H\!}C$ on subsets of $G$ preserves inclusions and satisfies ${}^{H\!}({}^{H\!}C)={}^{H\!}C$.
\end{notn}

We will frequently use the following well known facts about algebraic groups.

\begin{lemma}\label{LemmaUnipotConnected}
Let $G$ be a linear algebraic group over an algebraically closed field $L$.
\begin{enumerate}
\item \label{LemmaUnipotConnected_A}
If $\phi\colon G\to H$ is a homomorphism of algebraic groups then $\phi(G\open)=\phi(G)\open$ by \cite[I.1.4~Corollary]{Borel91}.
\item \label{LemmaUnipotConnected_B}
If $G\open$ is reductive and $Z$ is its center, then the group $G\open/Z$ is semi-simple with trivial center and $G\open/[G\open,G\open]$ is a torus by \cite[IV.11.21~Proposition, IV.14.11~Corollary and III.10.6~Theorem]{Borel91}
\end{enumerate}
Now let $L$ have characteristic zero.
\begin{enumerate}
\setcounter{enumi}{2}
\item \label{LemmaUnipotConnected_C}
Let $g\in G$ and let $n$ be a positive integer such that $g^n$ is semi-simple, then $g$ is semi-simple. 
\item \label{LemmaUnipotConnected_D}
All unipotent elements of $G$ are contained in $G\open$. In particular, all unipotent groups in characteristic zero are connected.
\end{enumerate}
\end{lemma}

\begin{proof}
\ref{LemmaUnipotConnected_B}
The center of $G\open/Z$ is trivial, because for any central element $\bar z\in G\open/Z$ and a preimage $z\in G\open$ of $\bar z$ the map $G\open\to Z\cap[G\open,G\open],\ g\mapsto gzg^{-1}z^{-1}$ factors through the identity component of $Z\cap[G\open,G\open]$ which is trivial by \cite[IV.14.2~Proposition]{Borel91}. Thus $z\in Z$ and $\bar z=1$.
  
\medskip\noindent
\ref{LemmaUnipotConnected_C}
If $g=g_sg_u$ is the multiplicative Jordan decomposition, where $g_s, g_u\in G$ are the semi-simple and unipotent parts of $g$, respectively, then $g^n=g_s^n g_u^n$ is the multiplicative Jordan decomposition of $g^n$. Consider a faithful representation $G\subset\GL_r$. Then $g_u$ is conjugate in $\GL_r$ to a unipotent upper triangular matrix (use \cite[IV.11.10~Theorem]{Borel91}), and hence if $g_u\ne 1$ it has infinite order as $\charakt(L)=0$. Therefore, $g^n$ is semi-simple, that is $g_u^n=1$, if and only if $g_u=1$ and $g$ is semi-simple. 

\medskip\noindent
\ref{LemmaUnipotConnected_D}
If $g\in G$ is a unipotent element, then its image in $G/G\open$ is of finite order and unipotent by \cite[I.4.4~Theorem]{Borel91}, whence trivial by \ref{LemmaUnipotConnected_C}.
\end{proof}

\bigskip\noindent
{\bfseries Acknowledgement.} We thank Friedrich Knop for providing a proof of Theorem~\ref{ThmQuasiTorusReductive}\ref{ThmQuasiTorusReductive_D} and Zakhar Kabluchko and Linus Kramer for some advice on measure theory and complex Lie groups, respectively. We thank the anonymous referees for their careful reading and many helpful suggestions to improve this article. We are also thankful for the support received through the EPSRC grants P19164 and P36794, and by the German Science Foundation (DFG) in form of SFB 878, Project-ID 427320536 -- SFB 1442, and Germany's Excellence Strategy EXC 2044--390685587 ``Mathematics M\"unster: Dynamics--Geometry--Structure''.

\section{Preliminaries on (Over)Convergent and Constant $F$-Isocrystals} \label{SectPrelim}

\subsection{Basic Definitions} \label{SubSectDefConj}

We recall the basic setup from Section~\ref{SectIntroduction} and give further details. For every $n\in\BN$ let $U_n:=U\otimes_{\BF_q}\BF_{q^n}$, let $K_n\subset\olK$ be the unramified extension of $K$ of degree $n$, and let $F$ be the Frobenius of $K_n$ over $K$. Then $F^n$ is the identity on $K_n$. Let $\FnIsoc_{K_n}(U_n)$ denote the $K_n$-linear rigid tensor category of $K_n$-linear convergent $F^n$-isocrystals on $U_n$; see \cite[Chapter~1]{Crew92} for details. We simply write $\FIsoc_K(U)$ for $F^1\textrm{-Isoc}_K(U)$. If $\CF$ is an object of $\FnIsoc_{K_n}(U_n)$ we let $\dal\CF\dar$ denote the Tannakian sub-category of $\FnIsoc_{K_n}(U_n)$ generated by $\CF$. There is a functor
\begin{equation}\label{EqFn}
(\,.\,)^{(n)}\colon \FIsoc_K(U)\longrightarrow \FnIsoc_{K_n}(U_n),\quad\CF\mapsto\CF^{(n)}
\end{equation}
given by pulling back under $U_n\to U$, that is, by tensoring the coefficients from $K$ to $K_n$, and replacing the Frobenius $F_\CF$ of $\CF$ by $F_\CF^n:=F_\CF\circ{\rm Fr}_{q,U}^*F_\CF\circ\ldots\circ{\rm Fr}_{q,U}^{(n-1)*}F_\CF$.

We fix a base point $\BasePoint\in U(\BF_{q^\BasePtDeg})=U_\BasePtDeg(\BF_{q^\BasePtDeg})$ with $\BasePtDeg=\deg(\BasePoint)\in\BN$. The pullback $\BasePoint^*\CF$ of any $\CF\in\FuIsoc_\olK(U)$ is a $K_\BasePtDeg$-linear $F^\BasePtDeg$-isocrystal on $\Spec\BF_{q^\BasePtDeg}$, that is, a finite dimensional $K_\BasePtDeg$-vector space together with a $K_\BasePtDeg$-linear automorphism coming from the Frobenius $F^\BasePtDeg$. Forgetting this automorphism yields the fiber functor $\omega_\BasePoint$ and makes $\FuIsoc_{K_\BasePtDeg}(U_\BasePtDeg)$ into a neutral Tannakian category. On the category $\FIsoc_K(U)$ we can still consider the fiber functor $\omega_\BasePoint\colon\CF\mapsto\BasePoint^*\CF$ to $K_\BasePtDeg$-vector spaces. It factors through the functor $(\,.\,)^{(\BasePtDeg)}$ from \eqref{EqFn} as $\omega_\BasePoint=\BasePoint^*\circ (\,.\,)^{(\BasePtDeg)}$. This makes $\FIsoc_K(U)$ into a $K$-linear Tannakian category, but $\omega_\BasePoint$ is non-neutral when $\BasePtDeg>1$.

To remedy this we consider for every $n$ the extension of scalars $\FnIsoc_\olK(U_n):=\FnIsoc_{K_n}(U_n)\otimes_{K_n} \olK$ which is a $\olK$-linear Tannakian category; see \cite[\S\,7.3.3]{AbeMarmora}, \cite[\S\,4.1]{Abe13} or more generally \cite[p.~155ff]{Deligne-Milne82}. This scalar extension comes with a tensor functor \mbox{${}_\bullet\otimes_{K_n} \olK\colon\FnIsoc_{K_n}(U_n)\to \FnIsoc_\olK(U_n)$}, $\CF\mapsto \CF\otimes_{K_n} \olK$, which satisfies 
\begin{equation*}
\Hom_{\FnIsoc_{K_n}(U_n)}(\CF,\CG)\otimes_{K_n} \olK\;=\;\Hom_{\FnIsoc_\olK(U_n)}(\CF\otimes_{K_n} \olK, \CG\otimes_{K_n} \olK)\,. 
\end{equation*}
Again we write $\FIsoc_\olK(U):=F^1\textrm{-Isoc}_\olK(U)$. Via the embedding $K_n\subset\olK$ we obtain a homomorphism $K_n\otimes_K \olK\to \olK$. Tensoring with this homomorphism, the functor $(\,.\,)^{(n)}$ described above induces a functor
\begin{equation}\label{EqFnKbar}
(\,.\,)^{(n)}\colon \FIsoc_\olK(U)\longrightarrow \FnIsoc_\olK(U_n),\quad\CF\mapsto\CF^{(n)}
\end{equation}
given by pulling back under $U_n\to U$ and replacing the Frobenius $F_\CF$ of $\CF$ by $F_\CF^n$. If $\omega$ is an $L$-rational fiber functor on $\FnIsoc_{K_n}(U_n)$ for some intermediate field $K_n\subset L\subset\olK$ then $\omega$ extends canonically to a $\olK$-rational, hence neutral fiber functor on $\FnIsoc_\olK(U_n)$. In particular, the base point $\BasePoint\in U(\BF_{q^\BasePtDeg})$ yields the neutral fiber functor
\[
\omega_\BasePoint=\BasePoint^*\circ (\,.\,)^{(\BasePtDeg)}\otimes\id_\olK\colon \FIsoc_\olK(U)\;\longto\;(\olK\text{-vector spaces})\,,\quad \CF\;\longmapsto\; \BasePoint^*\CF \otimes_{K_\BasePtDeg\otimes\olK} \olK
\]
and makes $\FIsoc_\olK(U)$ into a neutral Tannakian category over $\olK$. For $\CF\in \FIsoc_\olK(U)$ let $\Gr(\CF/U,\BasePoint):=\Aut^\otimes(\omega_\BasePoint|_{\dal\CF\dar})$ denote the \emph{monodromy group of $\CF$}, that is, the Tannakian fundamental group of $\dal\CF\dar$ with respect to the fiber functor $\omega_\BasePoint$; see \cite[Theorem~2.11]{Deligne-Milne82}. By \cite[Proposition~2.20(b)]{Deligne-Milne82} it is a linear algebraic group over $\olK$ and $\dal\CF\dar$ is tensor equivalent to the category of $\olK$-rational representations of $\Gr(\CF/U,\BasePoint)$. By Cartier's theorem all linear algebraic groups over a field of characteristic zero are smooth; see for example \cite[\S\,11.4]{Waterhouse79}. We explain the relation of $\Gr(\CF/U,\BasePoint)$ with Crew's~\cite{Crew92} monodromy group $\DGal(\CF)$ in Remark~\ref{Rem3.6} and Proposition~\ref{crew_sequence}. Since $\omega_\BasePoint=\omega_\BasePoint\circ (\,.\,)^{(\BasePtDeg)}$ the tensor functor $(\,.\,)^{(\BasePtDeg)}$ from \eqref{EqFnKbar} induces a homomorphism of linear algebraic groups over $\olK$
\[
h_\BasePtDeg(\CF)\colon\Gr(\CF^{(\BasePtDeg)}/U_\BasePtDeg,\BasePoint)\longrightarrow \Gr(\CF/U,\BasePoint)\,,
\]
which we study further in Proposition~\ref{Prop1.16}.

\begin{rem}\label{RemIgnoreBasePoint}
The group $\Gr(\CF/U,\BasePoint)$ is independent of the base point $\BasePoint$ up to conjugation in the following sense. Let $\BasePoint'\in U(\BF_{q^{\BasePtDeg'}})$ be another base point. By \cite[Theorem~3.2]{Deligne-Milne82} there is a (non-canonical) isomorphism of fiber functors $\alpha=\alpha_{\BasePoint',\BasePoint}\colon\omega_{\BasePoint'}\isoto\omega_\BasePoint$ over $\olK$. Every other isomorphism differs from $\alpha$ by composition with an element $g\in \Aut^\otimes(\omega_\BasePoint)=\Gr(\CF/U,\BasePoint)$. The isomorphism $\alpha$ induces an isomorphism of algebraic groups $\alpha_*\colon\Gr(\CF/U,\BasePoint')\isoto\Gr(\CF/U,\BasePoint)$ over $\olK$ and the isomorphism $(g\circ\alpha)_*$ induced by $g\circ\alpha$ differs from $\alpha_*$ by conjugation with $g$. In this way we may move the base point whenever it is convenient. We will thus drop the base point from the notation. Furthermore, if no confusion can arise we will drop the base scheme $U$ from the notation, and hence often simply write $\Gr(\CF):=\Gr(\CF/U,\BasePoint)$.
\end{rem}

\begin{rem}
The issue that the fiber functor $\omega_\BasePoint\colon\FIsoc_K(U) \to (K_\BasePtDeg\text{-vector spaces})$ is non-neutral and one needs the scalar extensions $\FIsoc_\olK(U)$ to $\olK$ cannot be avoided without restricting the generality. One could reformulate the scalar extension equivalently in terms of the non-neutral Tannakian category $\FIsoc_K(U)$ with the fiber functor $\omega_\BasePoint\colon \FIsoc_K(U)\to(K_\BasePtDeg\text{-vector spaces})$. This category is tensor equivalent to the $K_\BasePtDeg$-rational representations of a $K_\BasePtDeg/K$-groupoid. The latter point of view was taken in the second version on \href{http://arxiv.org/abs/1811.07084v2}{arXiv:1811.07084v2} of this article. It is also taken in Crew's seminal work \cite[\S\,2.2]{Crew92} by the concept of ``Frobenius structure'' on the group $\Gr(\CF)$ and is inherent to the monodromy groups of $F$-isocrystals. It would suffice to extend scalars to $K_\BasePtDeg$ instead of $\olK$, or to the maximal unramified extension $K^\un\subset \olK$. The latter also allows to change the base point $\BasePoint$ to a point $\BasePoint'$ with $\BasePtDeg':=\deg(\BasePoint') \ne \deg(\BasePoint) = \BasePtDeg$.

The readers who want to avoid the scalar extension from $K$ to $\olK$ (or $K_\BasePtDeg$ or $K^\un$) may assume that $U$ has an $\BF_q$-rational point, thus assume $\BasePtDeg=1$ in which case $\FIsoc_K(U)$ is neutral. However, they should note that one of our techniques is to pass to open subschemes $V\subset U$ which might lack rational points. If $U(\BF_q)=\emptyset$, the Tannakian category $\FIsoc_K(U)$ cannot be made neutral by naively enlarging the coefficient field $K$ to some $\wt K$, because the new category will still be linear only over the fixed field of the Frobenius $F$ in $\wt K$. It can be made neutral by enlarging the constant field $\BF_q$ of the scheme $U$ to $\BF_{q^\BasePtDeg}$ via the functor $(\,.\,)^{(\BasePtDeg)}$ from \eqref{EqFn}. But this will kill part of the monodromy; see Proposition~\ref{Prop1.16}. In any case there is no easy way out of this dilemma.
\end{rem}

To introduce the Frobenius conjugacy classes let $\CF\in\FIsoc_\olK(U)$ and let $x\in|U|$ be a closed point with residue field $\BF_{q^n}$. Choose a point $y\in U(\BF_{q^{n}})$ above $x$ and use $\BasePoint=y$ as base point. Since $y=({\rm Fr}_{q,U})^n\circ y$ as morphisms $\Spec\BF_{q^n}\to U$, the Frobenius $F^n_\CF\colon({\rm Fr}_{q,U})^{n*}\CF^{(n)}\isoto\CF^{(n)}$ of the $F^n$-isocrystal $\CF^{(n)}$ induces an automorphism $y^*F^n_\CF$ of the fiber functor $\omega_y$, that is an element of $\Gr(\CF^{(n)}/U_n)$. We denote by $\Frob_y(\CF)$ the $\Gr(\CF/U)$-conjugacy class of its image $h_n(\CF)(y^*F_\CF^n)$ under $h_n(\CF)$ in $\Gr(\CF/U)$. We claim that $\Frob_y(\CF)$ only depends on the closed point $x$ of $U$ lying below $y$. Indeed, there is a point $\tilde y\in U(\BF_{q^n})$ above $x$ with ${\rm Fr}_{q,U}\circ \tilde y=y$ and $\tilde y^*{\rm Fr}_{q,U}^*\CF=y^*\CF$. The isomorphism $F_\CF\colon{\rm Fr}_{q,U}^*\CF\isoto\CF$ of the $F$-isocrystal $\CF$ induces an isomorphism $\tilde y^*F_\CF\colon y^*\CF\isoto\tilde  y^*\CF$ under which the $K_n$-linear automorphisms $F_\CF^n$ on the fibers at $y$ and $\tilde y$ are mapped to each other. So $\Frob_{\tilde y}(\CF)$ and $\Frob_y(\CF)$ yield the same conjugacy class under the change of base point isomorphism $\alpha_{\tilde y,y}$ from Remark~\ref{RemIgnoreBasePoint}. We denote this conjugacy class by $\Frob_x(\CF)$.

\begin{defn}\label{FrobClass}
The subset $\Frob_x(\CF)\subset\Gr(\CF/U)$ defined above is called the \emph{Frobenius conjugacy class} of $\CF$ at the closed point $x\in |U|$. For an overconvergent $F$-isocrystal $\CF\in\FIsoc^\dagger_\olK(U)$ we similarly define $\Frob^\dagger_x(\CF)\subset\Gr^\dagger(\CF/U)$.
\end{defn}

Note that the $\olK$-linear automorphism $y^*F^n_\CF$ of $\omega_y(\CF)$ from the previous paragraph is the Frobenius of the $F^n$-isocrystal $y^*\CF^{(n)}$ over $y=\Spec\BF_{q^n}$. As explained by Katz~\cite[\S\,1.3]{Katz79}, the Newton slopes of $\CF$ at $x$ are equal to $\tfrac{1}{n}$ times the Newton slopes of $y^*\CF^{(n)}$, that is, equal to $\tfrac{1}{n}$ times the $p$-adic valuations of the eigenvalues of the automorphism $y^*F^n_\CF$ of $\omega_y(\CF)$, and hence equal to $\tfrac{1}{n}$ times the $p$-adic valuations of the common eigenvalues of the elements of $\Frob_x(\CF)\subset\Gr(\CF/U,\BasePoint)$ acting on $\omega_\BasePoint(\CF)$, no matter which base point $\BasePoint$ we choose.

\subsection{Constant $F$-Isocrystals} \label{SubSectConst}

In order to explain the relation of $\Gr(\CF/U)$ with Crew's monodromy group \cite{Crew92} we discuss constant $F$-isocrystals. We begin with the following

\begin{defn}\label{DefAlgEnvelope}
For any topological group $\BG$ and any topological field $L$ we let $\Rep^\cont_{L} \BG$ be the neutral Tannakian category of continuous representations of $\BG$ on finite dimensional $L$-vector spaces equipped with the forgetful fiber functor $\omega_f$ which sends a representation $\rho\colon \BG\to\Aut_{L}(W)$ to the $L$-vector space $W$. We define the \emph{$L$-linear algebraic envelope} $\BG^{L\text{\rm-alg}}$ of $\BG$ as the Tannakian fundamental group $\Aut^\otimes(\omega_f)$ of $\Rep^\cont_{L} \BG$. By Tannakian duality \cite[Theorem~2.11]{Deligne-Milne82}, $\omega_f$ induces a tensor equivalence between $\Rep^\cont_L \BG$ and the category $\Rep_L \BG^{L\text{\rm-alg}}$ of algebraic representations of $\BG^{L\text{\rm-alg}}$.
\end{defn}

We recall the following lemma; compare also with Saavedra~\cite[Chapter~V.0.3.1]{Saavedra72} and \cite[Example~(2.33)]{Deligne-Milne82}.

\begin{lemma}[{\cite[b) on page~66]{SerreGebres}}] \label{LemmaAlgEnvelope}
There is a natural continuous homomorphism $\BG\to\Aut^\otimes(\omega_f)(L)$ with dense image. In particular, for every continuous finite-dimensional $L$-linear representation $\rho\colon \BG\to\GL_{n,L}$ of $\BG$, the monodromy group $\Aut^\otimes(\omega_f|_{\dal\rho\dar})$ of $\rho$ is canonically isomorphic to the closure of the image $\rho(\BG)\subset\GL_n(L)$ of $\BG$. Therefore, we may describe $\BG^{L\text{\rm-alg}}$ as the limit of the closures of the images $\im(\rho)$ over the diagram of all continuous finite-dimensional $L$-linear representations $\rho$ of $\BG$ (in some suitably large universe).
\end{lemma}

\begin{example} \label{ExAlgEnvOfZ}
When $L=\olK$ then the $L$-linear algebraic envelope of $\BZ$ is $(\BG_{m,L}^{\kappa}\times_L\BG_{a,L})\times\wh\BZ$ where $\kappa$ is the cardinality of $L$. We leave the verification of this fact to the reader. What we only need is Theorem~\ref{ThmMonodrOfConstant} below.
\end{example}

We turn towards constant $F$-isocrystals and recall their

\begin{defn}\label{DefConstantFIsoc}
An $F$-isocrystal with $K$-coefficients on $\Spec\BF_q$ is by definition a pair $(W,f)$ consisting of a finite dimensional $K$-vector space $W$ together with a $K$-linear automorphism $f\in\Aut_K(W)$, its \emph{Frobenius}. The category $\FIsoc_K(\BF_q)$ is tensor equivalent to the category $\Rep_K^\cont\BZ$, where $\BZ$ carries the discrete topology, by sending a representation $\rho\colon\BZ\to\GL_r(K)$ to $(K^{\oplus r},\rho(1))\in\FIsoc_K(\BF_q)$. The scalar extension to $\olK$ is the category $\FIsoc_\olK(\BF_q)$ whose objects are pairs $(W,f)$ consisting of a finite dimensional $\olK$-vector space $W$ together with a $\olK$-linear Frobenius automorphism $f\in\Aut_\olK(W)$. This category is equivalent to $\Rep_\olK^\cont\BZ$.

Any object $(W,f)$ of $\FIsoc_\olK(\BF_q)$ can be pulled back under the structure morphism $\pi\colon U\to\Spec\BF_q$ to an (over)convergent $F$-isocrystal $\pi^*(W,f)$ on $U$, and any $F$-isocrystal on $U$ arising in this way is called \emph{constant}. We denote by $\FConst_\olK(U)$ the category of $\olK$-linear constant $F$-isocrystals on $U$. An $F$-isocrystal is constant if and only if it is trivial as an isocrystal, i.e.~it is generated by its horizontal sections. In particular, an $F$-isocrystal $\CF\in\FIsoc_\olK(U)$ is constant if and only if the associated $F^n$-isocrystal $(\CF)^{(n)}$ is constant for some (any) $n$.

\end{defn}

\begin{thm}\label{ThmMonodrOfConstant}
Let $\CC=\pi^*(W,f)\in \FIsoc_\olK(U)$ be a constant $F$-isocrystal. Then the following holds.
\begin{enumerate}
\item \label{ThmMonodrOfConstant_A}
The functor $\pi^*$ induces a tensor equivalence between the Tannakian sub-categories $\dal(W,f)\dar\subset\FIsoc_\olK(\BF_q)$ and $\dal\CC\dar\subset\FIsoc_\olK(U)$. In particular, every $\CF\in\dal\CC\dar$ is constant.
\item \label{ThmMonodrOfConstant_B}
The monodromy group $\Gr(\CC/U)$ of $\CC$ is isomorphic to $\Gr((W,f)/\BF_q)$. This group equals the closure of $f^\BZ$ in $\Aut_\olK(W)$. In particular, it is commutative and isomorphic to $T\TimesQQQ\BG_{a,\olK}^\epsilon$ where $T$ is the product of the finite abelian group $T/T\open$ by the torus $T\open$ and $\epsilon=0$ or $1$.
\item \label{ThmMonodrOfConstant_C}
For every $x\in |U|$ the set $\Frob_x(\CC)$ consists of the single element $f^{\deg(x)}$. 
\item \label{ThmMonodrOfConstant_E}
The categories $\FIsoc_\olK(\BF_q)$ and $\FConst_\olK(U)$ are tensor equivalent to the category $\Rep^\cont_\olK\BZ$, where $\BZ$ carries the discrete topology. Their Tannakian fundamental groups $\pi_1^{\FIsoc}(\BF_q)$ and $\pi_1^{\FConst}(U)$ are equal to the $\olK$-linear algebraic envelope $\BZ^{\olK\text{\rm-alg}}$ of $\BZ$.
\end{enumerate}
\end{thm}

\begin{proof}
\ref{ThmMonodrOfConstant_A} 
Since $\pi^*$ is a tensor functor we only need to see that every subobject of $\pi^*(W,f)$ is constant. This is trivial if $U$ has a rational point, because we can pull back to this point and use that this functor is fully faithful. The general case follows by applying the previous case to $F^n$-isocrystals, and then use that the image of the subobject actually must be invariant under $F$, too, by its uniqueness. 

\medskip\noindent
\ref{ThmMonodrOfConstant_B}
The isomorphy of monodromy groups follows from \ref{ThmMonodrOfConstant_A}. Since the category of $F$-isocrystals over $\BF_q$ is just the category of linear representations of $\BZ$, the second claim follows from Lemma~\ref{LemmaAlgEnvelope}. Since the monodromy group is the closure of $f^\BZ$, it is commutative and is the direct product of its unipotent radical by the subgroup consisting of its semi-simple elements, see \cite[I.4.7~Theorem]{Borel91}. The latter is the product of a finite \'etale abelian group scheme over $\olK$ by a torus, see \cite[III.8.7~Proposition]{Borel91}, and the unipotent radical is isomorphic to $\BG_{a,\olK}^\epsilon$ by \cite[Chapter~VIII, \S\,2.7, Corollary]{Serre88}. Since the image of $f$ must be dense in the quotient $\BG_{a,\olK}^\epsilon$, we have $\epsilon=0$ or $\epsilon=1$.

\medskip\noindent
\ref{ThmMonodrOfConstant_C} follows from the definition of $\Frob_x(\CC)$ from (before) Definition~\ref{FrobClass}. The set $\Frob_x(\CC)$ consists of only one element $f^{\deg(x)}$ here, because $\Gr(\CC/U)$ is commutative.

\medskip\noindent
\ref{ThmMonodrOfConstant_E}
The tensor equivalences follow from \ref{ThmMonodrOfConstant_A} and Definition~\ref{DefConstantFIsoc}. The last assertion follows from Lemma~\ref{LemmaAlgEnvelope}.
\end{proof}

\begin{defn}\label{Def3.5}
For every $\CF\in\FIsoc_\olK(U)$ let $\dal\CF\dar_{const}$ be the full Tannakian sub-category of constant $F$-isocrystals in $\dal\CF\dar$ and let $\mathbf{W}(\CF)$ denote the Tannakian fundamental group of $\dal\CF\dar_{const}$ with respect to the fiber functor $\omega_\BasePoint$. Let $\beta\colon\Gr(\CF/U)\onto\mathbf{W}(\CF)$ be the homomorphism induced by the inclusion $\dal\CF\dar_{const}\subset\dal\CF\dar$; see Lemma~\ref{LemmaCompatibility}\ref{LemmaCompatibility_A} below. We call the kernel
\[
\Gr(\CF/U)^{\rm geo} \;:=\;\ker\bigl(\Gr(\CF/U) \onto\mathbf{W}(\CF)\bigr)
\]
the \emph{geometric monodromy group} of $\CF$. This terminology is motivated by Corollary~\ref{CorCrewsThmGeo} below.
\end{defn}

\begin{rem} \label{Rem3.6}
Let $\Isoc_K(U)$ be the category of $K$-linear convergent isocrystals on $U$ and let $\Isoc_\olK(U):=\Isoc_K(U)\otimes_K \olK$ be its scalar extension. For every $\CF\in\FIsoc_\olK(U)$ let $\CF^{\sim}$ denote the underlying convergent isocrystal and let $\dal\CF^{\sim}\dar$ denote the Tannakian sub-category generated by $\CF^{\sim}$ in $\Isoc_\olK(U)$. For $\CF\in\FIsoc_\olK(U)$ R.~Crew~\cite{Crew92} has defined and studied the monodromy group $\DGal(\CF)$ of the neutral Tannakian category $\dal\CF^\sim\dar$ with respect to the fiber functor $\omega_\BasePoint$. If $\BasePoint\in U(\BF_q)$ he also defined the \emph{Weil group} $W^\CF(U/\olK)=\DGal(\CF)\rtimes\BZ$ as the semi-direct product, where $1\in\BZ$ operates on $\DGal(\CF)$ by conjugation with the Frobenius $\BasePoint^*F_\CF$.

Let $\alpha\colon \DGal(\CF)\to\Gr(\CF/U)$ be the homomorphism induced by the forgetful functor $(\,.\,)^{\sim}\colon \dal\CF\dar\to\dal\CF^{\sim}\dar$. It is natural to expect that in our setting $\DGal(\CF)$ plays the role of the geometric monodromy group from Definition~\ref{Def3.5}. However, we are only able to prove this in the semi-simple case.
\end{rem}

\begin{prop}\label{crew_sequence}
Assume that $\CF^{\sim}$ is semi-simple. Then $\DGal(\CF)$ is canonically isomorphic to the geometric monodromy group $\Gr(\CF/U)^{\rm geo}$, that is, the bottom row in the following diagram is exact. Moreover, if $\BasePoint\in U(\BF_q)$, then the bottom row can be canonically extended to a diagram with exact rows
\[
\xymatrix { 0 \ar[r] & \DGal(\CF) \ar[r]\ar@{=}[d] & W^\CF(U/\olK) \ar[r] \ar[d] & \BZ \ar[r] \ar[d] & **{!L(1) =<1.5pc,1.5pc>} \objectbox{0\,\;}\\
0 \ar[r] & \DGal(\CF) \ar[r]^\alpha & \Gr(\CF/U) \ar[r]^\beta &  \mathbf{W}(\CF) \ar[r] & **{!L(1) =<1.5pc,1.5pc>} \objectbox{0\,.}
}
\]
\end{prop}

The proof of Proposition~\ref{crew_sequence} will use the following criterion:

\begin{thm} \label{ThmLP}
Let
\begin{equation}\label{3.1.2}
\xymatrix { G_1\ar[r]^{q} & G_2\ar[r]^{p} & G_3 \ar[r] & 1}
\end{equation}
be a sequence of affine group schemes over a field $L$ such that $p$ is faithfully flat. Consider the conditions:
\begin{enumerate}
\item \label{ThmLP_A}
if $V\in\Rep_L G_2$, then $q^*(V)$ is trivial in $\Rep_L G_1$ if and only if $V\cong p^*(V')$ for some $V'\in\Rep_L G_3$,
\item \label{ThmLP_B}
for any $V\in\Rep_L G_2$, if $W_0\subset q^*(V)$ is the maximal trivial subobject in $\Rep_L G_1$, then there exists a subobject $V_0\subset V$ in $\Rep_L G_2$ such that $q^*(V_0)=W_0\subset q^*(V)$,
\item \label{ThmLP_C} any $W\in\Rep_L G_1$ is a subobject (or a quotient) of $q^*(V)$ for some $V\in\Rep_L G_2$. 
\end{enumerate}
If \ref{ThmLP_A} and \ref{ThmLP_B} hold, then in sequence~\eqref{3.1.2} the composition $p\circ q$ is trivial and $\ker(p)=q(G_1)^{\rm norm}$ is the smallest closed normal subgroup containing $q(G_1)$. If in addition, \ref{ThmLP_C} holds, then $G_1$ is isomorphic to $\ker(p)$ via $q$.
\end{thm}

\begin{proof}
The first assertion is \cite[Theorem~2.4]{LP}, and the second is \cite[Theorem~A.1]{EHS}.
\end{proof}

Now we are ready to give the

\begin{proof}[Proof of Proposition~\ref{crew_sequence}]
We first prove the exactness of the lower sequence, which yields the first assertion. Since $\dal\CF\dar_{const}$ is a sub-category of $\dal\CF\dar$ the map $\beta$ is surjective and faithfully flat. We apply Theorem~\ref{ThmLP} and verify its conditions \ref{ThmLP_A}, \ref{ThmLP_B} and \ref{ThmLP_C}. Condition \ref{ThmLP_A} holds, because an $F$-isocrystal on $U$ is constant if and only if it is trivial as an isocrystal. Next we show \ref{ThmLP_B}. The maximal trivial convergent sub-isocrystal $\CH_0$ of a convergent $F$-isocrystal $\CG$ is generated by horizontal sections of $\CG$. Since the Frobenius $F_\CG$ respects horizontal sections, the isocrystal $\CH_0$ underlies a convergent $F$-isocrystal. To prove \ref{ThmLP_C} let $\CG\in\dal\CF^\sim\dar$. Since the image of $\dal \CF\dar$ under $(\,.\,)^{\sim}$ is closed under direct sums, tensor products and duals, there is an object $\CH$ of $\dal \CF\dar$ such that $\CG$  is a subquotient of $\CH^{\sim}$. Since $\CF^{\sim}$ is semi-simple, so is every object in $\dal \CF^{\sim}\dar$. Therefore, $\CG$  is isomorphic to a subobject of $\CH^{\sim}$.

The upper sequence is exact by definition of the group $W^\CF(U/\olK)$. To prove the commutativity of the diagram we consider the morphism $\BZ\to\Gr(\CF/U)$ which sends $1\in\BZ$ to the element $\BasePoint^*F_\CF\in\Gr(\CF/U)$. This extends to a morphism $W^\CF(U/\olK)\to\Gr(\CF/U)$ by definition of $W^\CF(U/\olK)=\DGal(\CF)\rtimes\BZ$.
\end{proof}

\section{Unit-Root $F$-Isocrystals and Groups of Connected Components} \label{SectionUnitRootAndConnComp}

We recall Crew's~\cite{Crew92} theory of convergent unit-root $F$-isocrystals in Subsection~\ref{SubSectionUnit}. As an application we study in Subsection~\ref{SubSectionComponents} the group of connected components of the monodromy groups $\Gr(\CF/U)$ and we prove in Proposition~\ref{reduction3->2} that Conjecture~\ref{MainConj} follows from Conjecture~\ref{MainConj3}. 

\subsection{Unit-Root $F$-Isocrystals} \label{SubSectionUnit}

We begin our discussion of unit-root $F$-isocrystals with the following useful criterion.

\begin{lemma} \label{Lemma3.2}
Let $\CF\in\FIsoc_\olK(U)$ and assume that the identity component $\Gr(\CF)\open$ of its monodromy group is unipotent. Then $\CF$ is a unit-root $F$-isocrystal. In particular, this is the case, when the monodromy group is finite and hence $\Gr(\CF)\open=\{1\}$.
\end{lemma}

\begin{proof} 
Let $N\in\BN$ be the order of the group $\Gr(\CF)/\Gr(\CF)\open$. Let $x\in|U|$, let $\BF_{q^n}$ be its residue field, and let $y\in U(\BF_{q^n})$ be a point above $x$, which we may take as our base point. Then $(y^*F^n_\CF)^N$ lies in the unipotent group $\Gr(\CF,y)\open$ and acts unipotently on $\omega_y(\CF)$. Thus $\omega_y(\CF)$ has a basis in which $(y^*F^n_\CF)^N$ is an upper triangular unipotent matrix. By the definition of the Newton slopes in \cite[page~121]{Katz79} this implies that all Newton slopes of $y^*\CF$ are zero. As this holds for all $x$, we conclude that $\CF$ is unit-root.
\end{proof}

To recall Crew's result on unit-root $F$-isocrystals, let $K^\un=\bigcup_n K_n$ be the maximal unramified extension of $K$ in $\olK$ and let $\wh K^\un$ be its $p$-adic completion. Let $\FUR_\olK(U)\subset\FIsoc_\olK(U)$ be the Tannakian sub-category of $\olK$-rational convergent unit-root $F$-isocrystals on $U$. It is tensor equivalent to the category of $\olK$-rational representations of its Tannakian fundamental group $\pi_1^{\FUR}(U):=\Aut^\otimes\bigl(\omega_\BasePoint|_{\FUR_\olK(U)}\bigr)$.

\begin{thm}\label{ThmCrewsThm}
The category $\FUR_\olK(U)$ is canonically tensor equivalent to the category $\Rep^\cont_\olK\pi_1^\et(U,\bar\BasePoint)$, such that the fiber functor $\omega_\BasePoint$ on $\FUR_\olK(U)$ and the forgetful fiber functor $\omega_f$ on $\Rep^\cont_\olK\pi_1^\et(U,\bar\BasePoint)$ become canonically isomorphic over the compositum $\wh K^\un\olK$ of $\wh K^\un$ and $\olK$. In particular, $\pi_1^{\FUR}(U)\times_{\olK} (\wh K^\un\olK)$ is canonically isomorphic to the base change to $\wh K^\un\olK$ of the $\olK$-linear algebraic envelope of the topological group $\pi_1^\et(U,\bar\BasePoint)$.
\end{thm}

\begin{proof}
The tensor equivalence of categories was established by Crew \cite[Theorem~2.1 and Remark~2.2.4]{Crew87}. As this equivalence is natural (see loc.\ cit.), it commutes with the pull-back to the base point $\BasePoint\in U(\BF_{q^\BasePtDeg})$. Since we will need it in the next corollary, we explicitly describe the tensor equivalence at $\BasePoint$ between $\Rep^\cont_\olK\pi_1^\et(\BasePoint,\bar\BasePoint)$ and $\FUR_\olK(\BF_{q^\BasePtDeg})$; see \cite[p.~119]{Crew87}. The objects in the latter category are pairs $(\CF,F_\CF)$ consisting of a finite free $K_\BasePtDeg\otimes_K \olK$-module $\CF$ and an $F\otimes\id_\olK$-semi linear automorphism $F_\CF$ of $\CF$. Also $\pi_1^\et(\BasePoint,\bar\BasePoint)=\Gal(\ol\BF_q/\BF_{q^\BasePtDeg})\cong\wh\BZ$. The tensor equivalence associates a Galois representation $\rho\colon\Gal(\ol\BF_q/\BF_{q^\BasePtDeg})\to\Aut_\olK(W)$ with a unit-root $F$-isocrystal $(\CF,F_\CF)$ in such a way that there is a canonical Galois and $F$-equivariant isomorphism
\begin{equation}\label{EqThmCrewsThm}
\alpha\colon W\otimes_\olK (\wh K^\un\otimes_K \olK) \; \isoto \; \CF\otimes_{K_\BasePtDeg\otimes\olK} (\wh K^\un\otimes_K \olK)
\end{equation}
where $\gamma\in\Gal(\ol\BF_q/\BF_{q^\BasePtDeg})=\Gal(K^\un/K_\BasePtDeg)=\Aut_{K_\BasePtDeg}^{\rm cont}(\wh K^\un)$ acts on the left hand side as $\rho(\gamma)\otimes (\gamma\otimes\id_\olK)$ and on the right hand side as $\id_\CF\otimes(\gamma\otimes\id_\olK)$, and where the Frobenius $F$ acts on the left hand side as $\id_W\otimes (F\otimes\id_\olK)$ and on the right hand side as $F_\CF\otimes (F\otimes\id_\olK)$. The isomorphism $\alpha$ allows to recover the $F$-isocrystal $(\CF,F_\CF)$ as $\bigl(W\otimes_K (\wh K^\un\otimes_K \olK),\id_W\otimes (F\otimes\id_\olK)\bigr)^{\Gal(\ol\BF_q/\BF_{q^\BasePtDeg})}$ and $W$ as $\bigl(\CF\otimes_{K_\BasePtDeg\otimes\olK} (\wh K^\un\otimes_K \olK)\bigr)^{F\otimes\id_\olK=\id}$. Let $\wh K^\un\otimes_K \olK\to \wh K^\un\olK$ be the multiplication homomorphism. Tensoring $\alpha$ with $\wh K^\un\olK$ over $\wh K^\un\otimes_K \olK$ yields a canonical isomorphism of fiber functors $\alpha\otimes\id_{\wh K^\un\olK}\colon\omega_f\otimes_\olK (\wh K^\un\olK)\isoto\omega_\BasePoint\otimes_{\olK} (\wh K^\un\olK)$. The latter induces an isomorphism of $\wh K^\un\olK$-group schemes $\alpha_*\colon\Aut^\otimes(\omega_f)\times_\olK (\wh K^\un\olK)\isoto\pi_1^{\FUR}(U)\times_{\olK} (\wh K^\un\olK)$ and so the last statement follows directly from Lemma~\ref{LemmaAlgEnvelope}.
\end{proof}

We formulate the consequence for the individual monodromy groups $\Gr(\CF)$. For each $x\in|U|$ let $\bar x$ be a geometric base point of $U$ lying above $x$ and choose an isomorphism of groups $\pi_1^\et(U,\bar x)\isoto\pi_1^\et(U,\bar\BasePoint)$. It is unique up to conjugation in $\pi_1^\et(U,\bar\BasePoint)$. Let $\Frob_x^{-1}\in\Gal(\ol\BF_x/\BF_x)=\pi_1^\et(x,\bar x)$ be the \emph{geometric Frobenius} which maps $a\in\ol\BF_x$ to $a^{1/\#\BF_x}$. It is the inverse of the \emph{arithmetic Frobenius} $\Frob_x\colon a\mapsto a^{\#\BF_x}$. Then the conjugacy class of $x_*\Frob_x^{-1}$ in $\pi_1^\et(U,\bar\BasePoint)$ is well defined.

\begin{cor}\label{CorCrewsThm}
Let $\CF\in \FUR_\olK(U)$ and let $\rho\colon\pi_1^\et(U,\bar\BasePoint)\to\Aut_\olK(W)$ be the representation corresponding to $\CF$ under the tensor equivalence from Theorem~\ref{ThmCrewsThm}. Then the categories $\dal\CF\dar\subset\FUR_\olK(U)$ and $\dal\rho\dar\subset\Rep^\cont_\olK\pi_1^\et(U,\bar\BasePoint)$ are tensor equivalent and there is a (non-canonical) $\olK$-rational isomorphism $\beta\colon\omega_f|_{\dal\rho\dar}\isoto\omega_\BasePoint|_{\dal\CF\dar}$ of tensor functors on these categories. In particular $\Gr(\CF)$ is the closure of the image of $\beta_*\circ\rho\colon\pi_1^\et(U,\bar\BasePoint)\to\Aut_{\olK}(\omega_\BasePoint(\CF))$ and for all points $x\in|U|$ the $\Gr(\CF)$-conjugacy classes of $\beta_*\circ\rho(x_*\Frob_x^{-1})$ and $\Frob_x(\CF)$ coincide.
\end{cor}

\begin{proof}
By Theorem~\ref{ThmCrewsThm} the categories $\dal\CF\dar$ and $\dal\rho\dar$ are tensor equivalent. Since $\Aut^\otimes(\omega_f|_{\dal\rho\dar})$ is a closed subgroup of $\Aut_\olK(\omega_f(\rho))$ the $\Aut^\otimes(\omega_f|_{\dal\rho\dar})$-torsor $\Isom^\otimes(\omega_f|_{\dal\rho\dar},\,\omega_\BasePoint|_{\dal\CF\dar})$ is a scheme of finite type over $\olK$ by \cite[IV$_2$, Proposition~2.7.1]{EGA}. By Theorem~\ref{ThmCrewsThm} it has a point over $\wh K^\un\olK$, and hence is non-empty. Therefore, it also has points over $\olK$. Every such point defines an isomorphism $\beta\colon\omega_f|_{\dal\rho\dar}\isoto\omega_\BasePoint|_{\dal\CF\dar}$ of tensor functors and an isomorphism of algebraic groups $\beta_*\colon\Aut^\otimes(\omega_f|_{\dal\rho\dar})\isoto\Gr(\CF)$, $g\mapsto\beta\circ g\circ\beta^{-1}$ over $\olK$. So the statement about the latter group follows from Lemma~\ref{LemmaAlgEnvelope}.

To prove the equality of conjugacy classes note that $\rho(\BasePoint^*\Frob_\BasePoint^{-1})\otimes\id_{\wh K^\un\otimes\olK}=\bigl(\rho(\BasePoint^*\Frob_\BasePoint^{-1})\otimes\Frob_\BasePoint^{-1}\otimes\id_{\olK}\bigr) \circ \bigl(\id_W\otimes F^\BasePtDeg\otimes\id_\olK\bigr)$ is mapped under the tensor isomorphism $\alpha$ from \eqref{EqThmCrewsThm} to 
\[
\alpha\circ\bigl(\rho(\BasePoint_*\Frob_{\BasePoint}^{-1})\otimes\id_{\wh K^\un}\bigr)\circ\alpha^{-1} = \bigl(\id_{\BasePoint^*\CF}\otimes\Frob_\BasePoint^{-1}\otimes\id_\olK\bigr) \circ (\BasePoint^*F_\CF^\BasePtDeg\otimes F^\BasePtDeg\otimes\id_\olK\bigr)=\BasePoint^*F_\CF^\BasePtDeg\,.
\]
If $h:=\beta\circ\alpha^{-1}\in\Gr(\CF)(\wh K^\un\olK)$ then $\beta_*\circ\rho(\BasePoint_*\Frob_{\BasePoint}^{-1})\cdot h\;=\;h\cdot\alpha\circ\bigl(\rho(\BasePoint_*\Frob_{\BasePoint}^{-1})\otimes\id_{\wh K^\un}\bigr)\circ\alpha^{-1}\;=\;h\cdot \BasePoint^*F_\CF^\BasePtDeg$. Since $\beta_*\circ\rho(\BasePoint_*\Frob_{\BasePoint}^{-1})$ and $\BasePoint^*F_\CF^\BasePtDeg$ lie in $\Gr(\CF)(\olK)$ this is an equation for $h$ with coefficients in $\olK$ which has a solution in $\wh K^\un\olK$. Thus, it has a solution $h\in\Gr(\CF)(\olK)$, too. This proves that the $\Gr(\CF)(\olK)$-conjugacy classes of $\beta_*\circ\rho(\BasePoint_*\Frob_{\BasePoint}^{-1})$ and $\BasePoint^*F_\CF^\BasePtDeg=\Frob_\BasePoint(\CF)$ coincide for $x=\BasePoint$. The equality for general $x$ follows moving the base point $\BasePoint$ to $x$.
\end{proof}

\begin{rem}\label{RemOtherTopologies} 
Note that Conjecture~\ref{MainConj} for convergent unit-root $F$-isocrystals on $U$, which we proved in Proposition~\ref{PropUnitRoot}, is considerably weaker than the classical Chebotar\"ev Density Theorem for $U$ to the same extent as the pro-finite topology on $\pi_1^\et(U,\bar\BasePoint)$ is finer than the Zariski-topology on its $\olK$-linear algebraic envelope. Namely, the classical Chebotar\"ev Density Theorem says that the Frobenii of a set $S$ of Dirichlet density $1$ are dense in $\pi_1^\et(U,\bar\BasePoint)$ for the pro-finite topology, see \cite[Theorem~7]{Serre63}. If $\CF$ is a unit-root $F$-isocrystal the representation $\pi_1^\et(U,\bar\BasePoint)\to\Gr(\CF)$ corresponding to $\CF$ by Corollary~\ref{CorCrewsThm} is continuous for the $p$-adic topology. So the Frobenii lie $p$-adically dense in the image of this representation, but this image itself is not $p$-adically dense in $\Gr(\CF)$, since it is $p$-adically closed, but not the whole group in general. This image is only Zariski-dense by Corollary~\ref{CorCrewsThm}. Therefore, the stronger assertion that the set $\bigcup_{x\in S}\Frob_x(\CF)$ is $p$-adically dense in $\Gr(\CF)$ is false in general. So it is unreasonable to expect a density statement for any topology other than the Zariski-topology, even in the most simple case of constant $F$-isocrystals. This can be seen from the following
\end{rem}

\begin{example}\label{ExConstantIsocr}
Let $\CC$ be the pullback to $U$ of the $F$-isocrystal on $\BF_q$ of rank $1$ given by $(K,F=p^s)$ with $s\in\BZ$. If $s\ne0$ then $\Gr(\CF)=\BG_{m,\olK}$. Indeed, $\Gr(\CF)$ is a closed subgroup of $\Aut_{\olK}(\BasePoint^*\CC)=\BG_{m,\olK}$ which contains $\Frob_{\BasePoint}(\CC)=\{p^{\BasePtDeg s}\}$, where $\BasePtDeg=\deg(\BasePoint)$. Since the set $p^{\BZ\BasePtDeg s}$ is infinite, the only such group is $\BG_{m,\olK}$. However, the set $\bigcup_{x\in |U|}\Frob_x(\CF)\subset p^{\BZ\BasePtDeg s}$ is discrete in $\BG_m(\olK)$ for the $p$-adic topology.
\end{example}

For the next corollary recall from Definition~\ref{Def3.5} the geometric monodromy group $\Gr(\CF/U)^{\rm geo}$ and $\mathbf{W}(\CF)$, the Tannakian fundamental group of $\dal\CF\dar_{const}$. Also recall that the geometric \'etale fundamental group $\pi_1^\et(U,\bar\BasePoint)^{\rm geo}$ is defined as the \'etale fundamental group $\pi_1^\et(U\otimes_{\BF_q}\ol\BF_q,\bar\BasePoint)$ of $U\otimes_{\BF_q}\ol\BF_q$. It sits in Grothendieck's fundamental exact sequence \cite[IX, Th\'eor\`eme~6.1]{SGA1}, which is the upper row in diagram~\eqref{EqCorCrewsThmGeoDiag}.

\begin{cor}\label{CorCrewsThmGeo}
In the situation of Corollary~\ref{CorCrewsThm} the homomorphism $\beta_*\circ\rho$ induces a commutative diagram
\begin{equation}\label{EqCorCrewsThmGeoDiag}
\xymatrix {
1 \ar[r] & \pi_1^\et(U,\bar\BasePoint)^{\rm geo} \ar[r] \ar[d]^{\beta_*\circ\rho} & \pi_1^\et(U,\bar\BasePoint) \ar[r] \ar[d]^{\beta_*\circ\rho} & \Gal(\ol\BF_q/\BF_q) \ar[r] \ar[d] & 1 \\
1 \ar[r] & \Gr(\CF/U)^{\rm geo} \ar[r] & \Gr(\CF/U) \ar[r] & \mathbf{W}(\CF) \ar[r] & 1 \\
}
\end{equation}
with exact rows in which the three vertical maps have dense image. In particular, if $\rho^{\rm geo}:=\rho|_{\pi_1^\et(U,\bar\BasePoint)^{\rm geo}}$ is the restriction of $\rho$, then $\beta_*$ induces an isomorphism $\Aut^\otimes(\omega_f|_{\dal\rho^{\rm geo}\dar})\isoto\Gr(\CF/U)^{\rm geo}$.
\end{cor}

\begin{proof}
The category $\dal\CF\dar_{const}$ has a tensor generator $\CC$. Let $(K^{\oplus r},f)\in\FIsoc_\olK(\BF_q)$ be an $F$-isocrystal on $\Spec\BF_q$ such that $\CC$ is the pullback of $(K^{\oplus r},f)$ under the structure morphism $U\to\Spec\BF_q$. Then $\mathbf{W}(\CF)=\Gr(\CC/U)=\Gr((K^{\oplus r},f)/\Spec\BF_q)$ equals the closure of $f^\BZ$ in $\GL_{r,\olK}$ by Theorem~\ref{ThmMonodrOfConstant}\ref{ThmMonodrOfConstant_B}. Since $\CF\in \FIsoc_\olK(U)$ is unit-root, also $\CC$ is unit-root, and possibly after a change of basis, we can assume that $f\in\GL_r(\CO_K)$, where $\CO_K$ denotes the valuation ring of $K$. Since the group $\GL_r(\CO_K)$ is pro-finite, the morphism $\BZ\to\GL_r(\CO_K)$, $n\mapsto f^n$ extends uniquely to a morphism $\rho_{const}\colon\Gal(\ol\BF_q/\BF_q)=\wh\BZ\cdot\Frob_q\to\GL_r(\CO_K)$, $\Frob_q\mapsto f$ with image contained and dense in $\mathbf{W}(\CF)$. Let $H\subset \Gr(\CF/U)^{\rm geo}$ be the closure of $\beta_*\circ \rho(\pi_1^\et(U,\bar\BasePoint)^{\rm geo})$. It is a normal subgroup of $\Gr(\CF/U)$ by the density of $\beta_*\circ \rho$. Let $\CG\in\dal\CF\dar$ with $\Gr(\CG/U)=\Gr(\CF/U)/H$. Then the representation $\rho\colon\pi_1^\et(U,\bar\BasePoint)\to\Aut_\olK(\omega_\BasePoint(\CG))$ factors through $\Gal(\ol\BF_q/\BF_q)$, and hence $\CG$ is constant and belongs to $\dal\CF\dar_{const}$ by Crew's result \cite[Theorem~2.1 and Remark~2.2.4]{Crew87} for $\Spec\BF_q$. This shows that $\Gr(\CF/U)^{\rm geo}\subset \ker(\Gr(\CF/U)\onto\Gr(\CF/U)/H)=H$.
\end{proof}

\begin{rem}
We do not know, whether the image $\beta_*\circ\rho\bigl(\pi_1^\et(U,\bar\BasePoint)^{\rm geo}\bigr)$ equals the intersection of $\beta_*\circ\rho\bigl(\pi_1^\et(U,\bar\BasePoint)\bigr)$ with $\Gr(\CF/U)^{\rm geo}$. To prove this, one needs to find a faithful representation $\rho'$ of the group $\rho\bigl(\pi_1^\et(U,\bar\BasePoint)\bigr)\big/\rho\bigl(\pi_1^\et(U,\bar\BasePoint)^{\rm geo}\bigr)$ on a finite dimensional $\olK$-vector space which belongs to the Tannakian sub-category $\dal\rho\dar\subset\Rep^\cont_\olK\pi_1^\et(U,\bar\BasePoint)$. Nevertheless, we can prove the following
\end{rem}

\begin{cor}\label{CorCrewsThmGeo2}
In the situation of Corollary~\ref{CorCrewsThmGeo} there is a constant unit-root $F$-isocrystal $\CF'$ on $U$, corresponding to a representation $\rho'\colon \pi_1^\et(U,\bar\BasePoint)\onto\Gal(\ol\BF_q/\BF_q)\to\GL_r(\olK)$ for $r\in\BN$, that factors through an isomorphism $\rho\bigl(\pi_1^\et(U,\bar\BasePoint)\bigr)\big/\rho\bigl(\pi_1^\et(U,\bar\BasePoint)^{\rm geo}\bigr)\isoto\im(\rho')$. If we replace $\CF$ by $\CF\oplus\CF'$ and $\rho$ by $\rho\oplus\rho'$, then in diagram~\eqref{EqCorCrewsThmGeoDiag} we have
\[
\beta_*\circ (\rho\oplus\rho')\bigl(\pi_1^\et(U,\bar\BasePoint)^{\rm geo}\bigr)\;=\;\beta_*\circ (\rho\oplus\rho')\bigl(\pi_1^\et(U,\bar\BasePoint)\bigr)\cap\Gr(\CF\oplus\CF')^{\rm geo}\,.
\]
\end{cor}

\begin{proof}
We first construct $\rho'$. By Cartan's theorem, see \cite[Part~II, \S\,V.9, Corollary to Theorem~1 on page~155]{Serre92} or Theorem~\ref{Cartan1} below, the images $\BK:=\rho\bigl(\pi_1^\et(U,\bar\BasePoint)\bigr)$ and $\BK^{\rm geo}:=\rho\bigl(\pi_1^\et(U,\bar\BasePoint)^{\rm geo}\bigr)$ are Lie groups over $\BQ_p$, and the quotient $\BK/\BK^{\rm geo}=\rho\bigl(\pi_1^\et(U,\bar\BasePoint)\bigr)\big/\rho\bigl(\pi_1^\et(U,\bar\BasePoint)^{\rm geo}\bigr)$ is again a Lie group over $\BQ_p$ by \cite[Part~II, \S\,IV.5, Remark~2 after Theorem~1 on page~108]{Serre92}. If the quotient $\BK/\BK^{\rm geo}$ is finite, then it has a faithful representation $\rho'$ on a finite dimensional $\olK$-vector space. So we now assume that $\BK/\BK^{\rm geo}$ is not finite. Note that $\Gal(\ol\BF_q/\BF_q)\cong\wh\BZ=\BZ_p\times\prod_{\ell\ne p}\BZ_\ell$ surjects onto $\BK/\BK^{\rm geo}$. By the incompatibility of the $\ell$-adic and the $p$-adic topologies, the image of $\prod_{\ell\ne p}\BZ_\ell$ in $\BK/\BK^{\rm geo}$ is a finite subgroup $H$ and thus has a faithful representation on a finite dimensional $\olK$-vector space $V'_1$. On the other hand, the map $\BZ_p\into\Gal(\ol\BF_q/\BF_q)\onto \BK/\BK^{\rm geo}$ is analytic by \cite[Part~II, \S\,V.9, Theorem~2]{Serre92} and its image is an at most one-dimensional Lie group over $\BQ_p$ by \cite[Part~II, \S\,IV.5, Theorems~1 and 3 and Corollary to Theorem~2]{Serre92}. If it were zero dimensional, then it would be finite because it is compact, and this was excluded. So it is one-dimensional and the map from $\BZ_p$ onto its image is a local isomorphism by \cite[Part~II, \S\,III.9, Theorem~2]{Serre92}. The kernel of this map is finite, and hence trivial, because $\BZ_p$ is compact and torsion free. Therefore, we obtain an epimorphism $\phi\colon\BZ_p\times H\onto \BK/\BK^{\rm geo}$, which is even an isomorphism, because if an element $(g,h)$ lies in the kernel, then $\phi(g)=\phi(h^{-1})$ is a torsion element of $\phi(\BZ_p)=\BZ_p$, and hence trivial. Therefore, $g=1$, and since $\phi|_H$ is injective also $h=1$. Now take a faithful representation of $\BZ_p$ on a finite dimensional $\olK$-vector space $V'_2$, for example in a unipotent group. The sum $V'_1\oplus V'_2$ is the desired representation $\rho'\colon\pi_1^\et(U,\bar\BasePoint)\onto\Gal(\ol\BF_q/\BF_q)\to\GL_r(\olK)$. The convergent $F$-isocrystal $\CF'$ on $U$ corresponding to $\rho'$ is constant by Crew's result \cite[Theorem~2.1 and Remark~2.2.4]{Crew87} for $\Spec\BF_q$.

To prove the last statement, note that the inclusion ``$\subset$'' is trivial. To prove the converse inclusion ``$\supset$'' let $\wt \BK:=(\rho\oplus\rho')\bigl(\pi_1^\et(U,\bar\BasePoint)\bigr)$ and $\wt \BK^{\rm geo}:=(\rho\oplus\rho')\bigl(\pi_1^\et(U,\bar\BasePoint)^{\rm geo}\bigr)$. By construction of $\rho'$ we have isomorphisms $\BK\isoto\wt \BK$ and $\BK^{\rm geo}\isoto\wt \BK^{\rm geo}$ given by $c\mapsto c\oplus (c\mod \BK^{\rm geo})$. Therefore, every element $\beta_*(c\oplus c\mod \BK^{\rm geo})\in\beta_*(\wt \BK)=\beta_*\circ (\rho\oplus\rho')\bigl(\pi_1^\et(U,\bar\BasePoint)\bigr)$ which lies in the kernel $\Gr(\CF\oplus\CF')^{\rm geo}$ of $\Gr(\CF\oplus\CF')\onto \mathbf{W}(\CF\oplus\CF')$ is mapped to $1$ in the quotient $\Gr(\CF')=\Aut^\otimes(\omega_f|_{\dal\rho'\dar})$ of $\mathbf{W}(\CF\oplus\CF')$. This implies $\rho'(c)=1$ and $c\in \BK^{\rm geo}$ as desired.
\end{proof}

Another consequence of Corollary~\ref{CorCrewsThm} is the following

\begin{cor}\label{CorUnitRootFromUToV}
Let $\CF\in\FUR_\olK(U)$ and let $f\colon V\into U$ be a non-empty open subscheme. Then the pullback functor $f^*\colon\CG\mapsto f^*\CG$ from $\dal\CF\dar$ to $\dal f^*\CF\dar$ is a tensor equivalence of Tannakian categories. In particular, the induced morphism of monodromy groups $\Gr(f^*\CF/V)\to\Gr(\CF/U)$ is an isomorphism.
\end{cor}

\begin{proof}
We use that $\pi_1^\et(f)\colon\pi_1^\et(V,\bar\BasePoint)\onto\pi_1^\et(U,\bar\BasePoint)$ is an epimorphism by \cite[V, Proposition~8.2]{SGA1}. Let $\rho\colon\pi_1^\et(U,\bar\BasePoint)\to\Aut_\olK(W)$ be the representation corresponding to $\CF$ under the tensor equivalence from Theorem~\ref{ThmCrewsThm}. Then $f^*\rho:=\rho\circ\pi_1^\et(f)\colon\pi_1^\et(V,\bar\BasePoint)\to\Aut_\olK(W)$ is the representation corresponding to $f^*\CF$. Therefore, the functor $f^*$ on $\dal\CF\dar$ is fully faithful. Moreover, every subobject of $\dal f^*\rho\dar$ is of the form $\rho'\colon\pi_1^\et(V,\bar\BasePoint)\to\Aut_\olK(W')$ for an invariant subspace $W'\subset W$. By the surjectivity of $\pi_1^\et(f)$ this subspace is also $\pi_1^\et(U,\bar\BasePoint)$-invariant, and hence $\rho'=f^*(\rho_{W'})$ for the subobject $\rho_{W'}\colon\pi_1^\et(U,\bar\BasePoint)\to\Aut_\olK(W')$ of $\rho$. Since $f^*$ clearly is a tensor functor, the corollary is proven.
\end{proof}

\subsection{Groups of Connected Components} \label{SubSectionComponents}

We first record Kedlaya's full faithfulness theorem.

\begin{thm}[{\cite[Theorem~1.1]{Kedlaya04}}]\label{ThmKedlayaFF}
The functor $j_U\colon \FIsoc^{\dagger}_\olK(U)\rightarrow \FIsoc_\olK(U)$ that sends an overconvergent $F$-isocrystal $\CF^\dagger$ on $U$ to its convergent completion $j_U(\CF^\dagger)$ is fully faithful.
\end{thm}

\begin{lemma}\label{LemmaOverconvMonodr}
\begin{enumerate}
\item \label{LemmaOverconvMonodr_A}
The homomorphism $\Gr(j_U(\CF^\dagger))\to\Gr^\dagger(\CF^\dagger)$ induced by the functor $j_U$ always is a closed immersion.
\item \label{LemmaOverconvMonodr_B}
It induces an epimorphism $\Gr(j_U(\CF^\dagger))/\Gr(j_U(\CF^\dagger))\open\onto\Gr^\dagger(\CF^\dagger)/\Gr^\dagger(\CF^\dagger)\open$ on the groups of connected components.
\item \label{LemmaOverconvMonodr_C}
If $\CF\in\FIsoc_\olK(U)$ has finite monodromy group $\Gr(\CF)$, then there is an $\CF^\dagger\in\FIsoc^\dagger_\olK(U)$ with $j_U(\CF^\dagger)\cong\CF$. In this case, $j_U\colon \dal\CF^\dagger\dar \isoto \dal \CF\dar$ and $\Gr(\CF)\isoto\Gr^\dagger(\CF^\dagger)$ are isomorphisms.
\end{enumerate}
\end{lemma}

\begin{proof}
\ref{LemmaOverconvMonodr_A} follows from \cite[Proposition~2.21(b)]{Deligne-Milne82}, because every object of $\dal j_U(\CF^\dagger)\dar$ is a subquotient of an object of the form $\bigoplus_i j_U(\CF^\dagger)^{\otimes m_i}\otimes (j_U(\CF^\dagger)\dual)^{\otimes n_i}=j_U\bigl(\bigoplus_i (\CF^\dagger)^{\otimes m_i}\otimes (\CF^\dagger{}\dual)^{\otimes n_i}\bigr)$.

\medskip\noindent
\ref{LemmaOverconvMonodr_C}
$\CF$ is unit-root by Lemma~\ref{Lemma3.2}. The corresponding representation $\beta_*\circ\rho\colon \pi_1^\et(U,\bar\BasePoint)\to \Gr(\CF)$ from Corollary~\ref{CorCrewsThm} has dense image, and hence is surjective. Its kernel equals $\pi_1^\et(V,\bar\BasePoint)$ for a finite \'etale Galois covering $f\colon V\to U$ and a lift of the base point $\bar\BasePoint$ to $V$. This implies that $f^*\CF\cong\UOne_V^{\oplus\rk\CF}$ is trivial and has an overconvergent extension on $V$. By \cite[Th\'eor\`eme~2]{Etesse02} we conclude that $\CF\cong j_U(\CF^\dagger)$ for an $\CF^\dagger\in\FIsoc^\dagger_\olK(U)$ with $f^*\CF^\dagger\cong\UOne_V^{\oplus\rk\CF}$.

To prove the rest we use \ref{LemmaOverconvMonodr_A} and Kedlaya's Theorem~\ref{ThmKedlayaFF}. By \cite[Proposition~2.21(a)]{Deligne-Milne82} it remains to show that for every $\CG^\dagger\in\dal\CF^\dagger\dar$ and for every convergent sub-$F$-isocrystal $\CH\subset j_U(\CG^\dagger)$ there is an overconvergent sub-$F$-isocrystal $\CH^\dagger\subset \CG^\dagger$ with $\CH=j_U(\CH^\dagger)$. In this situation, $\CH\in\dal\CF\dar$ and $\Gr(\CH)$ is finite by Lemma~\ref{LemmaCompatibility}\ref{LemmaCompatibility_A} below. Then $\CH\cong j_U(\CH^\dagger)$ for an $\CH^\dagger\in\FIsoc^\dagger_\olK(U)$ by what we have just proven for $\CF$. By Theorem~\ref{ThmKedlayaFF}, the inclusion $\CH=j_U(\CH^\dagger)\into j_U(\CG^\dagger)$ comes from an inclusion $\CH^\dagger\into\CG^\dagger$.

\medskip\noindent
\ref{LemmaOverconvMonodr_B}
Let $\CU^\dagger\in\dal\CF^\dagger\dar$ such that $\Gr^\dagger(\CF^\dagger)\onto\Gr^\dagger(\CU^\dagger)$ has kernel equal to the normal subgroup $\Gr^\dagger(\CF^\dagger)\open\subset\Gr^\dagger(\CF^\dagger)$; see Lemma~\ref{LemmaCompatibility}\ref{LemmaCompatibility_C} below. Then $\Gr^\dagger(\CU^\dagger)$ is finite and isomorphic to $\Gr(j_U(\CU^\dagger))$ by \ref{LemmaOverconvMonodr_C}. Since $j_U(\CU^\dagger)\in \dal j_U(\CF^\dagger)\dar$, the corresponding homomorphism $\Gr(j_U(\CF^\dagger))\onto\Gr(j_U(\CU^\dagger))\isoto\Gr^\dagger(\CU^\dagger)= \Gr^\dagger(\CF^\dagger)/\Gr^\dagger(\CF^\dagger)\open$ is surjective by Lemma~\ref{LemmaCompatibility}.
\end{proof}

We now give an application of unit-root $F$-isocrystals. We consider a finite \'etale Galois covering $f\colon V\to U$ of smooth, geometrically irreducible, quasi-compact schemes and the pull back functor
\begin{equation}\label{EqPullbackFunctor}
f^*\colon \FIsoc_\olK(U)\longrightarrow \FIsoc_\olK(V)\,,\quad\CF\mapsto f^*\CF\,;
\end{equation}
see for example \cite[p.~431]{Crew92}. In addition, recall the functor $(\,.\,)^{(n)}$ from \eqref{EqFnKbar}. Both functors possess right adjoints 
\begin{equation}\label{EqAdjoints}
\begin{array}{r@{\,\colon \;}l@{\;\longrightarrow\;}l@{\,,\quad}r@{\;\mapsto\;}l}
f_* & \FIsoc_\olK(V) & \FIsoc_\olK(U) & \CG & f_*\CG \qquad\text{and}\\[2mm]
(\,.\,)_{(n)} & \FnIsoc_\olK(U_n) & \FIsoc_\olK(U) & \CG & \CG_{(n)}\,.
\end{array}
\end{equation}
For $f_*$ see \cite[1.7]{Crew92}. The functor $(\,.\,)_{(n)}$ can explicitly be described as $\CG_{(n)}:=\bigoplus_{i=0}^{n-1} F^{i*}\CG$ with the Frobenius $F_{\CG_{(n)}}\colon F^*{\CG_{(n)}}\isoto{\CG_{(n)}}$ be given by the matrix
\[
\left( \raisebox{3.7ex}{$
\xymatrix @C=0.5pc @R=0pc {
0 \ar@{.}[drdrdr] & & & F_{\CG}\\
1\ar@{.}[drdr]\\
& & & \\
& & 1 & 0\\
}$}
\right)\,,
\]
where $F_\CG\colon F^{n*}\CG\isoto\CG$ is the Frobenius of $\CG$. Note that due to the scalar extension from $K_n$ to $\olK$ there is no need to apply the pushforward $pr_*$ under the projection $pr\colon U_n\to U$. Indeed, for an $F^n$-isocrystal $\CG_0\in\FnIsoc_{K_n}(U_n)$ and its extension $\CG_0\otimes_{K_n}\olK$, the $F^n$-isocrystal $(pr_*\CG_0)\otimes_K \olK\in\FnIsoc_\olK(U)$ is a module over $K_n\otimes_K \olK$, and the tensor product $\bigl((pr_*\CG_0)\otimes_K \olK\bigr)\otimes_{K_n\otimes\olK} \olK$ equals $\CG_0\otimes_{K_n}\olK$. In this way, every $\CG\in \FnIsoc_\olK(U_n)$ can be viewed as an $F^n$-isocrystal $\CG\in\FnIsoc_\olK(U)$.

Let $\CK$ be the pullback to $U$ of the constant $F$-isocrystal on $\Spec\BF_q$ given by $\bigoplus_{i\in\BZ/n\BZ}\olK\cdot e_i$ with $\olK$-linear Frobenius $F(e_i)=e_{i+1}$. Then $\CK$ is unit-root, and even $\CK\in\FIsoc^\dagger_\olK(U)$ is overconvergent. Moreover, for the trivial $F^n$-isocrystal ${}^n\UOne_{U_n}$ on $U_n$ one has $F^{i*}({}^n\UOne_{U_n})={}^n\UOne_{U_n}$ and $({}^n\UOne_{U_n})_{(n)}=\CK$. The adjunction satisfies the projection formula $(\CF^{(n)}\otimes\CG)_{(n)}\cong\CF\otimes\CG_{(n)}$, and $(\CF^{(n)})_{(n)}\cong\CF\otimes \CK$, as well as $(\CG_{(n)})^{(n)}=\bigoplus_{i=0}^{n-1} F^{i*}\CG$. In particular, via the counit morphism of the adjunction, $\CG$ is a quotient of $(\CG_{(n)})^{(n)}$. Similarly, let $\CL:=f_*\UOne_V$ where $\UOne_V$ is the trivial $F$-isocrystal on $V$. Then $\CL$ is unit-root by \cite[bottom of p.~436]{Crew92}. Moreover, $f_*(f^*(\CF)\otimes\CG)\cong\CF\otimes f_*\CG$, and $f_*f^*\CF=\CF\otimes\CL$. And if we set $\Gamma=\Gal(V/U)$ then $f^*f_*\CG\cong\bigoplus_{\gamma\in\Gamma}\gamma^*\CG$ and $f^*\CL\cong\bigoplus_\Gamma\UOne_V$. So again, via the counit morphism of the adjunction, $\CG$ is a quotient of $f^*f_*\CG$.

Analogous functors to \eqref{EqAdjoints} exist for overconvergent $F$-isocrystals. This is obvious for $(\,.\,)_{(n)}$. For $f_*$ the condition $\CG\in\FIsoc^\dagger_\olK(U)$ implies $f^*f_*(j_V(\CG))\cong j_V(\bigoplus_{\gamma\in\Gamma}\gamma^*\CG)$ with $\bigoplus_{\gamma\in\Gamma}\gamma^*\CG\in\FIsoc^\dagger_\olK(V)$, and hence $f_*\CG\in\FIsoc^\dagger_\olK(U)$ by \cite[Th\'eor\`eme~2]{Etesse02}. In particular, $\CL\in\FIsoc^\dagger_\olK(U)$ is overconvergent.

\begin{lemma}\label{LemmaExampleGr}
The monodromy groups satisfy $\Gr^\dagger(\CK/U) = \Gr(\CK/U)\cong\BZ/n\BZ$ and $\Gr^\dagger(\CL/U) = \Gr(\CL/U)\cong\Gal(V/U)$.
\end{lemma}

\begin{proof}
We prove both assertions simultaneously by writing $\Gamma:=\BZ/n\BZ$ and $V:=U_n$ for the treatment of $\Gr(\CK/U)$. Under Crew's equivalence between unit-root $F$-isocrystals and representations of the \'etale fundamental group (Theorem~\ref{ThmCrewsThm}), the trivial $F$-isocrystal $\UOne_V$ and the trivial $F^n$-isocrystal ${}^n\UOne_{U_n}$ correspond to the trivial representations $\pi_1^\et(V,\bar\BasePoint)\to \olK\mal$, where we lift $\bar\BasePoint$ to $V$. Moreover, under this equivalence the functors $f^*$ and $(\,.\,)^{(n)}$ correspond to the restriction functors
\begin{eqnarray*}
\Res\colon & \Rep^\cont_\olK\pi_1^\et(U,\bar\BasePoint) \longrightarrow \Rep^\cont_\olK\pi_1^\et(V,\bar\BasePoint)\,, & \quad \rho\longmapsto\rho|_{\pi_1^\et(V,\bar\BasePoint)}\,.
\end{eqnarray*}
For an open subgroup $\BH$ of a compact topological group $\BG$ the right and left adjoint to $\Res_\BH^\BG$ is the induction functor $\Ind_\BH^\BG\colon\Rep_K^\cont \BH\to\Rep_K^\cont \BG$ with
\[
\Ind_\BH^\BG(\rho,W_\rho)\;:=\;\bigl\{\,r\colon \BG\to W_\rho\text{ continuous}\colon r(hg)=\rho(h)r(g)\;\forall\,h\in \BH, g\in \BG\,\bigr\}
\]
for a continuous representation $\rho\colon \BH\to \Aut_K(W_\rho)$; see \cite[Footnotes on pp.~61 and 63]{NSW08}. So the right adjoints $f_*$ and $(\,.\,)_{(n)}$ correspond to the right adjoints
\begin{eqnarray*}
\Ind\colon & \Rep^\cont_\olK\pi_1^\et(V,\bar\BasePoint) \longrightarrow \Rep^\cont_\olK\pi_1^\et(U,\bar\BasePoint)\,, & \quad \rho\longmapsto \Ind_{\pi_1^\et(V,\bar\BasePoint)}^{\pi_1^\et(U,\bar\BasePoint)} \rho \,.
\end{eqnarray*}
Thus, $\CL=f_*\UOne_V$ and $\CK=({}^n\UOne_{U_n})_{(n)}$ correspond to the representations $\bigoplus_{\Gal(V/U)}\olK$ and $\bigoplus_{\BZ/\!n\BZ}\olK$ in $\Rep^\cont_\olK\pi_1^\et(U,\bar\BasePoint)$ on which $\pi_1^\et(U,\bar\BasePoint)/\pi_1^\et(V,\bar\BasePoint)=\Gal(V/U)$, respectively $\pi_1^\et(U,\bar\BasePoint)/\pi_1^\et(U_n,\bar\BasePoint)=\BZ/n\BZ$, act as permutation representations. By Corollary~\ref{CorCrewsThm} the groups $\Gr(\CK/U)$ and $\Gr(\CL/U)$ equal (the closure of) $\Gamma$. Now use Lemma~\ref{LemmaOverconvMonodr}\ref{LemmaOverconvMonodr_C}.
\end{proof}

Let $\pi_1^{\FIsoc}(U)$ be the Tannakian fundamental group of $\FIsoc_\olK(U)$, that is, the automorphism group of the fiber functor $\omega_\BasePoint\colon\FIsoc_\olK(U)\to\{\olK\text{-vector spaces}\}$. It is an affine group scheme over $\olK$. We similarly define $\pi_1^{\FIsoc}(V)$ and $\pi_1^{\FnIsoc}(U_n)$.

\begin{prop} \label{Prop1.17}
Let $f\colon V\to U$ be a finite \'etale Galois covering with Galois group $\Gamma:=\Gal(V/U)$.
\begin{enumerate}
\item \label{Prop1.17a}
There is an exact sequence of affine group schemes over $\olK$
\[
\xymatrix { 0 \ar[r] & \pi_1^{\FIsoc}(V) \ar[r]^{\alpha\enspace} & \pi_1^{\FIsoc}(U) \ar[r]^{\enspace\beta} & \Gal(V/U) \ar[r] & 0\,,
}
\]
where the morphism $\alpha$ is induced by the pullback functor $f$ from \eqref{EqPullbackFunctor}, and $\beta$ comes from the epimorphism $\pi_1^{\FIsoc}(U)\onto\Gr(\CL/U)\cong\Gal(V/U)$ using Lemma~\ref{LemmaExampleGr}. The same holds for $\pi_1^{\FIsoc^\dagger}$.
\item \label{Prop1.17b}
For every $\CF\in\FIsoc_\olK(U)$, and analogously $\CF\in\FIsoc^\dagger_\olK(U)$, the sequence in \ref{Prop1.17a} induces the following exact sequence of affine group schemes over $\olK$
\begin{equation}\label{EqLemma1.17b}
\xymatrix { 0\ar[r] & \Gr(f^*\CF/V)\ar[r] & \Gr(\CF/U)\ar[r] & \Gamma' \ar[r] & 0\,,}
\end{equation}
where $\Gamma'$ is a finite quotient group of $\Gal(V/U)$. In particular, if $\Gr(\CF/U)$ is connected then $\Gr(f^*\CF/V)\isoto\Gr(\CF/U)$. 
\end{enumerate}
\end{prop}

\begin{proof}
\ref{Prop1.17a} 
We use Theorem~\ref{ThmLP} \footnote{An elementary and explicit proof of \ref{Prop1.17a} was given in the second version on \href{http://arxiv.org/abs/1811.07084v2}{arXiv:1811.07084v2} of this article.} and check the three conditions \ref{ThmLP_A}, \ref{ThmLP_B}, and \ref{ThmLP_C}. To prove \ref{ThmLP_A} assume $f^*\CF=\UOne_V^{\oplus r}$ is trivial. Then $\CF$ is unit-root, because every point $x\in |U|$ is the image of a point $y\in |V|$ and $1=\Frob_y(f^*\CF)$ is a power of $\Frob_x(\CF)$. The representation of $\pi_1^\et(U,\bar\BasePoint)$ corresponding to $\CF$ restricts to the trivial representation of $\pi_1^\et(V,\bar\BasePoint)$, that is, it factors through $\Gal(V/U)$, and hence $\CF\in\dal\CL\dar$. If $\CF\in\FIsoc^\dagger_\olK(U)$, this shows that $j_U(\CF)\in\dal j_U(\CL)\dar$, i.e.~$j_U(\CF)\cong j_U(\CG)$ for some $\CG\in\dal\CL\dar$ by Lemma~\ref{LemmaOverconvMonodr}\ref{LemmaOverconvMonodr_C}. Then $\CF\cong\CG$ by Kedlaya's Theorem~\ref{ThmKedlayaFF}. Condition~\ref{ThmLP_B} holds by Galois descent \cite[Theorem~4.5]{Ogus84} (respectively \cite[Th\'eor\`eme~1]{Etesse02} in the overconvergent case), because the largest trivial subobject of $f^*\CF$ is invariant under $\Gal(V/U)$, and hence of the form $f^*\CH$ for a subobject $\CH\subset\CF$. Condition~\ref{ThmLP_C} holds, because every $\CG\in\FIsoc_\olK(V)$ is a quotient of $f^*f_*\CG$.

\smallskip\noindent
\ref{Prop1.17b} The group scheme $\Gr(\CF/U)$ is the image of the representation $\pi_1^{\FIsoc}(U)\to\Aut_{\olK}(\omega_\BasePoint(\CF))$ corresponding to $\CF$ and likewise for $f^*\CF$. Since $\omega_\BasePoint(\CF)=\omega_{v}(f^*\CF)$ for a base point $v\in V$ above $\BasePoint$, the group $\Gr(f^*\CF/V)$ is a closed normal subgroup of $\Gr(\CF/U)$ and the quotient $\Gamma'$ is a quotient of $\Gal(V/U)$. If $\Gr(\CF/U)$ is connected then (its image in) $\Gamma'$ will be trivial. This proves the lemma.
\end{proof}

\begin{prop}\label{Prop1.16}
Let $n\in\BN$.
\begin{enumerate}
\item \label{Prop1.16a}
There is an exact sequence of affine group schemes over $\olK$
\[
\xymatrix { 0 \ar[r] & \pi_1^{\FnIsoc}(U_n) \ar[r]^{\alpha} & \pi_1^{\FIsoc}(U) \ar[r]^{\quad\beta} & \BZ/n\BZ \ar[r] & 0\,,
}
\]
where the morphism $\alpha$ is induced by the functor $(\,.\,)^{(n)}$ from \eqref{EqFnKbar}, and $\beta$ comes from the epimorphism $\pi_1^{\FIsoc}(U)\onto\Gr(\CK/U)\cong\BZ/n\BZ$ using Lemma~\ref{LemmaExampleGr}. The same holds for $\pi_1^{\FIsoc^\dagger}$.
\item \label{Prop1.16b}
For every $\CF\in\FIsoc_\olK(U)$, and analogously $\CF\in\FIsoc^\dagger_\olK(U)$ the sequence in \ref{Prop1.16a} induces the following exact sequence of affine group schemes over $\olK$
\[
\xymatrix { 0 \ar[r] & \Gr(\CF^{(n)}/U_n) \ar[r] & \Gr(\CF/U) \ar[r] & \Gamma' \ar[r] & 0\,,
}
\]
where $\Gamma'$ is a finite group which is a quotient of $\BZ/n\BZ$. In particular, if $\Gr(\CF/U)$ is connected then $\Gr(\CF^{(n)}/U_n)\isoto\Gr(\CF/U)$.
\end{enumerate}
\end{prop}

\begin{proof}
The proof is analogous to Proposition~\ref{Prop1.17}.
\end{proof}

\begin{cor}\label{CorConnMonodromy}
Let $\CF\in\FIsoc_\olK(U)$, or analogously $\CF\in\FIsoc^\dagger_\olK(U)$. Then there exists a finite \'etale Galois covering $f\colon V\to U$ such that $\Gr(f^*\CF/V)=\Gr(\CF/U)\open$ in sequence~\eqref{EqLemma1.17b} and $\Gal(V/U)$ is isomorphic to $\Gr(\CF/U)/\Gr(\CF/U)\open$.
\end{cor}

\begin{proof}
Let $\CG\in\dal\CF\dar$ be the object with $\Gr(\CG/U)=\Gamma:=\Gr(\CF/U)\big/\Gr(\CF/U)\open$; see Lemma~\ref{LemmaCompatibility}\ref{LemmaCompatibility_C} below. Since $\Gamma$ is a finite group, $\CG$ is a convergent unit-root $F$-isocrystal by Lemma~\ref{Lemma3.2}. Let $\rho_\CG\colon\pi_1^\et(U,\bar\BasePoint)\onto \Gamma$ be the representation of the \'etale fundamental group corresponding to $\CG$ by Theorem~\ref{ThmCrewsThm} which is surjective onto $\Gamma$ by Corollary~\ref{CorCrewsThm}. Lifting $\bar\BasePoint$ to $V$, the kernel of $\rho_\CG$ equals $\pi_1^\et(V,\bar\BasePoint)$ for a finite \'etale Galois covering $f\colon V\to U$, that is $\Gamma=\pi_1^\et(U,\bar\BasePoint)/\pi_1^\et(V,\bar\BasePoint)=\Gal(V/U)$. From sequence~\eqref{EqLemma1.17b} it follows that $\Gr(f^*\CF/V)$ equals the kernel $\Gr(\CF/U)\open$ of $\Gr(\CF/U)\onto \Gamma=\Gal(V/U)$.
\end{proof}

For the next corollary note that $\pi_1^{\FIsoc}(U,\bar\BasePoint)=\pi_1^{\FIsoc}(U)$ and $\Gr(\CF/U,\bar\BasePoint)=\Gr(\CF/U)$ for any geometric base point $\bar\BasePoint$ above our base point $\BasePoint\in U(\BF_{q^\BasePtDeg})$ and for every $\CF\in\FIsoc_\olK(U)$.

\begin{cor}\label{CorGpOfConnComp}
The exact sequence from Proposition~\ref{Prop1.17}\ref{Prop1.17a} induces an exact sequence of affine group schemes over $\olK$
\begin{equation}\label{EqCorGpOfConnComp}
\xymatrix { 0 \ar[r] & \pi_1^{\FIsoc}(U)\open \ar[r] & \pi_1^{\FIsoc}(U) \ar[r] & \pi_1^\et(U) \ar[r] & 0\,.
}
\end{equation}
In particular, the pro-group of connected components of $\pi_1^{\FIsoc}(U)$ equals the \'etale fundamental group $\pi_1^\et(U)$.
\end{cor}

\begin{proof}
For every $\CF\in\FIsoc_\olK(U)$ we obtain from Corollary~\ref{CorConnMonodromy} a finite \'etale Galois covering $f\colon V\to U$ and an exact sequence
\[
\xymatrix { 0\ar[r] & \Gr(\CF/U)\open \ar[r] & \Gr(\CF/U)\ar[r] & \Gal(V/U) \ar[r] & 0\,.}
\]
We now take the limit of these sequences over the diagram of the Tannakian sub-categories $\dal\CF\dar$ of $\FIsoc_\olK(U)$. This limit is taken in the category of sheaves of groups on $\olK$ for the \'etale topology. We claim that this limit is the sequence~\eqref{EqCorGpOfConnComp}. First of all, by Lemma~\ref{LemmaCompatibility} below, the projective system of the $\Gr(\CF/U)\open$ consists of epimorphisms and so satisfies the Mittag-Leffler condition. Therefore, \eqref{EqCorGpOfConnComp} is exact at $\pi_1^\et(U)=\varprojlim_V\Gal(V/U)$. The group $\pi_1^{\FIsoc}(U)$ is the projective limit of the $\Gr(\CF/U)$, which can equivalently be taken in the category of affine group schemes over $\olK$. It remains to identify $\pi_1^{\FIsoc}(U)\open$ with the limit of the projective system of the $\Gr(\CF/U)\open$. By construction this limit is representable by a closed subgroup scheme of $\pi_1^{\FIsoc}(U)$. Moreover, it is connected, because if $d$ is an idempotent in its structure sheaf then $d$ lies in the structure sheaf of some $\Gr(\CF/U)\open$ and satisfies $d^2=d$ after maybe replacing $\dal\CF\dar$ by a larger such category. Since $\Gr(\CF/U)\open$ is connected we have $d=0$ or $d=1$, whence the limit is connected and a closed subgroup scheme of $\pi_1^{\FIsoc}(U)\open$. On the other hand, the limit contains $\pi_1^{\FIsoc}(U)\open$ since the latter maps to the identity component in $\pi_1^\et(U)$, which is trivial. Therefore, this limit equals $\pi_1^{\FIsoc}(U)\open$ and the corollary is proven.
\end{proof}

\begin{prop}\label{reduction3->2}
Let $\CF\in\FIsoc_\olK(U)$, respectively $\CF\in\FIsoc^\dagger_\olK(U)$, for which Conjecture~\ref{MainConj3} is true. Then  Conjecture~\ref{MainConj} also holds true for $\CF$.
\end{prop} 

\begin{proof} Let $S\subset|U|$ be a subset of upper Dirichlet density one. The group $\Gamma:=\Gr(\CF/U)/\Gr(\CF/U)\open$ is a quotient of $\pi_1^\et(U,\bar\BasePoint)$ by Corollary~\ref{CorConnMonodromy}. For every conjugacy class $C\subset \Gamma$ let $R_C\subset |U|$ denote the set of those closed points $x\in|U|$ whose Frobenius class $\Frob_x(\CF)$ maps onto $C\subset \Gamma$. By Corollary~\ref{CorCrewsThm} the image of $\Frob_x(\CF)$ in $\Gamma$ coincides with the image of the conjugacy class of the Frobenius $\Frob_x^{-1}$ of $x$ in $\pi_1^\et(U,\bar\BasePoint)$. Thus by the classical Chebotar\"ev Density Theorem \cite[Theorem~7]{Serre63}, $R_C$ has positive Dirichlet density. Therefore, $S_C=S\cap R_C$ has positive upper Dirichlet density by Lemma~\ref{density_intersection} below. By Conjecture~\ref{MainConj3} for $\CF$ the closure of $\bigcup_{x\in S_C}\Frob_x(\CF)$ contains a connected component of $\Gr(\CF/U)$, which maps to a point in $C$. Since the closure of $\bigcup_{x\in S_C}\Frob_x(\CF)$ is conjugation-invariant, we get that this closure is equal to the union of all connected components mapping into $C$. The closure of $\bigcup_{x\in S}\Frob_x(\CF)$ contains the union of the closures of $\bigcup_{x\in S_C}\Frob_x(\CF)$ for varying $C$, and hence it must be the whole group $\Gr(\CF/U)$.
\end{proof}

\section{Chebotar\"ev for Constant $F$-Isocrystals and General Proof Strategy} \label{SectStrategy}

Theorems~\ref{ThmIsoclinic} and \ref{ThmOverconvIntro} use Theorem~\ref{ThmCheboForConstants} for constant $F$-isocrystals (more precisely, its Corollary~\ref{Cor4.2}) as a stepping stone. We prove Theorem~\ref{ThmCheboForConstants} in Subsection~\ref{SubSectCheboForConst} using $\ell$-adic Lie groups for an auxiliary prime $\ell$ that may be different from $p$; see Remark~\ref{RemAuxPrime}. In Subsection~\ref{SubSectTannRed}, we present our strategy to prove Theorems~\ref{ThmIsoclinic} and \ref{ThmOverconvIntro}, and the formal Tannakian reduction steps used in these proofs.

\subsection{Chebotar\"ev for Constant $F$-Isocrystals} \label{SubSectCheboForConst}

We work throughout over a field $L$ of characteristic zero, and all undecorated fiber products $\times$ are over $L$. Theorem~\ref{ThmCheboForConstants} will be an immediate consequence of the following

\begin{thm} \label{TheoremMordellLang}
Let $T$ be a commutative linear algebraic group over a field $L$ of characteristic zero, whose identity component $T\open$ is the product $\BG_{m,L}^r\TimesL\BG_{a,L}^n$ of a split torus $\BG_{m,L}^r$ with an additive group scheme $\BG_{a,L}^n$ for $r,n\ge0$. Let $\Gamma\subset T$ be an infinite cyclic dense subgroup in $T$. Then $n=0$ or $n=1$, and for any connected component $T^c$ of $T$, any infinite subset $S$ of $\Gamma\cap T^c$ is still dense in $T^c$.
\end{thm}

\begin{proof}[Proof of Theorem~\ref{ThmCheboForConstants}.]
We use Theorem~\ref{ThmMonodrOfConstant} and let $\CF=\pi^*(W,f)\in\FIsoc_\olK(U)$ be a constant $F$-isocrystal on $U$. Then the group $T:=\Gr(\CF/U)$ has the form as in Theorem~\ref{TheoremMordellLang} for $\Gamma:=f^\BZ$ and the field $L=\olK$. Let $X\subset |U|$ be an infinite subset. Since for a given degree $d$ there are only finitely many points $x\in |U|$ with $\deg(x) = d$, the subset $S:=\{\Frob_x(\CF)\colon x\in X\}=\{f^{\deg(x)}\colon x\in X\}\subset \Gamma$ is infinite. By the pigeonhole principle there is a connected component $T^c$ of $T$ for which $S\cap T^c$ is still infinite, and hence dense in $T^c$ by Theorem~\ref{TheoremMordellLang}.
\end{proof}

\begin{proof}[Proof of Theorem~\ref{TheoremMordellLang}.]
Clearly the image of $\Gamma$ under the projection $T\to\BG_{a,L}^n$ is dense. Since this image lies in the at most $1$-dimensional linear subspace generated by the image of a generator of $\Gamma$, the dimension $n$ is either zero or one. By choosing an element $h$ of $S$ and considering the translate $S-h$ we can assume that $S\subset T\open$, and we must show that $S$ is dense in $T\open$. Thus, by replacing $\Gamma$ with $\Gamma\cap T\open=\ker(\Gamma\to T/T\open)$ which is still infinite cyclic and dense in $T\open$, we may assume that $T=T\open$ is connected.

The $n=0$ case can be easily deduced from the the Mordell-Lang Conjecture for tori, proved by Michel Laurent. Indeed, by \cite[Th\'eor\`eme~2]{Laurent} the closure of $S$ is the finite union of finitely many translates of sub-tori of $T$. By shrinking $S$, if it is necessary, we may assume that this finite union consists of just one translate of a sub-torus $\wt T$. It will be enough to show that $\wt T=T$. Pick an element $h\in S$. Then the translate $S-h$ lies in $\wt T$, and hence the intersection $\Gamma'=\Gamma\cap \wt T$ is a subgroup of $\Gamma$ which contains the infinite set $S-h$. Since $\Gamma\cong\BZ$ we get that $\Gamma'$ is a subgroup of $\Gamma$ of finite index, say $i$. Then the $i$-th power isogeny $[i]\colon x\mapsto ix$ on $T$ maps $\Gamma$ into $\Gamma'$, and hence into $\wt T$. Because $\Gamma$ is dense in $T$ and $[i]$ is surjective, $\Gamma'$ is dense in $T$ too. We get that $T=\wt T$.

We will prove the $n=1$ case by a different method (which nevertheless can be applied to the $n=0$ case as well). Assume that the claim is false and let $Y\subset T=\BG_{m,L}^r\TimesL\BG_{a,L}$ be a proper hypersurface such that $S$ lies in $Y$. Let $\gamma\in \Gamma$ be a generator. Since the coefficients of the defining polynomial of $Y$ and the coordinates of $\gamma$ form a finite set $I$, we may assume without loss of generality that $L$ is finitely generated over $\BQ$, by replacing it with the field generated by $I$. Let $V$ be a smooth irreducible variety over $\BQ$ whose function field is $L$. By shrinking $V$ if necessary, we may assume that $\gamma$ extends to a section $\wt \gamma$ of the projection map 
$$T\times_\BQ V\to V\,,$$
where from now on we consider $T=\BG_{m,\BQ}^r\times_\BQ \BG_{a,\BQ}^n$ with its canonical $\BQ$-structure. The projection onto the first factor induces a morphism
\begin{equation}\label{EqTildeg}
\wt{\gamma}\colon V \to T\,.
\end{equation}
By shrinking $V$ further, we may also assume that the projection of $\wt \gamma$ onto the factor $\BG_{a,\BQ}$ is nowhere zero on $V$. Similarly, we may assume that $Y$ extends to a closed subscheme $\wt Y$ of $T\times_\BQ V$ which is a proper hypersurface in the fiber over any point of $V$. 

Let $v$ be a closed point of $V$ and let $E$ be the residue field of $v$. Then $E$ is a number field. As usual we identify $\BG_{a,\BQ}$ with $\BA_\BQ^1$ and we embed $\BG_{m,\BQ}$ into $\BA_\BQ^2$ as the closed subscheme on which the product of the coordinates of $\BA_\BQ^2$ equals $1$. We consider the coordinates of the point $\wt{\gamma}(v)\in T(E)\subset (\BA_\BQ^{2r}\times_\BQ \BA_\BQ^1)(E)$. For all but finitely many valuations $\mu$ of $E$ all these coordinates of $\wt \gamma(v)$ are $\mu$-adic units. Fix such a valuation $\mu$ and let $E_{\mu}$ be the completion of $E$ with respect to $\mu$. It is a finite extension of $\BQ_\ell$ for a prime number $\ell$. (In Remark~\ref{RemAuxPrime} below we explain that in the most interesting case of the application of Theorem~\ref{TheoremMordellLang} to Theorem~\ref{ThmCheboForConstants}, when $\CF$ is not unit-root, we should expect $\ell\ne p$.) By continuity there is an open ball $B$ around $v$ in $V(E_{\mu})$ such that for every $x\in B$ all the coordinates of $\wt \gamma(x)$ are still $\mu$-adic units. We next prove the following

\medskip\noindent
{\itshape Claim.} There is a $x\in B$ such that the group $\wt{\gamma}(x)^\BZ$ generated by $\wt \gamma(x)$ is dense in $T_{E_\mu}$.

\medskip\noindent
If this is false, the closure of the group $\wt{\gamma}(x)^\BZ$ is a proper closed subgroup $H\subsetneq T_{E_\mu}$. Note that for every such subgroup $H$ the locus $W_H\subset V_{E_\mu}:=V\otimes_\BQ E_\mu$ where the map $\wt \gamma$ from \eqref{EqTildeg} factors through $H$ is a closed subset. Moreover, $W_H$ does not contain any irreducible component of $V_{E_\mu}$, because it does not contain any generic point $\eta$ of $V_{E_\mu}$. Namely, the residue field $\kappa(\eta)$ of $\eta$ contains $L$, and so the element $\wt{\gamma}(\eta)=\gamma$ generates a dense subgroup of $T_{\kappa(\eta)}$ by assumption. We claim that the intersection $W_H(E_\mu)\cap B$ is a proper analytic subset of $B$ which is nowhere dense in $B$, see \cite[Chapter~9, \S\,5.1, Definition~1]{BourbakiTopol5}. Indeed, assume contrarily that there is a point $b\in B$ and a small open neighborhood around $b$ which is contained in $W_H(E_\mu)\cap B$. Since $V$ is smooth in $b$, a Zariski open neighborhood of $b$ in $V$ is \'etale over some affine space $\BA^d_{E_\mu}$ where $d=\dim V$. By shrinking the small open neighborhood around $b$ if necessary, we can assume that it maps isomorphically onto an open ball in $\BA^d(E_\mu)$. The latter ball is contained inside the scheme theoretic image of $W_H\subset V_{E_\mu}$ in $\BA^d_{E_\mu}$, which is a proper closed subset of dimension $<d$. This is a contradiction and proves that $W_H(E_\mu)\cap B$ is nowhere dense in $B$. Now by Lemma~\ref{countable} below, the set of proper closed subgroups $H\subsetneq T_{E_\mu}$ is countable, so the union $\bigcup_HW_H(E_\mu)\cap B\subset B$ cannot equal $B$ by Baire's category theorem \cite[Chapter~9, \S\,5.3, Theorem~1 and Definition~3]{BourbakiTopol5}. Every point $x$ in the complement of $\bigcup_HW_H(E_\mu)\cap B$ satisfies our claim.

Since all the coordinates of the generator $\wt{\gamma}(x)$ of $\wt{\gamma}(x)^\BZ$ are $\ell$-adic units, $\wt{\gamma}(x)^\BZ$ lies in $T(\CO_{E_\mu})$. The latter contains a pro-$\ell$-group of some finite index $m$. The map $\BZ\to T(\CO_{E_\mu})$, $1\mapsto \wt{\gamma}(x)^m$ has image in this pro-$\ell$-group and extends to a map $\BZ_\ell\to T(\CO_{E_\mu})$ by the universal property of the pro-$\ell$-completion. Let $\BK=\wt{\gamma}(x)^{m\BZ_\ell}$ be the image. We write $S=\gamma^{S'}\subset\Gamma=\gamma^\BZ$ for the corresponding subset $S'\subset\BZ$, and put $\wt S(x):=\wt{\gamma}(x)^{S'}$. Since $S'$ is infinite, there is an $a\in\BZ$ such that $S'\cap a+m\BZ$ is still infinite by the pigeonhole principle. We replace $S'$ by $S'\cap a+m\BZ$ and $S$ and $\wt S(x)$ correspondingly. Let $\eta$ be a generic point of $V_{E_\mu}$ containing $x$ in its closure. Then $\wt{\gamma}(x)^i$ lies in the closure of $\wt{\gamma}(\eta)^i=\gamma^i\in\wt Y(\eta)$ for every $i\in S'$. So the infinite set $\wt S(x)$ is contained in the proper hypersurface $\wt Y(x)$ of $T(E_\mu)$. Under the inverse of the $\ell$-adic analytic isomorphism $\BZ_\ell\isoto \wt{\gamma}(x)^a\cdot \BK$, $z\mapsto \wt{\gamma}(x)^{a+mz}$ the intersection $\wt Y(x)\cap \wt{\gamma}(x)^a\cdot \BK$ is mapped isomorphically onto an $\ell$-adic analytic subset $A$ of $\BZ_\ell$, that is $A$ is locally in the $\ell$-adic topology on $\BZ_\ell$ the zero locus of power series. The set $\wt Y(x)\cap \wt{\gamma}(x)^a\cdot \BK$ contains the infinite set $\wt S(x)$, so $A$ contains an infinite set. Since $\BZ_\ell$ is compact, this infinite set has an accumulation point $y$. In a neighborhood $U=y+\ell^n\BZ_\ell$ of $y$ for suitable $n\gg0$ the power series defining $A$ have infinitely many zeros. But this implies, that $A$ contains $U$, because the zeros of a power series in one variable are $\ell$-adically discrete by \cite[Proposition~2]{Lazard}. It follows that $\wt Y(x)$ contains $\wt{\gamma}(x)^{a+m\cdot U}=\wt{\gamma}(x)^{a+my+m\ell^n\BZ_\ell}$. Since $\wt{\gamma}(x)^\BZ$ is dense in $T_{E_\mu}$ and the $m\ell^n$-power isogeny $[m\ell^n]$ on $T_{E_\mu}$ is surjective, we get that $\wt{\gamma}(x)^{m\ell^n\BZ}$ and $\wt{\gamma}(x)^{a+my+m\ell^n\BZ}$ and $\wt{\gamma}(x)^{a+my+m\ell^n\BZ_\ell}$ are dense in $T_{E_\mu}$, too. So $\wt Y(x)$ contains a dense subset of $T_{E_\mu}$, which is a contradiction. 
\end{proof}

It remains to give a proof of the following well known

\begin{lemma} \label{countable}
Over a field $L$ of characteristic zero, all closed subgroups $G$ of $\BG_{m,L}^r\TimesL\BG_{a,L}$ are of the form $G_s\TimesL\BG_{a,L}^{\epsilon}$, where $G_s\subset \BG_{m,L}^r$ is a closed subgroup and $\epsilon$ is either $0$ or $1$. In particular, the set of such subgroups is countable.
\end{lemma}

\begin{proof} 
Since $G$ is commutative, it is the direct product $G=G_u\TimesL G_s$ of the set $G_u$ of its unipotent elements and the set $G_s$ of its semi-simple elements, which are both closed subgroups; see \cite[I.4.7~Theorem]{Borel91}. The projections $G_u\to \BG_{m,L}^r$ and $G_s\to\BG_{a,L}$ are both zero, because $1$ is the only element which is at the same time unipotent and semi-simple. Therefore, $G_u\subset\ker(\BG_{m,L}^r\TimesL\BG_{a,L}\to \BG_{m,L}^r) = \BG_{a,L}$ and $G_s\subset\ker(\BG_{m,L}^r\TimesL\BG_{a,L}\to \BG_{a,L}) = \BG_{m,L}^r$. Since $G_u$ is connected by Lemma~\ref{LemmaUnipotConnected}, there are only the possibilities $G_u=\{1\}$ or $G_u=\BG_{a,L}$. This proves the first assertion.

For the last assertion we only have to show that the set of all closed subgroups $G_s\subset \BG_{m,L}^r$ is countable. By \cite[III.8.2~Proposition]{Borel91} these subgroups correspond to quotients of the free abelian group $X^*(\BG_{m,L}^r)=\BZ^r$ and so there are only countably many.
\end{proof}

\begin{rem}\label{RemAuxPrime}
One should expect that the prime number $\ell$ used in the proof of Theorem~\ref{TheoremMordellLang}, and hence also of Theorem~\ref{ThmCheboForConstants}, can only be taken equal to the characteristic $p$ of $\BF_q$, if the constant $F$-isocrystal $\CF=\pi^*(W,f)$ is unit-root. Indeed, with respect to some $\olK$-basis of $W$ we may write $f\in\GL(W)$ as an upper triangular matrix with only $0$ and $1$ outside the diagonal. Then $\CF$ is unit-root if and only if the diagonal entries of $f$ lie in $\CO_\olK\mal$.

We now use the notation from the proof of Theorem~\ref{TheoremMordellLang}. The finitely generated subfield $L\subset\olK$ is contained in a finite extension $E_\mu$ of $\BQ_p$. The inclusion $L\into E_\mu$ yields a point $x_0\in V(E_\mu)$ that maps to the generic point of $V$. If $\CF$ is unit-root, then $f\in T(\CO_{E_\mu})$ and there is a $p$-adic open ball $B\subset V(E_\mu)$ around $x_0$ such that $\wt\gamma(x)\in T(\CO_{E_\mu})$ for every $x\in B$. Since $V(\ol\BQ)$ is $p$-adically dense in $V(E_\mu)$ we can find the number field $E$ and the points $v\in V(E)$ and $x\in V_{E_\mu}$ as in the proof of Theorem~\ref{TheoremMordellLang} for the $p$-adic field $E_\mu$ (or a finite extension of it).

On the other hand, if $\CF$ is not unit-root, and hence $f\notin T(\CO_\olK)$ we should not expect to find a point $x\in V(\olK)$ with $\wt\gamma(x)\in T(\CO_\olK)$, although strictly speaking this is possible in special examples. In one such example $T=\BG_m$ and $f\in p\cdot\BZ_p\mal$ is a solution of the equation $x^2-x+p=0$. Then $L=\BQ(f)$ and $V=\Spec L$. Since $V_{\BQ_p}=\Spec L\otimes_\BQ \BQ_p = \Spec (\BQ_p\times\BQ_p)$, we may take $E=L$ and $E_\mu=\BQ_p$, and the points $v=x\in V(\BQ_p)$ for which $\wt\gamma(x)=p/f\in \BZ_p\mal$. This point is the connected component of $V_{\BQ_p}$ different from the point $x_0$ with $\wt\gamma(x_0)=f$.
\end{rem}

The following corollary will be used in the proofs of Theorems~\ref{ThmIsoclinic} and \ref{ThmOverconvIntro} given in Sections~\ref{SectIsoclinic} and \ref{Section11}.

\begin{cor} \label{Cor4.2}
Let $L$, $T$ and $\Gamma$ be as in Theorem~\ref{TheoremMordellLang}. Let $G$ be a linear algebraic group over $L$ and let $\phi\colon G\to T$ be a surjective morphism of algebraic groups over $L$. Assume that every connected component of $\ker(\phi)$ contains an $L$-rational point. Let $T^c$ be a connected component of $T$ and let $Y\subset \phi^{-1}(T^c)$ be a closed subset which does not contain any irreducible component of $\phi^{-1}(T^c)$. Then the set of those $\gamma\in \Gamma\cap T^c$ such that $Y$ contains a connected component of $\phi^{-1}(\gamma)$ is finite. 
\end{cor}

\begin{proof}
By assumption, every connected component $b\in\pi_0\bigl(\ker(\phi)\bigr)$ of $\ker(\phi)$ is of the form $\ker(\phi)\open\cdot g_b$ for an $L$-rational point $g_b\in\ker(\phi)$. For every $\gamma\in \Gamma\subset T$ fix $g_\gamma\in \phi^{-1}(\gamma)\subset G$. Let $S$ be the set of all those $\gamma\in \Gamma\cap T^c$ for which $Y$ contains a connected component $\ker(\phi)\open g_{b(\gamma)} g_\gamma$ of $\phi^{-1}(\gamma)=\ker(\phi) g_\gamma$ for some $b(\gamma)\in\pi_0\bigl(\ker(\phi)\bigr)$. For fixed $\gamma\in S$ we obtain $\bigcup_b g_b Y \supset \bigcup_b g_b (\ker(\phi)\open g_{b(\gamma)} g_\gamma) \supset \ker(\phi) g_\gamma=\phi^{-1}(\gamma)$. In particular, $\bigcup_{b}g_b\cdot Y$ contains the closure $W$ of $\phi^{-1}(S)$.

Let $V$ be the closure of $S$ in $T^c$. Since $\phi^{-1}(S)$ is invariant under translation by the closed subgroup $\ker(\phi)$, the same holds for the closure $W$ of $\phi^{-1}(S)$, and hence $W=\phi^{-1}(\phi(W))$. Therefore, $\phi(W)$ is a closed subset of $T^c$ which contains $S$, so it contains $V$. We get that $\phi^{-1}(V)\subset \phi^{-1}(\phi(W))=W$. If $S$ was infinite then $V=T^c$ by Theorem~\ref{TheoremMordellLang}, and hence $\phi^{-1}(T^c)\subset W$ is contained in the closed subset $\bigcup_{b}g_b\cdot Y$. Let $G'$ be an irreducible component of $\phi^{-1}(T^c)$. Then $G'$ is contained in $\bigcup_{b}g_b\cdot Y$, and since it is irreducible it is contained in $g_b\cdot Y$ for one $g_b$. But this implies that $Y$ contains the irreducible component $g_b^{-1}G'$ of $\phi^{-1}(T^c)$, which is a contradiction. Thus, $S$ is finite.
\end{proof}

\subsection{Tannakian Reduction Techniques} \label{SubSectTannRed}

We recall the following facts from Tannakian theory. 

\begin{lemma}\label{LemmaSemiSimplification}
Let $\CF$ be an object of a neutral Tannakian category, which is linear over a field of characteristic $0$. Let $\Gr(\CF)=\Aut^\otimes(\omega_\BasePoint|\dal\CF\dar)$ be the monodromy group of $\CF$, that is, the Tannakian fundamental group of $\dal\CF\dar$ with respect to a neutral fiber functor $\omega_\BasePoint$. Then $\CF$ is semi-simple if and only if the category $\dal\CF\dar$ is semi-simple, if and only if the identity component $\Gr(\CF)\open$ is reductive. More generally, let $\CF^{ss}$ be the semi-simplification of $\CF$. Then $r\colon\Gr(\CF)\onto\Gr(\CF^{ss})$ is the maximal reductive quotient of $\Gr(\CF)$. In particular, $r$ induces an isomorphism on the groups of connected components. 
\end{lemma}

\begin{proof} 
The first statement is proven in (the proof of) \cite[Proposition~2.23 and Remark~2.28]{Deligne-Milne82}. We prove the rest. In a suitable basis the representation of $\Gr(\CF)$ on $\omega_\BasePoint(\CF)$ can be written in block matrix form such that the diagonal block entries are representations corresponding to the simple constituents of $\CF$, and the kernel of $r$ is a subgroup of the unipotent upper triangular matrices. Thus, this kernel is normal, unipotent, connected by Lemma~\ref{LemmaUnipotConnected}, and hence contained in the unipotent radical $R_u\!\Gr(\CF)$ of $\Gr(\CF)$. On the other hand, $R_u\!\Gr(\CF)$ is mapped to $\{1\}$ in $\Gr(\CF^{ss})$, because the latter is reductive by our first statement.
\end{proof}

In the rest of this section we describe the three main Tannakian reduction techniques which we use. 

\medskip
\technique{Passage to a larger Tannakian category} \label{TechniqueLarger}

\smallskip

We will freely use the following

\begin{lemma}\label{LemmaCompatibility}
Let $\CF$ be an object of a neutral $\olK$-linear Tannakian category as in Lemma~\ref{LemmaSemiSimplification}. 
\begin{enumerate}
\item \label{LemmaCompatibility_Z}
The category $\dal\CF\dar$ is tensor equivalent to the category of $\olK$-linear algebraic representations of $\Gr(\CF)$.
\item \label{LemmaCompatibility_A}
For every $\CG\in\dal\CF\dar$ there are canonical epimorphisms of linear algebraic groups $\Gr(\CF)\onto\Gr(\CG)$ and $\Gr(\CF)\open\onto\Gr(\CG)\open$.
\item \label{LemmaCompatibility_C}
Conversely, every epimorphism of linear algebraic groups $\Gr(\CF)\onto G$ comes from an object $\CG\in\dal\CF\dar$ and an isomorphism $G\cong\Gr(\CG)$.
\end{enumerate}
\end{lemma}

\begin{proof}
\ref{LemmaCompatibility_Z} is \cite[Theorem~2.11]{Deligne-Milne82}. The epimorphism of groups $\Gr(\CF)\onto\Gr(\CG)$ and of their identity components in \ref{LemmaCompatibility_A} comes from \cite[Proposition~2.21]{Deligne-Milne82}. Statement~\ref{LemmaCompatibility_C} is \cite[Propositions~2.20(b) and 2.21]{Deligne-Milne82}.
\end{proof}

The following lemma is specific to $F$-isocrystals and the question of Chebotar\"ev density. It allows to prove Conjectures~\ref{MainConj} and \ref{MainConj3} for an $F$-isocrystal $\CG$ by proving them for a different $F$-isocrystal $\CF$, that generates a larger Tannakian category $\dal\CF\dar\supset\dal\CG\dar$.

\begin{lemma}\label{Lemma1.10Converse}
Let $\CF\in\FIsoc_\olK(U)$ (or $\CF\in\FIsoc^\dagger_\olK(U)$) and let $\CG\in\dal\CF\dar$
\begin{enumerate}
\item \label{Lemma1.10Converse_A}
Under the map $\Gr(\CF/U)\onto\Gr(\CG/U)$ the Frobenius conjugacy class $\Frob_x(\CF)$ is mapped onto $\Frob_x(\CG)$ for every $x\in |U|$. 
\item \label{Lemma1.10Converse_B}
If one of Conjectures~\ref{MainConj} or \ref{MainConj3} holds for $\CF$, then this conjecture also holds for the semi-simplification $\CF^{ss}$ and more generally for every $\CG\in\dal\CF\dar$.
\item \label{Lemma1.10Converse_C}
If the epimorphism $\pi\colon\Gr(\CF/U)\onto\Gr(\CG/U)$ has \emph{finite} kernel and Conjecture~\ref{MainConj3} holds for $\CG$, then it also holds for $\CF$.
\end{enumerate}
\end{lemma}

\noindent
{\itshape Remark.} Note that part~\ref{Lemma1.10Converse_C} of the lemma might be false for Conjecture~\ref{MainConj}.

\begin{proof}[Proof of Lemma~\ref{Lemma1.10Converse}]
\ref{Lemma1.10Converse_A} and \ref{Lemma1.10Converse_B} are obvious from Lemma~\ref{LemmaCompatibility}\ref{LemmaCompatibility_A} and the definition of the Frobenius conjugacy classes.

\smallskip\noindent
To prove \ref{Lemma1.10Converse_C}, let $S\subset|U|$ be a subset of positive upper Dirichlet density and let $C$ be a connected component of $\Gr(\CG)$ in which $\bigcup_{x\in S}\Frob_x(\CG)$ is dense. Let $C_1,\ldots,C_n$ be the connected components of $\Gr(\CF)$ which map to $C$. By the finiteness assumption on the kernel of $\pi$, the dimensions of all $C_i$ are the same as the dimension of $C$. If $\bigcup_{x\in S}\Frob_x(\CF)$ is not dense in $C_i$ for all $i$, then the closure $Y_i$ of $C_i\cap\bigcup_{x\in S}\Frob_x(\CF)$ has dimension strictly less than $\dim(C_i)$ for all $i$. The images $\pi(Y_i)$ are closed in $C$, because $\pi$ is a finite morphism, and their union contains $\bigcup_{x\in S}\Frob_x(\CG)$ by Lemma~\ref{LemmaCompatibility}\ref{LemmaCompatibility_A}. Since the latter is dense in $C$ and $C$ is irreducible, we must have $C=\pi(Y_i)$ for one $i$. But this contradicts the dimension estimate $\dim\pi(Y_i)=\dim(Y_i)<\dim(C_i)=\dim(C)$, and proves the lemma.
\end{proof}

We end the description of the Reduction Technique~\ref{TechniqueLarger} by giving the

\begin{proof}[Proof of Corollary~\ref{Cor1.9}]
For every $x\in S$ the Frobenius conjugacy class $\Frob_x(\CF)$ maps to $\Frob_x(\CF^{ss})$ under the natural surjective map $r\colon \Gr(\CF)\onto \Gr(\CF^{ss})$ by Lemma~\ref{LemmaCompatibility}. For every $g\in \Gr(\CF)$ we have $\Tr(g)=\Tr(r(g))$ where we take traces with respect to the representations $\omega_\BasePoint(\CF)$ and $\omega_\BasePoint(\CF^{ss})$, because in a suitable basis of $\omega_\BasePoint(\CF)$ the kernel of $r$ consists of unipotent upper triangular matrices by Lemma~\ref{LemmaSemiSimplification}, and so the diagonal entries of $g$ and $r(g)$ coincide. Therefore, we get that $\Tr(\Frob_x(\CF))=\Tr(\Frob_x(\CF^{ss}))$, and consequently $\Tr(\Frob_x(\CF^{ss}))=\Tr(\Frob_x(\CG^{ss}))$ for every $x\in S$.

Let $\rho_1$ and $\rho_2$ denote the representations of $\Gr(\CF^{ss}\oplus\CG^{ss})$ on $\omega_\BasePoint(\CF^{ss})$ and $\omega_\BasePoint(\CG^{ss})$, respectively. For every $x\in S$ the Frobenius conjugacy class $\Frob_x(\CF^{ss}\oplus\CG^{ss})$ maps by Lemma~\ref{LemmaCompatibility} to $\Frob_x(\CF^{ss})$ and $\Frob_x(\CG^{ss})$ under $\rho_1$ and $\rho_2$, respectively. Thus the trace functions of the representations $\rho_1$ and $\rho_2$ on the group $\Gr(\CF^{ss}\oplus\CG^{ss})$ are equal on the Frobenius conjugacy classes $\Frob_x(\CF^{ss}\oplus\CG^{ss})$ for all $x\in S$. By assumption the latter are Zariski-dense in $\Gr(\CF^{ss}\oplus\CG^{ss})$, so the trace functions of the representations $\rho_1$ and $\rho_2$ on the group $\Gr(\CF^{ss}\oplus\CG^{ss})$ are equal.

Let $\Lambda\subset\End_{\olK}\bigl(\omega_\BasePoint(\CF^{ss}\oplus\CG^{ss})\bigr)$ be the smallest $\olK$-linear subspace (viewed as a scheme) containing the image of $\Gr(\CF^{ss}\oplus\CG^{ss})$. Then $\Lambda$ is the $\olK$-linear span of $\Gr(\CF^{ss}\oplus\CG^{ss})$. Moreover, $\omega_\BasePoint(\CF^{ss})$ and $\omega_\BasePoint(\CG^{ss})$ are semi-simple $\Lambda$-modules, because every submodule invariant under $\Gr(\CF^{ss}\oplus\CG^{ss})$ is also invariant under $\Lambda$. Finally, by their linearity the trace functions of $\Lambda$ on both representations coincide, because they do on $\Gr(\CF^{ss}\oplus\CG^{ss})$. Therefore, by \cite[Lemma in \S\,I.2.3 on p.~I-11]{SerreAbelian} the two representations are isomorphic and this implies that $\CF^{ss}\cong\CG^{ss}$.
\end{proof}

\bigskip

\technique{Dividing by the center and the derived group} \label{TechniqueDividing}

\begin{prop}\label{PropGroupOfSum}
Let $\CF,\CG$ be objects of a neutral Tannakian category $\CC$ as in Lemma~\ref{LemmaSemiSimplification}. Then
\begin{enumerate}
\item\label{PropGroupOfSum_A} 
The strictly full sub-category $\dal\CF\dar\cap\dal\CG\dar$ of $\CC$, consisting of all objects $\CH$ which are both isomorphic to an object of $\dal\CF\dar$ and to an object of $\dal\CG\dar$, is a Tannakian sub-category with $\Gr(\dal\CF\dar\cap\dal\CG\dar)=\Gr(\CF\oplus\CG)/N_1N_2$ where $N_1=\ker\bigl(\Gr(\CF\oplus\CG)\onto\Gr(\CF)\bigr)$ and $N_2=\ker\bigl(\Gr(\CF\oplus\CG)\onto\Gr(\CG)\bigr)$.

\item\label{PropGroupOfSum_B} 
$\Gr(\CF\oplus\CG)$ is a closed subgroup of $\Gr(\CF)\TimesQQQ\Gr(\CG)$ which sits in a cartesian diagram of epimorphisms
\begin{equation}\label{EqPropGroupOfSum}
\xymatrix @C=-1pc @R=1pc {
 & \ar@{->>}[ld]\Gr(\CF\oplus\CG)\ar@{->>}[rd] & \\
\Gr(\CF)\ar@{->>}[rd] & \qed\quad &
\ar@{->>}[ld]\;\,\Gr(\CG)\;. \\ 
 & \Gr(\dal\CF\dar\cap\dal\CG\dar) & }
\end{equation}
\end{enumerate}
\end{prop}

\begin{proof}
Write $G:=\Gr(\CF\oplus\CG)/N_1N_2$.

\medskip\noindent
\ref{PropGroupOfSum_A} follows from the obvious facts that tensor products, direct sums, duals, internal Hom's and subquotients of objects in $\dal\CF\dar\cap\dal\CG\dar$ again lie in $\dal\CF\dar\cap\dal\CG\dar$. We prove that $G$ equals $\Gr(\dal\CF\dar\cap\dal\CG\dar)$. Lemma~\ref{LemmaCompatibility}\ref{LemmaCompatibility_C} applied to the epimorphism $\Gr(\CF\oplus\CG)\onto G$ shows that $G$ is the monodromy group $\Gr(\CK)$ of an object $\CK\in\dal\CF\oplus\CG\dar$. Since $\Gr(\CK)$ is also a quotient of $\Gr(\CF)$, respectively of $\Gr(\CG)$, the object $\CK$ is both isomorphic to an object of $\dal\CF\dar$ and to an object of $\dal\CG\dar$, that is, it belongs to $\dal\CF\dar\cap\dal\CG\dar$. This yields an epimorphism $\Gr(\dal\CF\dar\cap\dal\CG\dar)\onto\Gr(\CK)=G$. Conversely, since $\dal\CF\dar\cap\dal\CG\dar$ is contained both in $\dal\CF\dar$ and $\dal\CG\dar$ the map $\Gr(\CF\oplus\CG)\onto\Gr(\dal\CF\dar\cap\dal\CG\dar)$ factors over $\Gr(\CF)$ and over $\Gr(\CG)$. So its kernel contains $N_1$ and $N_2$. This provides the epimorphism in the other direction $G\onto\Gr(\dal\CF\dar\cap\dal\CG\dar)$ and shows that both are isomorphisms.

\medskip\noindent
\ref{PropGroupOfSum_B} The object $\CF\oplus\CG$ corresponds to a faithful representation of $\Gr(\CF\oplus\CG)$ which factors through $\Gr(\CF)\TimesQQQ\Gr(\CG)$, because a tensor automorphism of $\omega_\BasePoint(\CF\oplus\CG)$ is trivial as soon as its restrictions to $\omega_\BasePoint(\CF)$ and $\omega_\BasePoint(\CG)$ are trivial.  Since $\CF$ and $\CG$ are objects of $\dal\CF\oplus\CG\dar$, the two upper arrows in diagram~\eqref{EqPropGroupOfSum} are epimorphisms. We claim that the diagram 
\[
\xymatrix @C=0pc @R=1pc {
 & \ar@{->>}[ld]_{\pi_1}\Gr(\CF\oplus\CG)\ar@{->>}[rd]^{\pi_2} & \\
\Gr(\CF)\ar@{->>}[rd]_{\rho_1} & \qed\quad &
\ar@{->>}[ld]^{\rho_2}\Gr(\CG) \\ 
 & G & }
\]
is cartesian. Since that diagram is commutative, we obtain a morphism from $\Gr(\CF\oplus\CG)$ to the fiber product $\Gr(\CF)\times_G\Gr(\CG)$, which is a closed immersion, because $\Gr(\CF\oplus\CG)\to\Gr(\CF)\TimesQQQ\Gr(\CG)$ is one. Consider a point $(g_1,g_2)\in\Gr(\CF)\times_G\Gr(\CG)$ with $\rho_1(g_1)=\rho_2(g_2)$. Since $\pi_i$ is surjective, there are elements $\tilde g_i\in\Gr(\CF\oplus\CG)$ with $\pi_i(\tilde g_i)=g_i$. The equation $\rho_1\pi_1(\tilde g_1)=\rho_1(g_1)=\rho_2(g_2)=\rho_2\pi_2(\tilde g_2)=\rho_1\pi_1(\tilde g_2)$ shows that $\tilde g_1^{-1}\tilde g_2$ lies in $\ker(\rho_1\pi_1)=N_1N_2$. So there are elements $n_i\in N_i$ with $\tilde g_1^{-1}\tilde g_2=n_1n_2^{-1}$. The element $\tilde g_1n_1=\tilde g_2n_2\in\Gr(\CF\oplus\CG)$ satisfies $\pi_i(\tilde g_in_i)=\pi_i(\tilde g_i)=g_i$ for $i=1,2$. This proves that $\Gr(\CF\oplus\CG)$ is isomorphic to the fiber product $\Gr(\CF)\times_G\Gr(\CG)$.
\end{proof}

We apply the previous proposition in the proof of Theorem~\ref{ThmOverconvIntro} to the following group theoretic situation.

\begin{prop}\label{PropRedGpAlmostProduct}
Let $\CF$ be a \emph{semi-simple} object of a neutral Tannakian category as in Lemma~\ref{LemmaSemiSimplification}. Let $G:=\Gr(\CF)$, let $Z$ be the center of $G\open=\Gr(\CF)\open$ and let $[G\open,G\open]$ be the derived group of $G\open$. Let $\CS,\CT\in\dal\CF\dar$ be the objects whose monodromy groups are $\Gr(\CS)=G/Z$ and $\Gr(\CT)=G/[G\open,G\open]$; see Lemma~\ref{LemmaCompatibility}\ref{LemmaCompatibility_C}. Then $\Gr(\CS)\open$ is semi-simple and has trivial center, and $\Gr(\CT)\open$ is a torus. $\Gr(\dal\CS\dar\cap\dal\CT\dar)=G/G\open$ is a finite group and in the diagram
\begin{equation}\label{EqPropRedGpAlmostProduct}
\Gr(\CF) \;\onto\;\Gr(\CS\oplus\CT)\;\isoto\;G/Z\times_{G/G\open}G/[G\open,G\open]
\end{equation}
there is a natural isomorphism on the right, and the kernel $Z\cap[G\open,G\open]$ of the surjection on the left is finite.
\end{prop}

\begin{proof}
Since $\CF$ is semi-simple, $G$ is reductive by Lemma~\ref{LemmaSemiSimplification}. From Proposition~\ref{PropGroupOfSum} we obtain the isomorphism on the right and the description of $\Gr(\dal\CS\dar\cap\dal\CT\dar)$ as $G/(Z\cdot[G\open,G\open])$, which equals $G/G\open$ by \cite[IV.14.2~Proposition]{Borel91}. The remaining assertions were listed in Lemma~\ref{LemmaUnipotConnected}.
\end{proof}

\smallskip

\technique{Twisting by rank one objects} \label{TechniqueTwist}

\smallskip

For given objects $\CF_1,\ldots,\CF_n$ and rank one objects $\CC_1,\ldots,\CC_n$ of a Tannakian category we set $\CU_i:=\CF_i\otimes\CC_i$ and $\CU:=\bigoplus_{i=1}^n\CU_i$ and $\CC:=\bigoplus_{i=1}^n\CC_i$. Then the object $\CF:=\bigoplus_{i=1}^n\CF_i=\bigoplus_{i=1}^n\CU_i\otimes\CC_i\dual$ belongs to $\dal\CU\oplus\CC\dar$, and hence it suffices to prove Conjectures~\ref{MainConj} and \ref{MainConj3} for $\CU\oplus\CC$ instead of $\CF$ by Lemma~\ref{Lemma1.10Converse}. We use this reduction technique in two situations:

\medskip\noindent
(1) In the proof of Theorem~\ref{ThmIsoclinic} on page~\pageref{ProofOfThmIsoclinic}, the $\CF_i\in\FIsoc_\olK(U)$ are isoclinic convergent $F$-isocrystals of slopes $s_i\in\BQ$. We let $\CC_i$ be the constant $F$-isocrystal of rank $1$ and slope $-s_i$ in order to assume that $\CU_i$ is unit root. Then the image $\BK=\im\bigl(\pi_1^\et(U,\bar\BasePoint)\to\Gr(\CU)\bigr)$ is a $p$-adic Lie group and Zariski-dense in the monodromy group $\Gr(\CU)$, by Corollary~\ref{CorCrewsThm}. On the factor $\Gr(\CU)$ we use analytic tools to prove a lower bound on the point count in (finite quotients of) $\BK$. It comes from the classical Chebotar\"ev Density Theorem for $\pi_1^\et(U,\bar\BasePoint)$, for which we prove an effective version in Theorem~\ref{useful_che}. On the other hand, with the help of Theorem~\ref{ThmCheboForConstants} and Corollary~\ref{Cor4.2} on the factor $\Gr(\CC)$, we derive from Oesterl\'e's result~\cite{Oesterle} an upper bound on the point count in $\BK$. Comparing the bounds, proves the Zariski-density claimed in Theorem~\ref{ThmIsoclinic}. We formulate the abstract, group theoretic and $p$-adic analytic essence of the proof as a separate Theorem~\ref{Thm4.3}.

\medskip\noindent
(2) In the proof of Theorem~\ref{ThmOverconvIntro} in Subsection~\ref{SubsectProofOf1.7}, the $\CF_i\in\FIsoc^\dagger_\olK(U)$ are irreducible overconvergent $F$-isocrystals. We base change all monodromy groups to $\BC$ via an isomorphism $\iota\colon\olK\isoto\BC$. Then we let $\CC_i$ be the constant $F$-isocrystal of rank $1$ such that $\det\CU_i$ is finite. By deep results on the Langlands correspondence of Abe~\cite{Abe13} ($p$-adic) and Lafforgue~\cite{Lafforgue02} ($\ell$-adic), this will imply that $\CU_i=\CF_i\otimes \CC_i$ is an irreducible, overconvergent, $\iota$-pure $F$-isocrystal of weight zero, see Corollary~\ref{langlands_implies_mixedness}. We apply Proposition~\ref{PropRedGpAlmostProduct} to $\CU$. The resulting $\CS$ is still $\iota$-pure of weight zero and we replace $\CU$ by $\CS$. For the resulting $\CT$, we choose a tensor generator $\CE$ of $\dal\CT\dar_{const}$, and replace $\CC$ by $\CC\oplus\CE$, to obtain $\Gr^\dagger(\CU\oplus\CC)\open=\Gr^\dagger(\CU)\open\times\Gr^\dagger(\CC)\open$ as a direct product (as opposed to a fiber product). By Lemma~\ref{Lemma1.10Converse}\ref{Lemma1.10Converse_C}, it will be enough to prove Conjecture~\ref{MainConj3} for $\CU\oplus\CC$ to obtain it for $\CF$. The weight zero part $\CU$ will be treated by the $p$-adic version of Deligne's Equidistribution Theorem on a maximal compact quasi-torus $\BT$ in $\Gr^\dagger(\CU)$, which we prove in Theorem~\ref{ThmEquiDistr}. It relies on the deep theory of Frobenius weights for overconvergent $F$-isocrystals of Kedlaya, Abe and Caro, and will play a role analogous to the analytic input from the classical Chebotar\"ev Density Theorem for point counting on $\BK$ in (1) in the proof of Theorem~\ref{ThmIsoclinic}. Maximal compact quasi-tori are studied in Subsection~\ref{SubSectMaxCompact}. The point counting measures $\mu_m$ on $\BT$, defined as the average of the Dirac measures on the conjugacy classes $\Frob_x(\CU)$ for $\deg(x)=m$, are bounded below by the positive upper Dirichlet density of the set $S\subset|U|$. By the Equidistribution Theorem~\ref{ThmEquiDistr}, the $\mu_m$ converge weakly to the Haar measure on $\BT$. If the Zariski-closure of $\bigcup_{x\in S}\Frob_X(\CU\oplus\CC)$ did not contain a connected component of $\Gr^\dagger(\CU\oplus\CC)$, then its Haar measure would be zero, which yields a contradiction using Theorem~\ref{ThmCheboForConstants}. This contradiction is achieved by using real algebraic geometry on the maximal compact quasi-torus $\BT$ to push forward and pull back measures, and a convergence result for complex hypersurfaces, which we prove in Theorem~\ref{new_oesterle1}. An outline of the proof of Theorem~\ref{ThmOverconvIntro} containing more details will be given in Subsection~\ref{SubsectOutline}.

\section{An Effective Version of the Classical Chebotar\"ev Density Theorem}\label{SectDensity}

In this section we want to make a few remarks on Dirichlet density. By \cite[formula~(18)]{Serre63} we have
\begin{equation}\label{EqHopital}
\lim\limits_{s{\scriptscriptstyle\searrow} \dim U}\frac{-\sum_{x\in |U|}q^{-\deg(x)s}}{\log(s-\dim U)}=1\,
\end{equation}

\begin{defn}\label{dirichlet_density_def} A subset $S\subset|U|$ has \emph{Dirichlet density $\delta(S)$} (in the sense of Serre~\cite[\S\,2.7]{Serre63}) if the limit
\[
\delta(S) \;:=\; \lim_{s{\scriptscriptstyle\searrow} \dim U}\frac{-\sum_{x\in S}q^{-\deg(x)s}}{\log(s-\dim U)}
\]
exists. The \emph{upper Dirichlet density $\ol\delta(S)$} of a subset $S\subset|U|$ is
\[
\ol\delta(S) \;:=\; \limsup_{s{\scriptscriptstyle\searrow} \dim U} \frac{-\sum_{x\in S}q^{-\deg(x)s}}{\log(s-\dim U)}.
\]
By \eqref{EqHopital} the limit superior $\ol\delta(S)$ always exists and is between $0$ and $1$. It is equal to the Dirichlet density of the set $S$, if the latter exists. Moreover, $S$ has \emph{positive (upper) Dirichlet density} if it has Dirichlet density $\delta(S)>0$ (or if $\ol\delta(S)>0$, respectively). Trivially $\ol\delta(R)\leq\ol\delta(S)$ when $R$ is a subset of $S$.
\end{defn}

\begin{lemma}\label{sub_additivity} Let $S\subset|U|$ be a subset, and assume that $S=S_1\cup\cdots\cup S_n,$ with pair-wise disjoint sets $S_i$. Then 
\[
\ol\delta(S) \;\le\; \ol\delta(S_1)+\cdots +\ol\delta(S_n).
\]
Moreover, if every $S_i$ has Dirichlet density $\delta(S_i)$ then $S$ has Dirichlet density $\delta(S_1)+\ldots+\delta(S_n)$. In addition, if $S$ has Dirichlet density $\delta(S)$, then its complement in $U$ has Dirichlet density $1-\delta(S)$.
\end{lemma}

\begin{proof} Note that
$$\frac{-\sum_{x\in S}q^{-\deg(x)s}}{\log(s-\dim U)}=
\frac{-\sum_{x\in S_1}q^{-\deg(x)s}}{\log(s-\dim U)}+\cdots+\frac{-\sum_{x\in S_n}q^{-\deg(x)s}}{\log(s-\dim U)},$$
so by the (sub-)additivity of the limit (superior) we get the lemma.
\end{proof}

\begin{lemma}\label{density_intersection} Let $S\subset|U|$ be a subset of upper Dirichlet density one, and let $R\subset|U|$ be a subset of positive Dirichlet density. Then $R\cap S$ also has positive upper Dirichlet density $\ol\delta(S\cap R)=\delta(R)$. 
\end{lemma}
\begin{proof} Let $R^c\subset|U|$ be the complement of $R$ in $|U|$. Then Lemma~\ref{sub_additivity} yields
$$
1 \;=\; \ol\delta(S) \;\le\; \ol\delta(S\cap R) + \ol\delta(S\cap R^c) \;\le\; \ol\delta(R) + \ol\delta(R^c) \;=\; \delta(R) + \delta(R^c) \;=\; 1\,;
$$
and hence $\ol\delta(S\cap R)=\delta(R)$ proving the claim.
\end{proof}

Since we could not find the following well known statement explicitly in the literature, we include a proof.

\begin{thm}\label{ThmWeilBounds}
Let $U(n):=\{ x\in |U|\colon \deg(x)=n\}$. Then $\#U(n) = \tfrac{q^{n \dim U}}{n} + O\bigl(\tfrac{q^{n(\dim U -1/2)}}{n}\bigr)$ for $n\to\infty$. Moreover, the Zeta-function
\[
Z(U,z)\;:=\;\prod_{x\in |U|}(1-z^{\deg(x)})^{-1}
\]
has a pole of order one at $z=q^{-\dim U}$ and no other pole or zero on the disc $\{z\in \BC\colon |z|<q^{1/2 -\dim U}\}$.
\end{thm}

\begin{proof}
Let $\ell$ be a prime different from the characteristic $p$ of $\BF_q$. By \cite[(1.5.1) and (1.5.4)]{DeligneWeil1} we have $\#U(\BF_{q^n})=\sum\limits_{i=0}^{2\dim U} (-1)^i \Tr\bigl({\rm Fr}_{q,U}^n\big| \Koh^i_c(U\otimes_{\BF_q}\ol\BF_q,\BQ_\ell)\bigr)$ and
\begin{equation*}\label{EqLIsZetaFcn}
Z(U,z)\;=\; \prod_{i=0}^{2\dim U} \det\bigl(1-z \,{\rm Fr}_{q,U}\big| \Koh^i_c(U\otimes_{\BF_q}\ol\BF_q,\BQ_\ell) \bigr)^{(-1)^{i+1}}.
\end{equation*}
By \cite[Th\'eor\`eme~I]{DeligneWeil2} all eigenvalues of ${\rm Fr}_{q,U}^n$ on the \'etale cohomology $\Koh^i_c(U\otimes_{\BF_q}\ol\BF_q,\BQ_\ell) $ have complex absolute value $q^{nm}$ for $m\le i/2$. Therefore, for $0\le i <2\dim U$ the summands of $\#U(\BF_{q^n})$ are in $O(q^{n(\dim U- 1/2)})$, and the factors of $Z(U,z)$ have no zeroes and poles on $\{z\in \BC\colon |z|<q^{1/2 -\dim U}\}$. Moreover, since $U$ is geometrically irreducible, $\Tr\bigl({\rm Fr}_{q,U}^n\big| \Koh^{2\dim U}_c(U\otimes_{\BF_q}\ol\BF_q,\BQ_\ell)\bigr)=q^{n\dim U}$, and $\det\bigl(1-z \,{\rm Fr}_{q,U}\big| \Koh^{2\dim U}_c(U\otimes_{\BF_q}\ol\BF_q,\BQ_\ell)\bigr)=1-z q^{\dim U}$ by \cite[XVIII, Remarque~2.10.1]{SGA4}. This proves the statement for $Z(U,z)$. Moreover, we obtain $\#U(\BF_{q^n})=q^{n\dim U}+O(q^{n(\dim U- 1/2)})$. To compute $\#U(n)$, we remove all points $x\in U(\BF_{q^n})$ for which $\deg(x)$ is a proper divisor of $n$. This set has cardinality $\le O(q^{\frac{n}{2}\dim U})$. Dividing by $n$ for the remaining points yields the proposition, because $\#\{x \in U(\BF_{q^n})\colon \deg(x)=n\}=n\cdot \#U(n)$.
\end{proof}

We will also need the following

\begin{lemma} \label{Lemma4.3} Let $S\subset|U|$ be a set with $\ol\delta(S)>0$. Then there is an infinite subset $R\subset\BN$ such that for every $n\in R$
$$\#\{x\in S\colon\deg(x)=n\} \;\geq\; \frac{\ol\delta(S)\cdot \#U(n)}{2}.$$
\end{lemma}

\begin{proof} Assume that the claim is false. Then there is a positive integer $m$ such that 
$$\#\{x\in S\colon\deg(x)=n\} \;<\;
\frac{\ol\delta(S)\cdot \#U(n)}{2}$$
for every $n>m$. Thus for every $s\in \BR$ with $\dim U < s< 1+\dim U$ we have:
$$\sum_{x\in S}q^{-\deg(x)s}
\;<\; \sum_{x\in S\colon\deg(x)\leq m}q^{-\deg(x)s}+\frac{\ol\delta(S)}{2}\sum_{n>m}\#U(n)q^{-ns}.$$
Since $\log(s-\dim U)$ is negative for such $s$, we get from the above that 
$$\ol\delta(S)\;\leq\;
\limsup_{s{\scriptscriptstyle\searrow} \dim U}
\frac{-\sum_{\deg(x)\leq m}q^{-\deg(x)s}}{\log(s-\dim U)}+
\frac{\ol\delta(S)}{2}\limsup_{s{\scriptscriptstyle\searrow} \dim U}
\frac{-\sum_{n>m}\#U(n)q^{-ns}}
{\log(s-\dim U)}\,.$$
The first limit on the right hand side is zero, while the second limit is $1$ by \eqref{EqHopital}. However, the resulting inequality $\ol\delta(S)\le\ol\delta(S)/2$ is a contradiction to $\ol\delta(S)>0$.
\end{proof}

The following well known property of the (upper) Dirichlet density obstructs the technique of replacing $U$ by a finite \'etale Galois covering.

\begin{example}\label{ExDensity}
Let $f\colon V\to U$ be a finite \'etale Galois covering of degree $n$, where $n$ is a prime number, let $S\subset|U|$ be a subset of upper Dirichlet density $\ol\delta(S)$, and let $S'=f^{-1}(S)\subset|V|$ be the preimage of $S$ under $f$.

\smallskip\noindent
(a) Assume that every $x\in S$ splits in $V$, that is, there are exactly $n$ points $x'$ of $V$ lying above $x$. They have $\BF_{x'}=\BF_x$ and $\deg(x')=\deg(x)$. We compute
\[
\sum_{x'\in S'}q^{-\deg(x')s}\;=\;n\cdot\sum_{x\in S}q^{-\deg(x)s}\,.
\]
Therefore, $S'$ has upper Dirichlet density $\ol\delta(S')=n\cdot\ol\delta(S)$.

\smallskip\noindent
(b) Assume that every $x\in S$ is inert in $V$. Then there is exactly one point $x'$ of $V$ lying above each $x\in S$, and it has $[\BF_{x'}:\BF_x]=n$ and $\deg(x')=n\deg(x)$. We compute
\[
\sum_{x'\in S'}q^{-\deg(x')s}\;=\;\sum_{x\in S}q^{-n\deg(x)s}\,.
\]
When $s{\scriptscriptstyle\searrow} \dim U$ this sum converges in $\BR$, whereas $\log(s-\dim U)$ goes to $\infty$. Therefore, $S'$ has upper Dirichlet density $\ol\delta(S')=0$.

\medskip
This is of course analogous to the situation for number fields, where one says that a subset $S\subset|M|$ of the set $|M|$ of places of a number field $M$ has Dirichlet density $\delta$ if 
\[
\liminf_{m\rightarrow\infty} \frac{\#\{x\in S\colon N(x)\le m\}}{\#\{x\in|M|\colon N(x)\le m\}} \;=\;\delta\,,
\]
with $N(x):=\#\CO_M/x$ denoting the norm of $x$. Let $N/M$ be a Galois extension of number fields whose degree $n$ is a prime number. Then the set $S$ of places in $M$ which split completely in $N$ has Dirichlet density $\delta(S)=\tfrac{1}{n}$, whereas its preimage $S'\subset|N|$ has Dirichlet density $\delta(S')=1$. The reason is that above every place $x\in S$ there are exactly $n$ places in $S'$ with the same norm as $x$.

On the other hand the set $S$ of places in $M$ which are inert in $N$ has Dirichlet density $\delta(S)=\tfrac{n-1}{n}$, whereas its preimage $S'\subset|N|$ has Dirichlet density $\delta(S')=0$. The reason is that above every place $x\in S$ there is exactly one place $x'$ in $S'$ whose norm is $N(x')=N(x)^n$.
\end{example}

\bigskip

For applications to isoclinic $F$-isocrystals we will need an effective version of the classical Chebotar\"ev Density Theorem for function fields.

\begin{notn} 
Consider a representation $\pi_1^\et(U,\bar\BasePoint)\onto \BG$ onto a finite group $\BG$. Let $\BG^{\rm geo}$ denote the image of the geometric \'etale fundamental group $\pi_1^\et(U,\bar\BasePoint)^{\rm geo}\,:=\,\pi_1^\et(U\otimes_{\BF_q}\ol\BF_q,\bar\BasePoint)\,\subset\,\pi_1^\et(U,\bar\BasePoint)$ and let $\BG^c=\BG/\BG^{\rm geo}$. It is the maximal constant quotient of $\BG$, that is, the largest quotient of $\BG$ which can be pulled back under $\pi_1^\et(U,\bar\BasePoint)\to\Gal(\ol{\BF}_q/\BF_q)$. In particular, $\BG^c$ is finite cyclic. There is an isomorphism $\iota_\BG\colon \BG^c\isoto\BZ/\#\BG^c$ which maps the geometric Frobenius in $\Gal(\ol\BF_q/\BF_q)$ to $1$.

Let $C\subset \BG$ be a conjugacy class. Then the image of $C$ under $\BG\to \BG^c$ is an element in $\BG^c$ which we will denote by $C^c$. In particular, for every $x\in|U|$ let $\Frob_x^{-1}$ be the geometric Frobenius at $x$ in $\pi_1^\et(U,\bar\BasePoint)$ which maps $a\in\ol\BF_x$ to $a^{1/\#\BF_x}$, and let $\textrm{Fr}_x^\BG\subset \BG$ denote the image of the conjugacy class of $\Frob_x^{-1}$. Then $(\textrm{Fr}_x^\BG)^c$ maps under $\iota_\BG$ to $\deg(x)\mod \#\BG^c$. 
\end{notn}

\begin{rem}
The notation $\BG^{\rm geo}$ follows the same spirit as the notation $\Gr(\CF)^{\rm geo}$ in Definition~\ref{Def3.5} and $\BK^{\rm geo}$ in Notation~\ref{Notation13.14} below. Namely, it denotes a closed normal subgroup of $\BG$, respectively $\Gr(\CF)$ or $\BK^{\rm arithm}$, such that the respective quotient is the largest quotient arising from constant objects, that is, from constant representations of the \'etale fundamental group, respectively from constant $F$-isocrystals. For example, if $\CF$ is a convergent unit-root $F$-isocrystal, then the compatibility of the notation was observed in Corollary~\ref{CorCrewsThmGeo}, where we saw that $\Gr(\CF)$, respectively $\Gr(\CF)^{\rm geo}$ are the Zariski closures of the images of $\pi_1^\et(U,\bar\BasePoint)$, respectively $\pi_1^\et(U,\bar\BasePoint)^{\rm geo}$.
\end{rem}

We use the following estimate, which can be found for example in \cite[Theorem~4.1]{Chavdarov}.

\begin{prop} \label{PropChebotarev} Fix a conjugacy class $C\subset \BG$. Then as $n$ goes to infinity we have 
\[
\#\{x\in |U|\colon\mathrm{Fr}_x^{\BG}\subset C,\ \deg(x)=n \} \;=\;
\begin{cases}
\frac{\#U(n)\cdot \#C}{\#\BG^{\rm geo}}+O(\#U(n)/\sqrt{q^n}) & \text{, if } n\equiv
\iota_\BG(C^c)\mod \#\BG^c,\\
0 & \text{, otherwise. } \end{cases}
\]
\end{prop}

\begin{proof} Let $n\equiv\iota_\BG(C^c)\mod \#\BG^c$, because otherwise the set on the left is obviously empty. Then \cite[Theorem~4.1]{Chavdarov} says that
\[
\#\{x\in U(\BF_{q^n})\colon(\mathrm{Fr}_x^{\BG})^{n/\deg(x)}\subset C \}\;=\;\frac{\#U(\BF_{q^n})\cdot \#C}{\#\BG^{\rm geo}}+O(\#U(\BF_{q^n})/\sqrt{q^n})\,.
\]
As in the proof of Theorem~\ref{ThmWeilBounds}, passing from $U(\BF_{q^n})$ to $U(n)$ yields an additional factor $n$ in the denominator.
\end{proof}

The aforementioned effective version is the following

\begin{thm} \label{useful_che}
Let $S\subset|U|$ be a set of positive upper Dirichlet density. Then there is a positive constant $\epsilon>0$ such that for every finite quotient group $\BG$ of $\pi_1^\et(U,\bar\BasePoint)$ there is an infinite subset $R_\BG \subset \BN$  such that for every $n\in R_\BG$ the union in $\BG$ of the Frobenius conjugacy classes ${\rm Fr}_x^\BG$ for all $x$ in $\{x\in S\colon\deg(x)=n\}$ has cardinality at least $\epsilon\cdot\#\BG^{\rm geo}$. 
\end{thm}

\begin{proof} We claim that $\epsilon=\frac{\ol\delta(S)}{4}$ will do. Now assume that the claim is false, and let $\BG$ be a finite quotient of $\pi_1^\et(U,\bar\BasePoint)$ violating the assertion of the theorem. By Lemma~\ref{Lemma4.3} there is an infinite subset $R\subset\BN$ such that for every $n\in R$
\begin{equation}\label{4.1.0}
\#\{x\in S\colon\deg(x)=n\} \;\geq\; \frac{\ol\delta(S)\#U(n)}{2}.
\end{equation}
For every $n\in\BN$ let $F_n$  be the union in $\BG$ of the Frobenius conjugacy classes ${\rm Fr}_x^\BG$ for all $x\in S$ with $\deg(x)=n$. It decomposes into a disjoint union of conjugacy classes $C$ in $\BG$. We apply Proposition~\ref{PropChebotarev} to each component $C$ of $F_n$ and add. By removing finitely many elements of $R$ if necessary, this tells us for all $n\in R$ that
\begin{equation}\label{4.1.1}
\#\{x\in S\colon\deg(x)=n\} \;\le\; \#\{x\in |U|\colon\mathrm{Fr}_x^{\BG}\subset F_n,\ \deg(x)=n \} \;\leq\; \frac{2\#U(n)\cdot\# F_n}{\#\BG^{\rm geo}}\,.
\end{equation}
Since we assumed that $\BG$ is a counter-example, by shrinking $R$ further, if this is necessary, we may assume that
\[
\#F_n \;<\; \frac{\ol\delta(S)\# \BG^{\rm geo}}{4} 
\]
for all $n\in R$. Together with \eqref{4.1.1} we get
\[
\#\{x\in S\colon\deg(x)=n\} \;<\; \frac{2\#U(n)}{\#\BG^{\rm geo}}\cdot\frac{\ol\delta(S)\#\BG^{\rm geo}}{4} \;=\;\frac{\ol\delta(S)\#U(n)}{2}\,.
\]
But this contradicts \eqref{4.1.0}, and hence the theorem follows.
\end{proof}

\section{Chebotar\"ev for Direct Sums of Isoclinic $F$-Isocrystals} \label{SectIsoclinic}

In this section we will prove Theorem~\ref{ThmIsoclinic} by working with $p$-adic analytic manifolds and Lie groups and using a point counting result of Oesterl\'e~\cite{Oesterle}. Again all undecorated fiber products $\times$ are over the field $L$ defined in

\subsection{An Abstract Group Theoretic and $p$-Adic Formulation of the Proof}

\begin{setup}\label{Setup4.3}
Let $L$ be a finite field extension of $\BQ_p$ and let $T$ be a commutative linear algebraic group over $L$ which is the product of a split torus $\BG_{m,L}^r$ for $r\ge0$ with a finite abelian group and possibly an additional factor $\BG_{a,L}$. Let $\Gamma\subset T(L)$ be a Zariski-dense subgroup in $T$, which is infinite cyclic, $\Gamma\cong\BZ$, and let $\gamma$ be a generator of $\Gamma$. We write the group $\Gamma$ additively as $\Gamma=\{n\gamma\text{ for }n\in\BZ\}$. Let $H$ and $\olT$ be linear algebraic groups over $L$ such that $\olT$ is of the same kind as $T$ and assume that there are two surjective homomorphisms $\phi_1\colon H\to \olT$ and $\phi_2\colon T\to \olT$ of algebraic groups over $L$. Let $G:=H\times_\olT T$ be the fiber product of these two maps. Note that $G$ is a closed algebraic subgroup of $H\TimesL T$. We denote by $pr_1\colon G\to H$ and $pr_2\colon G\to T$ the projections. Let $\BK\subset H(L)$ be a compact, but not necessarily open subgroup with respect to the $p$-adic topology, let $H^{\rm geo}\subset H$ be the kernel of $\phi_1$ and set $\BK^{\rm geo}=H^{\rm geo}(L)\cap \BK$. Finally let $F\subset(\BK\times \Gamma)\cap G(L)$ be a subset. 
\[
\xymatrix {
& & **{!R(0.27) =<1pc,2pc> } \objectbox{F\;\subset\;(\BK\times \Gamma)\cap G(L) \;\subset\; G\;=\; H\times_\olT T \quad} \ar[dl]_{pr_1} \ar[dr]^{pr_2} & \ar@{^{ (}->}[rrrr] & & & & H\TimesL T \\
& **{!R(0.5) =<5pc,2pc>} \objectbox{\BK\;\subset\;H(L)\;\subset\; H} \ar@{->>}[dr]_{\phi_1} & & **{!L(0.55) =<5pc,2pc>} \objectbox{\quad T\;\supset\;T(L)\;\supset\;\Gamma \;\ni\;n\gamma\,,} \ar@{->>}[dl]^{\phi_2} & & & & \BK\times \Gamma \ar@{^{ (}->}[u]\\
& & \olT
}
\]
Note that $pr_1$ induces an isomorphism between $\ker(pr_2)\subset G$ and $H^{\rm geo}$ whose inverse is given as
\begin{equation}\label{EqIsomHgeo}
H^{\rm geo} \;\isoto \;\ker(pr_2)\;\subset\; G\;= \;H\times_\olT T\,,\quad y\;\mapsto \;(y,1)\,.
\end{equation}
\end{setup}

\begin{thm} \label{Thm4.3}
Assume that the following hold:
\begin{enumerate}
\item \label{Thm4.3_A}
the Zariski-closure of $\BK^{\rm geo}$ is $H^{\rm geo}$,
\item \label{Thm4.3_B}
there is a positive constant $\epsilon>0$ such that for every open normal subgroup $\BJ\subset \BK$ there is an infinite subset $R_\BJ\subset\BN$ such that for every $n\in R_\BJ$ the image of $pr_1(F\cap pr_2^{-1}(n\gamma))\subset \BK$ under the quotient map $\BK\to \BK/\BJ$ has cardinality at least
$\epsilon\cdot\#(\BK^{\rm geo}/\BJ\cap \BK^{\rm geo})$. 
\end{enumerate}
Then the Zariski-closure of $F$ contains a connected component of $G$. 
\end{thm}

This theorem is the group theoretic essence needed to give the 

\begin{proof}[Proof of Theorem~\ref{ThmIsoclinic}] \refstepcounter{proofnumber}\label{ProofOfThmIsoclinic}
Let $\CF=\bigoplus_i\CF_i\in \FIsoc_\olK(U)$ be a direct sum of isoclinic convergent $F$-isocrystals $\CF_i$ on $U$. Recall that we have fixed a base point $\BasePoint\in U$ with $\deg(\BasePoint)=\BasePtDeg$, and $\omega_\BasePoint$ is a $K_\BasePtDeg$-rational fiber functor on $\FIsoc_K(U)$ that gives rise to the $\olK$-rational fiber functor $\omega_\BasePoint$ on $\FIsoc_\olK(U)=\FIsoc_K(U)\otimes_K \olK$. There is a finite field extension $L\subset\olK$ of $K_\BasePtDeg$ such that the $\CF$, $\CF_i$ arise as scalar extensions under $L\into \olK$ from $L$-linear $F$-isocrystals on $U$, which we denote again by $\CF,\CF_i\in \FIsoc_L(U)$ by abuse of notation. In the following we will work in the $L$-linear Tannakian category $\FIsoc_L(U)$ which is neutral with respect to the $L$-rational fiber functor $\omega_\BasePoint$. All monodromy groups will be linear algebraic groups over $L$. Let $\tfrac{m_i}{n}$ be the slope of $\CF_i$ with $m_i\in\BZ$ and $n\in\BN$. We may enlarge $L$ if necessary such that $p^{m_i/n}\in L$. Then $\CU_i:=\CF_i\otimes\CC_i\dual$ is unit-root where $\CC_i$ is the pullback to $U$ of the constant $F$-isocrystal on $\Spec\BF_q$ given by $(L,F=p^{m_i/n})$. Then $\CU:=\bigoplus_i \CU_i$ is unit-root and $\CC:=\bigoplus_i\CC_i$ is a direct sum of constant $F$-isocrystals of rank one. Moreover, $\CF$ lies in the Tannakian category $\dal\CU\oplus\CC\dar$, and so the monodromy group $\Gr(\CU\oplus\CC)$ surjects onto $\Gr(\CF)$. By Lemma~\ref{Lemma1.10Converse} it will be enough to see that Conjecture~\ref{MainConj3} holds for $\CU\oplus\CC$. Let $\dal\CU\dar_{const}$ be the full Tannakian sub-category of constant $F$-isocrystals in $\dal\CU\dar$. It is generated by some constant unit-root $F$-isocrystal $\wt\CC$. Then $\dal\wt\CC\dar\subset\dal\CU\dar\cap\dal\CC\oplus\wt\CC\dar\subset\dal\CU\dar_{const}=\dal\wt\CC\dar$. We now replace $\CC$ by $\CC\oplus\wt\CC$ and thus may assume that $\dal\CU\dar\cap\dal\CC\dar=\dal\CU\dar_{const}$. We can assume that $\CC$ is itself not unit-root, by adding a constant $F$-isocrystal with non-zero slope to it if necessary. Let $\rho\colon \pi_1^\et(U,\bar\BasePoint)\to\GL_r(L)$ be the representation corresponding to $\CU$ from Theorem~\ref{ThmCrewsThm}. By Corollary~\ref{CorCrewsThm} this induces a tensor equivalence between $\dal\CU\dar$ and $\dal\rho\dar$, and after possibly enlarging $L$ there is an isomorphism $\beta\colon\omega_f|_{\dal\rho\dar} \isoto\omega_\BasePoint|_{\dal\CU\dar}$ between the forgetful fiber functor $\omega_f$ on $\dal\rho\dar$ and the fiber functor $\omega_\BasePoint$ on $\dal\CU\dar$.

The proof of Conjecture~\ref{MainConj3} for $\CU\oplus\CC$ will rely on Theorem~\ref{Thm4.3}. Thus we will resort to Setup~\ref{Setup4.3}. Let $H=\Gr(\CU)$ and $T=\Gr(\CC)$, and let $G=\Gr(\CU\oplus\CC)$ and $\olT=\Gr(\dal\CU\dar\cap\dal\CC\dar)=\Gr(\dal\CU\dar_{const})$. Then $H^{\rm geo}:=\ker(H\to \olT)$ is the geometric monodromy group $\Gr(\CU)^{\rm geo}$ of $\CU$ from Definition~\ref{Def3.5}, and $G$ is the fiber product $H\times_\olT T$ by Proposition~\ref{PropGroupOfSum}\ref{PropGroupOfSum_B}. After enlarging $L$ if necessary, the groups $\olT$ and $T$ are each the product of a split torus with a finite group and possibly an additional factor $\BG_{a,L}$ by Theorem~\ref{ThmMonodrOfConstant}\ref{ThmMonodrOfConstant_B}. The isomorphism $\beta$ induces an isomorphism $\beta_*\colon\Aut^\otimes(\omega_f|_{\dal\rho\dar}) \isoto\Gr(\CU)$ and the image $\BK:=\beta_*\circ\rho\bigl(\pi_1^\et(U,\bar\BasePoint)\bigr)$ of the induced representation $\beta_*\circ\rho\colon\pi_1^\et(U,\bar\BasePoint)\to \Gr(\CU)(L)$ is Zariski-dense in $\Gr(\CU)$ by Corollary~\ref{CorCrewsThm}. This image is a compact group, because $\pi_1^\et(U,\bar\BasePoint)$ is pro-finite. By adding a constant unit-root $F$-isocrystal to $\CU$ as in Corollary~\ref{CorCrewsThmGeo2}, we can assume that $\BK^{\rm geo}:=H^{\rm geo}(L)\cap \BK$ equals the image $\beta_*\circ\rho\bigl(\pi_1^\et(U,\bar\BasePoint)^{\rm geo}\bigr)$. Then $\BK^{\rm geo}$ is Zariski-dense in $H^{\rm geo}$ by Corollary~\ref{CorCrewsThmGeo}.

For every $x\in|U|$ the Frobenius conjugacy class $\Frob_x(\CU)\subset \Gr(\CU)$ is by Corollary~\ref{CorCrewsThm} generated by the image with respect to $\beta_*\circ\rho$ of the conjugacy class in $\pi_1^\et(U,\bar\BasePoint)$ of the geometric Frobenius $\Frob_x^{-1}$ at $x$. If $\gamma\in T(L)$ is the image of the Frobenius of $\CC$ then $\Frob_x(\CC)$ consists of the single element $\deg(x)\cdot \gamma$ for every $x\in |U|$ by Theorem~\ref{ThmMonodrOfConstant}\ref{ThmMonodrOfConstant_C}. Since $\CC$ was assumed to be not unit-root, the group $\Gamma\subset T(L)$ generated by $\gamma$ is infinite cyclic and dense in $T$, see Theorem~\ref{ThmMonodrOfConstant}\ref{ThmMonodrOfConstant_B}. Consider the Frobenius conjugacy class $\Frob_x(\CU\oplus\CC)\subset\Gr(\CU\oplus\CC)$. By Lemma~\ref{LemmaCompatibility} it is mapped to the conjugacy classes of $\Frob_x(\CU)$ in $H$, respectively of $\Frob_x(\CC)$ in $T$. Therefore, $\Frob_x(\CU\oplus\CC)$ has a representative in $\BK\times\Gamma\subset G(L)$. We now let $S\subset|U|$ be a subset of positive upper Dirichlet density and set 
\[
F\;=\;\bigcup_{x\in S}\Frob_x(\CU\oplus\CC)\cap (\BK\times \Gamma) \;\subset\; (\BK\times \Gamma)\cap G(L)\,.
\]
Theorem~\ref{ThmIsoclinic} is therefore a consequence of the following 

\medskip\noindent
{\itshape Claim.} The octuple $(T,H,\olT,G,\BK,\Gamma,\gamma,F)$ satisfies the hypothesis of Theorem~\ref{Thm4.3}.

\medskip\noindent
Condition~\ref{Thm4.3_A} was already checked. Denoting the morphism $G\to H$ by $pr_1$ and the morphism $G\to T$ by $pr_2$, we see in particular, that for $n\in\BN_{>0}$
\begin{align*}
F\cap pr_2^{-1}(n\gamma)\;=\enspace & \bigcup_{x\in S\colon\deg(x)=n}\Frob_x(\CU\oplus\CC)\cap (\BK\times \Gamma) \qquad\text{and} \\[2mm]
pr_1\bigl(F\cap pr_2^{-1}(n\gamma)\bigr)\;=\enspace & \bigcup_{x\in S\colon\deg(x)=n}\Frob_x(\CU)\cap \BK\\[2mm]
\;\supset\enspace & \beta_*\circ\rho\Biggl(\;\bigcup_{x\in S\colon\deg(x)=n}\text{conjugacy class of }\Frob_x^{-1}\text{ in }\pi_1^\et(U,\bar\BasePoint)\Biggr)\,.
\end{align*}
Therefore, condition~\ref{Thm4.3_B} of Theorem~\ref{Thm4.3} follows from Theorem~\ref{useful_che}. This proves the claim and Theorem~\ref{ThmIsoclinic}.
\end{proof}

\subsection{Proof of Theorem~\ref{Thm4.3}}
The idea of the proof is that under assumption~\ref{Thm4.3_A} Zariski-closed subsets $Y$ of $G$ not containing a connected component of $G$ have lower-dimensional intersection with $\BK$, and hence meet $\BK$ only in few points as made precise in Proposition~\ref{Prop4.5}. This will follow from a point counting result of Oesterl\'e~\cite{Oesterle} and will lead to a contradiction with condition~\ref{Thm4.3_B}. We will first explain Oesterl\'e's result, then analyze condition~\ref{Thm4.3_A}, and then formulate and prove Proposition~\ref{Prop4.5}.

To explain Oesterl\'e's result let $\BZ_p\langle z_1,\ldots,z_N\rangle$ denote the integral Tate ring of restricted power series in $N$ variables, i.e.~the ring of formal power series whose coefficients converge to zero $p$-adically, and let  $\BQ_p\langle z_1,\ldots,z_N\rangle=\BZ_p\langle z_1,\ldots,z_N\rangle\otimes_{\BZ_p}\BQ_p$ denote the Tate algebra over $\BQ_p$ in $N$ variables. Let $0\ne f\in \BQ_p\langle z_1,\ldots,z_N\rangle$ and let
\[
\Var(f)\;:=\;\{\, x=(x_1,\ldots,x_N)\in\BZ_p^N\colon f(x)=0\,\}\;\subset\;\BZ_p^N
\]
be the analytic hypersurface defined by $f$. Note that $\dim \Var(f)=N-1$ by \cite[\S\,5.2.4, Proposition~1, \S\,5.2.2, Theorem~1, \S\,5.2.3, Proposition~3 and the Remark after \S\,6.1.2, Corollary~2]{BGR}. For all $\nu\in\BN_{>0}$ let $\Var(f)_\nu$ denote the image of $\Var(f)$ in  $(\BZ_p/p^\nu\BZ_p)^N$. By multiplying with a suitable constant we normalize $f$ in such a way that all its coefficients are in $\BZ_p$, but not all are in $p\BZ_p$. The reduction of $f$ modulo $p$ is a non-zero polynomial in $z_1,\ldots,z_N$ with coefficients in $\BF_p=\BZ_p/(p)$. Let $\deg(f)$ denote the degree of this polynomial and call it the \emph{Oesterl\'e degree} of $f$. Then Oesterl\'e~\cite[Theorem~4]{Oesterle} proves the following inequality on the cardinality of $\Var(f)_\nu$:
\begin{equation}\label{EqOesterle}
\#\Var(f)_\nu\;\leq\;\deg(f)p^{\nu(N-1)}\quad\text{for every $\nu>0$.}
\end{equation}

To apply this result we need a bound on the Oesterl\'e degree $\deg(f)$ of $f$ as in the following

\begin{prop} \label{PropUniformOesterleBound}
Let $V$ be a finite dimensional $\BQ_p$-linear subspace of the ring $\BQ_p\langle z_1,z_2,\ldots,z_N\rangle$. Then there is a constant $c_V$ only depending on $V$ such that for every non-zero $f\in V$ the Oesterl\'e degree $\deg(f)$ of $f$ is at most $c_V$.
\end{prop}

\begin{proof} 
For $f=\sum_{\ul k} a_{\ul k} \,z_1^{k_1}\cdots z_N^{k_N}\in\BQ_p\langle z_1,\ldots,z_N\rangle$ the Gau{\ss} norm of $f$ is defined as $\|f\|:=\max\{|a_{\ul k}|\colon \ul k\in(\BN_{\ge 0})^N\}$. By \cite[Theorem~50.8 and Proposition~50.4(ii)]{Schikhof} the unit ball $B:=\{f\in V\colon \|f\|\le 1\}\subset V$ with respect to the Gau{\ss} norm is a finitely generated $\BZ_p$-lattice in $V$, having as generating system an orthogonal basis $f_1,\ldots, f_r$ with $\|f_i\|=1$ for all $i$, because the Gau{\ss} norm takes values in $|\BQ_p|$. Let $c_V$ be the maximum of the Oesterl\'e degrees of $f_1,\ldots, f_r$. Since every $f\in V$ with $\|f\|=1$ is a $\BZ_p$-linear combination of the $f_i$, the Oesterl\'e degree of every $0\ne f\in V$ is less or equal to $c_V$.
\end{proof}

As a further preparation we recall the $p$-adic version of Cartan's theorem:

\begin{thm} \label{Cartan1} 
Let $G$ be a linear algebraic group over a finite field extension $L$ of $\BQ_p$, and let $\BK\subset G(L)$ be a subgroup which is compact in the $p$-adic topology. Then $\BK$ is a Lie group over $\BQ_p$.
\end{thm}

\begin{proof} 
Since $\BK$ is compact, it is closed in the Hausdorff space $G(L)=(\Res_{L/\BQ_p}G)(\BQ_p)$ which is an analytic group over $\BQ_p$, so the claim follows from \cite[Part~II, \S\,V.9, Corollary to Theorem~1 on page~155]{Serre92}.
\end{proof}

\begin{defn}\label{DefSerreStandardGp}
We will recall what Serre calls a \emph{standard group}; see \cite[Part~II, \S\,IV.8]{Serre92}. Let $F=(F_1,F_2,\ldots,F_N)\in\BZ_p\dbl x_1,\ldots,x_N,y_1,\ldots,y_N\dbr^N$ be a formal group law in $N$ variables over $\BZ_p$. Serre equips $(p\BZ_p)^N$ with the structure of a Lie group over $\BQ_p$. We need a re-normalization which identifies $\BZ_p$ with $p\BZ_p$ by multiplication with $p$. So we equip the $p$-adic analytic space $\BH_F:=\BZ_p^N$ of dimension $N$ with the structure of a Lie group over $\BQ_p$ where the multiplication is given by the formula
\begin{equation}\label{EqSerreStandardGp}
x\cdot_F \,y\;:=\;\tfrac{1}{p}\cdot F(px,py)\quad\text{for }x,y\in \BH_F:=\BZ_p^N,
\end{equation}
and the identity is $(0,0,\ldots,0)$.
\end{defn}

The proof of Theorem~\ref{Thm4.3} will proceed by exhibiting an open subgroup $\BK_1$ of $\BK^{\rm geo}$ that is Zariski-dense in $(H^{\rm geo})\open$ and isomorphic to $\BZ_p^N$ equipped with a formal group law. If there was a closed subset $Y$ of $G$ not containing a connected component of $G$ we show that $Y\cap\BK_1$ would be contained in a hypersurface of bounded Oesterl\'e degree. By applying Oesterl\'e's result \eqref{EqOesterle} and summing over the $\BK_1$ cosets in $\BK$ we give a precise upper bound on the number of points in $\BK$ in Proposition~\ref{Prop4.5}. This will contradict condition~\ref{Thm4.3_B}.

For this purpose, we next analyze condition~\ref{Thm4.3_A}. Using the isomorphism \eqref{EqIsomHgeo} it implies that every connected component of $\ker(pr_2)$ contains an $L$-rational point. Since $\BK$ is compact, and $\BK^{\rm geo}$ and $\BK^{\rm geo}_0:=\BK\cap (H^{\rm geo})\open(L)$ are closed subgroups we get that $\BK^{\rm geo}$ and $\BK^{\rm geo}_0$ are also compact. Then $\BK^{\rm geo}_0\subset \BK^{\rm geo}\subset \BK\subset H(L)$, and they are Lie groups over $\BQ_p$ by Theorem~\ref{Cartan1}. Note that $\BK^{\rm geo}_0\subset \BK^{\rm geo}$ has finite index, and hence is open as the kernel of the homomorphism $\BK^{\rm geo}\to H^{\rm geo}/(H^{\rm geo})\open$. By \cite[Part~II, \S\,IV.8, Theorem]{Serre92} there is an open subgroup $\BK_1$ of $\BK^{\rm geo}_0$ which is standard in the sense of Definition~\ref{DefSerreStandardGp}. This means that there is a formal group law $F$ in $N$ variables over $\BZ_p$ and an isomorphism
\begin{equation}\label{EqIsomPsi}
\psi\colon \BH_F\isoto \BK_1
\end{equation}
of Lie groups over $\BQ_p$, where $\BH_F=\BZ_p^N$ is equipped with the group law \eqref{EqSerreStandardGp} given by $F$. Since $\BK^{\rm geo}$ is compact the index $[\BK^{\rm geo}:\BK_1]$ is finite. Thus, assumption~\ref{Thm4.3_A} of Theorem~\ref{Thm4.3} implies that $\BK_1$ is dense in $(H^{\rm geo})\open$, as $(H^{\rm geo})\open$ is irreducible.

To establish Theorem~\ref{Thm4.3} we will choose closed embeddings $H\subset\BA^a_{L}$ and $T\subset\BA^b_{L}$ into affine spaces, and consider the induced embedding $G\subset H\TimesL T\subset\BA^{a+b}_{L}$. Recall that by \cite[Part~II, \S\,IV.9, Theorem~1]{Serre92} for every $\nu\in\BN_{>0}$ the subset $p^\nu\BZ_p^N\subset \BH_F=\BZ_p^N$ is actually an open normal subgroup under the group law \eqref{EqSerreStandardGp} given by $F$, and two elements $x,y$ of $\BH_F$ are congruent to each other modulo the subgroup $p^\nu\BZ_p^N$ if and only if $x\equiv y\mod p^{\nu}$. In the next Proposition~\ref{Prop4.5} we derive from Oesterl\'e's result \eqref{EqOesterle} that a closed subset of $\BA^{a+b}_{L}$ which does not contain an irreducible component of $G$ only contains few points in a precise sense which will then contradict condition~\ref{Thm4.3_B}. 
 
\begin{prop} \label{Prop4.5}
Let $T^c$ be a connected component of $T$ and let $Y\subset pr_2^{-1}(T^c)\subset\BA^{a+b}_{L}$ be a closed subset which does not contain any irreducible component of $pr_2^{-1}(T^c)$. Then there is a positive integer $d_{c,Y}$ such that for all but finitely many $n\in\BZ$ and for every $\nu\in\BN_{>0}$ the image of $pr_1(Y\cap pr_2^{-1}(n\gamma))\cap \BK$ under the quotient map $\BK\onto \BK/\psi(p^\nu\BZ_p^N)$ has cardinality at most $d_{c,Y}\cdot p^{\nu(N-1)}$. 
\end{prop}

\begin{proof} 
We can write $Y\subset\BA^{a+b}_{L}$ as an intersection of finitely many hypersurfaces defined by polynomials $f_\ell=f_\ell(\ul h,\ul t)\in L[\ul h,\ul t]$, where $\ul h$ and $\ul t$ denote the coordinates on $\BA^a_L$ and $\BA^b_L$, respectively. Let $d_Y$ be the maximum of the degrees of these polynomials. Note that for $n\gamma\notin \Gamma\cap T^c$ the lemma holds trivially, because then $Y\cap pr_2^{-1}(n\gamma)$ is empty. On the other hand, Corollary~\ref{Cor4.2} implies that for all but finitely many $n\gamma\in \Gamma\cap T^c$ the intersection $Y\cap pr_2^{-1}(n\gamma)$ does not contain an entire connected component of $pr_2^{-1}(n\gamma)$. We now fix an $n\gamma\in \Gamma\cap T^c$ for which this holds. 

We claim that $G\cap (\BK\times\{n\gamma\})$ is either a $\BK^{\rm geo}\times\{1\}$-coset in $\BK\times\{n\gamma\}$ or empty (in which case the assertion of the lemma again holds trivially). Indeed, if this set is non empty, let $g_n=(h_n,n\gamma)$ be a point in it with $h_n=pr_1(g_n)\in \BK$. Then $\phi_1(h_n)=\phi_1 pr_1(g_n)=\phi_2 pr_2(g_n)=\phi_2(n\gamma)$ and so every other point $\tilde g_n=(\tilde h_n,n\gamma)\in G\cap (\BK\times\{n\gamma\})$ satisfies $g_n\cdot(h_n^{-1}\tilde h_n,1)=\tilde g_n$ with $\phi_1(h_n^{-1}\tilde h_n)=\phi_1( h_n^{-1})\cdot \phi_1(\tilde h_n)=\phi_2(n\gamma)^{-1}\cdot \phi_2(n\gamma)=1$, that is $h_n^{-1}\tilde h_n\in \BK\cap H^{\rm geo}=\BK^{\rm geo}$. This implies $G\cap (\BK\times\{n\gamma\})\subset g_n\cdot(\BK^{\rm geo}\times\{1\})$. For the converse inclusion note that for every $c\in \BK^{\rm geo}$ the point $g_n\cdot(c,1)=(h_nc,n\gamma)\in H\TimesL T$ lies in $G$ because $\phi_1(h_nc)=\phi_1(h_n)=\phi_2(n\gamma)$. This proves that $G\cap (\BK\times\{n\gamma\})= g_n\cdot(\BK^{\rm geo}\times\{1\})$. As a consequence, $G\cap (\BK\times\{n\gamma\})$ is the pairwise disjoint union of $m$ cosets $g_{n,1}\cdot(\BK_1\times\{1\}),\ldots, g_{n,m}\cdot(\BK_1\times\{1\})$ of $\BK_1\times\{1\}$, where we let $m:=[\BK^{\rm geo}:\BK_1]$ be the index. 

Thus, for every such coset $\BK'=g_{n,i}\cdot(\BK_1\times\{1\})$ the intersection $Y\cap \BK'$ is a proper subset of $\BK'$, because $pr_1(g_{n,i}^{-1}\cdot \BK')=\BK_1$ is dense in $(H^{\rm geo})\open$ under assumption~\ref{Thm4.3_A}, as remarked above. Under the isomorphism from \eqref{EqIsomHgeo} this implies that $g_{n,i}^{-1}\cdot \BK'$ is dense in the identity component $\ker(pr_2)\open$, and so $\BK'$ is dense in $g_{n,i}\cdot\ker(pr_2)\open$. The latter is a connected component of $pr_2^{-1}(n\gamma)$ and not contained in $Y$ by our assumption on $n\gamma$. Therefore, $Y\cap \BK'$ is a proper closed subset of $\BK'$ cut out by the finitely many polynomials $f_\ell(\ul h,n\gamma)\in L[\ul h]$ of degree $\le d_Y$ obtained from $f_\ell(\ul h,\ul t)$ by plugging in (the coordinates of) the point $n\gamma$. Under the projection $pr_1$, which induces like in \eqref{EqIsomHgeo} an isomorphism $pr_1\colon pr_2^{-1}(n\gamma)\isoto pr_1(g_{n,i})\cdot H^{\rm geo}$ of varieties, $pr_1(Y\cap \BK')$ is a proper closed subset of $pr_1(\BK')$ cut out by the same polynomials $f_\ell(\ul h,n\gamma)\in L[\ul h]$. We consider the images $\bar f_\ell(\ul h,n\gamma)$ of these $f_\ell(\ul h,n\gamma)$ in the coordinate ring $L[H]$ of $H$.

For every such coset $\BK'=g_{n,i}\cdot(\BK_1\times\{1\})$, we consider the subset $Y'_{n,i}=pr_1(g_{n,i})^{-1}\cdot pr_1(Y\cap \BK')\subset \BK_1\subset H^{\rm geo}$, which is a proper subset of $\BK_1$ cut out by finitely many hypersurfaces of $H^{\rm geo}$. Namely, if we set $h_{n,i}:=pr_1(g_{n,i})\in \BK$, then $Y'_{n,i}$ is cut out by the pullbacks $\bar f'_{\ell,n,i}:=t_{h_{n,i}}^*(\bar f_\ell(\ul h,n\gamma))\in L[H]$ of the polynomials $\bar f_\ell(\ul h,n\gamma)$ under the translation $t_{h_{n,i}}$ by $h_{n,i}$. Let $\olW\subset L[H]$ be the $L$-linear subspace spanned by $t_h^*(\bar f)$ for all $h\in H(L)$ and all $\bar f\in L[H]$, which are images of polynomials $f\in L[\ul h]$ of degree $\le d_Y$. Then $\olW$ has finite dimension by \cite[I.1.9~Proposition]{Borel91} which only depends on $d_Y$, and the polynomials $\bar f'_{\ell,n,i}$ cutting out $Y'_{n,i}$ in $\BK_1$ belong to $\olW$. Let $W\subset L[\ul h]$ be a finite dimensional $L$-linear subspace that surjects onto $\olW$ and choose preimages $f'_{\ell,n,i}\in W$ of all $\bar f'_{\ell,n,i}$. Then $Y'_{n,i}$ is cut out in $\BK_1$ by the polynomials $f'_{\ell,n,i}$. We note that $pr_1(Y\cap pr_2^{-1}(n\gamma))\cap \BK$ equals the disjoint union $\coprod_{i=1}^m h_{n,i}\cdot Y'_{n,i}$ by using again the isomorphism $pr_1\colon pr_2^{-1}(n\gamma)\isoto pr_1(g_{n,i})\cdot H^{\rm geo}$. So it remains to count the elements in the finite set $Y'_{n,i}/\psi(p^\nu\BZ_p^N)$ or equivalently its preimage $\psi^{-1}(Y'_{n,i})/(p^\nu\BZ_p)^N$ under the isomorphism $\psi$ from \eqref{EqIsomPsi}. 

We claim that this preimage is contained in a proper analytic hypersurface $X$ in $\BH_F=\BZ_p^N$ whose degree is bounded independently of $n$ and $Y'_{n,i}$ by a constant depending only on the degree $d_Y$ and the isomorphism $\psi$. Indeed, we choose a $\BQ_p$-basis $(\alpha_1,\ldots,\alpha_s)$ of $L$, where we set $s=[L:\BQ_p]$. The map $\psi\colon \BH_F\isoto \BK_1\subset H(L)\subset(\BA^a_{L})(L)$ is given with respect to coordinates on $\BA^a_{L}$ by power series $\psi_1,\ldots,\psi_a\in L\langle z_1,\ldots,z_N\rangle$ which converge for every $x=(x_1,\ldots,x_N)$ in $\BH_F=\BZ_p^N$. Writing $f'_{\ell,n,i}\in W$ for the polynomial equations cutting out $Y'_{n,i}\subset \BK_1$ we see that $\psi^{-1}(Y'_{n,i})$ is the zero locus of the $f'_{\ell,n,i}(\psi_1,\ldots,\psi_a)\in L\langle z_1,\ldots,z_N\rangle$. With respect to the $\BQ_p$-basis $(\alpha_j)_j$ of $L$ we can write $f'_{\ell,n,i}(\psi_1,\ldots,\psi_a)=\sum_j \alpha_j\cdot \tilde{f}_{\ell,n,i,j}$ with $\tilde{f}_{\ell,n,i,j}\in\BQ_p\langle z_1,\ldots,z_N\rangle$. Then $\psi^{-1}(Y'_{n,i})$ is the simultaneous zero locus of all $\tilde{f}_{\ell,n,i,j}$. More precisely, we view $f'_{\ell,n,i}$ as a morphism $H\to\BA_L$ and consider its Weil restriction $\Res_{L/\BQ_p}f'_{\ell,n,i}\colon\Res_{L/\BQ_p}H\to\Res_{L/\BQ_p}\BA_L$. Here $\Res_{L/\BQ_p}\BA_L$ is the Weil restriction of $\BA_L$, which is isomorphic to $\BA^{s}_{\BQ_p}$ under the identification $(\Res_{L/\BQ_p}\BA_L)(\BQ_p)=L=\bigoplus_j\alpha_j\BQ_p$. Then $\psi^{-1}(Y'_{n,i})$ is the simultaneous zero locus of all the morphisms (for all $\ell$)
\begin{align*}
& \BH_F\;\xrightarrow[\raisebox{+2.5mm}{$\scriptstyle\sim$}]{\enspace\psi\;} \;\BK_1\;\subset\;(\Res_{L/\BQ_p}H)(\BQ_p)\;\xrightarrow{\;\Res_{L/\BQ_p}f'_{\ell,n,i}\,}\;(\Res_{L/\BQ_p}\BA_L)(\BQ_p)\;\cong\;\BA^{s}_{\BQ_p}(\BQ_p)\,.\\[-1mm]
& \; x \;\raisebox{0.1ex}{$\shortmid$}\hspace{-1.7mm}\xrightarrow{\hspace{8.5cm}}\bigl(\tilde{f}_{\ell,n,i,j}(x_1,\ldots,x_N)\bigr)_{j=1,\ldots,s}
\end{align*}
Since $Y'_{n,i}\ne \BK_1$, at least one $\tilde{f}_{\ell,n,i,j}$ is non-zero, and this defines the hypersurface $X$.

Let $V\subset\BQ_p\langle z_1,\ldots,z_N\rangle$ be the $\BQ_p$-vector space generated by all $\tilde{f}_j$ where $f'$ runs through a $\BQ_p$-basis of $W$ and $f'(\psi_1,\ldots,\psi_a)=\sum_j \alpha_j\cdot \tilde{f}_j$. Then $V$ is a finite dimensional $\BQ_p$-vector space which only depends on $d_Y$ and $\psi$ and not on $n\gamma$ and $g_{n,i}$. By Proposition~\ref{PropUniformOesterleBound} there is a constant $c_V$ such that the Oesterl\'e degree $\deg(\tilde{f})\le c_V$ for all $0\ne \tilde{f}\in V$. It now follows from Oesterl\'e's result \eqref{EqOesterle} that the cardinality of $Y'_{n,i}/\psi(p^\nu\BZ_p^N)$ is at most $c_V\,p^{\nu(N-1)}$. Thus the image of $pr_1(Y\cap pr_2^{-1}(n\gamma))\cap \BK$ under the quotient map $\BK\onto \BK/\psi(p^\nu\BZ_p^N)$ has cardinality at most $m\,c_V\,p^{\nu(N-1)}$. 
\end{proof}

After these preparations it is easy to finish the 

\begin{proof}[Proof of Theorem~\ref{Thm4.3}] 
We assume to the contrary that the closure of $F$ does not contain any connected component of $G$. Fix a connected component $T^c$ of $T$ and let $Y\subset pr_2^{-1}(T^c)$ be the closure of $F\cap pr_2^{-1}(T^c)$. Then the assumption implies that $Y$ does not contain any irreducible component of $pr_2^{-1}(T^c)$. Let $d_{c,Y}$ be the positive integer from Proposition~\ref{Prop4.5}. Then for all but finitely many $n\in\BN$ and for every $\nu\in\BN_{>0}$ the image of the set $pr_1\bigl(F\cap pr_2^{-1}(n\gamma)\cap pr_2^{-1}(T^c)\bigr)$ under the quotient map $\BK\onto \BK/\psi(p^\nu\BZ_p^N)$ has cardinality at most $d_{c,Y}\cdot p^{\nu(N-1)}$. Let $d$ be the sum $\sum_{T^c}d_{c,Y}$ of the $d_{c,Y}$ over all connected components $T^c$ of $T$. Taking the union over all $T^c$ we see that for all but finitely many $n\in\BN$ and for every $\nu\in\BN_{>0}$ the image of the set $pr_1(F\cap pr_2^{-1}(n\gamma))\subset \BK$ under the quotient map $\BK\onto \BK/\psi(p^\nu\BZ_p^N)$ has cardinality at most $d\,p^{\nu(N-1)}$.

Now let $\epsilon>0$ be the constant from assumption~\ref{Thm4.3_B} of Theorem~\ref{Thm4.3} and choose $\nu$ so large that $\epsilon\,p^{\nu N} > d\,p^{\nu(N-1)}$. We consider the subgroup $\BJ_1^{\rm geo}:=\psi(p^\nu\BZ_p^N)\subset \BK^{\rm geo}$ which is open in $\BK^{\rm geo}$. Then $\#(\BK^{\rm geo}/\BJ_1^{\rm geo})\ge\#(\BK_1/\BJ_1^{\rm geo})=\#(\BZ_p/p^\nu\BZ_p)^N=p^{\nu N}$. Since $\BK^{\rm geo}$ carries the subspace topology induced from $\BK$ and the topology on $\BK$ is the topology of $\BK$ as a profinite group by \cite[Part~II, \S\,IV.8, Corollary~2]{Serre92}, there is an open normal subgroup $\BJ\subset \BK$ such that $\BJ^{\rm geo}:=\BK^{\rm geo}\cap \BJ\subset \BJ_1^{\rm geo}$. By assumption~\ref{Thm4.3_B} of Theorem~\ref{Thm4.3} there is an infinite subset $R_{\BJ}\subset\BN$ such that for every $n\in R_{\BJ}$ the image of $pr_1(F\cap pr_2^{-1}(n\gamma))\subset \BK$ under the quotient map $\BK\to \BK/\BJ$ has cardinality at least $\epsilon\cdot\#(\BK^{\rm geo}/\BJ^{\rm geo})$. Since the map $\BK\to \BK/\BJ$ factors through $\BK\to \BK/\BJ^{\rm geo}\to \BK/\BJ$, the image of $pr_1(F\cap pr_2^{-1}(n\gamma))\subset \BK$ under the quotient map $\BK\to \BK/\BJ^{\rm geo}$ has also cardinality at least $\epsilon\cdot\#(\BK^{\rm geo}/\BJ^{\rm geo})$. The fibers of the map $\BK/\BJ^{\rm geo}\onto \BK/\BJ_1^{\rm geo}$ are principal homogeneous spaces under the group $\BJ_1^{\rm geo}/\BJ^{\rm geo}$. This implies that the number of points in the image of $pr_1(F\cap pr_2^{-1}(n\gamma))$ in $\BK/\BJ^{\rm geo}$ which are mapped to the same point in $\BK/\BJ_1^{\rm geo}$ is at most $\#(\BJ_1^{\rm geo}/\BJ^{\rm geo})$. In particular, the image of $pr_1(F\cap pr_2^{-1}(n\gamma))$ in $\BK/\BJ_1^{\rm geo}$ has cardinality at least $\epsilon\cdot\#(\BK^{\rm geo}/\BJ^{\rm geo})/\#(\BJ_1^{\rm geo}/\BJ^{\rm geo})=\epsilon\cdot\#(\BK^{\rm geo}/\BJ_1^{\rm geo})\ge\epsilon \,p^{\nu N}$. But this contradicts the estimate from the previous paragraph. This completes the proof of Theorem~\ref{Thm4.3}.
\end{proof}

\section{Maximal Quasi-Tori, Maximal Compact Subgroups, and Zariski-Density Criteria} \label{SectMaxQuasiTori}

We need some results in the theory of not necessarily connected algebraic groups. In Subsections~\ref{SubSectQuasiTori} and \ref{SubSectConjMQT} the letter $L$ denotes an algebraically closed field of characteristic $0$ and all undecorated fiber products $\times$ are over $L$.

\subsection{The Theory of Maximal Quasi-Tori} \label{SubSectQuasiTori}

In the sequel, we will need the following mild generalization of a classical result of Steinberg~\cite[Theorem~8.1]{Steinberg68}. It was announced in \cite[Theorem~1.1.A]{Kottwitz-Shelstad99} with a brief sketch of proof. We include a full proof for the convenience of the reader. We recall Notation~\ref{InitialNotation}.

\begin{thm} \label{steinberg} 
Let $G$ be a reductive linear algebraic group over $L$, let $h\in G$ be a semi-simple element which normalizes a Borel subgroup $B\subset G\open$ and a maximal torus $T\subset B$. Then $G^h$ is reductive, $T^h{}\open$ is a maximal torus in $G^h{}\open$, and $B^h{}\open$ is a Borel subgroup in $G^h{}\open$. 
\end{thm}

Note that $G^h$, $T^h$ and $B^h$ are not connected in general as can be seen from the following 

\begin{example}\label{ExKS}
Let $p$ be a prime number and let $G$ be the $p-1$-dimensional torus:
$$G \;=\; \{(x_1,x_2,\ldots,x_p)\in\BG_m^p
\colon x_1\cdot x_2\cdots x_p=1\}.$$
Then the cyclic permutation
$$h\colon (x_1,x_2,\ldots,x_p) \;\longmapsto\; (x_2,x_3,\ldots,x_p,x_1)$$
is an automorphism of $G$ of order $p$ whose fixed points are
$$(\zeta,\zeta,\ldots,\zeta),$$
where $\zeta$ is any $p$-th root of unity. The semi-direct product $G\rtimes
\BZ/p\BZ$ where the generator $h$ of $\BZ/p\BZ$ acts by the automorphism above is a counter-example to the connectivity of both the $h$-fixed points of a maximal torus and a Borel subgroup of $G$, since the latter are both equal to $G$.
\end{example}

For the proof of Theorem~\ref{steinberg} we need a

\begin{prop}\label{PropUnivCov}
Let $G$ be a linear algebraic group over $L$ whose identity component $G\open$ is semi-simple. Then there is a linear algebraic group $\wt G$ over $L$ and a surjective group homomorphism $\phi\colon\wt G \to G$ such that $\wt G\open\to G\open$ is the universal simply connected cover. The kernel $\wt K$ of $\phi$ is finite and $\wt K\cap \wt G\open$ lies in the center of $\wt G\open$. 
\end{prop}

\noindent
\emph{Remark.} Note that in general there might not exist a $\wt G$ for which the surjective map $\wt G/\wt G\open\onto G/G\open$ is bijective.

\begin{proof}[Proof of Proposition~\ref{PropUnivCov}.]
Since the question is independent of the field $L$ we may assume that $L=\BC$. Then by \cite{Borovoi98} the algebraic fundamental group $\pi_1(G\open)$ equals the topological fundamental group of the complex Lie group $G\open(\BC)$. By \cite[Corollary~5.6]{BrownMucuk} the group $\wt G$ with the described properties exists as a complex Lie group with finite component group $\wt  G/\wt G\open$. Since $\wt G\open$ is linear algebraic, $\wt G$ is too. More precisely, by \cite[Theorem~5.2]{BrownMucuk} a covering $\phi\colon\wt G\to G$ of topological groups with finite component group $\Gamma:=\wt G/\wt G\open$ and $\wt G\open$ simply connected exists if and only if a certain obstruction class in $\Koh^3\bigl(\Gamma,\pi_1(G\open)\bigr)$ vanishes. This class is the image of the corresponding class in $\Koh^3\bigl(G/G\open,\pi_1(G\open)\bigr)$. Since the latter cohomology group is finite, it can be annihilated by a surjection $\Gamma\onto G/G\open$ for a suitable finite group $\Gamma$. Then $\wt K$ is an extension of the finite groups $\ker(\Gamma\onto G/G\open)$ by $\wt K\cap\wt G\open$, and the latter lies in the center of $\wt G\open$ by \cite[V.24.1]{Borel91}.
\end{proof}

\begin{proof}[Proof of Theorem~\ref{steinberg}.] Steinberg~\cite[Theorem~8.1]{Steinberg68} proved the claim when $G\open$ is simply connected. We reduce the general case to this one via two reduction steps. First assume that $G\open$ is semi-simple, let $\phi:\wt G\to G$ be the surjective homomorphism from Proposition~\ref{PropUnivCov} and let $\wt K\subset\wt G$ be the kernel of $\phi$. Pick an element $\wt h\in\wt G$ in the preimage of $h$, and let $\wt T,\wt B$ be the preimages of $T,B$ in $\wt G$. Note that $\wt h$ is semi-simple, because when we write $\wt h=\wt h_s\wt h_u$ as a product of its semi-simple part $\wt h_s$ and its unipotent part $\wt h_u$, then $\wt h_u$ lies in $\wt K$ by \cite[I.4.4~Theorem]{Borel91}, because $\phi(\wt h)=h$ is semi-simple. If we denote by $n$ the order of $\wt K$, then $\wt h^n=\wt h_s^n\wt h_u^n=\wt h_s^n$ is semi-simple. Therefore, already $\wt h$ is semi-simple by Lemma~\ref{LemmaUnipotConnected}\ref{LemmaUnipotConnected_C}.

Since $T$ is connected the restriction $\phi|_{\wt T\open}:\wt T\open\to T$ is surjective with finite kernel $\wt K\cap\wt G\open$, therefore the connected group $\wt T\open$ must be a torus. It has the same dimension as $T$, so it must be a maximal torus in $\wt G\open$, as the ranks of $G\open$ and $\wt G\open$ are the same. Likewise $\wt B$ is an extension of the solvable group $B$ by the commutative group $\wt K\cap\wt G\open$. Therefore, $\wt B$ is connected solvable and of the same dimension as $B$, and hence a Borel subgroup of $\wt G\open$. Since
$$\phi(\wt h^{-1}\wt T\wt h)=\phi(\wt h^{-1})
\phi(\wt T)\phi(\wt h)=h^{-1}Th=T,$$
we get that $\wt h$ normalizes $\wt T$. A similar computation shows that $\wt h$ normalizes $\wt B$, too. Therefore, by Steinberg's theorem quoted above the subgroup $\wt G^{\wt h}$ is reductive, $\wt T^{\wt h}{}\open$ is a maximal torus in $\wt G^{\wt h}$, and $\wt B^{\wt h}{}\open$ is a Borel subgroup in $\wt G^{\wt h}$.

The regular map
$$\phi^{-1}(G^h)\;\longto\; \wt K\,,\quad \wt x\mapsto\wt h^{-1}\wt x\wt h\wt x^{-1}$$
has finite image, so it is locally constant, hence $1$ on $\phi^{-1}(G^h)\open$. Therefore, $\phi^{-1}(G^h)\open\subset\wt G^{\wt h}$. Clearly
$\wt G^{\wt h}\subset\phi^{-1}(G^h)$, and hence $\wt G^{\wt h}{}\open=\phi^{-1}(G^h)\open$. A similar argument shows that 
$\wt T^{\wt h}{}\open=\phi^{-1}(T^h)\open$ and $\wt B^{\wt h}{}\open=\phi^{-1}(B^h)\open$. Since the connected reductive group $\wt G^{\wt h}{}\open$ surjects onto $G^h{}\open$, we get that $G^h{}\open$ is reductive by \cite[IV.14.11~Corollary]{Borel91}. Similarly $\wt T^{\wt h}{}\open$ surjects onto $T^h{}\open$, so the latter is a maximal torus in $G^h{}\open$ by \cite[IV.11.14~Proposition]{Borel91}. The same reasoning shows that $B^h{}\open$ is a Borel subgroup in $G^h{}\open$. 

Consider now the general case and let $Z\subset G\open$ be the identity component of the center of $G\open$. It equals the radical of $G\open$ and is a torus contained in $T$ by \cite[IV.11.21~Proposition]{Borel91}. Set $\olG=G/Z$ and let $\psi:G\to\olG$ be the quotient map. Then $\olG$ is semi-simple. Let $\olT,\olB$ be the image of $T,B$ in $\olG$, respectively. The subgroup $\olT$ is a maximal torus in $\olG$ and $\olB$ is a Borel subgroup in $\olG$. Clearly $\olT\subset\olB$ and $\psi(h)$ normalizes this pair, so by the case which we have just proven the subgroup $\olG{}^{\psi(h)}$ is reductive, $\olT{}^{\psi(h)}{}\open$ is a maximal torus in $\olG{}^{\psi(h)}{}\open$, and $\olB{}^{\psi(h)}{}\open$ is a Borel subgroup in $\olG{}^{\psi(h)}{}\open$. Now let $\wt G,\wt T, \wt B$ denote the preimage of $\olG{}^{\psi(h)},\olT{}^{\psi(h)}{}\open$ and $\olB{}^{\psi(h)}{}\open$ with respect to $\psi$, respectively. Clearly $G^h{}\subset\wt G$ and $T^h{}\open\subset\wt T\open\subset T$ and $B^h{}\open\subset\wt B\subset B$, because $Z\subset T\subset B$.

\bigskip\noindent
{\itshape Claim.} Under $\psi$ the groups $G^h{}\open,T^h{}\open,B^h{}\open$ surject onto $\olG{}^{\psi(h)}{}\open,\olT{}^{\psi(h)}{}\open$ and $\olB{}^{\psi(h)}{}\open$, respectively. 

\bigskip\noindent
To prove the claim note that $h^{-1}Zh\subset Z$ because $Z$ is a normal subgroup of $G$. Therefore, the map $$Z\to Z,\quad z\mapsto h^{-1}zhz^{-1}$$ is a homomorphism of groups. Let $Z'$ be the image of $Z$ under this homomorphism. It is a closed subgroup of $Z$ invariant under conjugation by $h$ as the computation $h(h^{-1}zhz^{-1})h^{-1}=h^{-1}(hzh^{-1})h(hz^{-1}h^{-1})$ shows. Note that it will be enough to show that the regular map
\begin{equation}\label{EqImageOfMapKappa}
\kappa:\wt G\open\to\wt G\open,\quad x\mapsto h^{-1}xhx^{-1}
\end{equation}
has image in $Z'$. Indeed, if this is the case then for every $x\in\wt G\open$ there is a $z\in Z$ such that 
$$h^{-1}xhx^{-1}=h^{-1}zhz^{-1},$$
so
\[
h^{-1}xz^{-1}h(xz^{-1})^{-1}=h^{-1}xh(h^{-1}z^{-1}h)zx^{-1}=
h^{-1}xhx^{-1}z(h^{-1}z^{-1}h)=h^{-1}zhz^{-1}zh^{-1}z^{-1}h=1
\]
using that both $z$ and $h^{-1}z^{-1}h$ are in the center of $G\open$. Therefore, $xz^{-1}\in G^h{}\open$, but $\psi(x)=\psi(xz^{-1})$ as $z\in Z$. So $\psi(G^h{}\open)=\psi(G^h)\open=\psi(\wt G\open)=\olG{}^{\psi(h)}{}\open$. Using $Z\subset T\subset B$, a similar argument shows that $T^h{}\open$ and $B^h{}\open$ surject onto $\olT^{\psi(h)}{}\open$ and $\olB^{\psi(h)}{}\open$ with respect to $\psi$, respectively. So the claim will follow.

We compute the image of the map $\kappa$ from \eqref{EqImageOfMapKappa}. The group $\wt G\open$ is the central extension of the reductive group $\olG{}^{\psi(h)}{}\open$ by the torus $Z$, so it is reductive. Therefore, semi-simple elements are dense in $\wt G\open$. So it will be enough to show that $\kappa$ maps every maximal torus $S\subset\wt G\open$ into $Z'$, since the latter is closed. Since $Z$ is a central torus, it is contained in $S$. Therefore, $\psi^{-1}(\psi(S))=S$. The image $\psi(S)$ is a torus on which the action of $\psi(h)$ is trivial, in particular $\psi(S)$ is normalized by $\psi(h)$. Therefore, $h$ normalizes $S$. Since $h\in G^h\subset \wt G$, for some positive integer $m$ we have $h^m\in\wt G{}\open$, so $h^m$ is in the normalizer of $S$ in $\wt G{}\open$. Since $S$ has finite index in its normalizer in $\wt G{}\open$, the conjugation action of $h$ on $S$ has finite order. 

Conjugation by $h$ leaves $Z'$ invariant, as we have already remarked, so there is an induced action on $\olS=S/Z'$. This action is trivial on the subgroup $\olZ=Z/Z'$ by definition. We also noted that the induced action on the quotient $\olS/\olZ=S/Z$ is also trivial, because under the morphism $\psi$ the latter is isomorphic to $\psi(S)$ on which $\psi(h)$ acts trivially. By Lemma~\ref{LemmaFiniteAutOfTorus} below this implies that this action on $\olS$ is trivial. This is equivalent to $\kappa|_S$ taking values in $Z'$ as promised. This proves the claim.

The proof of the theorem is now easy. By the claim $G^h{}\open$ is the extension of $\olG{}^{\psi(h)}{}\open$ by a subgroup of $Z^h$. The group $\olG{}^{\psi(h)}$ is reductive by the above, while the above-mentioned subgroup of $Z^h$ is a subgroup of the torus $Z$, so it is also reductive. Therefore, $G^h{}\open$ is also reductive. Since $Z$ lies in $T$, the subgroup $Z^h{}\open$ lies in $T^h{}\open$. Therefore, $T^h{}\open$ is the extension of a maximal torus in $\olG{}^{\psi(h)}$ by a group containing the identity component of the kernel of the restriction of $\psi$ onto $G^h{}\open$, so it is a maximal torus in $G^h{}\open$. A similar argument shows that $B^h{}\open$ is a Borel subgroup in $G^h{}\open$.
\end{proof}

\begin{lemma}\label{LemmaFiniteAutOfTorus} Let $T$ be a torus over a field of arbitrary characteristic, and let $\sigma$ be an automorphism of $T$ of \emph{finite order}. Assume that there is a sub-torus $T'\subset T$, such that $\sigma$ fixes every point of $T'$, and the automorphism of the quotient $T/T'$ induced by $\sigma$ is also the identity. Then $\sigma$ is trivial.
\end{lemma}

\begin{proof} 
For every torus $T$ let $X_*(T)$ denote the group of its cocharacters. Then the rule $T\mapsto X_*(T)\otimes_\BZ\BQ$ is a full and exact functor. In particular we have a short exact sequence:
$$\xymatrix{0\ar[r] &X_*(T')\otimes_\BZ\BQ\ar[r]
& X_*(T)\otimes_\BZ\BQ\ar[r] & X_*(T/T')\otimes_\BZ\BQ
\ar[r] & 0.}$$
which is equivariant relative to $\sigma$. Let $\mathbf b'$ be a $\BQ$-basis of $X_*(T')\otimes_\BZ\BQ\subset X_*(T)\otimes_\BZ\BQ$, and extend this to a $\BQ$-basis $\mathbf b$ of $X_*(T)\otimes_\BZ\BQ$. Since the induced action is trivial both on $X_*(T')\otimes_\BZ\BQ$ and $X_*(T)\otimes_\BZ\BQ/X_*(T')\otimes_\BZ\BQ=X_*(T/T')\otimes_\BZ\BQ$, the matrix of the action on $X_*(T)\otimes_\BZ\BQ$ in the basis $\mathbf b$ is upper triangular with ones on the diagonal. In particular it is unipotent. But this is also a matrix of finite order, so it is semi-simple, too. Therefore, this matrix is the identity by Lemma~\ref{LemmaUnipotConnected}.
\end{proof}

\begin{defn}\label{DefQuasiTorus}
Let $G$ be a linear algebraic group over $L$.
\begin{enumerate}
\item \label{DefQuasiTorus_A}
If $G$ is reductive, a closed subgroup $T\subset G$ is called a \emph{maximal quasi-torus} if $T$ equals the intersection $N_G(B)\cap N_G(T\open)$ of the normalizers in $G$ of a Borel subgroup $B\subset G\open$ and a maximal torus $T\open\subset B$.
\item \label{DefQuasiTorus_B}
For general $G$, a closed subgroup $T\subset G$ is called a \emph{maximal quasi-torus} if the quotient morphism $r_G\colon G\onto G^\red$ maps $T$ isomorphically onto a maximal quasi-torus in the reductive group $G^\red$.
\end{enumerate}
\end{defn}

\begin{rem}\label{RemQuasiTorus}
If $G$ is reductive and $T\subset G$ is a maximal quasi-torus of the form $T=N_G(B)\cap N_G(T\open)$ then $T\cap G\open=N_{G\open}(B)\cap N_{G\open}(T\open)=N_{B}(T\open)=Z_{B}(T\open)=T\open$ by \cite[IV.11.16~Theorem, III.10.6~Theorem and IV.13.17~Corollary~2]{Borel91}. In particular, the identity component of $T$ equals the initially given maximal torus $T\open$ (and that was the reason for calling the latter $T\open$). 
\end{rem}

\begin{lemma}\label{LemmaQuasiTorus}
Let $G$ be an arbitrary linear algebraic group and let $T$ be a maximal quasi-torus in $G$. Then $T\open$ is a maximal torus in $G$ and all elements of $T$ are semi-simple.
\end{lemma}

\begin{proof}
By the above $\wt T:=r_G(T)$ is a maximal quasi-torus in $G^\red$ and $\wt T\open=r_G(T\open)$ is a maximal torus in $G^\red$. In particular, $T\open$ is a torus and contained in a maximal torus $T'$ of $G\open$. It follows that $\wt T\open$ is contained in and hence equal to the torus $r_G(T')$. Since $T'\cap\ker r_G=\{1\}$, it follows that $r_G\colon T'\onto r_G(T')=r_G(T\open)$ is an isomorphism and so $T\open=T'$ is a maximal torus in $G$.

If $g$ lies in $T$ and $n$ is the order of the finite group $T/T\open$ then $g^n\in T\open$. That is, $g^n$ is semi-simple, and so $g$ is semi-simple by Lemma~\ref{LemmaUnipotConnected}\ref{LemmaUnipotConnected_C}.
\end{proof}

We need to establish a few more facts about maximal quasi-tori. We are grateful to Friedrich Knop for providing a proof of part~\ref{ThmQuasiTorusReductive_D} in the following theorem. Since we were not able to find a correct proof of this statement in the literature we include Knop's argument on \href{https://mathoverflow.net/questions/280874/are-the-semi-simple-elements-in-a-non-connected-reductive-algebraic-group-dense}{https:/\!/mathoverflow.net/questions/280874} for the reader's convenience.

\begin{thm}\label{ThmQuasiTorusReductive}
Assume that $G$ is reductive. 
\begin{enumerate}
\item \label{ThmQuasiTorusReductive_A}
Let $T$ be a maximal quasi-torus in $G$. Then $T/T\open=G/G\open$. Every other maximal quasi-torus is conjugate to $T$ under $G\open$.
\item \label{ThmQuasiTorusReductive_B}
An element $g\in G$ is semi-simple if and only if it is contained in a maximal quasi-torus, if and only if its $G$-conjugacy class is closed.
\item \label{ThmQuasiTorusReductive_D}
Every connected component of $G$ contains a dense open subset consisting of semi-simple elements.
\end{enumerate}
\end{thm}

\begin{proof}
\ref{ThmQuasiTorusReductive_A} The conjugacy statement follows from the fact that all pairs $T\open\subset B$ of a maximal torus and a Borel subgroup in $G\open$ are conjugate under $G\open$ by \cite[IV.11.19~Proposition]{Borel91}. To show that $T$ surjects onto $G/G\open$ let $g\in G$. Then $gT\open g^{-1}$ is a maximal torus of the Borel subgroup $gB g^{-1}$ of $G\open=gG\open g^{-1}$. So there exists $h\in G\open$ such that $gT\open g^{-1}=hT\open h^{-1}$ and $gB g^{-1}=hB h^{-1}$. Then $h^{-1}g\in N_G(T\open)\cap N_G(B)=T$ and $T\onto G/G\open$ is surjective. Since $T\cap G\open = T\open$ by Remark~\ref{RemQuasiTorus}, we conclude that $T/T\open\isoto G/G\open$.

\medskip\noindent
\ref{ThmQuasiTorusReductive_B} If $g$ is semi-simple then it normalizes a maximal torus $T\open$ and a Borel subgroup $B$ of $G\open$ containing $T\open$ by \cite[Theorem~7.5]{Steinberg68}, and hence lies in the maximal quasi-torus $T=N_G(B)\cap N_G(T\open)$. The converse was proven in Lemma~\ref{LemmaQuasiTorus}. The characterization in terms of the $G$-conjugacy class of $g$ is given in \cite[Corollaire~II.2.22]{Spaltenstein82}.

\medskip\noindent
\ref{ThmQuasiTorusReductive_D} Let $T\subset G$ be a maximal quasi-torus. By \ref{ThmQuasiTorusReductive_A} every connected component of $G$ is of the form $hG\open$ for an $h\in T$. To show that $hG\open$ contains a dense open subset consisting of semi-simple elements, we consider the conjugation action 
\[
\Phi\colon G\open\TimesL hT^h{}\open\to hG\open,\quad(g,ht)\mapsto g ht g^{-1} = h(h^{-1}gh)tg^{-1},
\]
where $T^h:=\{g\in T\open\colon gh=hg\}\subset T\open$ as in Notation~\ref{InitialNotation}. All elements in the image of $\Phi$ are conjugate to elements in $hT^h{}\open\subset T$, and so are semi-simple by \ref{ThmQuasiTorusReductive_B}. Since the image of $\Phi$ is constructible by Chevalley's theorem~\cite[$\mathrm{IV}_1$, Corollaire~1.8.5]{EGA}, this image contains an open subset of its closure by \cite[$0_{\rm III}$, Proposition~9.2.2]{EGA}. It thus suffices to show that $\Phi$ is dominant. Under the isomorphism of varieties $hG\open\isoto G\open$,\; $x\mapsto h^{-1}x$ the morphism $\Phi$ corresponds to the morphism 
\[
\Phi'\colon G\open\TimesL T^h{}\open\to G\open,\quad(g,t)\mapsto h^{-1}g ht g^{-1}. 
\]
To prove that $\Phi'$ is dominant, we use Theorem~\ref{steinberg} which says that $T^h{}\open$ is a maximal torus in $G^h{}\open$. Therefore, the conjugation action $G^h{}\open\TimesL T^h{}\open\to G^h{}\open$,\; $(\tilde g,t)\mapsto \tilde gt\tilde g^{-1}$ is dominant by \cite[IV.11.10~Theorem and IV.13.17, Corollary~2]{Borel91}. Note that $h^{-1}(g\tilde g)ht(g\tilde g)^{-1} = h^{-1}gh(\tilde gt\tilde g^{-1})g^{-1}$ for $\tilde g\in G^h{}\open$ and $g\in G\open$. Thus it suffices to show that the morphism
\[
\Phi''\colon G\open\TimesL\, G^h{}\open \to G\open,\quad(g,g')\mapsto h^{-1}g hg' g^{-1}
\]
is dominant. In the point $(g,g')=(1,1)$ the morphism $\Phi''$ has differential 
\begin{eqnarray}
(\Ad(h^{-1})-1)\oplus \id\colon & \!\!\!\!\!\Lie G\open \oplus \Lie G^h{}\open\!\!\!\! & \longrightarrow \;\Lie G\open,  \label{EqThmQuasiTorusReductive1}\\
& (\,X\;,\;X'\,) & \longmapsto \;(\Ad(h^{-1})-1)(X)+X'\;=\;h^{-1}Xh-X+X'\,,\nonumber
\end{eqnarray}
where $\Lie G\open$ denotes the Lie algebra of $G\open$ and $\Ad\colon G\to\Aut_L(\Lie G\open)$ denotes the adjoint representation. Since $G^h{}\open=\{g\in G\open\colon h^{-1}ghg^{-1}=1\}\open$ we obtain $\Lie G^h{}\open=\ker(\Ad(h^{-1})-1)$ and since $h$ is semi-simple, also $\Ad(h^{-1})$ is semi-simple, and hence $\ker(\Ad(h^{-1})-1)+\im(\Ad(h^{-1})-1)=\Lie G\open$. This shows that the differential \eqref{EqThmQuasiTorusReductive1} is surjective in $(1,1)$, and therefore $\Phi''$ is dominant (for example by \cite[\S\,2.2, Proposition~8]{BLR}) and the theorem is proven.
\end{proof}

If $G$ is not assumed to be reductive this implies the following

\begin{thm}\label{ThmQuasiTorusGeneral}
Let $G$ be a not necessarily connected, linear algebraic group over $L$.
\begin{enumerate}
\item \label{ThmQuasiTorusGeneral_A}
Let $T\open$ be a maximal torus of $G\open$. Then there exists a maximal quasi-torus $T\subset G$ with $T\cap G\open=T\open$. 
\item \label{ThmQuasiTorusGeneral_D}
Every maximal quasi-torus $T$ in $G$ satisfies $T/T\open=G/G\open$ and normalizes a Borel subgroup $B$ of $G\open$, which contains the maximal torus $T\open\subset G\open$. In particular, $G\open\cap T=T\open$. Moreover, the group $TB$ equals the semi-direct product $T\ltimes R_u(TB)$, and $T$ is a maximal reductive subgroup of $TB$.
\item \label{ThmQuasiTorusGeneral_E}
Conversely, a closed subgroup $T\subset G$ is a maximal quasi-torus if $T\onto G/G\open$ is surjective, the identity component $T\open$ is a maximal torus of $G\open$, and $T$ normalizes a Borel subgroup $B\subset G\open$ containing $T\open$.
\item \label{ThmQuasiTorusGeneral_B}
Any two maximal quasi-tori in $G$ are conjugate under $G\open$. 
\item \label{ThmQuasiTorusGeneral_C}
An element $g\in G$ is semi-simple if and only if it is contained in a maximal quasi-torus. 
\end{enumerate}
\end{thm}

\begin{proof}
\ref{ThmQuasiTorusGeneral_A} Note that $\wt T\open:=r_G(T\open)\subset(G^\red)\open$ is a maximal torus by \cite[IV.11.14~Proposition]{Borel91}. Choose a Borel subgroup $\wt B\subset(G^\red)\open$ containing $\wt T\open$ and consider the maximal quasi-torus $\wt T=N_{G^\red}(\wt T\open)\cap N_{G^\red}(\wt B)$ of $G^\red$. Let $\wt H:= \wt T\wt B$ and $H:=r_G^{-1}(\wt H)$. By \cite[Chapter~VIII, Theorem~4.3]{HochschildBasic} there is a reductive subgroup $L\subset H$ with $T\open\subset L$ such that $H=L\ltimes R_uH$ is a semi-direct product. Then $L\cap R_uG=\{1\}$, and hence $r_G\colon L \isoto r_G(L)$ is an isomorphism. Moreover, $r_G(L)$ normalizes $\wt B$. The subgroup $r_G(L)\open\subset \wt B$ is reductive by \cite[IV.14.11~Corollary]{Borel91} and contains the maximal reductive subgroup $\wt T\open$ of $\wt B$; use \cite[III.10.6~Theorem]{Borel91}. Therefore, $\wt T\open=r_G(L)\open$ is normalized by $r_G(L)$, and hence $r_G(L)\subset\wt T$. Since $r_G(R_uH)=R_u\wt H\subset \wt B$, the map $r_G\colon L \onto \wt H/\wt B=\wt T/\wt T\open$ is surjective, where the last equality comes from $\wt T\open\subset \wt T \cap \wt B\subset \wt T\cap (G^\red)\open=\wt T\open$. We conclude that $r_G(L)=\wt T$ and $L$ is a maximal quasi-torus in $G$. (The groups $H,\wt H$ could be called ``quasi-Borel subgroups'' and $L, r_G(L)$ could be called ``quasi-Levi subgroups'', because their groups of connected components are isomorphic to $G/G\open$, respectively $G^\red/(G^\red)\open$ and their identity components are Borel subgroups, respectively Levi subgroups).

\medskip\noindent
\ref{ThmQuasiTorusGeneral_D} If $\wt T:=r_G(T)$, then $T/T\open=\wt T/\wt T\open=G^\red/(G^\red)\open=G/G\open$ by Theorem~\ref{ThmQuasiTorusReductive}, and $T\open$ is a maximal torus in $G\open$ by Lemma~\ref{LemmaQuasiTorus}. To prove that $T$ normalizes a Borel subgroup of $G$, let $\wt B$ be a Borel subgroup of $G^\red$ containing $\wt T\open$ with $\wt T=N_{G^\red}(\wt B)\cap N_{G^\red}(\wt T\open)$. Let $B\subset G\open$ be a Borel subgroup with $r_G(B)=\wt B$. Since $r_G^{-1}(\wt B)$ contains $B$ and is an extension of $\wt B$ by $R_uG$, and hence connected solvable, we conclude that $r_G^{-1}(\wt B)=B$ is a Borel subgroup of $G\open$. By construction, it is normalized by $T$. By \cite[Chapter~VIII, Theorem~4.3]{HochschildBasic} there exists a maximal reductive subgroup $L\subset H:=TB=r_G^{-1}(\wt T \wt B)$ with $T\subset L$ such that $H=L\ltimes R_uH$ is a semi-direct product. As in \ref{ThmQuasiTorusGeneral_A} we have $r_G\colon L\isoto r_G(L)=\wt T=r_G(T)$ and this proves $L=T$ as desired.

\medskip\noindent
\ref{ThmQuasiTorusGeneral_E} Let $n:=\#(T/T\open)$. For every element of $T$ its $n$-th power lies in the torus $T\open$ and hence is semi-simple. Therefore, all elements of $T$ are semi-simple by Lemma~\ref{LemmaUnipotConnected} and $T\cap R_uG=\{1\}$. So the map $r_G$ restricted to $T$ is injective, and maps $T$ isomorphically onto $r_G(T)\subset G^\red$. The identity component $r_G(T)\open=r_G(T\open)$ is a maximal torus in $(G^\red)\open$ and $r_G(B)$ is a Borel subgroup in $(G^\red)\open$. Since $r_G(T)$ normalizes the pair $r_G(T\open)\subset r_G(B)$, it is contained in the maximal quasi-torus $\wt T=N_{G^\red}(r_G(B))\cap N_{G^\red}(r_G(T\open))$, which satisfies $\wt T\open=r_G(T\open)$ by Remark~\ref{RemQuasiTorus}. Since $T\onto G/G\open \isoto G^\red/(G^\red)\open=\wt T/\wt T\open$, we have $r_G\colon T\isoto r_G(T)=\wt T$, and hence $T$ is a maximal quasi-torus in $G$.

\medskip\noindent
\ref{ThmQuasiTorusGeneral_B} Let $T_1,T_2$ be two maximal quasi-tori in $G$ and let $\wt T_1$ and $\wt T_2$ be their isomorphic images in $G^\red$ under $r_G$. By Theorem~\ref{ThmQuasiTorusReductive} we can conjugate $\wt T_2$ into $\wt T_1$ under $(G^\red)\open=r_G(G\open)$ and thus assume that they are equal $\wt T:=\wt T_1=\wt T_2$. Let $\wt B$ be a Borel subgroup in $G^\red$ which is normalized by $\wt T$. As in the proof of \ref{ThmQuasiTorusGeneral_D}, $T_1$ and $T_2$ are both maximal reductive subgroups of $H:=r_G^{-1}(\wt T\wt B)$. Then $T_1$ and $T_2$ are conjugate under $H\open$ by \cite[Chapter~VIII, Theorem~4.3]{HochschildBasic}.

\medskip\noindent
\ref{ThmQuasiTorusGeneral_C} By Lemma~\ref{LemmaQuasiTorus} every element of a maximal quasi-torus is semi-simple. Conversely, to show that every semi-simple element $g\in G$ lies in a maximal quasi-torus we use that $g$ normalizes a Borel subgroup $B\subset G\open$ and a maximal torus $T\open\subset B$ by \cite[Theorem~7.5]{Steinberg68}. Then $\wt T\open:=r_G(T\open)$ and $\wt B:=r_G(B)$ are a maximal torus and a Borel subgroup of $(G^\red)\open$, which are normalized by $r_G(g)\in G^\red$. In particular, $r_G(g)$ lies in the maximal quasi-torus $\wt T:=N_{G^\red}(\wt B)\cap N_{G^\red}(\wt T\open)$ of $G^\red$. By the proof of \ref{ThmQuasiTorusGeneral_A} we can choose a maximal quasi-torus $T'\subset G$ containing $T\open$ and mapping isomorphically onto $\wt T$ under $r_G$. By \ref{ThmQuasiTorusGeneral_D} there is a Borel subgroup $B'\supset T\open$ of $G\open$ normalized by $T'$. Let $g'\in T'$ be the preimage of $r_G(g)$ under $r_G\colon T'\isoto\wt T$. Then $g$ and $g'$ both map to the same element in $\wt T$, and hence also in $G/G\open=G^\red/(G^\red)\open=\wt T/\wt T\open=T'/T\open$. Let $n$ be the order of $g$ in $G/G\open$. Then $g^n\in N_{G\open}(B)\cap N_{G\open}(T\open)=N_{B}(T\open)=Z_{B}(T\open)$ as in Remark~\ref{RemQuasiTorus}. Since $g^n$ is semi-simple, we have $g^n\in T\open$ by \cite[12.1~Theorem(d)]{Borel91}. We consider the subgroup $G_1$ of $G$ generated by $g,g'$ and $G\open$, which is automatically closed and equals $\coprod_{i=0}^{n-1}g^i G\open$. By \ref{ThmQuasiTorusGeneral_E} the subgroups $T_1:=\coprod_{i=0}^{n-1}g^i T\open$, respectively $T'_1:=\coprod_{i=0}^{n-1}(g')^i T\open=T'\cap G_1$, generated by $g,T\open$, respectively by $g',T\open$, are maximal quasi-tori in $G_1$, because they have the identity component $T\open$, normalize the Borel subgroup $B$, respectively $B'$ of $G_1\open=G\open$, and surject onto $G_1/G_1\open$. By \ref{ThmQuasiTorusGeneral_B} there is an element $h\in G\open$ with $h T'_1 h^{-1} =T_1$, and hence $g$ lies in the maximal quasi-torus $h T' h^{-1}$ of $G$.
\end{proof}

\begin{cor}\label{Cor6.10} 
Let $f\colon G\onto G'$ be a surjection of algebraic groups. Then the image of a maximal quasi-torus (resp.\ maximal torus, resp.\ Borel subgroup) in $G$ is again a maximal quasi-torus (resp.\ maximal torus, resp.\ Borel subgroup) in $G'$. Conversely, every maximal quasi-torus (resp.\ maximal torus, resp.\ Borel subgroup) in $G'$ arises in this way.
\end{cor}

\begin{proof}
For maximal tori and Borel subgroups this is just \cite[IV.11.14~Proposition]{Borel91}. So let $T\subset G$ be a maximal quasi-torus. By Theorem~\ref{ThmQuasiTorusGeneral}\ref{ThmQuasiTorusGeneral_D} there is a Borel subgroup $B\subset G\open$ which contains the maximal torus $T\open\subset G\open$ and is normalized by $T$. Then $f(T)$ normalizes the Borel subgroup $f(B)$ of $G'$, and $f(T)\open=f(T\open)$ is a maximal torus in $G'$. Since the surjection $T\onto G/G\open\onto G'/G'{}\open$ factors through $f(T)$, we see that $f(T)\onto G'/G'{}\open$ is surjective, and hence $f(T)$ is a maximal quasi-torus in $G'$ by Theorem~\ref{ThmQuasiTorusGeneral}\ref{ThmQuasiTorusGeneral_E}.

Conversely, if $T\subset G$ and $T'\subset G'$ are any maximal quasi-tori, then by Theorem~\ref{ThmQuasiTorusGeneral}\ref{ThmQuasiTorusGeneral_B} there is an element $h\in G'$ with $T'=h^{-1}f(T)h$. For any preimage $g\in G$ of $h$ the maximal quasi-torus $g^{-1}Tg$ in $G$ surjects onto $T'$. This proves the corollary.
\end{proof}

The corollary leads to the following 

\begin{rem}\label{handy} 
Let $G=G_1\times_{G_0}G_2$ be the fiber product of two linear algebraic groups $G_1$ and $G_2$ over a third $G_0$ for epimorphisms $G_1\onto G_0$ and $G_2\onto G_0$. Then $G$ is a closed subgroup of $G_1\TimesL G_2$ and the restrictions of the projections $\pi_1:G_1\TimesL G_2\to G_1$ and $\pi_2:G_1\TimesL G_2\to G_2$ are surjective. Let $T\subset G$ be a maximal quasi-torus. Then for $i=0,1,2$ the images $T_i\subset G_i$ of $T$ are maximal quasi-tori by Corollary \ref{Cor6.10} and $T=T_1\times_{T_0}T_2$. In particular, if $G_0\open$ is trivial then $G\open=G_1\open\TimesL G_2\open$ and $G\open\cap T$ is the maximal torus $T\open=T_1\open\TimesL T_2\open$. And if $G_0$ is trivial then $G=G_1\TimesL G_2$ and $T=T_1\TimesL T_2$.
\end{rem}

\begin{prop}\label{reduce_to_quasi} 
Let $f\colon G\onto G'$ be a surjection of linear algebraic groups, and let $T\subset G$ be a maximal quasi-torus in $G$. Then $T$ is a maximal quasi-torus in $H=f^{-1}(f(T))$.
\end{prop}

\begin{proof}
We will use Theorem~\ref{ThmQuasiTorusGeneral}\ref{ThmQuasiTorusGeneral_E} for the pair $T\subset H$. The identity component $T\open$ is a maximal torus in $G\open$, so it is a maximal torus in the smaller group $H\open$, too. Let $B\supset T\open$ be a Borel subgroup of $G\open$ which is normalized by $T$. Then $T$ also normalizes $(H\open\cap B)\open$, and we claim that the latter is a Borel subgroup of $H\open$. To prove the claim, let $\wt B\supset T\open$ be a Borel subgroup of $H\open$. Then $\wt B=(H\open\cap B_1)\open$ for a Borel subgroup $B_1\subset G\open$ by \cite[11.14~Proposition]{Borel91}. Since $T\open\subset B_1$ there is an $n\in N_{G\open}(T\open)$ with $B=n B_1 n^{-1}$ by \cite[11.19~Proposition]{Borel91}. Moreover, $f(H\open)=f(H)\open=f(T)\open=f(T\open)$ implies that $f(n H\open n^{-1})=f(n T\open n^{-1})=f(T\open)$, and hence $n H\open n^{-1}=H\open$. This shows that $(H\open\cap B)\open = n (H\open\cap B_1)\open n^{-1}$ is a Borel subgroup in $H\open$ as claimed.

By Theorem~\ref{ThmQuasiTorusGeneral}\ref{ThmQuasiTorusGeneral_E} it suffices to show that the map $H/H\open\to G/G\open$ induced by the inclusion $H\hookrightarrow G$ is injective, because then the surjection $T\onto G/G\open$ will factor through a surjection $T\onto H/H\open$. So let $h\in G\open\cap H$. Then $f(h)$ lies in $f(G\open)\cap f(H) = G'{}\open\cap f(T)=f(T)\open=f(T\open)$, because $f(T)$ is a maximal quasi-torus of $G'$ by Corollary~\ref{Cor6.10}. Therefore, $f(h)=f(t)$ for an element $t\in T\open\subset H\open\subset G\open$. Thus we have to show that the element $\tilde h:=ht^{-1}\in G\open\cap\ker f$ actually lies in $H\open$. For this purpose note, that $(\ker f)\open$ is a characteristic subgroup of $\ker f$, which in turn is normal in $G$. Therefore, $(\ker f)\open$ is normal in $G$. Let $\pi\colon G\onto G/(\ker f)\open=:\olG$ be the quotient morphism and set $\olH:=\pi(H)$ and $\ol T\open:=\pi(T)\open=\pi(T\open)$. The latter is a maximal torus in $\olG$. Now $\pi(\tilde h)$ lies in the finite group $\Gamma:=\olG\open\cap\,\ker f/(\ker f)\open$, which is normal in $\olG\open$. The operation of $\olG\open$ by conjugation on $\Gamma$ factors through the finite automorphism group of $\Gamma$, and hence is trivial because $\olG\open$ is connected. It follows that $\Gamma$ is contained in the center of $\olG\open$. Since $\pi(\tilde h)$ has finite order, it is semi-simple, and hence lies in $\ol T\open=\pi(T\open)$ by \cite[IV.11.11~Corollary]{Borel91}. Thus $\pi(\tilde h)=\pi(\tilde t)$ for an element $\tilde t\in T\open$, and so $ht^{-1}\tilde t^{-1}=\tilde h\tilde t^{-1}\in\ker\pi=(\ker f)\open\subset H\open$, and $\tilde h, h\in H\open$. We conclude that $G\open\cap H=H\open$ and $H/H\open\to G/G\open$ is injective as desired. This proves the proposition.
\end{proof}

The proposition has the following consequence.

\begin{thm}\label{new_gambit} 
Let $\phi\colon G_1\into G_2$ be an injective homomorphism of algebraic groups, and assume that there is a closed normal subgroup $N\triangleleft G_1$ such that $\phi(N)\triangleleft G_2$ is also normal. Let $T\subset G_1$ be a maximal quasi-torus. If its image in $G_2/\phi(N)$ is a maximal quasi-torus, then also its image in $G_2$ is a maximal quasi-torus.
\end{thm}

\begin{proof}
Let $f_1\colon G_1\to G_1/N$ and $f_2\colon G_2\to G_2/\phi(N)$ be the quotient maps. Then $f_1(T)\subset G_1/N$ is a maximal quasi-torus by Corollary~\ref{Cor6.10} and its image $f_2\phi(T)$ in $G_2/\phi(N)$ is a maximal quasi-torus by assumption. Note that $\phi$ maps the preimage $H_1=f_1^{-1}(f_1(T))=NT\subset G_1$ isomorphically to the preimage $H_2=f_2^{-1}(f_2\phi(T))=\phi(NT)\subset G_2$. By Corollary~\ref{Cor6.10} there is a maximal quasi-torus $T_2$ in $G_2$ with $f_2(T_2)=f_2\phi(T)$. By Proposition~\ref{reduce_to_quasi} the subgroups $T\subset H_1$ and $T_2\subset H_2$ are maximal quasi-tori. Since $\phi|_{H_1}\colon H_1\to H_2$ is an isomorphism, $\phi(T)$ is a maximal quasi-torus in $H_2$. Therefore, $\phi(T)$ is conjugate to $T_2$ under $H_2\open$ by Theorem~\ref{ThmQuasiTorusGeneral}\ref{ThmQuasiTorusGeneral_B}. So $\phi(T)$ is conjugate to a maximal quasi-torus in $G_2$, and hence is itself a maximal quasi-torus in $G_2$. 
\end{proof}

Another useful criterion for being a maximal quasi-torus is given in the following

\begin{thm} \label{ThmAbelianQuasiTorus}
Let $G$ be reductive, and let $H\subset G$ be a closed subgroup with the following properties:
\begin{enumerate}
\item\label{ThmAbelianQuasiTorus_A} 
the identity component $H\open$ of $H$ is a maximal torus in $G\open$,
\item\label{ThmAbelianQuasiTorus_B} 
the natural map $H/H\open\to G/G\open$ is surjective,
\item\label{ThmAbelianQuasiTorus_C} 
the group $H$ is commutative.
\end{enumerate}
Then $H$ is a maximal quasi-torus in $G$, and it is the only maximal quasi-torus in $G$ containing $H\open$.
\end{thm}

Condition \ref{ThmAbelianQuasiTorus_C} can not be dropped as the following example shows.

\begin{example}\label{ExThmAbelianQuasiTorus}
Consider the group $G=\PGL_2\rtimes \BZ/2$ with the action of the generator $1\in\BZ/2$ on $\PGL_2$ via conjugation by $\left(\begin{smallmatrix} 0 & 1 \\ 1 & 0 \end{smallmatrix}\right)$, and the subgroup $H=T\open\rtimes \BZ/2$ where $T\open\subset\PGL_2$ is the torus of diagonal matrices. Then $H$ satisfies conditions \ref{ThmAbelianQuasiTorus_A} and \ref{ThmAbelianQuasiTorus_B} but is not a maximal quasi-torus in $G$, because the element $(1,1)\in H$ does not normalize any of the two Borel subgroups $B$ containing $T\open$. Instead, the maximal quasi-torus normalizing $B$ and $T\open$ is $\bigl(\left(\begin{smallmatrix} 0 & \mbox{\large $*$} \\ \mbox{\large $*$} & 0 \end{smallmatrix}\right),1\bigr)\amalg \bigl(T\open,0\bigr)$.
\end{example}

To prove the theorem we will need the following 

\begin{lemma}\label{8.2} 
In the situation of the theorem, two elements of $H$ are conjugate under $G\open$ if and only if they are conjugate under $N_{G\open}(H\open)$. In particular, the intersection of every $G\open$-conjugacy class with $H$ is finite. (Note however, that there is no action of $N_{G\open}(H\open)$ on $H$ in general, because $N_{G\open}(H\open)$ only normalizes $H\open$ and not necessarily $H$.)
\end{lemma}

\begin{proof}
In order to prove the first claim, note that one direction follows from the  inclusion $N_{G\open}(H\open)\subset G\open$. To prove the converse, let $h,h'\in H$ be conjugate under $G\open$, say $h=x^{-1}h'x$ for an $x\in G\open$. Since $H$ is commutative by condition~\ref{ThmAbelianQuasiTorus_C}, $h'$ centralizes $H\open$, and hence the conjugate $h=x^{-1}h'x$ centralizes $x^{-1}H\open x$. But $h$ also centralizes $H\open$, so we get that $H\open$ and $x^{-1}H\open x$ lie in the centralizer $G^h$ of $h$ in $G\open$. Since $H\open$ and  $x^{-1}H\open x$ are maximal tori in $G\open$ by condition~\ref{ThmAbelianQuasiTorus_A}, they are maximal tori in $G^h$, too. So there is a $y\in G^h$ such that $y^{-1}x^{-1}H\open xy=H\open$. Set $w=xy$. Then $w\in N_{G\open}(H\open)$, and $w^{-1}h'w=y^{-1}x^{-1}h'xy=y^{-1}hy=h$, as $y\in G^h$. Therefore, $h$ is conjugate to $h'$ under $N_{G\open}(H\open)$.

To prove the second claim, note that the identity component of $N_{G\open}(H\open)$ is $H\open$ by assumption~\ref{ThmAbelianQuasiTorus_A} and \cite[III.8.10, Corollary~2 and IV.13.17, Corollary~2]{Borel91}. By assumption~\ref{ThmAbelianQuasiTorus_C}, $H\open$ acts trivially by conjugation on $H$. Therefore, if two elements $h,h'$ of $H$ are conjugate under $N_{G\open}(H\open)$, they are conjugate under the Weyl group $W=N_{G\open}(H\open)/H\open$, which is finite. 
\end{proof}

\begin{proof}[Proof of Theorem~\ref{ThmAbelianQuasiTorus}.] 
By condition~\ref{ThmAbelianQuasiTorus_A} and Remark~\ref{RemQuasiTorus} there is a maximal quasi-torus $T\subset G$ such that $T\open=H\open$. We have to show that $T=H$. To this end fix any $t\in T$. By condition~\ref{ThmAbelianQuasiTorus_B} there is an $h\in H$ such that $h^{-1}t\in G\open$, and hence $tG\open=hG\open$. We make the following

\medskip\noindent
\emph{Claim.} We have $h^{-1}t\in T\open$.

\medskip\noindent
We prove the claim. Since $T\open$ is commutative, every element $tx\in tT\open$ centralizes $T^t\subset T\open$. Therefore, the quotient group $Q=T\open/T^t$ acts faithfully on $tT\open$ by conjugation. More precisely, $q\in Q$ acts as $tx\mapsto qtxq^{-1}=tx(t^{-1}qtq^{-1})$ for $x\in T\open$. The map $\psi\colon q\mapsto t^{-1}qtq^{-1}$ from $Q$ to $T\open$ is an injective homomorphism of tori, and hence we may form the quotient $Y=t\cdot T\open/\psi(Q)$ of $tT\open$ by the $Q$-action. Let $\pi\colon tT\open\to Y$ be the quotient map. Its fibers are the orbits of the $Q$-action. Consider the set
\[
X=\{(x,y)\in hT\open\TimesL Y\colon\exists\, a\in tT\open,\, \exists\, b\in G\open\text{ such that }\pi(a)=y \text{ and } b^{-1}ab=x)\,\}.
\]
We show that $X$ is a constructible set. Namely consider the morphism 
\[
\phi\colon\;tT\open\TimesL G\open\;\longto\; hG\open\TimesL Y,\quad(a,b) \;\longmapsto\;(b^{-1}ab,\pi(a))\,.
\]
The preimage $\phi^{-1}(hT\open\TimesL Y)\subset tT\open\TimesL G\open$ is a closed subset, and $X=\phi\bigl(\phi^{-1}(hT\open\TimesL Y)\bigr)$. Therefore, $X$ is a constructible set by Chevalley's theorem~\cite[$\mathrm{IV}_1$, Corollaire~1.8.5]{EGA}. Let $\pi_1\colon X\to hT\open$ and $\pi_2\colon X\to Y$ be the projections onto the first and the second factor, respectively. Every element of $hT\open=hH\open\subset H$ is semi-simple by Lemma~\ref{LemmaUnipotConnected}, because some power of it lies in the torus $H\open$. Therefore, this element is conjugate under $G\open$ to an element of $tT\open$ by Theorem~\ref{ThmQuasiTorusReductive}\ref{ThmQuasiTorusReductive_A},\ref{ThmQuasiTorusReductive_B}. Thus the map $\pi_1$ is surjective, and hence the dimension of $X$ is at least $\dim(T\open)=\dim(hT\open)$ by \cite[IV$_2$, Corollaire~5.6.8]{EGA}. For every $y\in Y$, the points in the fiber $\pi^{-1}(y)$ of $\pi\colon tT\open\to Y$ are conjugate under $T\open$, so the fiber of $\pi_2\colon X\to Y$ above $y\in Y$, which equals 
\[
\{\,x\in hT\open\colon \exists\,a\in\pi^{-1}(y),\, \exists\,b\in G\open \text{ with }x=b^{-1}ab\,\}\;=\;H\cap\{\,b^{-1}ab\colon a\in\pi^{-1}(y),b\in G\open\,\}\,,
\]
is the intersection of $H$ with the $G\open$-conjugacy class of $a$ for any $a\in\pi^{-1}(y)$. This is a finite set by Lemma~\ref{8.2}. The image $\pi_2(X)\subset Y$ is constructible by Chevalley's theorem, and the fibers of the surjective map $\pi_2\colon X\to\pi_2(X)$ are finite by the above. Thus $\dim(\pi_2(X))=\dim(X)$ by \cite[IV$_2$, Corollaire~5.6.8]{EGA}, and this is at least $\dim(T\open)=\dim(tT\open)$. Together we obtain $\dim(Y)\ge\dim(\pi_2(X))\ge\dim(tT\open)=\dim(Y)+\dim(Q)$. This means that $Q$ is zero-dimensional and connected as a quotient of $T\open$, and hence $T^t=T\open$. Therefore, both $h$ and $t$ centralize $T\open$, and so $h^{-1}t$ centralizes $T\open$, too. But $h^{-1}t\in G\open$ and the centralizer $Z_{G\open}(T\open)$ of $T\open$ in $G\open$ is $T\open$ itself by \cite[IV.13.17, Corollary~2]{Borel91}. This proves the claim.

By the claim $tT\open=hT\open=hH\open$, so we get that $H$ contains $T$. Now we only need to show the reverse inclusion. Let $h\in H$ be again arbitrary. Since the natural map $T/T\open\to G/G\open$ is surjective by Theorem~\ref{ThmQuasiTorusReductive}\ref{ThmQuasiTorusReductive_A}, there is a $t\in T$ such that $t^{-1}h\in G\open$. Since $T$ is in $H$, we get that $t^{-1}h$ is in $H$, too. But $H$ centralizes $H\open$, so $t^{-1}h$ is in $Z_{G\open}(H\open)=H\open$. Therefore $hH\open=tH\open=tT\open$, and hence $T$ contains $H$. So $T=H$, and this finishes the proof of Theorem~\ref{ThmAbelianQuasiTorus}.
\end{proof}

\subsection{Intersecting Conjugacy Classes with Maximal Quasi-Tori}\label{SubSectConjMQT}

In this subsection we collect several results which we will need in the following sections. The following setup was already briefly used in the proof of Theorem~\ref{ThmAbelianQuasiTorus}.

\begin{notn} \label{Notation9.4NEW}
Let $G$ be a linear algebraic group, let $T\subset G$ be a maximal quasi-torus. Note that $T\open$ commutes with $T^h:=\{g\in T\open\colon gh=hg\}$ for every $h\in T$. So we have $T^{ht}=T^h$ for every $t\in T\open$. The conjugation action of $T\open$ on $hT\open$ is given for $t\in T\open$ and $hx\in hT\open$ by $thxt^{-1}=hh^{-1}thxt^{-1}=hxh^{-1}tht^{-1}$, because $h^{-1}th\in T\open$ commutes with $x\in T\open$. Therefore, the map
\begin{equation}\label{EqT0Conjugates_hT0}
\phi\colon T\open\longto T\open,\quad t\longmapsto h^{-1}tht^{-1}
\end{equation}
satisfies $thxt^{-1}=hx\phi(t)$ for every $t\in T\open$ and $hx\in hT\open$. Moreover, $\phi$ is a homomorphism of algebraic groups, namely the product of the endomorphisms $t\mapsto h^{-1}th$ and $t\mapsto t^{-1}$ of the commutative group $T\open$. The kernel of $\phi$ is $T^h$. Let $Q_h$ be the quotient $T\open/T^h$. Then $\phi$ induces a closed immersion of tori
\[
\ol\phi\colon Q_h\longinto T\open,\quad \bar t=t\mod T^h\longmapsto h^{-1}tht^{-1}.
\]
\end{notn}

\begin{prop}\label{comment27NEW} 
In the situation of Notation~\ref{Notation9.4NEW} the following holds:
\begin{enumerate}
\item\label{comment27NEW_A} 
The natural group homomorphisms $T^h{}\open\TimesL Q_h\to T\open$, $(t_0,\bar t\,)\mapsto t_0\cdot\ol\phi(\bar t)$ and $T^h{}\open\to T\open/\ol\phi(Q_h)$ are surjective with finite kernels.
\item\label{comment27NEW_B} 
Every element of $hT\open$ is conjugate under $T\open$ to an element of $hT^h{}\open$,
\item\label{comment27NEW_C} 
The intersection of $hT^h{}\open$ with any $G\open$-conjugacy class (respectively $G$-conjugacy class) is finite.
\end{enumerate}
\end{prop}

\begin{proof} 
\ref{comment27NEW_A} The kernel of the first homomorphism is $\{(t_0,\bar t\,)\colon t_0^{-1}=\ol\phi(\bar t)=h^{-1}tht^{-1}\}$. This condition is equivalent to $h^{-1}th=t_0^{-1}t$. Since $t_0\in T^h$ we obtain $h^{-n}th^n=t_0^{-n}t$ for every positive integer $n$. If $n$ equals the order of $h$ in the group $T/T\open$, then $h^n\in T\open$ and $h^{-n}th^n=t$. This shows that $t_0^n=1$ and $\ol\phi(\bar t\,{}^n)=\ol\phi(\bar t)^n=t_0^{-n}=1$. In particular, the kernel of the first (respectively second) homomorphism is contained in the $n$-torsion subgroup of the torus $T^h{}\open\TimesL Q_h$ (respectively $T^h{}\open$). Both are finite groups. The surjectivity now follows from the irreducibility of the targets, and from the comparison of dimensions $\dim Q_h=\dim T\open-\dim T^h$ and $\dim T\open/\ol\phi(Q_h)=\dim T\open-\dim Q_h=\dim T^h=\dim T^h{}\open$.

\medskip\noindent
\ref{comment27NEW_B} By \ref{comment27NEW_A} every $hx\in hT\open$ is of the form $hx=ht_0\ol\phi(\bar t)=ht_0\phi(t)=t(ht_0)t^{-1}$ for $ht_0\in hT^h{}\open$ and $t\in T\open$.

\medskip\noindent
\ref{comment27NEW_C} Since the projection $G\to G^\red$ maps $T$ isomorphically to its image, and it maps conjugate elements to conjugate elements, we may assume that $G$ is reductive. (Note that there may be elements in $T$ which are not conjugate under $G$, but whose images are conjugate under $G^\red$. So the cardinality of the intersection in question may grow by passing to $G^\red$.) The proposition is now a consequence of the following more precise statement.
\end{proof}

\begin{prop}\label{PropComment27NEW}
In the situation of Notation~\ref{Notation9.4NEW} let $G$ be reductive and let $\Omega\subset \{\,w\in N_G(T^h{}\open)\colon whw^{-1}\in hT^h{}\open\,\}$ be a finite subset which is maximal (under the inclusion of subsets) such that $\Omega\into N_G(T^h{}\open)/N_G(T^h{}\open)\open$ is injective. Moreover, let $m>0$ be the smallest positive integer with $h^m\in T^h{}\open$ and consider $\Delta:=\{z\in T^h{}\open\colon z^m=1\}$, which is a finite group. If two elements $u,v\in hT^h{}\open$ are conjugate under $G$, then there are elements $w\in \Omega$ and $z\in \Delta$ with $u=zwv w^{-1}$.
\end{prop}

\begin{rem}\label{RemSteinbergConjugacy}
(a) The set $\{\,w\in N_G(T^h{}\open)\colon whw^{-1}\in hT^h{}\open\,\}$ is a subgroup. But in general it does not contain $N_G(T^h{}\open)\open$.

\medskip\noindent
(b) Note that we do \emph{not} claim that for every $w\in \Omega$ and $z\in \Delta$ the element $zwv w^{-1}$ is conjugate to $v$.

\medskip\noindent
(c) When $G$ is connected, $T=T\open$ and so $h\in T\open$ and $T^h=T^h{}\open=T\open$ is a maximal torus in $G$. Thus we can take $\Omega=W(T,G)$ as (a system of representatives of) the Weyl group of $T\open$. Also $m=1$ and $\Delta=\{1\}$. In this way we recover the result of Steinberg~\cite[\S\,III.3.4, Corollary~2]{Steinberg74}: Two elements of $T$ are conjugate under $G$ if and only if they are conjugate under $W(T,G)$.
\end{rem}

\begin{proof}[Proof of Proposition~\ref{PropComment27NEW}.]
First note that some power $h^n$ of $h$ lies in $T\open$, and since $h^n$ commutes with $h$, it also lies in $T^h$. Multiplying the exponent further by the order of $T^h/T^h{}\open$ produces an integer $m>0$ such that $h^m\in T^h{}\open$. Also note that $\Delta$ is a finite group because $T^h{}\open$ is commutative.

Let $v$ and $u=xvx^{-1}\in hT^h{}\open$ be conjugate via $x\in G$. Since $v$ centralizes $T^v{}\open$, the conjugate $u=xvx^{-1}$ centralizes $xT^v{}\open x^{-1}$. By Theorem~\ref{steinberg} the subgroups $T^u{}\open\subset G^u{}\open$ and $T^v{}\open\subset G^v{}\open$ are maximal tori, and hence $xT^v{}\open x^{-1}$ is also a maximal torus in $xG^v{}\open x^{-1}=G^u{}\open$. So there is a $y\in G^u{}\open$ such that $yxT^v{}\open x^{-1}y^{-1}=T^u{}\open=T^h{}\open=T^v{}\open$. Then $w:=yx\in N_G(T^h{}\open)$, and $y\in G^u{}\open$ implies $u=yuy^{-1}=yxvx^{-1}y^{-1}=wvw^{-1}$.

Writing $u=h\tilde u=\tilde uh$ and $v=h\tilde v=\tilde vh$ with $\tilde v,\tilde u\in T^h{}\open$, we see that $h\tilde u=u=wvw^{-1}=whw^{-1}w\tilde vw^{-1}$ and thus $whw^{-1}=h\tilde u(w\tilde v^{-1}w^{-1})\in hT^h{}\open$. So there is an element $\tilde w\in \Omega$ with $w=\tilde wn$ for an $n\in N_G(T^h{}\open)\open$. Let $t\in T^h{}\open$ with $\tilde wh\tilde w^{-1}=ht=th$. Then $u^{-1}(\tilde w v\tilde w^{-1})=\tilde u^{-1}h^{-1}(\tilde w h\tilde v\tilde w^{-1})=\tilde u^{-1}t(\tilde w \tilde v\tilde w ^{-1})=(\tilde w \tilde v\tilde w ^{-1})t\tilde u^{-1}=(\tilde w \tilde vh\tilde w^{-1})h^{-1}\tilde u^{-1}=(\tilde w v\tilde w^{-1})u^{-1}$, because $T^h{}\open$ is commutative. We set $z:=u(\tilde w v\tilde w^{-1})^{-1}\in T^h{}\open$ and compute $z^m=u^m(\tilde w v^m\tilde w^{-1})^{-1}$. By \cite[III.8.10, Corollary~2]{Borel91} we have $n\in N_{G\open}(T^h{}\open)\open=Z_{G\open}(T^h{}\open)\open$, and from $v^m=h^m\tilde v^m \in T^h{}\open$ we conclude that $u^m=w v^mw^{-1}=\tilde wn v^mn^{-1}\tilde w^{-1}=\tilde w v^m\tilde w^{-1}$. Therefore, $z^m=1$ and $z\in \Delta$, whence $u=z\tilde w v\tilde w^{-1}$ as claimed.
\end{proof}

\begin{cor}\label{comment27b} 
In the situation of Notation~\ref{Notation9.4NEW} let $C$ be the intersection of $hT\open$ with a $G\open$-conjugacy class (respectively a $G$-conjugacy class) in $G$. Then $C$ is a finite union of $T\open$-conjugacy classes on $hT\open$.
\end{cor}

\begin{proof} 
Write $C$ as a disjoint union $C=\coprod_{i\in I}C_i$ such that each $\emptyset\ne C_i\subset hT\open$ is the intersection of $hT\open$ with a $T\open$-conjugacy class. By Proposition~\ref{comment27NEW}\ref{comment27NEW_B} for each $i\in I$ the intersection $C_i\cap hT^h{}\open$ is non-empty. Therefore, $\#(C\cap hT^h{}\open)\geq\#I$. But the set $C\cap hT^h{}\open$ is finite by Proposition~\ref{comment27NEW}\ref{comment27NEW_C}.
\end{proof}

Recall the symbol $^{H\!}C$ from Notation~\ref{InitialNotation}.

\begin{prop}\label{Prop8.8} 
In the situation of Notation~\ref{Notation9.4NEW} let $C\subset hT\open$ be a subset and let $H\subset G$ be a closed subgroup containing $T\open$. Then ${}^{H\!}C\cap hT^h{}\open$ is dense in $hT^h{}\open$ if and only if ${}^{H\!}C\cap hT\open$ is dense in $hT\open$.
\end{prop}

\begin{proof}
We set $C_1:={}^{H\!}C\cap hT^h{}\open$ and observe that ${}^{H\!}C_1={}^{H\!}C$, because the inclusion ${}^{H\!}C_1\subset {}^{H\!}C$ is obvious and the opposite inclusion follows from Proposition~\ref{comment27NEW}\ref{comment27NEW_B}, because $T\open\subset H$. 

First assume that $hT^h{}\open$ equals the closure $\ol{C_1}$ of $C_1$, and let $x\in hT\open$. By Proposition~\ref{comment27NEW}\ref{comment27NEW_B} there is a $t\in T\open\subset H$ with $txt^{-1}\in hT^h{}\open=\ol{C_1}$. Then $x\in t^{-1}\ol{C_1}t=\ol{t^{-1}C_1t\,}$. Since $t^{-1}C_1t$ is contained in ${}^{H\!}C_1\cap hT\open={}^{H\!}C\cap hT\open$ we conclude that $x\in \ol{t^{-1}C_1t\,}\subset\ol{{}^{H\!}C\cap hT\open}$. Therefore, ${}^{H\!}C\cap hT\open$ is dense in $hT\open$.

For the converse implication we assume that $\ol{C_1}\ne hT^h{}\open$ and consider the closed subset
\[
D\;:=\;\{\,(c_1,b,g)\in C_1\TimesL hT\open\TimesL H\text{ such that } b=gc_1g^{-1}\,\}\;\subset\;C_1\TimesL hT\open \TimesL H
\]
and the projections $\pi_1\colon D\to C_1$ and $\pi_2\colon D\to hT\open$. Then $\pi_2(D)={}^{H\!}C_1\cap hT\open={}^{H\!}C\cap hT\open$. We consider the following diagram which is \emph{not} commutative
\[
\xymatrix @C+2pc {
D \ar[r]^{\textstyle\pi_2} \ar[d]_{\textstyle\pi_1} & hT\open \ar[d]^{\textstyle\gamma} \\
\qquad\qquad C_1 \subset hT^h{}\open \ar[r]^(0,6){\textstyle\beta} & hT\open/\ol\phi(Q_h)
}
\]
where $\beta$ is induced from the homomorphism from Proposition~\ref{comment27NEW}\ref{comment27NEW_A}. Although the diagram is not commutative, we claim that $\gamma\pi_2(D)\subset \beta\pi_1(D)$. Indeed, let $x\in\gamma\pi_2(D)$ and let $(c_1,b,g)\in D$ be a preimage of $x$, that is $b=gc_1g^{-1}$. By Proposition~\ref{comment27NEW}\ref{comment27NEW_B} there is an element $t\in T\open\subset H$ such that $c:=tbt^{-1}\in hT^h{}\open$. Then $c=(tg)c_1(tg)^{-1}\in {}^{H\!}C_1\cap hT^h{}\open=C_1$. Moreover, $b=t^{-1}ct=c\phi(t^{-1})$, and hence $(c,b,t^{-1})\in D$. This shows that $x=\gamma\pi_2(c_1,b,g)=b\cdot\ol\phi(Q_h)=c\cdot\ol\phi(Q_h)=\beta\pi_1(c,b,t^{-1})\in\beta\pi_1(D)$ and proves the claim.

Since $\ol{C_1}\ne hT^h{}\open$ and $hT^h{}\open$ is irreducible, we get $\dim\ol{C_1}<\dim hT^h{}\open$. Since $\beta$ is a finite morphism, $\beta(\ol{C_1})$ is closed with $\dim\beta(\ol{C_1})=\dim\ol{C_1}<\dim hT^h{}\open=\dim hT\open/\ol\phi(Q_h)$, and therefore $\beta(\ol{C_1})$ is a proper closed subset of $hT\open/\ol\phi(Q_h)$ which contains $\gamma\pi_2(D)$. Since $\gamma$ is surjective, the preimage of $\beta(\ol{C_1})$ under $\gamma$ is a proper closed subset of $hT\open$ which contains $\pi_2(D)={}^{H\!}C\cap hT\open$. This shows that ${}^{H\!}C\cap hT\open$ is not dense in $hT\open$ and finishes the proof.
\end{proof}

\begin{prop}\label{finite_to_oneNEW} 
Let $G$ be a linear algebraic group, let $T\subset G$ be a maximal quasi-torus, let $h\in T$, and let $C\subset hT\open$ be a subset. Let $H\subset G$ and $\wt H\subset G^\red$ be closed subgroups with $T\open\subset H$ and $r_G(H)\subset\wt H$. Then ${}^{H\!}C\cap hT\open$ is dense in $hT\open$ if and only if ${}^{\wt H\!}r_G(C)\cap r_G(hT\open)$ is dense in $r_G(hT\open)$. 
\end{prop}

\noindent
{\itshape Remark.} Note that the ``if''-direction is not obvious, because there may be elements in $hT\open$ which are not conjugate under $H$, but whose images are conjugate under $r_G(H)$.

\begin{proof} 
Note that $r_G({}^{H\!}C)\subset{}^{\wt H\!}r_G(C)$, and hence $r_G({}^{H\!}C\cap hT\open)\subset{}^{\wt H\!}r_G(C)\cap r_G(hT\open)$. Since $r_G\colon hT\open\isoto r_G(hT\open)$ is an isomorphism, the ``only if''-direction is clear.

To prove the converse, we use Notation~\ref{Notation9.4NEW} and let $C_1:={}^{H\!}C\cap hT^h{}\open$. Then ${}^{H\!}C_1={}^{H\!}C$ as at the beginning of the proof of Proposition~\ref{Prop8.8}. Moreover,  ${}^{\wt H\!}r_G(C_1)={}^{\wt H\!}r_G(C)$, because the inclusion ${}^{\wt H\!}r_G(C_1)\subset {}^{\wt H\!}r_G(C)$ follows from $r_G(C_1)\subset r_G({}^{H\!}C)\subset{}^{\wt H\!}r_G(C)$, and the opposite inclusion follows from $r_G(C)\subset r_G({}^{H\!}C_1)\subset{}^{\wt H\!}r_G(C_1)$. So by Proposition~\ref{Prop8.8} it suffices to show that $C_1$ is dense in $hT^h{}\open$ provided that ${}^{\wt H\!}r_G(C_1)\cap r_G(hT^h{}\open)$ is dense in $r_G(hT^h{}\open)$. 

Let $u\in{}^{\wt H\!}r_G(C_1)\cap r_G(hT^h{}\open)$, that is $u=\tilde gv\tilde g^{-1}$ for some $\tilde g\in\wt H$ and $v\in r_G(C_1)\subset r_G(hT^h{}\open)$. Then in the notation of Proposition~\ref{PropComment27NEW} applied to $r_G(T^h{}\open)\subset G^\red$, there are elements $z\in \Delta$ and $w\in \Omega$ with $u=zwvw^{-1}$. We conclude that 
\[
{}^{\wt H\!}r_G(C_1)\cap r_G(hT^h{}\open)\;\subset\; \bigcup_{z\in \Delta,w\in \Omega}zw \,r_G(C_1)\,w^{-1}\;\subset\;r_G(hT^h{}\open)\,. 
\]
If the closure $\ol{C_1}$ of $C_1$ is strictly contained in $hT^h{}\open$, then $\ol{\,r_G(C_1)\,}=r_G(\ol{C_1})\ne r_G(hT^h{}\open)$, because $r_G|_T\colon T\isoto r_G(T)$ is an isomorphism. We obtain an inequality of dimensions $\dim\ol{\,r_G(C_1)\,}<\dim r_G(hT^h{}\open)$, because $r_G(hT^h{}\open)$ is irreducible. On the other hand the closure 
\[
\ol{{}^{\wt H\!}r_G(C_1)\cap r_G(hT^h{}\open)}\;\subset\;\bigcup_{z\in \Delta,w\in \Omega}zw\ol{\,r_G(C_1)\,}w^{-1}\;\subset\;r_G(hT^h{}\open)\,,
\]
because the union is closed as a finite union of closed subsets. Since $\dim zw\ol{\,r_G(C_1)\,}w^{-1}=\dim\ol{\,r_G(C_1)\,}<\dim r_G(hT^h{}\open)$ and $r_G(hT^h{}\open)$ is irreducible, it cannot be the finite union of proper closed subsets. This implies that ${}^{\wt H\!}r_G(C_1)\cap r_G(hT^h{}\open)$ is not dense in $r_G(hT^h{}\open)$ and proves the proposition.
\end{proof}

\begin{lemma}\label{useful}
Let $G$ be reductive (but not necessarily connected) and let $T\subset G$ be a maximal quasi-torus.
\begin{enumerate}
\item \label{useful_A}
For every $g\in G$ we have $\ol{{}^{G\!}\{g\}}\cap T= {}^{G\!}\{g_s\}\cap T$, where $g_s$ is the semi-simple part of $g$.
\item \label{useful_B}
Let $\{C_x\colon x\in S\}$ be a collection of conjugacy classes in $G$ and let $h\in T$. Then $hG\open\cap\bigcup_{x\in S}C_x$ is dense in $hG\open$ if and only if $hT\open\cap\bigcup_{x\in S}\ol{C_x}$ is dense in $hT\open$.
\end{enumerate}
\end{lemma}

\begin{proof} 
Let $c\colon G\TimesL G\to G$ be the map given by the rule $(h,g)\mapsto hgh^{-1}$.

\medskip\noindent
\ref{useful_A} 
Theorem~\ref{ThmQuasiTorusReductive}\ref{ThmQuasiTorusReductive_A}, \ref{ThmQuasiTorusReductive_B} imply that ${}^{G\!}\{g\}$ is closed in $G$ if and only ${}^{G\!}\{g\}\cap T\ne\emptyset$. So if ${}^{G\!}\{g\}\cap T=\emptyset$ there is a $g'\in\ol{{}^{G\!}\{g\}}\smallsetminus {}^{G\!}\{g\}$. Since ${}^{G\!}\{g\}$ is open in $\ol{{}^{G\!}\{g\}}$ by \cite[I.1.8~Proposition]{Borel91} we have $\ol{{}^{G\!}\{g'\}}\subsetneq\ol{{}^{G\!}\{g\}}$. Proceeding in this way will eventually produce a $g'\in\ol{{}^{G\!}\{g\}}\cap T$. Now we use that the map $s\colon h\mapsto h_s$ on $\ol{{}^{G\!}\{g\}}$ sending $h\in G$ to its semi-simple part $h_s$ is actually a morphism of schemes by Lemma~\ref{algebraic_jordan} below. Since ${}^{G\!}\{g\}$ is mapped into ${}^{G\!}\{g_s\}$ under $s$, its closure $\ol{{}^{G\!}\{g\}}$ is mapped into $\ol{{}^{G\!}\{g_s\}}$. But $g_s$ is semi-simple, so ${}^{G\!}\{g_s\}$ is closed,  and hence the image of $\ol{{}^{G\!}\{g\}}$ under $s$ lies in ${}^{G\!}\{g_s\}$. Therefore, $g'=g'_s$ also lies in ${}^{G\!}\{g_s\}$, and this shows that ${}^{G\!}\{g_s\}={}^{G\!}\{g'\}\subset \ol{{}^{G\!}\{g\}}$. Moreover, all semi-simple elements $h=h_s\in\ol{{}^{G\!}\{g\}}$ are mapped under $s$ to ${}^{G\!}\{g_s\}$. Since $h=s(h)$, we conclude that $\ol{{}^{G\!}\{g\}}\cap T\subset {}^{G\!}\{g_s\}\cap T$ proving \ref{useful_A}.

\medskip\noindent
\ref{useful_B} Let $A:=\ol{hG\open\cap\bigcup_{x\in S}C_x}$ and $B:=\ol{hT\open\cap\bigcup_{x\in S}\ol{C_x}}$ denote the closures. Note that $A$ contains $B$, and since $A$ is invariant under conjugation by $G\open$, it also contains $\ol{c(G\open\TimesL B)}$. From Theorem~\ref{ThmQuasiTorusReductive}\ref{ThmQuasiTorusReductive_D} we conclude that the set of semi-simple elements in $hG\open$, which equals $c(G\open\TimesL hT\open)$, is dense in $hG\open$. Now if $B=hT\open$ then $A$ contains $\ol{c(G\open\TimesL hT\open)}=hG\open$.

Conversely, we have to show that $A=hG\open$ implies $B=hT\open$. Assume that this is not the case and let $V\subset hT\open$ be the open complement of $B$. We claim that $\ol{c(G\open\TimesL V)}=hG\open$. Indeed, $(G\open\TimesL hT\open)\smallsetminus c^{-1}(\ol{c(G\open\TimesL V)})$ is open in $G\open\TimesL hT\open$ and its image in $hT\open$ with respect to the projection does not meet $V$. Since the projection $G\open\TimesL hT\open\to hT\open$ is flat of finite presentation, this image is open by \cite[IV$_2$, Th\'eor\`eme~2.4.6]{EGA}, and hence empty because $V$ is open and dense in the irreducible variety $hT\open$. Thus $(G\open\TimesL hT\open)\subset c^{-1}(\ol{c(G\open\TimesL V)})$ and $c(G\open\TimesL hT\open)\subset\ol{c(G\open\TimesL V)}$. Since $\ol{c(G\open\TimesL hT\open)}=hG\open$, we have $\ol{c(G\open\TimesL V)}=hG\open$ as claimed. By Chevalley's theorem $c(G\open\TimesL V)$ is constructible, so it must contain a non-empty open subset $O\subset hG\open$ by \cite[0$_{\rm III}$, Proposition~9.2.2]{EGA}. Since $hG\open\cap\bigcup_{x\in S}C_x$ is dense in $hG\open$, there is an $x\in S$ such that $C_x\cap O\neq\emptyset$. Thus, there is a $(h,v)\in G\open\times V$ with $hvh^{-1}\in C_x$, and hence $V\cap C_x\neq\emptyset$, because $C_x$ is a conjugacy class. This is a contradiction.
\end{proof}

The following lemma is well known, but we could not find it in the literature. So we include a proof.

\begin{lemma}\label{algebraic_jordan} 
Let $G$ be a linear algebraic group, and consider the maps of sets $s\colon G\to G,\,h\mapsto h_s$ and $u\colon G\to G,\,h\mapsto h_u$, which send every element $h$ to its semi-simple part $h_s$, respectively unipotent part $h_u$. Let $g\in G$. Then the restriction of these maps to the reduced closed subscheme $\ol{{}^{G\!}\{g\}}$, where ${}^{G\!}\{g\}$ is the conjugacy class of $g$, are morphisms of schemes. 
\end{lemma}

\begin{proof} 
Let $\rho\colon G\into \GL_n=:H$ be a faithful linear representation. Since $\rho(\ol{{}^{G\!}\{g\}})\subset\ol{{}^{H\!}\{\rho(g)\}}$ and the multiplicative Jordan decomposition is compatible with $\rho$, it will be enough to show the claim for $\ol{{}^{H\!}\{\rho(g)\}}$, or in other words, we may assume that $G=\GL_n$. Note that the characteristic polynomial $\chi_h(t)\in L(t)$ of $h\in H$, as a conjugation-invariant regular function, is constant on $\ol{{}^{G\!}\{g\}}$. By \cite[I.4.2~Proposition]{Borel91} there is a polynomial $P(t)\in L(t)$ that only depends on $\chi_h(t)$, such that $s(h) = P(h)$ for every $h\in\ol{{}^{G\!}\{g\}}$. Therefore, the map $h\mapsto s(h)$ is the polynomial map on $\ol{{}^{G\!}\{g\}}$ given by $P(t)$, and hence a morphism of schemes. Thus $u(h)=s(h)^{-1}\cdot h$ is also a morphism.
\end{proof}

\begin{cor}\label{Cor_ss-ize} Let $G$ be a reductive group (which is not necessarily connected). Let $F\subset G$ be a union of conjugacy classes and let $F^{ss}=\{g_s\colon g\in F\}$ be the set consisting of the semi-simple parts $g_s$ of the elements $g$ of $F$. Then $F$ is dense in a connected component of $G$ if and only if $F^{ss}$ is dense in that connected component. 
\end{cor}

\begin{proof}
For every $g\in G$ the closure $\ol{{}^{G\!}\{g\}}$ contains ${}^{G\!}\{g_s\}$ by Lemma~\ref{useful}\ref{useful_A}. So the closure of $F^{ss}$ lies in the closure of $F$, and hence one implication holds. On the other hand let $hG\open\subset G$ be a connected component in which $F\cap hG\open$ is dense. By Theorem~\ref{ThmQuasiTorusReductive}\ref{ThmQuasiTorusReductive_D} there is a dense open subset $O\subset hG\open$ such that $O$ only contains semi-simple elements. Then $F\cap O$ is also dense in $hG\open$, but $F\cap O\subset F^{ss}\cap hG\open$. Therefore, $F^{ss}\cap hG\open$ is also dense in $hG\open$.
\end{proof}

\subsection{Maximal Compact Subgroups of Complex Linear Algebraic Groups}\label{SubSectMaxCompact}

In Section~\ref{Section11} we want to look at the overconvergent analog of Conjectures~\ref{MainConj} and \ref{MainConj3}, and prove Theorem~\ref{ThmOverconvIntro}. To this end, we record some facts about maximal compact subgroups of complex linear algebraic groups in the present subsection. All algebraic groups and schemes in this subsection are over $\BC$, and hence all undecorated fiber products $\times$ are over $\BC$. If $G$ is a linear algebraic group over $\BC$, then $\BG:=G(\BC)$ has the structure of a complex Lie group such that $G\open(\BC)$ is the identity component $\BG\open$ in the usual topological sense. In particular the group of connected components of the Lie group $\BG$ is finite. Therefore, \cite[Theorem~14.1.3]{HilgertNeeb}\footnote{The reference even applies to real Lie groups and every complex Lie group is also a real Lie group.} applies to $\BG$, and provides the following result. We also continue to identify $G$ with the \emph{algebraic} group $G(\BC)$.

\begin{thm}\label{ThmHoch2} 
The following hold:
\begin{enumerate}
\item\label{ThmHoch2_0}
$\BG$ contains \emph{maximal compact} subgroups, that is subgroups which are maximal among the compact subgroups ordered by inclusion.
\item\label{ThmHoch2_A}
For every maximal compact subgroup $\BK\subset \BG$ the inclusion $\BK\subset \BG$ induces an isomorphism $\BK/\BK\open \cong \BG/\BG\open\cong G/G\open$. In particular, $\BK\cap \BG\open=\BK\open$, and this is a maximal compact subgroup of $\BG\open$.
\item\label{ThmHoch2_B}
Any two maximal compact subgroups in $\BG$ are conjugate by an element in $\BG\open$.
\item\label{ThmHoch2_D} 
Every compact subgroup $\BL\subset \BG$ is contained in a maximal compact subgroup of $\BG$.
\item\label{ThmHoch2_F}
Every element $x$ of a compact subgroup $\BL\subset \BG$ is semi-simple.
\item\label{ThmHoch2_G}
There is a closed differentiable sub-manifold $\BE\subset \BG\open$ diffeomorphic to $\BR^n$ for some $n\in\BN$ and containing the identity element such that the multiplication map $\BE\times \BK\to \BG$ is a diffeomorphism and $x^{-1}\BE x=\BE$ for every $x\in \BK$.
\end{enumerate}
\end{thm}

\begin{proof}
Everything is stated in \cite[Theorem~14.1.3]{HilgertNeeb} except for \ref{ThmHoch2_G} which is \cite[Chapter~XV, Theorem~3.1]{Hoch65}, and \ref{ThmHoch2_F} which we now prove. Let $H\subset G$ be the Zariski-closure of the group $x^\BZ$ generated by $x$. Then $H$ is a commutative algebraic group and is isomorphic to the product $H_s\TimesBC H_u$ of the closed subgroups $H_s$ and $H_u$ consisting of all semi-simple, respectively unipotent elements of $H$; see \cite[I.4.7~Theorem]{Borel91}. The element $x$ lies in $H\cap \BL$, which is a compact group, because $\BL$ is compact and $H\subset G$ is a closed subgroup. The image of $x$ under the projection homomorphism $\pi_u\colon H\onto H_u$ lies in the compact subgroup $\pi_u(H\cap \BL)$. Since $H_u$ is a successive extension of additive groups $\BG_{a,\BC}$ and $\BG_{a,\BC}$ contains no compact subgroups other than $\{0\}$, the image $\pi_u(H\cap \BL)$ must be $\{0\}$. This shows that $x\in H_s$, and proves the claim.
\end{proof}

\begin{prop}\label{PropRealFormOfSemiSimple}
Let $G$ be a (not necessarily connected) \emph{reductive} linear algebraic group over $\BC$ and let $\BK\subset G$ be a maximal compact subgroup. Then
\begin{enumerate}
\item \label{PropRealFormOfSemiSimple_A}
$\BK$ is Zariski-dense in $G$.
\item \label{PropRealFormOfSemiSimple_B}
$G$ has in a unique way the structure of an algebraic group over $\BR$ such that $\BK=G(\BR)$.
\item \label{PropRealFormOfSemiSimple_C}
The restriction to $\BK$ defines an equivalence of categories
\begin{equation}\label{EqResGK}
\Res^G_\BK\colon \Rep_\BC G \;\isoto\; \Rep^\cont_\BC \BK\,,\qquad \rho\;\mapsto\; \rho|_\BK
\end{equation}
between the algebraic representations of $G$ and the continuous representations of $\BK$.
\end{enumerate}
\end{prop}

\noindent
{\itshape Remark.} Note that we need to assume that $G$ is reductive. The claim is not true for $G=\BG_a^n$, for example, where the maximal compact subgroup is just the zero element.

\begin{proof}[Proof of Proposition~\ref{PropRealFormOfSemiSimple}]
By \cite[\S\,5.2 and Th\'eor\`eme~1]{SerreGebres} the $\BR$-linear algebraic envelope $G_\BR:=\BK^{\BR\text{-alg}}$ of $\BK$, see Definition~\ref{DefAlgEnvelope}, is a linear algebraic group over $\BR$ satisfying $G_\BR(\BR)=\BK$. By \cite[Beginning of \S\,5.3 and Th\'eor\`eme~4 and Remarque]{SerreGebres} the $\BC$-linear algebraic envelope $\BK^{\BC\text{-alg}}$ of $\BK$ satisfies $\BK^{\BC\text{-alg}}=G_\BR\times_\BR\BC=G$. Thus, $G_\BR$ is the desired unique real form. This proves \ref{PropRealFormOfSemiSimple_B} and \ref{PropRealFormOfSemiSimple_C}. Finally, \ref{PropRealFormOfSemiSimple_A} was already recalled in Lemma~\ref{LemmaAlgEnvelope}.
\end{proof}

\begin{example}\label{compact-torus}
The group $G=\BG_{m,\BC}^n=\Spec\BC[z_\nu^{\pm1}\colon \nu=1,\ldots,n]$ is commutative, and hence has a unique maximal compact subgroup $\BK$ by Theorem~\ref{ThmHoch2}\ref{ThmHoch2_B}. A real structure on $G$ is given by $G_\BR:=\Spec\BR[a_\nu,b_\nu\colon \nu=1,\ldots,n]/(a_\nu^2+b_\nu^2-1)$ with isomorphism given by $a_\nu=\frac{1}{2}(z_\nu+z_\nu^{-1})$ and $b_\nu=\frac{1}{2i}(z_\nu-z_\nu^{-1})$, as well as $z_\nu=a_\nu+ib_\nu$ and $z_\nu^{-1}=a_\nu-ib_\nu$. Clearly $G_\BR(\BR)=(S^1)^n$ is compact, so by Theorem~\ref{ThmHoch2}\ref{ThmHoch2_D} it is contained in the maximal compact subgroup $\BK\subset G$. The map $x\mapsto \bigl(\log|z_\nu(x)|\bigr)_{\nu=1\ldots n}$ has kernel $G_\BR(\BR)$, because $|z_\nu(x)|=1$ if and only if $z_\nu^{-1}(x)=\ol{z_\nu(x)}$. Under this map the quotient Lie group $G(\BC)/G_\BR(\BR)$ is isomorphic to $(\BR^n,+)$, whose only compact subgroup is the zero element. Since the image of $\BK$ under the continuous quotient map $G(\BC)\to G(\BC)/G_\BR(\BR)$ is compact, it is the zero element. Therefore, $G_\BR(\BR)=\BK$ and hence $G_\BR(\BR)$ is the maximal compact subgroup in $G$ and $G_\BR$ is the unique associated compact real form of $G$ from Proposition~\ref{PropRealFormOfSemiSimple}.
\end{example}

\begin{prop}\label{PropWt0InMaxCp}
Let $G\subset \GL_n$ be a closed algebraic subgroup over $\BC$ and let $\BK\subset G$ be a maximal compact subgroup. Let $z\in G$ be a semi-simple element such that every eigenvalue of $z$ on the standard representation $\BC^n$ of $\GL_n$ has complex norm $1$. Then $z$ is conjugate to an element of $\BK$ under $G\open$.
\end{prop}

\begin{proof} 
Since $z$ is semi-simple, a conjugate $gzg^{-1}$ for $g\in\GL_n(\BC)$ is a diagonal matrix, and hence contained in the compact group $(S^1)^n$. The subgroup $\BT:=g^{-1}(S^1)^ng\cap G\subset G$ contains $z$. It is the intersection of a compact and a closed algebraic subgroup, hence is compact, too.  By Theorem~\ref{ThmHoch2} there is an $x\in G\open$ with $x^{-1}zx\in x^{-1}\BT x\subset \BK$.
\end{proof}

We will also need the following lemma due to Deligne \cite[Lemma~2.2.2]{DeligneWeil2} whose proof in loc.\ cit.\ lacks references and is formulated in the connected semi-simple case only. So we include a proof.

\begin{lemma}\label{deligne} 
Let $G$ be a linear algebraic group over $\BC$, let $\BK\subset G$ be a maximal compact subgroup, and let $x,y\in \BK$ be two elements conjugate under $G$, respectively under $G\open$. Then they are already conjugate under $\BK$, respectively under $\BK\open$.
\end{lemma}

\begin{proof} 
Let $\BE$ be as in Theorem~\ref{ThmHoch2}\ref{ThmHoch2_G}. Let $a\in G$ be such that $axa^{-1}=y$. Then we may write $a$ uniquely as $a=ec$ with $e\in \BE$ and $c\in \BK$. If $a\in G\open$ then  $\BE\subset G\open$ implies $c\in\BK\cap G\open=\BK\open$; use Theorem~\ref{ThmHoch2}\ref{ThmHoch2_A}. Then $e(cxc^{-1})e^{-1}=y$. Since $c\in \BK$, the element $z=cxc^{-1}$ is in $\BK$. So it will be enough to show that $z=y$. Note that $eze^{-1}=y$ implies $ezy^{-1}=yey^{-1}$. By Theorem~\ref{ThmHoch2}\ref{ThmHoch2_G} we have $yey^{-1}\in \BE$. Since $y,z\in \BK$ we have $zy^{-1}\in \BK$, therefore $yey^{-1}=e$ and $zy^{-1}=1$ by Theorem~\ref{ThmHoch2}\ref{ThmHoch2_G}.
\end{proof}

We next combine the above with the theory of maximal quasi-tori from Subsection~\ref{SubSectQuasiTori}.

\begin{defn} \label{DefMaxCompQT}
A subgroup $\BT\subset G$ is a \emph{maximal compact quasi-torus} in $G$ if there is an (algebraic) maximal quasi-torus $T\subset G$ in the sense of Definition~\ref{DefQuasiTorus}, such that $\BT$ is a maximal compact subgroup in $T$.
\end{defn}

\begin{prop} \label{PropMaxCPQuasiTorus}
The following holds:
\begin{enumerate}
\item\label{PropMaxCPQuasiTorus_A}
any two maximal compact quasi-tori in $G$ are conjugate under $G\open$,
\item\label{PropMaxCPQuasiTorus_B}
every subgroup in $G$ conjugate to a maximal compact quasi-torus under $G$ is a maximal compact quasi-torus,
\item\label{PropMaxCPQuasiTorus_C}
every maximal compact subgroup $\BK\subset G$ contains a maximal compact quasi-torus $\BT$,
\item\label{PropMaxCPQuasiTorus_D}
for every $\BK$ and $\BT$ as in \ref{PropMaxCPQuasiTorus_C} the following holds: every $x\in \BK$ is conjugate under $\BK\open$ to an element of $\BT$.
\end{enumerate}
\end{prop}

\begin{proof} 
\ref{PropMaxCPQuasiTorus_A}, \ref{PropMaxCPQuasiTorus_B}, and \ref{PropMaxCPQuasiTorus_C} follow directly from Theorems~\ref{ThmQuasiTorusGeneral} and \ref{ThmHoch2}.

\medskip\noindent
\ref{PropMaxCPQuasiTorus_D} The element $x$ is semi-simple by Theorem~\ref{ThmHoch2}\ref{ThmHoch2_F}. Let $T\subset G$ be an algebraic maximal quasi-torus in $G$ such that $\BT$ is a maximal compact subgroup in $T$. Then there is an $a\in G\open$ such that $a^{-1}xa\in T$ by Theorem~\ref{ThmQuasiTorusGeneral}. Since $\BL=a^{-1}\BK a\cap T$ is the intersection of a compact Lie group and a closed algebraic subgroup, it is a compact subgroup in $T$. Therefore, there is a $b\in T\open$ such that $b^{-1}\BL\, b\subset \BT$ by Theorem~\ref{ThmHoch2}. Then $b^{-1}a^{-1}xab\in b^{-1}\BL\, b\subset \BT$. Claim \ref{PropMaxCPQuasiTorus_D} now follows from Deligne's Lemma~\ref{deligne} applied to the two elements $x\in \BK$ and $(ab)^{-1}xab\in \BT\subset \BK$. 
\end{proof}

Next let us prove the analog of Proposition~\ref{comment27NEW} for maximal compact quasi-tori. We consider the following situation: Let $G$ be a reductive algebraic group over $\BC$, let $\BK\subset G$ be a maximal compact subgroup, and let $\BT\subset G$ be a maximal compact quasi-torus contained in $\BK$, such that $\BT$ is a maximal compact subgroup of a maximal quasi-torus $T$ of $G$. Let $h\in\BT$ and recall that we defined $T^h:=\{g\in T\open\colon gh=hg\}$ in Notation~\ref{InitialNotation}. Set $\BT^h=\{g\in\BT\open\colon gh=hg\}=T^h\cap\BT\open$, which is a compact subgroup. 

\begin{lemma}\label{no_prob} The group $\BT^h$ is the unique maximal compact subgroup in $T^h$.
\end{lemma}

\begin{proof} 
Let $\BS$ be a maximal compact subgroup of $T^h$ which contains the compact subgroup $\BT^h$. Then $\BS$ is contained in a maximal compact subgroup of $T\open$. Since $T\open$ is abelian, it has a unique maximal compact subgroup by Theorem~\ref{ThmHoch2}\ref{ThmHoch2_B}, which is $\BT\open=\BT\cap T\open$ by Theorem~\ref{ThmHoch2}\ref{ThmHoch2_A}. We get that $\BT^h\subset \BS\subset T^h\cap\BT\open=\BT^h$.
\end{proof}

\begin{prop}\label{compact_conj} The following hold:
\begin{enumerate}
\item\label{compact_conj_A} Every element of $h\BT\open$ is conjugate under $\BT\open$ to an element of $h\BT^h{}\open$.
\item\label{compact_conj_B} There is an integer $M$ such that the intersection of $h\BT^h{}\open$ with any $G$-conjugacy class has at most $M$ elements.
\end{enumerate}
\end{prop}

Before we start the proof of this proposition, we will need a lemma. Every torus $T$ over $\BC$ is abelian, so it has a  unique maximal compact subgroup by Theorem~\ref{ThmHoch2}\ref{ThmHoch2_B}, which we (temporarily) denote by $c(T)$.

\begin{lemma}\label{Lemma_c_tori} 
Let $P,Q$ be two tori over $\BC$. Then the following holds:
\begin{enumerate}
\item\label{Lemma_c_tori_A} we have $c(P\TimesBC Q)=c(P)\times c(Q)$,
\item\label{Lemma_c_tori_B} every surjective homomorphism $\phi:P\to Q$ maps $c(P)$ onto $c(Q)$.
\end{enumerate}
\end{lemma}

\begin{proof}
\ref{Lemma_c_tori_A} By Tychonoff's theorem $c(P)\times c(Q)$ is a compact subgroup of $P\TimesBC Q$. Therefore, $c(P\TimesBC Q)$ contains $c(P)\times c(Q)$. On the other hand the image of $c(P\TimesBC Q)$ with respect to the projection $P\TimesBC Q\to P$ is a compact subgroup of $P$, so it lies in $c(P)$. Similarly the image of $c(P\TimesBC Q)$ with respect to the projection $P\TimesBC Q\to Q$ lies in $c(Q)$. Therefore, $c(P\TimesBC Q)$ lies in $c(P)\times c(Q)$, so claim \ref{Lemma_c_tori_A} holds. 

\medskip\noindent
\ref{Lemma_c_tori_B} First assume that $\phi$ is an isogeny, that is it has finite kernel. Recall from Example~\ref{compact-torus} that the maximal compact subgroup of a torus $T\open=(\BC\mal)^n$ is the real torus $(S^1)^n$ whose (real) dimension $n$ is the same as the dimension of $T\open$. Since $\phi$ induces an isomorphism on tangent spaces we get that $\phi(c(P))$ is a compact Lie group whose dimension is the same as the dimension of $c(Q)$, which contains it. Since both groups are connected, they are equal. 

Next assume that $P=R\TimesBC Q$ and $\phi$ is the projection $R\TimesBC Q\to Q$ to the second factor. Then $c(P)=c(R)\times c(Q)$ by part \ref{Lemma_c_tori_A}, so claim \ref{Lemma_c_tori_B} holds in this case, too. Finally let $\phi$ be arbitrary. The identity component of the kernel of $\phi$ is a torus $R$. Since over the algebraically closed field $\BC$ any extension of tori splits, we have $P=R\TimesBC S$ for $S:=P/R$. Then $\phi:P=R\TimesBC S\to Q$ is the composition of the projection $R\TimesBC S\to S$ followed by the isogeny $S\to Q$. By the above the claim holds for the projection and the isogeny, so it holds for $\phi$, too.
\end{proof}

\begin{proof}[Proof of Proposition~\ref{compact_conj}] 
Let  $T\subset G$ be a maximal quasi-torus such that $\BT$ is a maximal compact subgroup in $T$. Let $\phi:T\open\to T\open$ be the map $\phi(t)=h^{-1}tht^{-1}$. As we saw in Notation~\ref{Notation9.4NEW} this map is a group homomorphism. The restriction of this homomorphism maps $\BT\open$ into itself, because $\phi(\BT\open)$ is contained in the maximal compact subgroup $\BT\open$. We denote the resulting homomorphism $\BT\open\to\BT\open$ by $\phi$ as well. In order to prove \ref{compact_conj_A} it will be sufficient to show that for every $hx\in h\BT\open$ there are elements $t_0\in \BT^h{}\open$ and $t\in\BT\open$ with
$$hx=ht_0\phi(t)=h\phi(t)t_0=hh^{-1}tht^{-1}t_0=tht^{-1}t_0=t(ht_0)t^{-1}.$$
So it will be sufficient to show that the map $\BT^h{}\open\times\BT\open\to\BT\open$ given by $(t_0,t)\mapsto t_0\phi(t)$ is surjective. 

Let $\psi:T^h{}\open\TimesBC T\open\to T\open$ be the homomorphism defined similarly, i.e.~by the formula $(t_0,t)\mapsto t_0\phi(t)$. By Proposition~\ref{comment27NEW}\ref{comment27NEW_A} we know that $\psi$ is surjective. By Lemmas~\ref{no_prob} and \ref{Lemma_c_tori}\ref{Lemma_c_tori_A} the group $\BT^h{}\open\times\BT\open$ is a maximal compact subgroup of $T^h{}\open\TimesBC T\open$.  Therefore, claim \ref{compact_conj_A} holds by Lemma~\ref{Lemma_c_tori}\ref{Lemma_c_tori_B}. Claim \ref{compact_conj_B} follows from Proposition~\ref{PropComment27NEW}.
\end{proof}

Lemma~\ref{Lemma_c_tori} also implies the following 

\begin{cor}\label{Cor6.10compact} 
Let $\phi\colon G\onto\wt G$ be an epimorphism of linear algebraic groups over $\BC$ and let $T\subset G$ be a maximal quasi-torus and $\BT\subset T\subset G$ a maximal compact quasi-torus. Then $\phi(\BT)$ is a maximal compact quasi-torus in $\phi(T)$ and $\wt G$.
\end{cor}

\begin{proof}
By Corollary~\ref{Cor6.10} the image $\phi(T)$ is a maximal quasi-torus in $\wt G$. Since $\phi(\BT)\subset \phi(T)$ is compact, it is contained in a maximal compact subgroup $\wt\BT$ of $\phi(T)$. We must show that $\phi(\BT)=\wt\BT$. In the identity component we have $\phi(\BT\open)=\phi(c(T\open))=c(\phi(T\open))=\wt\BT\open$ by Lemma~\ref{Lemma_c_tori}\ref{Lemma_c_tori_B}. The claim now follows from the surjectivity $\BT/\BT\open=T/T\open=G/G\open\onto\wt G/\wt G\open=\wt\BT/\wt\BT\open$ for which we use Theorem~\ref{ThmHoch2}\ref{ThmHoch2_A} and Theorem~\ref{ThmQuasiTorusGeneral}\ref{ThmQuasiTorusGeneral_D}.
\end{proof}

\section{Chebotar\"ev for Overconvergent $F$-Isocrystals}\label{Section11}

In this section we prove Theorem~\ref{ThmOverconvIntro} about Chebotar\"ev density for overconvergent $F$-isocrystals by reducing to the case when $\CF$ is $\iota$-mixed; see Definition~\ref{DefIotaMixed}. A second proof will be given in Remark~\ref{RemPinkForOverconv}.

\begin{rem} \label{Rem12.5}
Theorem~\ref{ThmOverconvIntro} establishes that the assumptions of Corollary~\ref{Cor1.9} are satisfied for overconvergent $F$-isocrystals. This corollary was proven earlier by N.~Tsuzuki when $\CF$ and $\CG$ are $\iota$-mixed via a simpler direct method, see \cite[Proposition~A.4.1]{Abe13}, but his argument also uses the $p$-adic analog of Weil II, like ours. However, the condition of being $\iota$-mixed is not as restrictive as it might look, by deep work of T.~Abe on the $p$-adic analog of the Langlands correspondence and by L.~Lafforgue on the Ramanujan-Petersson Conjecture. We explain this in Corollary~\ref{langlands_implies_mixedness} below. 
\end{rem}

\begin{defn}\label{DefIotaMixed}
Fix an isomorphism of fields $\iota\colon \olK\isoto\BC$ and let $|\cdot|\colon \BC\rightarrow\BR_{\geq0}$ be the usual archimedean absolute value on $\BC$. We say that an overconvergent $F$-isocrystal $\CF\in\FIsoc^\dagger_\olK(U)$ is (point-wise) \emph{$\iota$-pure} of weight $w$, where $w\in\BZ$, if for every $x\in|U|$ and for every eigenvalue $\alpha\in\olK$ of $\Frob^{\dagger}_x(\CF)$ acting on $\omega_\BasePoint(\CF)\otimes_{\olK,\iota}\BC$ we have $|\iota(\alpha)|=q^{w\deg(x)/2}$. We say that $\CF\in\FIsoc^\dagger_\olK(U)$ is \emph{$\iota$-mixed} if it is a successive extension of $\iota$-pure overconvergent $F$-isocrystals.
\end{defn}

\subsection{Outline of the Proof of Theorem~\ref{ThmOverconvIntro}} \label{SubsectOutline} 

It is inspired by the proof of Theorem~\ref{ThmIsoclinic} and proceeds in several steps.

\medskip\noindent
{\bfseries Step~1.} \ By deep results on the Langlands correspondence of Abe~\cite{Abe13} ($p$-adic) and Lafforgue~\cite{Lafforgue02} ($\ell$-adic), every overconvergent $F$-isocrystal with finite determinant is $\iota$-pure of weight zero; see Corollary~\ref{langlands_implies_mixedness}. Using the reduction techniques~\ref{TechniqueDividing} and \ref{TechniqueTwist} and Proposition~\ref{PropRedGpAlmostProduct}, this allows us to reduce to the case that $\CF=\CU\oplus\CC$ for a $\iota$-pure $\CU\in\FIsoc^\dagger_\olK(U)$ of weight zero with unit-root determinant, and a constant $F$-isocrystal $\CC$, such that $\Gr^\dagger(\CU)\open$ is a semi-simple group, $\Gr^\dagger(\CF)\open=\Gr^\dagger(\CU)\open \TimesBC \Gr^\dagger(\CC)\open$, and $\Gamma := \mathbf{W}(\CU) = \BZ/\#\Gamma$ is finite cyclic.

\medskip\noindent
{\bfseries Step~2.} \ We base change all groups via $\iota$ to $\BC$ and consider a maximal compact subgroup $\BK\subset \Gr^\dagger(\CU)\times_{\olK,\iota}\BC$ and a maximal compact quasi-torus $\BT\subset\BK$ which is of the form $\BT=T(\BR)$ for a maximal quasi-torus $T\subset\Gr^\dagger(\CU)$ equipped with a suitable real structure. There will be a connected component $h_1 \Gr^\dagger(\CF)\open = h_2 \Gr^\dagger(\CU)\open \TimesBC h_3 \Gr^\dagger(\CC)\open$ and a subset $S'$ of $S$ with positive upper Dirichlet density such that the semi-simple parts $\Frob^\dagger_x(\CF)^{ss}$ meet $h_1 \Gr^\dagger(\CF)\open$ for all $x\in S'$. We may assume that $h_2\in \BT$. Since $\CU$ is $\iota$-pure of weight zero, $h_2\BT^{h_2}{}\open\cap \Frob^\dagger_x(\CU)^{ss}\ne\emptyset$ for every $x\in S'$.

\medskip\noindent
{\bfseries Step~3.} \ Let $\gamma\in \Gamma$ be the image of $h_2$. Then the image of $\Frob^\dagger_x(\CU)^{ss}$, which equals $\deg(x)$ in $\BZ/\#\Gamma = \Gamma$, coincides with $\gamma$. Let $(h_2\BK\open)^\#$ denote the set of conjugacy classes which meet $h_2\BK\open$. On $(h_2\BK\open)^\#$ we consider for $m\in\BN$ the counting measures $\mu_m$ defined as the average of the Dirac measures at the points $\Frob^\dagger_x(\CU)^{ss}\in (h_2\BK\open)^\#$ for $x\in U(m):=\{x\in |U|\colon \deg(x)=m\}$, and the measure $\mu_{{\rm Haar},\gamma}^\#$ given as the push forward of the Haar measure from $\BK$, normalized by $\mu_{{\rm Haar},\gamma}^\#((h_2\BK\open)^\#)=1$. Using the deep theory of Frobenius weights for overconvergent $F$-isocrystals of Kedlaya, Abe and Caro, we prove the $p$-adic analog of Deligne's Equidistribution Theorem in Theorem~\ref{ThmEquiDistr}, which says that the $\mu_m$ for $m\mod\#\Gamma=\gamma=h_2$ in $\Gamma$ converge weakly to $\mu_{{\rm Haar},\gamma}^\#$ when $m\to\infty$.

\medskip\noindent
{\bfseries Step~4.} \ We prove in Lemma~\ref{LemmaPhiNice} and Proposition~\ref{PropWeakConvOfPullBack}, that under the natural map $\phi\colon h_2 \BT^{h_2}{}\open \to (h_2\BK\open)^\#$ one can pull back these measures to $h_2\BT^{h_2}{}\open$ such that the $\phi^*\mu_m$ converge weakly to $\lambda:=\phi^*\mu_{{\rm Haar},\gamma}$. The measure $\lambda$ will have finite volume and satisfy $\lambda(X\cap h_2\BT^{h_2}{}\open)=0$ for every proper hypersurface $X\subsetneq h_2 T^{h_2}{}\open$.

\medskip\noindent
{\bfseries Step~5.} \ Now we will assume that $\bigcup_{x\in S'} \Frob^\dagger_x(\CF)$ is not dense in $h_1\Gr^\dagger(\CF)\open$, but contained in a proper hypersurface $W\subsetneq h_1\Gr^\dagger(\CF)\open$. By Corollary~\ref{Cor4.2} to the Chebotar\"ev Density Theorem for constant $F$-isocrystals, it follows that $X_m:=\{z\in h_2 T^{h_2}{}\open\colon (z,f^m)\in W\}$ defines a proper hypersurface in $h_2T^{h_2}{}\open$ for almost all $m$. By Theorem~\ref{new_oesterle1} a subsequence of the $X_m$ will converge to a proper hypersurface $X\subsetneq h_2T^{h_2}{}\open$, and by Proposition~\ref{simple_estimate2}  there is then an $\epsilon>0$ giving us the contradiction
\[
\frac{\ol\delta(S')}{2} \;\le\; \frac{\#S'\cap U(m)}{\#U(m)} \;\le\;\limsup_{m\to\infty}\phi^*\mu_m\bigl(\ol{(X\cap h_2\BT^{h_2}{}\open)(\epsilon)}\bigr)\;\le\; \lambda\bigl(\ol{(X\cap h_2\BT^{h_2}{}\open)(\epsilon)}\bigr)\;<\;\frac{\ol\delta(S')}{2}\;,
\]
where the first inequality comes from Lemma~\ref{Lemma4.3} and the second from Step~2 and an explicit description of the counting measure $\phi^*\mu_m$.

\medskip

Next we describe the ingredients of Steps~1, 3, 4, and 5 in detail.

\subsection{Ingredients for Step~1: The Langlands Correspondence}

\begin{thm}[{$p$-adic Langlands correspondence, Abe~\cite[Theorem~4.2.2]{Abe13}}] \label{abe} 
Let $\olC$ be a smooth, projective, geometrically irreducible curve over $\BF_q$ and let $C\subset \olC$ be a non-empty open subset. Denote by $\BA$ the ring of ad\`eles of the function field $\BF_q(\olC)=\BF_q(C)$. Let $\CF\in\FIsoc^\dagger_\olK(C)$ be irreducible of rank $r$ such that the determinant $\det\CF=\bigwedge^r(\CF)$ of $\CF$ is a unit-root $F$-isocrystal with finite monodromy. Let $\rho$ be the $p$-adic representation $\rho\colon \pi_1^\et(C,\bar\BasePoint)\rightarrow \olK\mal$ corresponding to $\det\CF$. Then  there is a cuspidal automorphic representation $\Pi$ of $GL_r(\BA)$ such that the central character of $\Pi$ is $\rho$ under the identification furnished by class field theory, and if $\Pi=\otimes_{x\in|\olC|}\Pi_x$ is the factorization of $\Pi$ into the tensor product of local representations then $\Pi_x$ is unramified for every $x\in|C|$ and its Hecke eigenvalues are equal to the images under the fixed $\iota$ of the eigenvalues of $\Frob^\dagger_x(\CF)$ acting on $\omega_\BasePoint(\CF)$. 
\end{thm}

In order to show how this implies $\iota$-purity, we next describe the overconvergent analog of Definition~\ref{Def3.5}, Remark~\ref{Rem3.6} and Proposition~\ref{crew_sequence}.

\begin{defn} 
Let $\Isoc^{\dagger}_K(U)$ denote the category of $K$-linear overconvergent isocrystals on $U$ and let $\Isoc_\olK(U):=\Isoc_K(U)\otimes_K \olK$ be its scalar extension. Let $(\,.\,)^{\sim}\colon \FIsoc^{\dagger}_\olK(U)\to\Isoc^{\dagger}_\olK(U)$ denote the functor furnished by forgetting the Frobenius structure. For every $\CF\in \FIsoc^{\dagger}_\olK(U)$ let $\DGal^\dagger(\CF)$ denote the Tannakian fundamental group with respect to $\omega_\BasePoint$ of the Tannakian sub-category $\dal\CF^{\sim}\dar$ of $\Isoc^{\dagger}_\olK(U)$ generated by $\CF^{\sim}$. Moreover for every such $\CF$ let $\dal\CF\dar_{const}$ and $\mathbf{W}(\CF)$ denote the Tannakian sub-category of constant objects of $\dal\CF\dar$ and the Tannakian fundamental group of $\dal\CF\dar_{const}$ with respect to $\omega_\BasePoint$, respectively. Then $\mathbf{W}(\CF)$ is commutative by Theorem~\ref{ThmMonodrOfConstant}\ref{ThmMonodrOfConstant_B}.
\end{defn}

The monodromy group $\DGal^\dagger(\CF)$ was introduced by Crew~\cite{Crew92}. Next we describe its relationship to $\Gr^{\dagger}(\CF)$. Let $\alpha\colon \DGal^\dagger(\CF)\to\Gr^{\dagger}(\CF)$ be the homomorphism induced by the forgetful functor $(\cdot)^{\sim}\colon \dal\CF\dar\to\dal\CF^{\sim}\dar$, and let $\beta\colon \Gr^{\dagger}(\CF)\to\mathbf{W}(\CF)$ be the homomorphism induced by the inclusion $\dal\CF\dar_{const}\subset\dal\CF\dar$.

\begin{prop}\label{PropCrewsGroup} 
Let $\CF\in\FIsoc^\dagger_\olK(U)$ and $j_U(\CF)\in\FIsoc_\olK(U)$ the associated convergent $F$-isocrystal. Then $\mathbf{W}(j_U(\CF))=\mathbf{W}(\CF)$. Moreover, if $\CF^\sim$ is semi-simple, then the sequence
$$
\xymatrix { 0 \ar[r] & \DGal^\dagger(\CF) \ar[r]^{\alpha} & \Gr^{\dagger}(\CF) \ar[r]^{\beta} & 
\mathbf{W}(\CF) \ar[r] & 0.}
$$
is exact. This is the case for example if the overconvergent $F$-isocrystal $\CF$ is semi-simple.
\end{prop}

\begin{proof}
The functor $j_U\colon \dal\CF\dar_{const}\to\dal j_U(\CF)\dar_{const}$ is fully faithful by Kedlaya's Theorem~\ref{ThmKedlayaFF}. It is essentially surjective, because every object of $\dal j_U(\CF)\dar_{const}$ is constant, and hence overconvergent. The rest follows by the same proof as in Proposition~\ref{crew_sequence}. The last statement is shown in the proof of \cite[Corollary~4.10]{Crew92} (without making use of the stated hypothesis that $U$ is a curve).
\end{proof}

\begin{prop}\label{PropCrew1.5}
Let $\CL\in\FIsoc^\dagger_\olK(U)$ be unit-root of rank $1$. Then the geometric monodromy group $\Gr^\dagger(\CL)^{\rm geo}:=\DGal^\dagger(\CL)$ is finite and some tensor power of $\CL$ is constant on $U$.
\end{prop}

\begin{proof}
We choose an affine open subset $V\subset U$ and a (not necessarily smooth) projectivization $Y$ of $V$. It suffices to show that $\CL|_V^{\otimes N} \cong \CC$ for a positive integer $N$ and a constant $F$-isocrystal $\CC\in\FIsoc_\olK(V)$, because then $\CC$ extends to a constant $F$-isocrystal on $U$ and $\CL^{\otimes N}\cong\CC$ in $\FIsoc^\dagger_\olK(U)$ by \cite[Corollary~4.1.2]{TsuzukiDuke}. By \cite[Theorem~1.3.1]{TsuzukiDuke} there is a smooth proper scheme $Y'$ and a generically \'etale, surjective morphism $f\colon Y'\to Y$, a unit-root $F$-isocrystal $\CN\in\FIsoc_\olK(Y')$, and an isomorphism of overconvergent $F$-isocrystals $\CN|_{V'}\cong f^*(\CL|_V)$ on $V' := f^{-1}(V)$. By Crew's Theorem~\ref{ThmCrewsThm}, $\CN$ corresponds to a representation $\pi_1^\et(Y',\bar\BasePoint)\to \olK\mal$. By \cite[Theorem~1.6]{Crew92} 
the image of $\pi_1^\et(Y',\bar\BasePoint)^{\rm geo}\to \olK\mal$ is finite. If $N$ is the order of this image, then $\CN^{\otimes N}$ has trivial geometric monodromy $\Gr(\CN^{\otimes N}/Y')^{\rm geo}$, and hence is isomorphic to a constant $F$-isocrystal on $Y'$ by Corollary~\ref{CorCrewsThmGeo}. Since the latter is pulled back from $\BF_q$, it is of the form $f^*\CC$ for a constant $F$-isocrystal $\CC\in\FIsoc_\olK(Y)$. If we shrink $V$ to the open subset over which $f$ is \'etale, the isomorphism $f^*(\CL|_V^{\otimes N}) \cong \CN^{\otimes N}|_{V'}\cong f^*(\CC|_V)$ descends by \cite[Th\'eor\`eme~1]{Etesse02} to an isomorphism of overconvergent $F$-isocrystals $\CL|_V^{\otimes N}\cong\CC|_V$ as desired.
\end{proof}

\begin{cor}\label{langlands_implies_mixedness}
Let $\CF\in\FIsoc^\dagger_\olK(U)$ be irreducible such that the determinant $\det\CF$ of $\CF$ is a unit-root $F$-isocrystal and the monodromy group $\mathbf{W}(\det \CF)$ of $\dal\det\CF\dar_{const}$ is finite. Then $\CF$ is $\iota$-pure of weight zero.
\end{cor}

\begin{proof}
By Propositions~\ref{PropCrew1.5} and \ref{PropCrewsGroup} the monodromy group $\Gr^\dagger(\det\CF/U)$ is finite. Let $x\in |U|$ be a closed point. By Theorem~\ref{ThmWeilBounds} we can choose an auxiliary point $y\in|U|$ with $\deg(y)$ coprime to $\deg(x)$. By \cite[Theorem~3.10]{Abe+Esnault} there is a smooth curve $f\colon X\to U$ such that $x$ and $y$ factor through $X$ and such that $f^*\CF\in\FIsoc^\dagger_\olK(X)$ is still irreducible. Let $\olX$ be the smooth projectivization of $X$. Since it contains the points $x$ and $y$ of coprime degree, $\olX$ is geometrically irreducible. The determinant $\det(f^*\CF)=f^*(\det\CF)$ is still unit-root and $\Gr^\dagger(\det(f^*\CF)/X)$ is still finite, because it is a closed subgroup of $\Gr^\dagger(\det\CF/U)$. So by Theorem~\ref{abe} there is a cuspidal automorphic representation $\Pi$ of the function field $\BF_q(\olX)$ corresponding to $f^*\CF$, which is unramified at every $x\in|X|$. By Lafforgue's proof \cite[Th\'eor\`eme~VI.10(i)]{Lafforgue02} of the Ramanujan-Petersson Conjecture for $\Pi$, the Hecke eigenvalues of $\Pi$ at all $x\in |X|$ have complex absolute value one. Theorem~\ref{abe} then implies that $\CF$ is $\iota$-pure of weight zero.
\end{proof}

\subsection{Ingredients for Step~3: The $p$-adic Equidistribution Theorem}

From the theory of Frobenius weights \cite{Kedlaya06} we will derive the $p$-adic analog of Deligne's Equidistribution Theorem. For the rest of Section~\ref{Section11} we keep the following

\begin{notn}\label{Notation13.14}
We fix a semi-simple $\iota$-pure overconvergent $F$-isocrystal $\CU\in\FIsoc^{\dagger}_K(U)$ of weight zero. We call $G^{\rm arith}:=\Gr^{\dagger}(\CU)$ and $G^{\rm geo}:=\DGal^\dagger(\CU)$ its arithmetic and geometric monodromy groups, respectively. The names are justified by Proposition~\ref{PropCrewsGroup}. Assume that $(G^{\rm arith})\open$ is a semi-simple group. In the exact sequence from Proposition~\ref{PropCrewsGroup} the group $\mathbf{W}(\CU)$ is commutative, and hence $(G^{\rm geo})\open$ contains the derived group of $(G^{\rm arith})\open$, which equals $(G^{\rm arith})\open$ by \cite[IV.14.2~Corollary]{Borel91}. Thus  $(G^{\rm arith})\open=(G^{\rm geo})\open$ and $\mathbf{W}(\CU)$ is finite. It follows further that $\mathbf{W}(\CU)$ is cyclic, because it is generated by the Frobenius $f$ of a tensor generator $(W,f)$ of $\dal\CU\dar_{const}$ by Theorem~\ref{ThmMonodrOfConstant}\ref{ThmMonodrOfConstant_B}.

Using the embedding $\iota$ we may extend scalars and view $G^{\rm arith}$ and $G^{\rm geo}$ as semi-simple algebraic groups over $\BC$. We will denote by $\BK^{\rm arith}$ and $\BK^{\rm geo}$ maximal compact subgroups of $G^{\rm arith}$ and $G^{\rm geo}$ such that $\BK^{\rm geo}\subset\BK^{\rm arith}$. Clearly $(\BK^{\rm geo})\open\subset(\BK^{\rm arith})\open$. Since $(G^{\rm arith})\open=(G^{\rm geo})\open$, we get that $(\BK^{\rm arith})\open=(\BK^{\rm geo})\open$, and therefore we deduce from Theorem~\ref{ThmHoch2}\ref{ThmHoch2_A} that  
\begin{align*}
\BK^{\rm arith}/\BK^{\rm geo} & \;\cong\; \bigl(\BK^{\rm arith}/(\BK^{\rm arith})\open\bigr)\big/\bigl(\BK^{\rm geo}/(\BK^{\rm geo})\open\bigr)\\
 & \;\cong\; \bigl(G^{\rm arith}_\BC/(G^{\rm arith}_\BC)\open\bigr)\big/\bigl(G^{\rm geo}_\BC/(G^{\rm geo}_\BC)\open\bigr)\\
 & \;\cong\; G^{\rm arith}_\BC/G^{\rm geo}_\BC\\
 & \;\cong\;\mathbf{W}(\CU)(\BC)\;=:\;\Gamma
\end{align*}
is a finite cyclic group with canonical generator being the Frobenius $f$ from the previous paragraph. Let $d:=\#\Gamma$. In what follows we consider the group homomorphism $\chi$ defined as the composition
\begin{equation}\label{EqClassInGamma}
\chi\colon \BK^{\rm arith}\;\longonto\;\BK^{\rm arith}/\BK^{\rm geo}\;\cong\;\Gamma\;\xleftarrow{\,\sim\;}\BZ/d\BZ
\end{equation}
where the last map is induced by $\BZ\onto\Gamma,\ m\mapsto[m]:=f^m$.
\end{notn}

\begin{defn}\label{DerMeasuresOnConjCl}
Fix an element $\gamma\in\Gamma\cong \BK^{\rm arith}/\BK^{\rm geo}$ and let $\BK^{\rm arith}_\gamma$ denote the inverse image of $\gamma$ in $\BK^{\rm arith}$.  We denote the set of $\BK^{\rm arith}$-conjugacy classes in $\BK^{\rm arith}$ by $\BK^{{\rm arith},\#}$ and the ones which meet $\BK^{\rm arith}_\gamma$ by $\BK^{{\rm arith},\#}_\gamma$. We equip these sets with the quotient topology. $\BK^{{\rm arith},\#}_\gamma$ is a union of connected components of $\BK^{{\rm arith},\#}$, because it equals the preimage of $\gamma$ under the induced map $\BK^{{\rm arith},\#}\to\Gamma$. Let $\mu_{\rm Haar}$ be the Haar measure on $\BK^{\rm arith}$ normalized so that $\mu_{\rm Haar}(\BK^{\rm geo})=1$. (We may take the left or right invariant measure as either is bi-invariant by \cite[Chapitre~VII, \S\,1.3, Corollaire]{BourbakiIntegration}.) Then the push-forward of $\mu_{\rm Haar}$ under $\BK^{\rm arith}\onto \BK^{{\rm arith},\#}$ is a measure on $\BK^{{\rm arith},\#}$ which we will denote by $\mu_{\rm Haar}^\#$, see Definition~\ref{DefPushForwMeasure}. For every $\gamma\in\Gamma$ let $\mu_{{\rm Haar},\gamma}^\#$ be the measure defined by the formula:
\[
\mu_{{\rm Haar},\gamma}^\#(M)\;:=\;\mu_{\rm Haar}^\#(M\cap\BK^{{\rm arith},\#}_\gamma)\quad\textrm{for all measurable }M\subset \BK^{{\rm arith},\#}.
\]
Then $\mu_{{\rm Haar},\gamma}^\#(\BK^{{\rm arith},\#}_\gamma)=\mu_{\rm  Haar}(\BK^{\rm geo})=1$. The main equidistribution statement will be that a suitably normalized sum of point masses corresponding to Frobenius elements converges to the measure $\mu_{{\rm Haar},\gamma}^\#$.
\end{defn}

\begin{lemma}\label{LemmaQuotientHausdorff} 
The space $\BK^{{\rm arith},\#}$ is a compact topological Hausdorff space. 
\end{lemma}

\begin{proof} 
By \cite[Chapter~I, \S\,8.3, Proposition~8]{BourbakiTopology} it suffices to show that the quotient map $\BK^{\rm arith}\to\BK^{{\rm arith},\#}$ is open and the graph of the conjugation action $\BK^{\rm arith}\times \BK^{\rm arith}\to \BK^{\rm arith}\times \BK^{\rm arith}$, $(x,g)\mapsto(x,gxg^{-1})$ is closed in the metric topology. The openness of the quotient map follows from \cite[Chapter~III, \S\,2.4, Lemma~2]{BourbakiTopology}, and the graph is closed because $\BK^{\rm arith}\times \BK^{\rm arith}$ is compact and Hausdorff.
\end{proof}

We will need more results on the structure of the quotient, which are due to Brumfiel by viewing $\BK^{\rm arith}$ as a closed, bounded affine semi-algebraic group, see \cite[\S\,7, Definition~3]{DeKn81}.

\begin{thm}[{\cite[Corollary~1.6]{Br}}] \label{ThmBrumfiel}
If $H$ is a closed, bounded affine semi-algebraic group which acts continuously and semi-algebraically on an affine semi-algebraic space $X$ then the quotient space $X/H$ exits as an affine semi-algebraic space and the quotient map $X\to X/H$ is continuous and semi-algebraic.
\end{thm}

\begin{rem}\label{RemBrumfiel}
Theorem~\ref{ThmBrumfiel} applies to all real closed fields, but we will only use it for the real number field $\BR$. It is pointed out in \cite[Remark~1.3]{Br} that in the latter case the quotient map $X\to X/H$ is \emph{topological}, which implies that the topology on $X/H$ is the usual quotient topology. We conclude that $\BK^{{\rm arith},\#}$ is an affine semi-algebraic space and the quotient map $\BK^{\rm arith}\to\BK^{{\rm arith},\#}$ is semi-algebraic. In particular, $\BK^{{\rm arith},\#}$ is metrizable.
\end{rem}

\begin{defn} \label{Def_theta(x)}
We consider the semi-simplification $\Frob^{\dagger}_x(\CU)^{ss}$ of the Frobenius conjugacy class $\Frob^{\dagger}_x(\CU)$. The eigenvalues $\alpha\in\olK$ of $\Frob^\dagger_x(\CU)^{ss}$ and of $\Frob^\dagger_x(\CU)$ acting on $\omega_\BasePoint(\CU)$ are the same. Since $\CU$ is $\iota$-pure of weight zero, all these eigenvalues have complex norm $1$ in the action on $\omega_\BasePoint(\CU)\otimes_{\olK,\iota}\BC$. So $\Frob^{\dagger}_x(\CU)^{ss}$ is conjugate under $(G^{\rm arith})\open$ to an element of $\BK^{\rm arith}$ by Proposition~\ref{PropWt0InMaxCp}. The $\BK^{\rm arith}$-conjugacy class of this element is well-defined, because by Lemma~\ref{deligne} two elements of $\BK^{\rm arith}$ which are conjugate under $G^{\rm arith}$ are already conjugate under $\BK^{\rm arith}$. We denote this $\BK^{\rm arith}$-conjugacy class by $\theta(x)$. Note that by definition of $\mathbf{W}(\CU)$, the image of $\Frob^{\dagger}_x(\CU)$ and $\Frob^{\dagger}_x(\CU)^{ss}$ in $\Gamma=\mathbf{W}(\CU)(\BC)$ is just the image $\gamma:=[\deg(x)]=[\chi(\theta(x))]$ of $\deg(x)$ under the map \eqref{EqClassInGamma}. Thus the conjugacy class $\theta(x)$ is an element of $\BK^{{\rm arith},\#}_\gamma$.
\end{defn}

\begin{defn}\label{DefMeasure_d_mu_m}
Recall that $U(m):=\{x\in|U|\colon \deg(x) =m\}$ for every positive integer $m$. For every $\theta\in\BK^{{\rm arith},\#}$ let $\delta_\theta$ denote the Dirac measure on $\BK^{{\rm arith},\#}$ centered at $\theta$. We define the discrete measure $\mu_m$ on $\BK^{{\rm arith},\#}$ to be:
\[
\mu_m\;:=\;\frac{1}{\#U(m)}\sum_{x\in U(m)}\delta_{\theta(x)}.
\]
Note that this measure is supported on $\BK^{{\rm arith},\#}_\gamma$ where $\gamma$ is equal to the image $[m]$ of $m$ in $\Gamma$ from \eqref{EqClassInGamma}.
\end{defn}

\begin{thm}[$p$-adic version of Deligne's Equidistribution Theorem]\label{ThmEquiDistr}
For every $\gamma\in\Gamma$ the measures $\mu_m$, for which $[m]=\gamma\in\Gamma$ in \eqref{EqClassInGamma}, converge weakly to $\mu^\#_{{\rm Haar},\gamma}$ on $\BK^{{\rm arith},\#}_\gamma$ as $m\to\infty$.
\end{thm}

Our proof is inspired by Deligne's method \cite[Chapitre~II]{DeligneWeil2} of estimating radii of convergence for $L$-functions as in the following

\begin{lemma}\label{LemmaClaimOnL-Fcn}
For every \emph{irreducible} representation $\rho\in\Rep^\cont_\BC\BK^{\rm arith}$ the $L$-function
\begin{equation}\label{EqLFctOfG}
L(\rho,z)\;:=\;\prod_{ x\in |U| }\det\nolimits_\BC\bigl(1-\rho(\theta(x))z^{\deg(x)}\bigr)^{-1}
\end{equation}
is a meromorphic function without zeros on an open neighborhood of the closed disc $\olD:=\{z\in\BC\colon |z|\le q^{-\dim U}\}$. It has a simple pole at $z=(\xi q^{\dim U})^{-1}$ if $\rho=\xi^\chi$ for a $d$-th root of unity $\xi\in\BC$ and is holomorphic on an open neighborhood of $\olD$ otherwise. Here $\chi$ is the homomorphism $\chi\colon\mathbb K^{\rm arith}\onto\mathbb K^{\rm arith}/\mathbb K^{\rm geo}=\Gamma=\BZ/d\BZ$ from \eqref{EqClassInGamma}.

\end{lemma}

\begin{proof}
We identify $\olK$ with $\BC$ via $\iota$. By Proposition~\ref{PropRealFormOfSemiSimple}\ref{PropRealFormOfSemiSimple_C} the restriction to $\BK^{\rm arith}$ defines an equivalence of categories between $\dal\CU\dar=\Rep_\BC G^{\rm arith}$ and $\Rep^\cont_\BC \BK^{\rm arith}$. Thus $\rho=\rho_\CG$ for an \emph{irreducible} overconvergent $F$-isocrystal $\CG\in\dal\CU\dar$. The latter is $\iota$-pure of weight zero. Its $L$-function is defined as
\[
L(\CG/\olK,z)\;:=\; \prod_{ x\in |U|}\det\nolimits_\olK\bigl(1-F_\CG^{\deg(x)}z^{\deg(x)}|\omega_x(\CG)\bigr)^{-1}\;=\;L(\rho_\CG,z)\,.
\]
Note that for any $\CF\in\FIsoc_K(U)$ and its base extension $\CF\otimes_K\olK\in\FIsoc_\olK(U)$ the fiber $\omega_x(\CF)$ is a $K_n$-vector space for $n=\deg(x)$ and $\omega_x(\CF\otimes_K\olK)=\omega_x(\CF)\otimes_{K_n}\olK$. Therefore
\[
\det\nolimits_\olK\bigl(1-F_\CF^{\deg(x)}z^{\deg(x)}|\omega_x(\CF\otimes_K\olK)\bigr)\;=\;\det\nolimits_{K_n}\bigl(1-F_\CF^{\deg(x)}z^{\deg(x)}|\omega_x(\CF)\bigr)\;=\;\det\nolimits_K\bigl(1-F_\CF^{\deg(x)}z^{\deg(x)}|\omega_x(\CF)\bigr)^{1/\deg(x)}
\]
and this proves the equality of $L(\CF/\olK,z)=L(\rho_\CF,z)$ with the $L$-function of \cite{ES}. It converges absolutely on the open disc $D:=\{z\in\BC\colon |z|<q^{-\dim U}\}$, and extends to a meromorphic function on $\BC$ by the \'Etesse--Le Stum trace formula
\begin{equation}\label{EqESTrace}
L(\CG/\olK,z)\;=\;\prod_{i=0}^2 \det\nolimits_\olK \bigl(1-z\,q^{\dim U}\cdot F_{\CG\dual}^{-1}|\Koh_{\rig}^i(U,\CG\dual)\bigr)^{(-1)^{i+1}},
\end{equation}
see \cite[Th\'eor\`eme~6.3.II and Corollaire~6.4]{ES}, where $F_{\CG\dual}$ is the Frobenius of $\CG\dual$ and $\Koh_{\rig}^i(U,\CG\dual)$ is the rigid cohomology on $U$. Namely, since $\CG$ and $\CG\dual$ are $\iota$-pure of weight zero (and hence $\iota$-realizable), Kedlaya~\cite[Theorem~5.3.2(b)]{Kedlaya06} shows that $\Koh_{\rig}^i(U,\CG\dual)$ is $\iota$-pure of weight $\ge i$, that is, (the images under $\iota$ of) the eigenvalues of the Frobenius $F_{\CG\dual}$ have absolute value $\ge q^{i/2}$, and hence those of $F_{\CG\dual}^{-1}$ have absolute value $\le q^{-i/2}$. In particular, the factors in \eqref{EqESTrace} for $i\ge 1$ have neither zeros nor poles on an open neighborhood of the closed disc $\olD$. For $i=0$ we distinguish two cases. If $\Koh^0_\rig(U,\CG\dual)$ is trivial, then the factor for $i=0$ equals $1$, and $L(\CG/\olK,z)=L(\rho_\CG,z)$ has neither zero nor pole on $\olD$.

On the other hand, if the $\olK$-vector space $\Koh^0_\rig(U,\CG\dual)$ of horizontal sections of $(\CG\dual)^\sim$ is non-trivial, it generates a non-trivial constant, and hence overconvergent sub-isocrystal $\CH\subset(\CG\dual)^\sim$. Since the Frobenius $F_{\CG\dual}$ respects horizontal sections, $\CH$ underlies an overconvergent sub-$F$-isocrystal. The assumption that $\CG\dual$ is irreducible implies that $\CG\dual=\CH$ is a constant $F$-isocrystal, and hence $\CG$ is constant, too. By Proposition~\ref{PropCrewsGroup} this means that the representation $\rho_\CG$ factors through $\chi$ and is of the form $\rho_\CG=\xi^\chi$ for some $d$-th root of unity $\xi$. In this case the $L$-function is $L(\xi^{\chi},z)=\prod_{x\in |U|}(1-(\xi z)^{\deg(x)})^{-1}$ and the statement follows from Theorem~\ref{ThmWeilBounds}.
\end{proof}

\begin{proof}[Proof of Theorem~\ref{ThmEquiDistr}.]
We only use the Lemma~\ref{LemmaClaimOnL-Fcn} and the equality $\chi(\theta(x))\equiv \deg(x)\mod d$ to deduce equidistribution. So one could formulate our argument as a method like it was done by Deligne~\cite[Chapitre~II]{DeligneWeil2}. By the Peter-Weyl theorem, see for example \cite[Propositions~5.21 and 5.25]{Folland}, the linear combinations of the traces $\Tr(\rho)$ for $\rho\in\Rep_\BC^c\BK^{\rm arith}$ are (metrically) dense in the $\BC$-algebra of continuous functions on $\BK^{{\rm arith},\#}$. So by the Portemanteau theorem~\cite[Theorem~13.16]{Klenke} we only need to show that for every irreducible representation $\rho\in\Rep^\cont_\BC\BK^{\rm arith}$
\begin{equation}\label{EqLRHandSide}
\frac{1}{\#U(m)}\sum_{x\in U(m)}\int_{\BK^{{\rm arith},\#}_\gamma}\Tr(\rho)\;d\delta_{\theta(x)}\;\longto\;\int_{\BK^{{\rm arith},\#}_\gamma}\Tr(\rho)\;d\mu_{{\rm Haar},\gamma}^\#\quad\text{if} \quad m\longto \infty \quad\text{with}\quad [m]=\gamma \quad\text{in}\quad \Gamma.
\end{equation}
First we compute the right hand side. Let $n\in\BZ$ with $[n]=\gamma$.

\bigskip\noindent
{\itshape Claim.} \quad We have \ $\displaystyle\int_{\BK^{{\rm arith},\#}_\gamma}\Tr(\rho)\;d\mu_{{\rm Haar},\gamma}^\# \;=\;\begin{cases} \xi^n & \text{if $\rho=\xi^\chi$ for some $d$-th root of unity $\xi$,} \\ 0 & \text{otherwise.} \end{cases}$
      
\bigskip\noindent
Fix a primitive $d$-th root of unity $\zeta$. Using the orthogonality relations $\sum_{k=0}^{d-1}\zeta^{kr}=0$ when $d\nmid r$ and $\sum_{k=0}^{d-1}\zeta^{kr}=d$ when $d| r$ we have:
\[
\int_{\BK^{{\rm arith},\#}_\gamma}\Tr(\rho)\;d\mu_{{\rm Haar},\gamma}^\#\;=\;\int_{\chi^{-1}(n)}\Tr(\rho)\;d\mu_{\rm Haar}\;=\;
\tfrac{1}{d}\,\sum_{k=0}^{d-1}\zeta^{kn}\int_{\BK^{\rm arith}}
\ol{\zeta^{k\chi}}\cdot\Tr(\rho)\;d\mu_{\rm Haar}\,,
\]
which is zero unless $\rho=\zeta^{k\chi}$ for some integer $k$ by the orthogonality of characters of irreducible representations; see \cite[Proposition~5.23]{Folland}. In the latter case, setting $\xi=\zeta^k$ we obtain
\[
\int_{\BK^{{\rm arith},\#}_\gamma}\Tr(\rho)\;d\mu_{{\rm Haar},\gamma}^\#\;=\;\int_{\chi^{-1}(n)}\zeta^{k\chi}\;d\mu_{\rm Haar}\;=\;\zeta^{kn}\cdot\mu_{\rm Haar}\bigl(\chi^{-1}(n)\bigr)\;=\;\xi^n,
\]
which proves the claim.

To compute the left hand side of \eqref{EqLRHandSide} we first compute the logarithmic derivative multiplied by $z$ of the $L$-function $L(\rho,z)$ from \eqref{EqLFctOfG} as follows:
\[
z\,\frac{L(\rho,z)'}{L(\rho,z)}\;=\;\sum_{m=1}^{\infty}a_mz^m, \qquad\text{where}\qquad a_m=\sum_{k|m}k\cdot\sum_{x\in U(k)}\Tr\rho(\theta(x)^{m/k})
\]
and where $\theta(x)^{m/k}$ denotes the conjugacy class of $m/k$-th powers of elements of $\theta(x)$.
Indeed, this is a completely standard computation using the identities
\[
z\,\frac{L(\rho,z)'}{L(\rho,z)}=\sum_{ x\in |U|}z\,\frac{(\det(1-\rho(\theta(x))z^{\deg(x)})^{-1})'}{\det(1-\rho(\theta(x))z^{\deg(x)})^{-1}}
\]
and
\[
z\,\frac{(\det(1-\rho(\theta(x))z^{\deg(x)})^{-1})'}{\det(1-\rho(\theta(x))z^{\deg(x)})^{-1}}=\deg(x)\cdot\sum_{j=1}^{\infty}\Tr\rho(\theta(x)^j)\cdot z^{j\deg(x)}.
\]
We estimate $a_m$ further and claim that
$$a_m\;=\;m\cdot\sum_{x\in U(m)}\Tr\rho(\theta(x))+O(q^{m\dim U/2})\;=\;m\cdot\sum_{x\in U(m)}\int_{\BK^{{\rm arith},\#}_\gamma}\Tr(\rho)\;d\delta_{\theta(x)}+O(q^{m\dim U/2}).$$
Namely, we note that $\rho$ is a unitary representation, since $\BK^{\rm arith}$ is compact, and so $|\Tr\rho(\theta)|\leq \dim(\rho)$ for every $\theta\in\BK^{{\rm arith},\#}$. Therefore, using the bound $\#U(k)\leq O(q^{k\dim U}/k)$ from Theorem~\ref{ThmWeilBounds} we estimate
\begin{align*}
\bigl|a_m-m\cdot\sum_{x\in U(m)}\Tr\rho(\theta(x))\bigr| \;&=\;
\bigl|\sum_{\substack{ k|m \\ k<m }}k\cdot\sum_{x\in U(k)}\Tr\rho(\theta(x)^{m/k}))\bigr|\\
&\leq\;\sum_{\substack{ k|m \\ k<m }}\dim(\rho)\cdot k\cdot\#U(k) \;\leq\; \sum_{\substack{k\le m/2 }}\dim(\rho)\cdot O(q^{k\dim U}) \;\leq\; O(q^{m\dim U/2}).
\end{align*}
Note that it suffices to compute the integral of $\Tr(\rho)\;d\delta_{\theta(x)}$ only over $\BK^{{\rm arith},\#}_\gamma$, because that contains the support of $\delta_{\theta(x)}$.

From Lemma~\ref{LemmaClaimOnL-Fcn} we conclude that the residue of the function $z\,\frac{L(\rho,z)'}{L(\rho,z)}$ at $(\xi q^{\dim U})^{-1}$ is $-(\xi q^{\dim U})^{-1}$ if $\rho=\xi^{\chi}$, and zero otherwise. Let $R$ be this residue. To finish the proof we compute $R$ in a different way. The function $z\,\frac{L(\rho,z)'}{L(\rho,z)}+\frac{R\xi q^{\dim U}}{1-z\xi q^{\dim U}}$ has a holomorphic extension to an open neighborhood of $\olD$. Since $\frac{R\xi q^{\dim U}}{1-z\xi q^{\dim U}}=\sum_{m=0}^{\infty}R(\xi q^{\dim U})^{m+1}z^m$ on $D$, we get that $\sum_{m=0}^{\infty}(a_m+R(\xi q^{\dim U})^{m+1})z^m$ converges absolutely on an open disk of center $0$ and of radius bigger than $q^{-\dim U}$. Let $n\in\BZ$ with $[n]=\gamma$ in $\Gamma$. Using that
$\xi^m=\xi^n$ for $m\equiv n \mod d$ we get that
\[
\lim_{\substack{m\to\infty \\ m\equiv n}}\,\Bigl|\frac{a_m}{q^{m\dim U}}+\xi^{n+1}Rq^{\dim U}\Bigr|\;=\;\lim_{\substack{m\to\infty \\ m\equiv n}}\,\bigl|a_m+R(\xi q^{\dim U})^{m+1}\bigr|\cdot q^{-m\dim U}\;=\;0\,.
\]
Together with Theorem~\ref{ThmWeilBounds} we obtain
\[
\lim_{\substack{m\to\infty \\ [m]=\gamma}}\ \frac{1}{\#U(m)}\sum_{x\in U(m)}\int_{\BK^{{\rm arith},\#}_\gamma}\Tr(\rho)\;d\delta_{\theta(x)} \;=\; \lim_{\substack{m\to\infty \\ m\equiv n}}\ \frac{a_m}{q^{m\dim U}} \;=\; -\xi^{n+1}Rq^{\dim U} \;=\; \int_{\BK^{{\rm arith},\#}_\gamma}\Tr(\rho)\;d\mu_{{\rm Haar},\gamma}^\#\,.
\qedhere
\]
\end{proof}

\subsection{Ingredients for Step~4: Real Algebraic Geometry}

\begin{lemma}\label{LemmaPhiNice}
Abbreviate $\BK:=\BK^{\rm arith}$, let $\BT\subset\BK$ be a maximal compact quasi-torus, and let $h\in \BT$. The inclusion $h\BT^h{}\open\into h\BK\open$ defines a map $\phi\colon h\BT^h{}\open \onto (h\BK\open)^\#$, which is surjective by Propositions~\ref{PropMaxCPQuasiTorus}\ref{PropMaxCPQuasiTorus_D} and \ref{compact_conj}\ref{compact_conj_A}.
\begin{enumerate}
\item \label{LemmaPhiNice_A}
The map $\phi$ is a continuous map between compact metrizable spaces. It is nice in the sense of Definition~\ref{DefNiceMap}. 
\item \label{LemmaPhiNice_B}
There is a positive integer $M$ such that all fibers of $\phi$ have cardinality at most $M$.
\item \label{LemmaPhiNice_D} 
For every semi-algebraic subset $X\subset (h\BT^{h}{}\open)$ we have $\dim X=\dim\phi(X)$.
\item \label{LemmaPhiNice_C}
There is a closed semi-algebraic subset $V\subset(h\BK\open)^\#$ with $\dim V<\dim(h\BK\open)^\#$ whose complement is a finite disjoint union $(h\BK\open)^\#\smallsetminus V=\coprod_{i=1}^n Y_i$ of open subsets $Y_i$ such that $\phi$ is trivial over $Y_i$ in the sense that there is a finite discrete set $F_i$ and a semi-algebraic isomorphism $\phi^{-1}(Y_i)\isoto F_i\times Y_i$ compatible with the projections onto $Y_i$.
\end{enumerate}
\end{lemma}

\begin{proof}
\ref{LemmaPhiNice_A} By Lemma~\ref{LemmaQuotientHausdorff} the quotient $(h\BK\open)^\#$ is compact. By Remark~\ref{RemBrumfiel} the quotient map $h\BK\open\to(h\BK\open)^\#$ is continuous and semi-algebraic as the restriction of $\BK\to\BK^\#$ to the connected component $h\BK\open$, and $(h\BK\open)^\#$ is an affine semi-algebraic space. The latter also holds for $h\BT^{h}{}\open$, and hence both are metrizable. Moreover, $\phi$ is continuous and semi-algebraic as the composition of the inclusion $h\BT^{h}{}\open\to h\BK\open$ and the quotient map $h\BK\open\to(h\BK\open)^{\#}$. Moreover, $h\BT^{h}{}\open$ is semi-algebraic and compact, hence complete by \cite[Theorem~9.4]{DeKn81}. Then \cite[Chapter~I, Remark~5.5(v) and \S\,6, Definition~4]{DeKn85} implies that $\phi$ is a finite semi-algebraic map. Therefore, \cite[Chapter~II, Theorem~6.13]{DeKn85} implies that there are semi-algebraic triangulations $\tau_1,\tau_2$ of $h\BT^{h}{}\open$ and $(h\BK\open)^{\#}$, respectively, such that the restriction of $\phi$ onto any simplex of $\tau_1$ is a simplicial map to a simplex of $\tau_2$, up to continuous semi-algebraic isomorphism. Since every finite to one simplicial map between simplices is trivially injective, we get that each such restriction is injective. The claim now follows from Remark~\ref{RemNiceMap}. 

\medskip\noindent
\ref{LemmaPhiNice_B} follows from Proposition~\ref{compact_conj}\ref{compact_conj_B}.

\medskip\noindent
\ref{LemmaPhiNice_C} By Hardt's Local-Triviality-Theorem~\cite[Theorem~6.4]{DeKn82} for the semi-algebraic map $h\BT^{h}{}\open\to(h\BK\open)^\#$ there is a decomposition $(h\BK\open)^\#=\coprod_{i=1}^{\tilde n} \wt Y_i$ into finitely many semi-algebraic subsets $\wt Y_i\subset (h\BK\open)^\#$, such that $\phi$ is trivial over $\wt Y_i$ in the above sense. By \cite[Proposition~8.2(b)]{DeKn81} there is an $n\le\tilde n$ with $\dim(h\BK\open)^\#=\dim\wt Y_1=\ldots=\dim\wt Y_n>\dim\wt Y_i$ for all $i>n$. For $i\le n$ let $Y_i$ be the open interior of $\wt Y_i$, which is a semi-algebraic subset of $(h\BK\open)^\#$. Let further $V:=\bigcup_{i\le n}(\wt Y_i\smallsetminus Y_i)\cup\bigcup_{i>n}\wt Y_i$. Then $\dim V<\dim(h\BK\open)^\#=\dim Y_i$ for every $i\le n$ by \cite[Theorem~8.10]{DeKn81}. Indeed, if $\dim(\wt Y_i\smallsetminus Y_i)$ were equal to $\dim(h\BK\open)^\#$ then $\wt Y_i\smallsetminus Y_i$ would contain a non-empty open subset by loc.\ cit.\ in contradiction to $Y_i$ being the largest open subset of $\wt Y_i$. This proves \ref{LemmaPhiNice_C}.

\medskip\noindent
\ref{LemmaPhiNice_D} The inequality $\dim X\ge\dim\phi(X)$ follows from \cite[Proposition~8.3]{DeKn81}. We next consider the decomposition $(h\BK\open)^\#=\coprod_{i=1}^{\tilde n} \wt Y_i$ from the previous paragraph and the semi-algebraic subsets $\phi(X)\cap\wt Y_i$ of $(h\BK\open)^\#$. Then $X\cap\phi^{-1}(\wt Y_i)\subset \phi^{-1}(\phi(X)\cap\wt Y_i)\cong F_i\times(\phi(X)\cap\wt Y_i)$. Therefore, \cite[Proposition~8.2(b)]{DeKn81} implies
\[
\dim X \;=\; \max\{\dim X\cap\phi^{-1}(\wt Y_i)\colon 1\le i\le \tilde n\} \;\le\; \max\{\dim \phi(X)\cap\wt Y_i\colon 1\le i\le \tilde n\} \;=\; \dim\phi(X) 
\qedhere
\]
\end{proof}

\subsection{Ingredients for Step~5: Convergence of Complex Hypersurfaces}

As usual let $|z|$ denote the absolute value of a complex number $z\in\BC$ and let $d(\,.\,,.\,):\BC^d\to\BR_{\geq0}$ be the Euclidean metric. The following will be applied to $Y=h_2 T^{h_2}{}\open$ with $Y(\BR)=h_2 \BT^{h_2}{}\open$.

\begin{defn}\label{DefMeasureMu}
Let $Y\subset\BA^d_{\BR}$ be a smooth affine scheme over $\BR$, and assume that its set of real points $Y(\BR)\subset\BR^d$ is compact and Zariski-dense in the base change $Y_{\BC}\subset\BA^d_{\BC}$ of $Y$ to $\BC$. For every $\epsilon>0$ and every subset $A\subset Y(\BR)$ let $A(\epsilon)$ denote the open $\epsilon$-neighborhood of $A$ in $Y(\BR)$:
\[
A(\epsilon)\;:=\;\{z\in Y(\BR)\colon \exists\ y\in A\textrm{ such that }d(z,y)<\epsilon\}\;=\;\bigcup_{y\in A}\{z\in Y(\BR)\colon d(z,y)<\epsilon\}\,.
\]
Note that for a subset $A\subset Y(\BR)$ which is closed, and hence compact, the closure $\ol{A(\epsilon)}$ with respect to the metric on $Y(\BR)$ is contained in $A(2\epsilon)$.
We say that a sequence $(A_m)_{m\in\BN}$ of subsets of $Y(\BR)$ \emph{converges} to a subset $A\subseteq Y(\BR)$ if for every $\epsilon>0$ there is an index $m_\epsilon\in\BN$ such that $A_m\subseteq A(\epsilon)$ for every $m\geq m_\epsilon$.
\end{defn}

Let $P_D\subset\BC[x_1,x_2,\ldots,x_d]$ denote the complex vector space of all complex polynomials on $\BC^d$ of total degree at most $D$ and let $\olP_D$ be its image in the coordinate ring $\BC[x_1,x_2,\ldots,x_d]/I(Y_\BC)$ of $Y_\BC$. The images of the monomials $x_1^{i_1}x_2^{i_2}\cdots x_d^{i_d}$ for $\ul i=(i_1,i_2,\ldots,i_d)\in\BN_0^d$ with $i_1+i_2+\cdots+i_d\leq D$ form a generating system of the $\BC$-vector space $\olP_D$. We may shrink this generating system to a basis $\CB$. For every
\[
F(x_1,x_2,\ldots,x_d)\;=\;\sum_{\ul i\in\CB}
c_{\ul i}\cdot x_1^{i_1}x_2^{i_2}\cdots x_d^{i_d}\;\in\;\olP_D\qquad\text{with}\quad c_{\ul i}\in\BC
\]
let $\|F\|\;:=\;\max_{\ul i\in\CB}|c_{\ul i}|$. This is clearly a norm on $\olP_D$. It satisfies the following estimates.

\begin{lemma}\label{simple_estimate1} In the situation of Definition~\ref{DefMeasureMu}, let $M\in\BR$, $M\ge 1$ be such that $d(0,x)\le M$ for all $x\in Y(\BR)$.
\begin{enumerate}
\item \label{simple_estimate1_A}
For every $x,y\in Y(\BR)$ and for every $F\in\olP_D$ we have:
$$|F(x)-F(y)| \;\leq\; \|F\|\cdot d(x,y)\cdot M^{D-1}\cdot D\cdot\#\CB.$$
\item \label{simple_estimate1_B}
For every $x\in Y(\BR)$ and for every $F_1,F_2\in\olP_D$ we have:
$$|F_1(x)-F_2(x)| \;\leq\; \|F_1-F_2\|\cdot M^D\cdot\#\CB.$$
\end{enumerate}
\end{lemma}

\begin{proof} \ref{simple_estimate1_A}
Write $x=(x_1,x_2,\ldots,x_d)$ and $y=(y_1,y_2,\ldots,y_d)$. Then
\[
|x_i-y_i| \;\leq\; d(x,y)\qquad\text{and}\qquad |x_i|,|y_i| \;\le\; M \quad(\forall i=1,2,\ldots,d).
\]
In particular for every multi-index $\ul i=(i_1,i_2,\ldots,i_d)\in\CB$ by telescoping we get:
\begin{eqnarray}\nonumber
|x_1^{i_1}x_2^{i_2}\cdots x_d^{i_d}-y_1^{i_1}y_2^{i_2}\cdots y_d^{i_d}| &=&
|(x_1^{i_1}x_2^{i_2}\cdots x_d^{i_d}-y_1^{i_1}x_2^{i_2}\cdots x_d^{i_d})+
(y_1^{i_1}x_2^{i_2}\cdots x_d^{i_d}-y_1^{i_1}y_2^{i_2}x_3^{i_3}\cdots x_d^{i_d})+\cdots|\\ \nonumber
&\leq &
\sum_{j=1}^d|x_j^{i_j}-y_j^{i_j}|\cdot|y_1^{i_1}\cdots y_{j-1}^{i_{j-1}}\cdot x_{j+1}^{i_{j+1}}\cdots x_d^{i_d}|
\\ 
\nonumber
&\leq &
\sum_{j=1}^d|x_j-y_j|\cdot|x_j^{i_j-1}+x_j^{i_j-2}y_j+\cdots+y_j^{i_j-1}|\cdot M^{i_1+\ldots+i_{j-1}+i_{j+1}+\ldots+i_d}
\\ \nonumber
&\leq & d(x,y)\cdot M^{D-1}\cdot
\sum_{j=1}^di_j \\ \nonumber
&\leq & d(x,y)\cdot M^{D-1}\cdot D.
\end{eqnarray}
Write $\displaystyle F(t_1,t_2,\ldots,t_d) \;=\;\sum_{\ul i\in\CB} c_{\ul i}\cdot t_1^{i_1}t_2^{i_2}\cdots t_d^{i_d}.$ By the above we compute
\begin{eqnarray}\nonumber
|F(x)-F(y)| &=& |\sum_{\ul i\in\CB}c_{\ul i}\cdot (x_1^{i_1}x_2^{i_2}\cdots x_d^{i_d}-y_1^{i_1}y_2^{i_2}\cdots y_d^{i_d})|
\\ \nonumber
&\leq &
\sum_{\ul i\in\CB}|c_{\ul i}|\cdot 
d(x,y)\cdot M^{D-1}\cdot D
\\ \nonumber
&\leq &
\|F\|\cdot d(x,y)\cdot M^{D-1}\cdot D\cdot\#\CB.
\end{eqnarray}

\medskip\noindent
\ref{simple_estimate1_B}
Writing $F_1(t_1,\ldots,t_d)=\sum_{\ul i\in\CB}b_{\ul i}\cdot t_1^{i_1}\cdots t_d^{i_d}$ and $F_2(t_1,\ldots,t_d)=\sum_{\ul i\in\CB}c_{\ul i}\cdot t_1^{i_1}\cdots t_d^{i_d}$ we compute
\[
|F_1(x)-F_2(x)| \;\le\; \sum_{\ul i\in\CB}|b_{\ul i} - c_{\ul i}|\cdot |x_1|^{i_1}|x_2|^{i_2}\cdots |x_d|^{i_d} \;\le\; \sum_{\ul i\in\CB}\|F_1-F_2\|\cdot M^D \;\le\; \|F_1-F_2\|\cdot M^D\cdot\#\CB. 
\qedhere
\]
\end{proof}

\begin{thm}\label{new_oesterle1}
In the situation of Definition~\ref{DefMeasureMu}, let $D\in\BN$, and let $(X_m)_{m\in\BN}$ be a sequence of algebraic hypersurfaces of $\BC^d$ of degree at most $D$ such that the intersection $X_m\cap Y_{\BC}$ is a proper hypersurface in
$Y_{\BC}$ for every $m$. Then there is a subsequence $(X_{m_n})_{n\in\BN}$ and an algebraic hypersurface $X\subset\BC^d$ of degree at most $D$ with $X\cap Y_\BC\subsetneq Y_\BC$, such that the sequence $\bigl(X_{m_n}\cap Y(\BR)\bigr)_{n\in\BN}$ converges to $X\cap Y(\BR)$.
\end{thm}

\begin{proof} For every $m\in\BN$ let $F_m\in\olP_D$ be a polynomial such that $\|F_m\|=1$ and the zero set of $F_m$ in $Y_\BC$ is $X_m\cap Y_\BC$. Because the unit sphere $\{F\in\olP_D\colon\|F\|= 1\}$ is compact, there is a subsequence $(F_{m_n})_{n\in\BN}$ converging to an element $F\in\olP_D$ with respect to the norm $\|\,.\,\|$ such that $\|F\|=1$. We may even assume this subsequence is the full sequence after re-indexing.

We claim that in this case $\bigl(X_m\cap Y(\BR)\bigr)_{m\in\BN}$ converges to $X\cap Y(\BR)$, where $X\cap Y_\BC$ is the zero set of $F$. Assume that this is false. Then there is a small $\epsilon>0$ such that, after taking a suitable subsequence,  $X_m\cap Y(\BR)$ does not lie in $\bigl(X\cap Y(\BR)\bigr)(\epsilon)$ for every $m\in\BN$. Choose an $x_m\in \bigl(X_m\cap Y(\BR)\bigr)\smallsetminus\bigl(X\cap Y(\BR)\bigr)(\epsilon)$ for every $m$. Since the set $Y(\BR)\smallsetminus\bigl(X\cap Y(\BR)\bigr)(\epsilon)$ is closed in $Y(\BR)$, it is compact, so we may assume, after taking a suitable subsequence, that $x_m$ converges to a point $x\in Y(\BR)\smallsetminus \bigl(X\cap Y(\BR)\bigr)(\epsilon)$. Note that
\begin{align*}
|F_m(x)|=|F_m(x)-F_m(x_m)| \; \leq \enspace & M^{D-1}\cdot D\cdot \#\CB\cdot d(x,x_m) \qquad\text{and} \\[2mm]
|F(x)-F_m(x)|\;\le\enspace & M^D\cdot\#\CB\cdot\|F-F_m\|
\end{align*}
by Lemma~\ref{simple_estimate1}. Therefore,
$$F(x)=\lim_{m\to\infty}F_m(x)+\lim_{m\to\infty}\bigl(F(x)-F_m(x)\bigr)=0,$$
and $x\in X\cap Y(\BR)$. But this is a contradiction and so our claim is proven. To finish the proof of Theorem~\ref{new_oesterle1} we note that $F\ne0$ in $\olP_D$, because $\|F\|=1$, and so $X\cap Y_\BC$ is a proper hypersurface in $Y_\BC$.
\end{proof}

To prove Theorem~\ref{ThmOverconvIntro} we will also need the following

\begin{prop}\label{simple_estimate2} 
Keep the notation of Definition~\ref{DefMeasureMu}. Let $\lambda$ be a measure on $Y(\BR)$ satisfying $\lambda(Y(\BR))<\infty$. Then for every algebraic hypersurface $X\subset\BC^d$ with $X\cap Y(\BR)\subsetneq Y(\BR)$ and $\lambda\bigl(X\cap Y(\BR)\bigr)=0$ we have:
$$\lim_{\epsilon\to0}\lambda\bigl(\bigl(X\cap Y(\BR)\bigr)(\epsilon)\bigr) \;=\; 0.$$
\end{prop}

\begin{proof}
Assume that the claim is false. Then there is a strictly decreasing sequence $\epsilon_1,\epsilon_2,\ldots,\epsilon_n,\ldots$ such that $\epsilon_n\to0$, and $\lambda\bigl((X\cap Y(\BR))(\epsilon_n)\bigr) \;\geq\; \delta$, for some positive $\delta$. Set $A_n=\bigl(X\cap Y(\BR)\bigr)(\epsilon_n)-\bigl(X\cap Y(\BR)\bigr)(\epsilon_{n+1})$. Then by $\sigma$-additivity applied to $\bigl(X\cap Y(\BR)\bigr)(\epsilon_n)=\bigl(X\cap Y(\BR)\bigr)\cup\coprod_{m\geq n}A_m$ we have
\begin{equation}\label{Eq10.1}
\delta\;\le\; \lambda\bigl(\bigl(X\cap Y(\BR)\bigr)(\epsilon_n)\bigr)\;=\;\lambda\bigl(X\cap Y(\BR)\bigr)+\sum_{m\geq n}\lambda(A_m)\;=\;\sum_{m\geq n}\lambda(A_m)\,.
\end{equation}
Since $\lambda\bigl(\bigl(X\cap Y(\BR)\bigr)(\epsilon_1)\bigr)\le\lambda(Y(\BR))<\infty$ the sum $\sum_{m=1}^{\infty}\lambda(A_m)$ is convergent. Therefore, $\lim_{n\to\infty}\sum_{m\geq n}\lambda(A_m)=0$, and by taking the limit as $n\to\infty$ in (\ref{Eq10.1}) we get a contradiction.
\end{proof}

\subsection{Proof of Theorem~\ref{ThmOverconvIntro}}\label{SubsectProofOf1.7}

Finally we are ready to give the

\begin{proof}[Proof of Theorem~\ref{ThmOverconvIntro}] \refstepcounter{proofnumber}\label{ProofOfThmOverconvIntro}

\noindent
{\bfseries Step~1.} \ We use the reduction technique~\ref{TechniqueTwist} from page~\pageref{TechniqueTwist}. Let $\CF_i$ be the simple summands of $\CF$. Fix an $i$ and look at the element $\alpha_i:=\Frob^\dagger_\BasePoint(\det\CF_i):=\BasePoint^*\det F_{\CF_i}^\BasePtDeg$ in the (abelian) monodromy group $\Gr^\dagger(\det\CF_i)\subset \olK\mal$, where $\BasePoint\in U(\BF_{q^\BasePtDeg})$ is our fixed base point. We take an $(r_i\BasePtDeg)$-th root $\wt\alpha_i\in\olK$ of $\alpha_i$ where $r_i$ is the rank of $\CF_i$. We let $\CC_i$ be the constant, hence overconvergent $F$-isocrystal on $U$ induced from the $F$-isocrystal $(K,F=\wt\alpha_i^{-1})\in\FIsoc_\olK(\Spec\BF_q)$. Set $\CU_i:=\CF_i\otimes\CC_i$. Then $\CU_i$ is irreducible, because $\CF_i$ is. Moreover, the conjugacy class $\Frob^\dagger_\BasePoint(\det\CU_i)$ of $\det\CU_i=\det(\CF_i\otimes \CC_i)=(\det\CF_i)\otimes \CC_i^{\otimes r_i}$ is equal to $\{1\}\subset\Gr^\dagger(\det\CU_i)\subset \olK\mal$. In particular, $\det\CU_i$ is unit-root at $\BasePoint$. Since it has rank one, it is unit-root on the entire scheme $U$. By Lemma~\ref{LemmaCompatibility} also the tensor generator $\CE_i$ of $\dal\det\CU_i\dar_{const}$ has $\Frob^\dagger_\BasePoint(\CE_i)=\{1\}$ and so $\Gr^\dagger(\CE_i)=\mathbf{W}(\det\CU_i)$ is a finite group by Theorem~\ref{ThmMonodrOfConstant}\ref{ThmMonodrOfConstant_B} and \ref{ThmMonodrOfConstant_C}. Then $\CU_i$ is $\iota$-pure of weight zero by T.~Abe's and L.~Lafforgue's results, see Corollary~\ref{langlands_implies_mixedness}. By the reduction technique~\ref{TechniqueTwist} from page~\pageref{TechniqueTwist} we may assume that $\CF=\CU\oplus\CC$, where $\CU=\bigoplus_i\CU_i$ is a semi-simple and $\iota$-pure overconvergent $F$-isocrystals of weight zero and $\CC=\bigoplus_i\CC_i$ is the direct sum of irreducible constant $F$-isocrystals of varying weights.

Let $Z_1$ be the center of $\Gr^{\dagger}(\CU)\open$. We use Proposition~\ref{PropRedGpAlmostProduct} and let $\CS,\CT\in\dal\CU\dar$ be the overconvergent $F$-isocrystals whose monodromy groups are $\Gr^\dagger(\CS)=\Gr^\dagger(\CU)/Z_1$ and $\Gr^\dagger(\CT)=\Gr^\dagger(\CU)/[\Gr^\dagger(\CU)\open,\Gr^\dagger(\CU)\open]$. Unit-root, $\iota$-pure overconvergent $F$-isocrystals of weight zero form a full Tannakian sub-category, and both $\CS$ and $\CT$ belong to this sub-category. Let $\CE$ be a tensor generator of $\dal\CT\dar_{const}$. It is a constant $F$-isocrystal whose monodromy group is $\Gr^{\dagger}(\CE)=\mathbf{W}(\CT)=\Gr^{\dagger}(\CT)/\DGal^\dagger(\CT)$ by Proposition~\ref{PropCrewsGroup}. There is a commutative diagram:
\[
\xymatrix @R-1pc
{\Gr^{\dagger}(\CU\oplus\CC)
\ar@{^{ (}->}[r]\ar@{->>}[d] &
\Gr^{\dagger}(\CU)\TimesBC\Gr^{\dagger}(\CC)\ar@{->>}[d] \\
\Gr^{\dagger}(\CS\oplus\CT\oplus\CC)\ar@{^{ (}->}[r]\ar@{->>}[d] &
\Gr^{\dagger}(\CS\oplus\CT)\TimesBC\Gr^{\dagger}(\CC)\ar@{^{ (}->}[r] &
\Gr^{\dagger}(\CS)\TimesBC\Gr^{\dagger}(\CT)\TimesBC
\Gr^{\dagger}(\CC) \ar@{->>}[d]\\
\Gr^{\dagger}(\CS\oplus\CE\oplus\CC)\ar@{^{ (}->}[rr] & &
\Gr^{\dagger}(\CS)\TimesBC\Gr^{\dagger}(\CE)\TimesBC
\Gr^{\dagger}(\CC).}
\]
By Proposition~\ref{PropRedGpAlmostProduct} the group $\Gr^{\dagger}(\CS)\open$ is semi-simple with trivial center, $\Gr^{\dagger}(\CT)\open$ is a torus, and the vertical map in the middle has finite kernel. The vertical map on the right has kernel $\DGal^\dagger(\CT)$ which is a finite group by \cite[Lemma~6.1]{Abe11}. Since the horizontal maps are injective by Proposition~\ref{PropGroupOfSum}\ref{PropGroupOfSum_B}, we get that the two vertical maps on the left have finite kernel, too. By Lemma~\ref{Lemma1.10Converse}\ref{Lemma1.10Converse_C} and Proposition~\ref{reduction3->2} it is enough to show that $\CS\oplus\CE\oplus\CC$ satisfies Conjecture~\ref{MainConj3}.

Note that $\CE\oplus\CC$ is the direct sum of constant $\iota$-pure $F$-isocrystals of varying weights. By slight abuse of notation writing $\CU$ for $\CS$ and $\CC$ for $\CE\oplus\CC$ we may assume without loss of generality that $\CF=\CU\oplus\CC$, where $\CU$ is a semi-simple and $\iota$-pure overconvergent $F$-isocrystals of weight zero such that $\Gr^{\dagger}(\CU)\open$ is semi-simple with trivial center, and $\CC$ is the direct sum of constant $\iota$-pure $F$-isocrystals of varying weights. In this case $\Gr^{\dagger}(\CF)$ is the fiber product of $\Gr^{\dagger}(\CU)$ and $\Gr^{\dagger}(\CC)$ over $\Gr^{\dagger}(\dal\CU\dar\cap\dal\CC\dar)$ by Proposition~\ref{PropGroupOfSum}\ref{PropGroupOfSum_B}. The group $\Gr^{\dagger}(\CC)$ is commutative by Theorem~\ref{ThmMonodrOfConstant}\ref{ThmMonodrOfConstant_B}, and hence its quotient $\Gr^{\dagger}(\dal\CU\dar\cap\dal\CC\dar)\open$ is commutative, too. Since $\Gr^{\dagger}(\dal\CU\dar\cap\dal\CC\dar)\open$ is also a quotient of the semi-simple group $\Gr^{\dagger}(\CU)\open$, which has no commutative quotients by \cite[IV.14.2~Proposition]{Borel91}, we conclude that $\Gr^{\dagger}(\dal\CU\dar\cap\dal\CC\dar)\open$ is trivial. So $\Gr^{\dagger}(\CF)\open$ is actually the direct product of $\Gr^{\dagger}(\CU)\open$ and $\Gr^{\dagger}(\CC)\open$.

\medskip\noindent
{\bfseries Step~2.} \ We now consider the base change of these groups to $\BC$ via $\iota$
\[
G_1\;:=\;\Gr^{\dagger}(\CF)\times_{\olK,\iota}\BC\,,\qquad  G_2\;:=\;\Gr^{\dagger}(\CU)\times_{\olK,\iota}\BC\,,\qquad G_3\;:=\;\Gr^{\dagger}(\CC)\times_{\olK,\iota}\BC\,.
\]
Let $T_1\subset G_1$ be a maximal quasi-torus. It is the fiber product of two maximal quasi-tori $T_2\subset  G_2$ and $T_3\subset G_3$ with $T_1\open=T_2\open\TimesBC T_3\open$ by Remark~\ref{handy}.  Note that actually $T_3= G_3$, because $ G_3\open$ is a torus. Let $\BT_1$ be a maximal compact quasi-torus in $T_1$, see Definition~\ref{DefMaxCompQT} and let $\BT_j$ for $j=2,3$ be the image of $\BT_1$ under the projections $T_1\onto T_j$. Then $\BT_j$ is a maximal compact quasi-torus in $T_j$ and $G_j$ by Corollary~\ref{Cor6.10compact}. 

Let $S\subset |U|$ be a subset of positive upper Dirichlet density $\ol\delta(S)>0$. For every connected component $h_1G_1\open$ of $G_1$ we consider the subset $S(h_1)$ of those $x\in S$ for which $h_1G_1\open$ contains a point of $\Frob^\dagger_x(\CF)$. Then $S$ is the finite union of the subsets $S(h_1)$. By Lemma~\ref{sub_additivity} at least one of them has positive upper Dirichlet density. We replace $S$ by this subset and then consider the connected component $h_1G_1\open$ of $G_1$ which meets $\Frob^\dagger_x(\CF)$ for every $x\in S$. By Theorem~\ref{ThmQuasiTorusReductive} and Theorem~\ref{ThmHoch2}\ref{ThmHoch2_A} we may assume that $h_1\in \BT_1$. For $j=2,3$ let $h_j\in \BT_j$ be the image of $h_1$ under the projection $T_1\onto T_j$. Then $T_3^{h_3}:=\{t_3\in T_3\open\colon t_3h_3=h_3t_3\}=T_3\open$ using Notation~\ref{Notation9.4NEW}, because $T_3=G_3$ is commutative. Therefore, $T_1^{h_1}:=\{t_1=(t_2,t_3)\in T_1\open=T_2\open\TimesBC T_3\open\colon t_1h_1=h_1t_1\}=T_2^{h_2}\TimesBC T_3\open$ and $h_1T_1^{h_1}{}\open=h_2T_2^{h_2}{}\open\TimesBC h_3T_3\open$. Moreover, $\BT_j^{h_j}=T_j^{h_j}\cap\BT_j\open$ is a maximal compact subgroup in $T_j^{h_j}$ for $j=1,2,3$ by Lemma~\ref{no_prob}. 

We consider the semi-simplification $\Frob^{\dagger}_x(\CF)^{ss}$ of the Frobenius conjugacy class $\Frob^{\dagger}_x(\CF)$ and similarly for the $F$-isocrystal $\CU$. In order to show that $h_1G_1\open\cap\bigcup_{x\in S}\Frob^\dagger_x(\CF)$ is dense in $h_1G_1\open$, it is by Lemma~\ref{useful}\ref{useful_B} and Proposition~\ref{Prop8.8} enough to show that $h_1T_1^{h_1}{}\open\cap\bigcup_{x\in S}\Frob^\dagger_x(\CF)^{ss}$ is dense in $h_1T_1^{h_1}{}\open$. Also note that $h_1T_1^{h_1}{}\open\cap\Frob^\dagger_x(\CF)^{ss}\ne\emptyset$ for every $x\in S$ by our choice of $h_1$ and by Theorem~\ref{ThmQuasiTorusReductive} and Proposition~\ref{comment27NEW}\ref{comment27NEW_B}.

As in Notation~\ref{Notation13.14} we now consider maximal compact subgroups $\wt\BK_2^{\rm arith}$ of $G_2$ and $\wt\BK_2^{\rm geo}$ of $G_2^{\rm geo}:=\DGal^\dagger(\CU)\times_{\olK,\iota}\BC$ such that $\wt\BK_2^{\rm arith}$ contains $\wt\BK_2^{\rm geo}$. By Theorem~\ref{ThmHoch2} there is an element $e\in G_2\open$ such that $\BK_2^{\rm arith}:=e\,\wt\BK_2^{\rm arith}e^{-1}$ contains $\BT_2$ and is a maximal compact subgroup of $G_2$. Since $G_2^{\rm geo}$ is normal in $G_2$ conjugation by $e$ is an automorphism of $G_2^{\rm geo}$, and hence $\BK_2^{\rm geo}:=e\,\wt\BK_2^{\rm geo}e^{-1}$ is a maximal compact subgroup of $G_2^{\rm geo}$ contained in $\BK_2^{\rm arith}$. To lighten the notation we drop the superscript ``arith'' and just write $\BK_2:=\BK_2^{\rm arith}$. We let $\gamma\in \Gamma$ be the image of the element $h_2\in\BT_2$, and we denote by $\BK_{2,\gamma}$ the preimage of $\gamma$ in $\BK_2$. It is a union of connected components containing $h_2\BK_2\open$. Note that by Lemma~\ref{deligne} two elements of $\BK_2$ which are conjugate under $G_2$ are already conjugate under $\BK_2$. We denote the set of conjugacy classes of $\BK_2$ by $\BK_2^\#$, the ones which meet $\BK_{2,\gamma}$ by $\BK_{2,\gamma}^\#$, and the ones which meet $h_2\BK_2\open$ by $(h_2\BK_2\open)^\#$. We equip these sets with the quotient topology. Then $(h_2\BK_2\open)^\#$ is a connected component of $\BK_{2,\gamma}^\#$. 
We consider the following diagram
\begin{equation}\label{EqDiagConjClass}
\xymatrix @C=3pc {
 & h_2\BK_2\open \ar@{->>}[d]^\psi \ar@{^{ (}->}[r] & \BK_{2,\gamma} \ar@{->>}[d]^\psi \\
h_2\BT_2^{h_2}{}\open \ar@{->>}[r]^\phi & (h_2\BK_2\open)^\# \ar@{^{ (}->}[r] & \BK_{2,\gamma}^\#\;.\!\!\!
}
\end{equation}
In this diagram the map $\psi$ is surjective by construction, and $\phi$ is surjective by Propositions~\ref{PropMaxCPQuasiTorus}\ref{PropMaxCPQuasiTorus_D} and \ref{compact_conj}\ref{compact_conj_A}. In particular, it follows from this and Definition~\ref{Def_theta(x)} that $\theta(x)\in(h_2\BK_2\open)^\#$, and hence
\begin{equation}\label{EqIntersectNonEmpty}
h_2\BT_2^{h_2}{}\open\cap\Frob^\dagger_x(\CU)^{ss} \;=\; \phi^{-1}(\theta(x))\;\ne\; \emptyset \qquad\text{for every }x\in S\,.
\end{equation}

\noindent
{\bfseries Step~3.} \ On the sets in Diagram~\eqref{EqDiagConjClass} we consider various measures: On $\BK_{2,\gamma}$ and $h_2\BK_2\open$ the (restriction of the) Haar measure $\mu_{{\rm Haar},\gamma}$ from Definition~\ref{DerMeasuresOnConjCl} and on $\BK_{2,\gamma}^\#$ and $(h_2\BK_2\open)^\#$ its push-forward $\mu_{{\rm Haar},\gamma}^\#$. Moreover, on $\BK_{2,\gamma}^\#$ and $(h_2\BK_2\open)^\#$ we consider the (restriction of the) measures $\mu_m$ from Definition~\ref{DefMeasure_d_mu_m}. By the Equidistribution Theorem~\ref{ThmEquiDistr}, when $m\to\infty$ the measures $\mu_m$, for which $[m]=\gamma\in\Gamma$ in \eqref{EqClassInGamma}, converge weakly to $\mu_{{\rm Haar},\gamma}^\#$ on $\BK_{2,\gamma}^\#$ and on $(h_2\BK_2\open)^\#$.

\medskip\noindent
{\itshape Claim.} Every closed semi-algebraic subset $Y\subset(h_2\BK_2\open)^\#$ with $\dim Y<\dim(h_2\BK_2\open)^\#$ has volume $\mu_{{\rm Haar},\gamma}^\#(Y)=0$.

\medskip\noindent
Proof of the Claim. By Lemma~\ref{LemmaPhiNice}\ref{LemmaPhiNice_D} the semi-algebraic set $\phi^{-1}(Y)$ satisfies $\dim\phi^{-1}(Y)<\dim h_2\BT_2^{h_2}{}\open$. By definition this means that the Zariski-closure $\ol{\phi^{-1}(Y)}$ in $h_2T_2^{h_2}{}\open$ is strictly contained in $h_2T_2^{h_2}{}\open$. Let ${}^{G_2}\phi^{-1}(Y)$ be the union of the $G_2$-conjugacy classes of the elements of $\phi^{-1}(Y)\subset h_2G_2\open$. The preimage $\psi^{-1}(Y)$ equals the intersection of ${}^{G_2}\phi^{-1}(Y)$ with $h_2\BK_2\open$. If $\psi^{-1}(Y)$ was dense in $h_2G_2\open$, then ${}^{G_2}\phi^{-1}(Y)$ would be dense in $h_2G_2\open$, too. This would imply by Lemma~\ref{useful}\ref{useful_B} and Proposition~\ref{Prop8.8} that ${}^{G_2}\phi^{-1}(Y)\cap h_2T_2^{h_2}{}\open$ is dense in $h_2T_2^{h_2}{}\open$. In the notation of Proposition~\ref{PropComment27NEW} we have ${}^{G_2}\phi^{-1}(Y)\cap h_2T_2^{h_2}{}\open\subset \bigcup_{w\in \Omega, z\in \Delta}zw\phi^{-1}(Y)w^{-1}\subset h_2T_2^{h_2}{}\open$. This is a finite union. Since $h_2T_2^{h_2}{}\open$ is irreducible, already one component $zw\phi^{-1}(Y)w^{-1}$ must be dense in $h_2T_2^{h_2}{}\open$ for certain $w$ and $z$. But then $\phi^{-1}(Y)$ would be dense in $w^{-1}z^{-1}(h_2T_2^{h_2}{}\open) w=h_2T_2^{h_2}{}\open$ which yields a contradiction. Therefore, $\ol{\psi^{-1}(Y)}$ must be contained in a proper hypersurface $W\subsetneq h_2G_2\open$ and then $\psi^{-1}(Y)\subset W\cap h_2\BK_2\open\subsetneq h_2\BK_2\open$, the latter being a strict inclusion by Proposition~\ref{PropRealFormOfSemiSimple}. Since $W$ is defined by a polynomial equation in the coordinates of $h_2G_2\open$ and $\mu_{{\rm Haar},\gamma}$ is absolutely continuous with respect to the Lebesgue measure on the charts of the differentiable manifold $h_2\BK_2\open$, we conclude $0=\mu_{{\rm Haar},\gamma}\bigl(\psi^{-1}(Y)\bigr)=\mu_{{\rm Haar},\gamma}^\#(Y)$. This proves the claim.

\medskip\noindent
{\bfseries Step~4.} \ Since $\phi$ is nice in the sense of Definition~\ref{DefNiceMap} by Lemma~\ref{LemmaPhiNice}, we can pull back measures along $\phi$, see Definition~\ref{DefPullbackMeasure}. By Lemma~\ref{LemmaPhiNice} and Proposition~\ref{PropWeakConvOfPullBack} the pullback measures $\phi^*\mu_m$, for which the class $[m]$ of $m$ in $\Gamma$ equals $\gamma$, converge weakly to the measure $\lambda:=\phi^*\mu_{{\rm Haar},\gamma}^\#$ on $h_2\BT_2^{h_2}{}\open$ when $m\to\infty$. 

\medskip\noindent
{\itshape Claim.} The pull-back measure $\lambda:=\phi^*\mu_{{\rm Haar},\gamma}^\#$ satisfies $\lambda(h_2\BT_2^{h_2}{}\open)<\infty$ and $\lambda(X\cap h_2\BT_2^{h_2}{}\open)=0$ for every proper hypersurface $X\subsetneq h_2T_2^{h_2}{}\open$.

\medskip\noindent
Proof of the Claim. By Lemma~\ref{LemmaPhiNice}\ref{LemmaPhiNice_B} the cardinality of every fiber $\phi^{-1}(y)$ is at most $M$. Therefore,
\[
\lambda(h_2\BT_2^{h_2}{}\open)\;:=\;\int_{(h_2\BK_2\open)^\#}\#\phi^{-1}(y)\;d\mu_{{\rm Haar},\gamma}^\#(y)\;\le\;M\cdot \mu_{{\rm Haar},\gamma}^\#\bigl((h_2\BK_2\open)^\#\bigr)\;:=\;M\cdot \mu_{{\rm Haar},\gamma}(h_2\BK_2\open)\;\le\;M
\]
by Definition~\ref{DerMeasuresOnConjCl} of the measure $\mu_{{\rm Haar},\gamma}$.

Next let $X$ be as in the second statement. By \cite[\S\,8, Definitions~1 and 2]{DeKn81} the dimension of $X\cap h_2\BT_2^{h_2}{}\open$ is the dimension of its Zariski-closure in $h_2T_2^{h_2}{}\open$. Therefore, $\dim(X\cap h_2\BT_2^{h_2}{}\open)\le \dim X<\dim h_2T_2^{h_2}{}\open=\dim h_2\BT_2^{h_2}{}\open=\dim (h_2\BK_2\open)^\#$, because $h_2T_2^{h_2}{}\open$ is irreducible. Here the second-to-last equality follows from \cite[Proposition~8.6]{DeKn81} by considering a real structure on $T_2^{h_2}{}\open$ with $\BT_2^{h_2}{}\open= T_2^{h_2}{}\open(\BR)$ as in Proposition~\ref{PropRealFormOfSemiSimple}. The last equality follows from Lemma~\ref{LemmaPhiNice}\ref{LemmaPhiNice_D}. By \cite[Proposition~8.3]{DeKn81} we have 
\[
\dim \phi(X\cap h_2\BT_2^{h_2}{}\open)\;\le\;\dim(X\cap h_2\BT_2^{h_2}{}\open)\;<\;\dim (h_2\BK_2\open)^\#\,,
\]
and hence the claim in Step~3 implies $\mu_{{\rm Haar},\gamma}^\#\bigr(\phi(X\cap h_2\BT_2^{h_2}{}\open)\bigl)=0$. Now for a point $y\in (h_2\BK_2\open)^\#$ the cardinality of $\phi^{-1}(y)\cap X\cap h_2\BT_2^{h_2}{}\open$ is zero if $y\notin\phi(X\cap h_2\BT_2^{h_2}{}\open)$ and otherwise at most $M$ by Lemma~\ref{LemmaPhiNice}\ref{LemmaPhiNice_B}. Thus we compute
\[
\lambda(X\cap h_2\BT_2^{h_2}{}\open)\;:=\;\int_{(h_2\BK_2\open)^\#}\#\bigl(\phi^{-1}(y)\cap X\cap h_2\BT_2^{h_2}{}\open\bigr)\;d\mu_{{\rm Haar},\gamma}^\#(y)\;\le\;M\cdot \mu_{{\rm Haar},\gamma}^\#\bigr(\phi(X\cap h_2\BT_2^{h_2}{}\open)\bigl)\;=\;0
\]
as desired. This proves the claim.

\medskip\noindent
{\bfseries Step~5.} \ The pullback measure $\phi^*\mu_m$ on $h_2\BT_2^{h_2}{}\open$ has the following description. Let $A\subset h_2\BT_2^{h_2}{}\open$ be a Borel-subset. Then 
\begin{eqnarray*}
\phi^*\mu_m(A) \; := \; \frac{1}{\#U(m)}\sum_{x\in U(m)}(\phi^*\delta_{\theta(x)})(A) \; = \; \frac{1}{\#U(m)}\sum_{x\in U(m)} \#(A\cap\Frob_x(\CU)^{ss})
\end{eqnarray*}
by Definition~\ref{DefPullbackMeasure}. Now let $S_m:=S\cap U(m)$. Then Lemma~\ref{Lemma4.3} gives us an infinite subset $R\subset\BN$ such that 
\[
\frac{\#S_m}{\#U(m)} \;\ge\; \frac{\ol\delta(S)}{2} \quad\text{for every }m\in R\,.
\]
In particular, for every $m\in R$ we have $S_m\ne\emptyset$, and $h_2\BT_2^{h_2}{}\open\cap\Frob^\dagger_x(\CU)^{ss}\ne\emptyset$ for every $x\in S_m$ by \eqref{EqIntersectNonEmpty}. Thus the image $[m]$ of $m$ in $\Gamma=\mathbf{W}(\CU)(\BC)$ under the map \eqref{EqClassInGamma} coincides with the image of $h_2$ which we called $\gamma$.

Consider closed immersions $h_2T_2^{h_2}{}\open\hookrightarrow\BC^{d_2}$ and $h_3T_3\open\hookrightarrow\BC^{d_3}$. It will be sufficient to prove the following

\medskip\noindent
{\itshape Claim.} No proper hypersurface $W\subset \BC^{d_2+d_3}$ with $W\cap h_1T_1^{h_1}{}\open\subsetneq h_1T_1^{h_1}{}\open=h_2T_2^{h_2}{}\open\TimesBC h_3T_3\open\subset \BC^{d_2}\TimesBC\BC^{d_3}$ contains $h_1T_1^{h_1}{}\open\cap\bigcup_{x\in S}\Frob^{\dagger}_x(\CF)^{ss}$.

\medskip\noindent
To prove the claim, assume the contrary and let $W$ be such a counterexample. Let $f\in G_3$ be the Frobenius of the constant $F$-isocrystal $\CC$. Then $f^m=\Frob^\dagger_x(\CC)=\Frob^\dagger_x(\CC)^{ss}\in h_3T_3\open$ for every $x\in S_m$. Let $D$ be the degree of $W$ in the variables of the first factor $\BC^{d_2}$ and for every $m\in R$ set
\[
X_m=\{z\in\BC^{d_2}\colon (z,f^m)\in W\}\,.
\]
Then $\Frob^\dagger_x(\CF)^{ss}=\Frob^\dagger_x(\CU)^{ss}\times\{f^{\deg(x)}\}$ implies that $h_2T_2^{h_2}{}\open\cap\Frob^\dagger_x(\CU)^{ss}\subset X_m$ for every $x\in S_m$. Each $X_m$ is a hypersurface in $\BC^{d_2}$ of degree $\le D$ such that $X_m\cap h_2T_2^{h_2}{}\open$ is properly contained in $h_2T_2^{h_2}{}\open$ for all but finitely many $m$ by Corollary~\ref{Cor4.2} (applied with $L=\BC, T=T_3, \Gamma=f^\BZ, G=G_1, T^c=h_3T_3\open$). So by shrinking $R$ we may assume that $X_m\cap h_2T_2^{h_2}{}\open$ is properly contained in $h_2T_2^{h_2}{}\open$ for every $m\in R$. By considering a real structure on $T_2^{h_2}$ with $h_2\BT_2^{h_2}{}\open =h_2T_2^{h_2}{}\open(\BR)$ as in Proposition~\ref{PropRealFormOfSemiSimple}, applying Theorem~\ref{new_oesterle1} to $Y=h_2 T_2^{h_2}{}\open$, and shrinking $R$ further we may even assume that there is a proper hypersurface $X\subsetneq h_2T_2^{h_2}{}\open\subset \BC^{d_2}$ of degree at most $D$ such that the sequence $(X_m\cap h_2\BT_2^{h_2}{}\open)_m$ converges to $X\cap h_2\BT_2^{h_2}{}\open$ in the sense of Definition~\ref{DefMeasureMu}. Then $X\cap h_2\BT_2^{h_2}{}\open\subsetneq h_2\BT_2^{h_2}{}\open$ by Proposition~\ref{PropRealFormOfSemiSimple} and $\lambda(X\cap h_2\BT_2^{h_2}{}\open)=0$ by the claim in Step~4.

By Proposition~\ref{simple_estimate2} there is a small $\epsilon>0$ such that $\lambda\bigl((X\cap h_2\BT_2^{h_2}{}\open)(2\epsilon)\bigr)<\tfrac{1}{2}\,\ol\delta(S)$. By the triangle inequality $(X\cap h_2\BT_2^{h_2}{}\open)(2\epsilon)$ contains the closure $\ol{(X\cap h_2\BT_2^{h_2}{}\open)(\epsilon)}$ of $(X\cap h_2\BT_2^{h_2}{}\open)(\epsilon)$ in the metric space $h_2\BT_2^{h_2}{}\open$, and hence $\lambda\bigl(\ol{(X\cap h_2\BT_2^{h_2}{}\open)(\epsilon)}\bigr)<\tfrac{1}{2}\,\ol\delta(S)$. Choose an $m_\epsilon\in\BN$ such that for every index $m\in R$ with $m\geq m_\epsilon$ we have $X_m\cap h_2\BT_2^{h_2}{}\open\subseteq (X\cap h_2\BT_2^{h_2}{}\open)(\epsilon)$. For every $x\in S_m$ the intersection $h_2\BT_2^{h_2}{}\open\cap\Frob^\dagger_x(\CU)^{ss}$ is contained in $X_m\cap h_2\BT_2^{h_2}{}\open$ by assumption. Moreover, $h_2\BT_2^{h_2}{}\open\cap\Frob^\dagger_x(\CU)^{ss}$ is non-empty by \eqref{EqIntersectNonEmpty}. So for every $m\in R$ with $m\ge m_\epsilon$ and for every $x\in S_m$ we have $\#\bigl(\ol{(X\cap h_2\BT_2^{h_2}{}\open)(\epsilon)}\cap\Frob_x(\CU)^{ss}\bigr) \ge 1$, and hence
\[
\phi^*\mu_m\bigl(\ol{(X\cap h_2\BT_2^{h_2}{}\open)(\epsilon)}\bigr)\; = \; \frac{1}{\#U(m)}\sum_{x\in U(m)} \#\bigl(\ol{(X\cap h_2\BT_2^{h_2}{}\open)(\epsilon)}\cap\Frob_x(\CU)^{ss}\bigr) \;\ge\; \frac{\#S_m}{\#U(m)} \;\ge\;\frac{\ol\delta(S)}{2}\,.
\]
But taking $\limsup_{m\to\infty}$ and by the weak convergence of $\phi^*\mu_m$ to $\lambda$ and the Portemanteau theorem~\cite[Theorem~13.16]{Klenke} we have
\[
\frac{\ol\delta(S)}{2}\;\le\;\limsup_{m\to\infty}\phi^*\mu_m\bigl(\ol{(X\cap h_2\BT_2^{h_2}{}\open)(\epsilon)}\bigr)\;\le\; \lambda\bigl(\ol{(X\cap h_2\BT_2^{h_2}{}\open)(\epsilon)}\bigr)\;<\;\frac{\ol\delta(S)}{2}\;,
\]
which is a contradiction. This rules out the existence of $W$ and finishes the proof of Theorem~\ref{ThmOverconvIntro}.
\end{proof}

\begin{rem}
Let us point out the similarity between theorems Theorem~\ref{Thm4.3} and Theorem~\ref{ThmOverconvIntro}. In both cases we show that a set $F$ is Zariski-closed in a group ($G$ or $G_1$) which is a fibered product of another group ($H$ or $G_2$) and a torus ($T$ or $G_3$) such that the fibers in $F$ of the projection to the torus satisfy an equidistribution property in some naturally occurring compact group ($\BK$ or $\BK_2$). Indeed, in the first case the Chebotar\"ev Density Theorem for $p$-adic Galois groups is an equidistribution theorem for the normalized Haar measure on the compact group $\BK$ which is the image of the \'etale fundamental group of the variety $U$. We think of Theorem~\ref{Thm4.3} as a $p$-adic, or non-archimedean variant, while Theorem~\ref{ThmOverconvIntro} is an archimedean variant, where the equidistribution result used is Theorem~\ref{ThmEquiDistr}. In our exposition, Theorem~\ref{Thm4.3} is a general, abstract result, axiomatizing the input data needed, while Theorem~\ref{ThmOverconvIntro} is not. 

It is possible to formulate an abstract group theoretic version of Theorem~\ref{ThmOverconvIntro}, too. However, we refrain from doing so for several reasons. Firstly, the statement would be complicated and not too different from the layout of the claim of Theorem~\ref{Thm4.3}. Secondly, we do not have any application in mind at the moment other than Theorem~\ref{ThmOverconvIntro}. Thirdly, what is really important here is not so much the precise statement of the claim, but the method. In fact, it is quite common in analytical number theory, which we consider these results to be part of, that such a result has to be modified every time to fit the actual situation it is applied to next, and it is impossible to find one master theorem which covers all imaginable situations at once. So we leave the details to the reader, in case the need arises. 
\end{rem}

\section{The Parabolicity Conjecture and its Consequences} \label{SectPink}

In this last section we study the consequences of the parabolicity Conjecture~\ref{ConjPink} for (over)convergent $F$-isocrystals $\CF$ on $U$ and a dense open subscheme $f\colon V\into U$. This conjecture was formulated by Richard Pink in private conversation with one of us as a question in the special case when $\CF$ comes from a $p$-divisible group on $U$. Note that by the following lemma, the conjecture is equivalent to the assertion that the natural injective morphism $\Gr(f^*\CF/V)\to\Gr(\CF/U)$ maps every Borel subgroup of $\Gr(f^*\CF/V)$ onto a Borel subgroup of $\Gr(\CF/U)$. We begin with a

\begin{lemma}\label{Lemma9.2}
The natural morphism $\Gr(f^*\CF/V)\to\Gr(\CF/U)$ always is a closed immersion and induces a surjection on the groups of connected components.
\end{lemma}

\begin{proof}
The injectivity follows from \cite[Proposition~2.21(b)]{Deligne-Milne82}, because every object of $\dal f^*\CF\dar$ is a subquotient of an object of the form $\bigoplus_i f^*\CF^{\otimes m_i}\otimes (f^*\CF\dual)^{\otimes n_i}=f^*\bigl(\bigoplus_i \CF^{\otimes m_i}\otimes (\CF\dual)^{\otimes n_i}\bigr)$. To prove the surjectivity on connected components, let $\CU$ be an object of $\dal\CF\dar$ such that the surjective homomorphism $\Gr(\CF/U)\onto\Gr(\CU/U)$ has kernel equal to the normal subgroup $\Gr(\CF/U)\open\subset \Gr(\CF/U)$; see Lemma~\ref{LemmaCompatibility}\ref{LemmaCompatibility_C}. Then $\CU$ is unit-root by Lemma~\ref{Lemma3.2}, and hence the inclusion map $\Gr(f^*\CU/V)\subset\Gr(\CU/U)$ is an isomorphism by Corollary~\ref{CorUnitRootFromUToV}. Since $f^*\CU$ is an object of $\dal f^*\CF\dar$ the corresponding homomorphism $\Gr(f^*\CF/V)\onto\Gr(f^*\CU/V)\isoto\Gr(\CU/U)= \Gr(\CF/U)\big/\Gr(\CF/U)\open$ is surjective by Lemma~\ref{LemmaCompatibility}. This finishes the proof.
\end{proof}

For our purpose the following weaker form of Conjecture~\ref{ConjPink} will even be more important.

\begin{con}\label{ConjWPink}
Let $\CF\in\FIsoc_\olK(U)$ and let $f\colon V\into U$ be a non-empty open subscheme. Then every maximal quasi-torus of $\Gr(f^*\CF/V)$ is also a maximal quasi-torus of the group $\Gr(\CF/U)$.
\end{con}

The reason why we are interested in this conjecture, is Theorem~\ref{ThmWPinkImpliesChebotarevIntro} which we will prove at the end of this section.

\begin{prop}\label{PropWeaklyPinkQuasiTorus}
Let $\CF\in\FIsoc_\olK(U)$ and let $f\colon V\into U$ be an open subscheme. Write $\beta\colon\Gr(f^*\CF/V)\into\Gr(\CF/U)$ for the inclusion.
\begin{enumerate}
\item \label{PropWeaklyPinkQuasiTorus_A}
Conjecture~\ref{ConjPink} for $\CF$ and $f$ implies Conjecture~\ref{ConjWPink} for $\CF$ and $f$.
\item \label{PropWeaklyPinkQuasiTorus_B}
If Conjecture~\ref{ConjWPink} holds for $\CF$ and $f$ then $\beta$ induces an isomorphism on the groups of connected components and every maximal torus of $\Gr(f^*\CF/V)\open$ is also a maximal torus of $\Gr(\CF/U)\open$. (See Remark~\ref{RemWPinkConverse} for the converse.)
\item \label{PropWeaklyPinkQuasiTorus_C}
If $\CF$ satisfies Conjecture~\ref{ConjPink}, respectively \ref{ConjWPink}, respectively the condition $\Gr(f^*\CF/V)=\Gr(\CF/U)$. Then the same is true for every object $\CG$ of $\dal\CF\dar$.
\item \label{PropWeaklyPinkQuasiTorus_D}
If $\CF$ is semi-simple and has an overconvergent extension $\CF^\dagger\in\FIsoc^\dagger_\olK(U)$, then Conjectures~\ref{ConjPink} and \ref{ConjWPink} hold true for $\CF$ and $f$.
\end{enumerate}
\end{prop}

\begin{proof}
Let $T_1$ be a maximal quasi-torus of $G_1:=\Gr(f^*\CF/V)$ and set $G_2:=\Gr(\CF/U)$.

\medskip\noindent
\ref{PropWeaklyPinkQuasiTorus_A} 
By Theorem~\ref{ThmQuasiTorusGeneral} there is a Borel subgroup $B_1\supset T_1\open$ of $G_1\open$ which is normalized by $T_1$. Since $\beta(G_1)\subset G_2$ is parabolic, \cite[IV.11.2~Corollary and IV.11.3~Corollary]{Borel91} imply that $\beta(T_1\open)$ is a maximal torus and $\beta(B_1)$ is a Borel subgroup of $G_2\open$, both normalized by $\beta(T_1)$. By Lemma~\ref{Lemma9.2} the map $T_1\onto G_1/G_1\open\onto G_2/G_2\open$ is surjective, and hence $\beta(T_1)$ is a maximal quasi-torus in $G_2$ by Theorem~\ref{ThmQuasiTorusGeneral}\ref{ThmQuasiTorusGeneral_E}.

\medskip\noindent
\ref{PropWeaklyPinkQuasiTorus_B}
By assumption $T_2:=\beta(T_1)$ is a maximal quasi-torus of $G_2$ and $\beta(T_1\open)=T_2\open$ is a maximal torus of $G_2$ by Lemma~\ref{LemmaQuasiTorus}. Furthermore, Theorem~\ref{ThmQuasiTorusGeneral} implies that $G_1/G_1\open \cong T_1/T_1\open\cong T_2/T_2\open\cong G_2/G_2\open$ is an isomorphism.

\medskip\noindent
\ref{PropWeaklyPinkQuasiTorus_C}
We have a commutative diagram:
\[
\xymatrix @R-0.5pc {
\Gr(f^*\CF/V)\ar@{^{ (}->}[r]^-\beta \ar@{->>}[d]_{\pi_1} & \Gr(\CF/U)\,\;\ar@{->>}[d]^{\pi_2}  \\
\Gr(f^*\CG/V)\ar@{^{ (}->}[r]^-\gamma & \Gr(\CG/U)\,.}
\]
If $\beta$ is an isomorphism, then $\gamma$ is also an isomorphism, because it is a surjective closed immersion. On the other hand, let $Q$ be a Borel subgroup (resp.~maximal quasi-torus) in $\Gr(f^*\CF/V)$. By assumption the image $\beta(Q)$ is also a Borel subgroup (resp.~a maximal quasi-torus) in $\Gr(\CF/U)$.
Since $\pi_1$ and $\pi_2$ are surjective, the images $\pi_1(Q)$ and $\pi_2(\beta(Q))=\gamma(\pi_1(Q))$ are also Borel subgroups (resp.~maximal quasi-tori) in $\Gr(f^*\CG/V)$ and $\Gr(\CG/U)$ by Corollary~\ref{Cor6.10}.

\medskip\noindent
\ref{PropWeaklyPinkQuasiTorus_D}
If $\CF$ is semi-simple, then $\CF^{\dagger}$ is semi-simple, too, because if $\CG^{\dagger}\subset\CF^{\dagger}$ is an overconvergent sub-isocrystal, then it has a convergent complement $\CH\subset\CF$. This $\CH$ is isomorphic to the convergent isocrystal underlying the quotient $\CF^{\dagger}/\CG^{\dagger}$, so by Kedlaya's extension Theorem~\ref{ThmKedlayaFF} the embedding $\CH\hookrightarrow\CF$ extends to an embedding $\CF^{\dagger}/\CG^{\dagger}\hookrightarrow\CF^{\dagger}$, and hence $\CG^{\dagger}$ has an overconvergent complement, too. Statement \ref{PropWeaklyPinkQuasiTorus_D} is now a consequence of Theorem~\ref{ThmDAddezio}.
\end{proof}

\begin{rem}\label{RemWPinkConverse}
We explain why we believe that the converse of Proposition~\ref{PropWeaklyPinkQuasiTorus}\ref{PropWeaklyPinkQuasiTorus_B} is false. If we consider a closed subgroup $G_1$ of a non-connected linear algebraic group $G_2$ and a maximal quasi-torus $T_1$ in $G_1$ one can ask whether there is a maximal quasi-torus of $G_2$ containing $T_1$. We believe that this is not true in general, even under the assumption that $G_1/G_1\open=G_2/G_2\open$ and that $T_1\open$ is a maximal torus in $G_2\open$. That is, we believe that the converse of Proposition~\ref{PropWeaklyPinkQuasiTorus}\ref{PropWeaklyPinkQuasiTorus_B} does not hold in general. In order to prove this converse one would have to show that (in the notation of the proof of Proposition~\ref{PropWeaklyPinkQuasiTorus}) every maximal quasi-torus $T_1$ of $G_1$ is mapped isomorphically to a maximal quasi-torus in $G_2^\red$. By our hypothesis we obtain an isomorphism $G_1/G_1\open \isoto G_2/G_2\open \isoto G_2^\red/(G_2^\red)\open$. The proof now reduces to the following group theoretic statement. By hypothesis $\wt T_2\open=r_{G_2}\circ\beta(T_1\open)$ is a maximal torus in $(G_2^\red)\open$ and we choose a Borel subgroup $\wt B_2\subset (G_2^\red)\open$ containing $\wt T_2\open$. Then $\wt T_2:=N_{G_2^\red}(\wt B_2)\cap N_{G_2^\red}(\wt T_2\open)$ is a maximal quasi-torus in $G_2^\red$. The isomorphism $G_2^\red/(G_2^\red)\open\isoto \wt T_2/\wt T_2\open\subset N_{G_2^\red}(\wt T_2\open)/\wt T_2\open$ from Theorem~\ref{ThmQuasiTorusReductive} yields a split exact sequence of groups
\begin{equation}\label{EqPropWeaklyPinkQuasiTorus}
\xymatrix {
1 \ar[r] & N_{(G_2^\red)\open}(\wt T_2\open)/\wt T_2\open \ar[r] & N_{G_2^\red}(\wt T_2\open)/\wt T_2\open \ar[r] & G_2^\red/(G_2^\red)\open \ar[r] & 1\,.
}
\end{equation}
Here $W:=N_{(G_2^\red)\open}(\wt T_2\open)/\wt T_2\open$ is the Weyl group. The morphism $G_2^\red/(G_2^\red)\open\isoto G_1/G_1\open \isoto T_1/T_1\open \xrightarrow{\,r_{G_2}\beta}N_{G_2^\red}(\wt T_2\open)/\wt T_2\open$ yields another splitting of \eqref{EqPropWeaklyPinkQuasiTorus}. One has to show that the two splittings are conjugate. Every conjugacy class of splittings $s\colon G_2^\red/(G_2^\red)\open\to N_{G_2^\red}(\wt T_2\open)/\wt T_2\open$ defines a cohomology class $\phi\colon G_2^\red/(G_2^\red)\open\to W$ in $\Koh^1(G_2^\red/(G_2^\red)\open, W)$ as follows. Let $g\in G_2^\red/(G_2^\red)\open$. Since $s(g)$ normalizes $\wt T_2\open$ it conjugates $\wt B_2$ to another Borel subgroup containing $\wt T_2\open$. The latter is of the form $s(g)^{-1}\wt B_2 s(g)=\phi(g)\wt B_2\phi(g)^{-1}$ for a uniquely determined element $\phi(g)\in W$ by \cite[IV.11.19 Proposition]{Borel91}. The cohomology class $\phi$ is trivial if and only if the splitting comes from a maximal quasi-torus, because if $s(g)\in \wt T_2\subset N_{G_2^\red}(\wt B_2)$ then $\phi(g)=1$. So $r_{G_2}\beta(T_1)$ is a maximal quasi-torus if and only if the corresponding cohomology class is trivial. Now it is not difficult to construct a group $G_2^\red$ with $\Koh^1(G_2^\red/(G_2^\red)\open, W)\ne0$ and to choose a splitting $s$ with non-zero cohomology class. However, we do not know, whether this situation can arise from a group homomorphism $G_1\to G_2$ or from a convergent $F$-isocrystal $\CF\in\FIsoc_\olK(U)$.
\end{rem}

\begin{prop}\label{PropWeaklyPinkPlusConservative} 
Let $\CF\in\FIsoc_\olK(U)$ satisfy Conjecture~\ref{ConjWPink} with respect to an open immersion $f\colon V\hookrightarrow U$, and let $\CG\in\FIsoc_\olK(U)$ satisfy $\Gr(f^*\CG/V) = \Gr(\CG/U)$. Then $\CF\oplus\CG$ satisfies Conjecture~\ref{ConjWPink} with respect to $f$, too. 
\end{prop}
\begin{proof} By Proposition~\ref{PropGroupOfSum}\ref{PropGroupOfSum_B} there are two Cartesian diagrams
$$\xymatrix @C=-1pc @R=1pc {
 & \ar@{->>}[ld]_{\pi_1}\Gr(f^*(\CF)\oplus f^*(\CG)/V)\ar@{->>}[rd]^{\rho_1} & \\
\Gr(f^*\CF/V)\ar@{->>}[rd] & \;\qed\qquad\, & \ar@{->>}[ld]^{\sigma_1}\Gr(f^*\CG/V) \\ 
& \Gr(\dal f^*(\CF)\dar\cap\dal f^*(\CG)\dar/V) & } \qquad\qquad
\xymatrix @C=0pc @R=1pc {
 & \ar@{->>}[ld]_{\pi_2}\Gr(\CF\oplus\CG/U)\ar@{->>}[rd]^{\rho_2} & \\
\Gr(\CF/U)\ar@{->>}[rd] & \;\qed\quad\quad &
\ar@{->>}[ld]^{\sigma_2}\Gr(\CG/U) \\ 
 & \Gr(\dal\CF\dar\cap\dal\CG\dar/U), & }$$
of algebraic groups, and the pull-back under $f$ induces a morphism $\beta$ from the left to the right diagram. Let $N_i$ denote the kernel of $\pi_i$. Then $N_i$ is mapped under $\rho_i$ isomorphically onto the kernel of $\sigma_i$. Since the map $\beta\colon\Gr(f^*\CG/V)\isoto \Gr(\CG/U)$ is an isomorphism by assumption, and since
$$
\beta\colon \Gr(\dal f^*\CF\dar\cap\dal f^*\CG\dar/V)\into\Gr(\dal\CF\dar\cap\dal\CG\dar/U)
$$
is a closed immersion by Lemma~\ref{Lemma9.2}, we conclude that $\beta|_{N_1}\colon N_1\to N_2$ is an isomorphism. If $T\subset \Gr(f^*(\CF)\oplus f^*(\CG)/V)$ is a maximal quasi-torus, its images in $\Gr(f^*\CF/V)$ and in $\Gr(\CF/U) = \Gr(\CF\oplus\CG/U)\big/\beta(N_1)$ are maximal quasi-tori by Corollary~\ref{Cor6.10} and by the assumption on $\CF$. Therefore, the image of $T$ in $\Gr(\CF\oplus \CG/U)$ is a maximal quasi-torus by Theorem~\ref{new_gambit}, and hence $\CF\oplus\CG$ and $f$ also satisfy Conjecture~\ref{ConjWPink}.
\end{proof}

\bigskip

Next we will establish the following

\begin{prop}\label{direct_isoclinic} 
Let $\CF\in\FIsoc_\olK(U)$ and let $f\colon V\into U$ be an open subscheme. Assume either that
\begin{enumerate}
\item \label{direct_isoclinic_A}
$\CF$ is a direct sum of isoclinic convergent $F$-isocrystals on $U$, or
\item \label{direct_isoclinic_B}
the identity component $\Gr(\CF/U)\open$ of its monodromy group is a torus.
\end{enumerate}
Then $\Gr(f^*\CF/V)=\Gr(\CF/U)$.
\end{prop}

To prove the proposition we need the following
\begin{lemma}\label{isoc10.2}
The full sub-category $\CS(U)$ of direct sums of isoclinic convergent $F$-isocrystals is a full Tannakian sub-category. 
\end{lemma}

\begin{proof} 
Clearly $\CS(U)$ is closed under direct sums. Since the tensor product and duals of isoclinic $F$-isocrystals are isoclinic, the category $\CS(U)$ is closed under tensor products and taking duals, too. Therefore, it will be sufficient to show that if $\CF$ is a direct sum of isoclinic $F$-isocrystals and $\CG\subset\CF$ is a sub-$F$-isocrystal, then $\CG$ and $\CF/\CG$ are also direct sums of isoclinic $F$-isocrystals with the same slopes than $\CF$. By looking at the dual $(\CF/\CG)\dual\subset\CF\dual$ and observing that the dual of an isoclinic $F$-isocrystal is again isoclinic with the negative slope, it suffices to prove the statement for $\CG$. Write $\CF$ as $\CF=\CF_1\oplus\cdots \oplus\CF_n,$ where each $\CF_i$ is an isoclinic $F$-isocrystal of slope $\lambda_i$ and the slopes $\lambda_i$ are pair-wise different. We show the claim by induction on $n$. The case $n=1$ is trivial. 

Now assume that $n\geq2$ and the claim is true for $n-1$. Let $\pi_i\colon\CF\to\CF_i$ be the projection onto the $i$-th factor. Set $\CG_1=\ker(\pi_1)\cap\CG$; it is a sub-$F$-isocrystal of $\CG$. It is also isomorphic to a sub-$F$-isocrystal of $\ker(\pi_1)=\CF_2\oplus\cdots\oplus\CF_n$, so it is a direct sum of isoclinic $F$-isocrystals by the induction hypothesis. Set $\CG_2=\ker(\pi_2\oplus\cdots\oplus\pi_n)\cap\CG$; it is a sub-$F$-isocrystal of $\CG$. It is also isomorphic to a sub-$F$-isocrystal of $\CF_1$, so it is isoclinic. Since $\CG_1\cap\CG_2\subset \ker(\pi_1)\cap\ker(\pi_2\oplus\cdots\oplus\pi_n)$, the intersection $\CG_1\cap\CG_2$ is the trivial isocrystal, and therefore we have an injection $\CG_1\oplus\CG_2\hookrightarrow\CF$ which is an isomorphism onto $\CG_1+\CG_2$, the sub-isocrystal generated by $\CG_1$ and $\CG_2$. Since $\CG$ contains both $\CG_1$ and $\CG_2$, it contains $\CG_1+\CG_2$, too. The quotient $\CG/(\CG_1+\CG_2)$ is isomorphic both to a subquotient of $\CF_1$ and $\CF_2\oplus\cdots\oplus\CF_n$. By our induction hypothesis, every subquotient of $\CF_1$ is isoclinic with slope $\lambda_1$ and every subquotient of $\CF_2\oplus\cdots\oplus\CF_n$ is a direct sum of isoclinics with slopes in $\{\lambda_2,\ldots,\lambda_n\}$. Since $\lambda_1$ does not lie in $\{\lambda_2,\ldots,\lambda_n\}$, this can only be the case if the quotient $\CG/(\CG_1+\CG_2)$ is trivial, and hence $\CG\cong\CG_1\oplus\CG_2$. The claim follows. 
\end{proof}

\begin{proof}[Proof of Proposition~\ref{direct_isoclinic}]
\ref{direct_isoclinic_A}
By \cite[Proposition~2.21]{Deligne-Milne82} we must show two things: (1) $f^*\colon \dal\CF\dar\to\dal f^*\CF\dar$ is fully faithful, and (2) if $\CG\subset f^*\CF$ is a sub-$F$-isocrystal (over $V$), then $\CG$ is of the form $f^*\CH$ for some sub-$F$-isocrystal $\CH\subset\CF$. By the proof of Lemma~\ref{isoc10.2}, $\CG$ is the direct sum of sub-$F$-isocrystals of pull-backs of the isoclinic direct summands of $\CF$ via $f$. So we may assume that $\CF$ is isoclinic of slope $\lambda$. Then there is a constant $F$-isocrystal $\mathcalD\in\FIsoc_\olK(U)$ of slope $\lambda$, such that $\CF\otimes\mathcalD\dual$ is unit-root. Then $\CG\otimes f^*\mathcalD\dual\subset f^*(\CF\otimes\mathcalD\dual)$ is of the form $f^*\CH$ for a unique $F$-isocrystal $\CH\subset\CF\otimes\mathcalD\dual$ on $U$ by Corollary~\ref{CorUnitRootFromUToV}. Therefore, $\CG=f^*(\CH\otimes\mathcalD)$ and (2) is proven. To prove (1) we use the same argument to reduce to full faithfulness of $f^*$ on unit-root $F$-isocrystals, which was stated in Corollary~\ref{CorUnitRootFromUToV}.

\medskip\noindent
\ref{direct_isoclinic_B}
By Corollary~\ref{CorConnMonodromy} there is a finite \'etale covering $g\colon \wt U\to U$, such that  $\Gr(g^*\CF/\wt U)\isoto\Gr(\CF/U)\open$ is an isomorphism. Being a torus, its one-dimensional $\olK$-rational representations generate the Tannakian category of all its representations. So $g^*\CF\in\dal \wt\CF\dar\subset\FIsoc_{\olK}(\wt U)$ for an $F$-isocrystal $\wt\CF$ on $\wt U$ which is a sum of one-dimensional, and hence isoclinic $F$-isocrystals. Let $pr_{\wt U}\colon\wt V:=\wt U\times_U V\to\wt U$ be the projection. By \ref{direct_isoclinic_A} for $\wt\CF$ and Proposition~\ref{PropWeaklyPinkQuasiTorus}\ref{PropWeaklyPinkQuasiTorus_C} the closed immersion $\Gr(pr_{\wt U}^*g^*\CF/\wt V)\into \Gr(g^*\CF/\wt U)$ is an isomorphism. Since $g\circ pr_{\wt U}= f\circ pr_V\colon\wt V\to U$, Proposition~\ref{Prop1.17} and Lemma~\ref{Lemma9.2} provide closed immersions $\Gr(\tilde pr_{\wt U}^*g^*\CF/\wt V)\into \Gr(f^*\CF/V)\into \Gr(\CF/U)$ whose composition is an isomorphism onto the identity component by the above. Now Lemma~\ref{Lemma9.2} implies that $\Gr(f^*\CF/V)\into \Gr(\CF/U)$ is an isomorphism.
\end{proof}

After these general results we now turn towards the proof of Theorem~\ref{ThmWPinkImpliesChebotarevIntro}. We recall the symbol $^{H\!}C$ from Notation~\ref{InitialNotation}, and make use of the following

\begin{defn}\label{DefFrobSS}
Let $\CF\in\FIsoc_\olK(U)$ and fix a maximal quasi-torus $T\subset\Gr(\CF/U)$ and an element $t\in T$. For every closed point $x$ of $U$ let 
\[
\Frob_x^{ss}(\CF,tT\open)\;:=\;\ol{\Frob_x(\CF)}\cap tT\open.
\]
\end{defn}

We will frequently use the following useful fact: let $\CG$ be an object of $\dal\CF\dar$ and let $h\colon \Gr(\CF/U)\onto\Gr(\CG/U)$ be the corresponding surjective homomorphism. Choose two maximal quasi-tori $T_1\subset\Gr(\CF/U)$ and $T_2\subset\Gr(\CG/U)$ such that $h$ maps $T_1$ into $T_2$. Let $t_1\in T_1$ and $t_2:=h(t_1)$. Then for every closed point $x$ of $U$ the morphism $h$ maps $\Frob_x^{ss}(\CF,t_1T_1\open)$ into $\Frob_x^{ss}(\CG,t_2T_2\open)$. (This is clear since $h$ maps $\Frob_x(\CF)$ into $\Frob_x(\CG)$ by Lemma~\ref{LemmaCompatibility}.)

\begin{proof}[Proof of Theorem~\ref{ThmWPinkImpliesChebotarevIntro}]
We write $G:=\Gr(\CF/U)$ and $H:=\Gr(f^*\CF/V)\subset G$. The semi-simplification $\CH=(f^*\CF)^{ss}$ is a direct sum of isoclinic convergent $F$-isocrystals on $V$. Then $r_H\colon  H\onto H^\red:=\Gr(\CH/V)$ is the maximal reductive quotient. We choose a maximal quasi-torus $T$ in $H$. Then $T\open$ is a maximal torus in $H\open$ by Lemma~\ref{LemmaQuasiTorus}, and $\wt T:=r_H(T)$ is a maximal quasi-torus in $H^\red$ with $\wt T\open=r_H(T\open)$, and $r_H|_{T}\colon T\isoto \wt T$ is an isomorphism. Moreover, the map $T/T\open\isoto H/H\open\onto G/G\open$ is surjective by Theorem~\ref{ThmQuasiTorusReductive}\ref{ThmQuasiTorusReductive_A} and Lemma~\ref{Lemma9.2}. 

\medskip\noindent
$\ref{ThmWPinkImpliesChebotarevIntro_A}\Rightarrow\ref{ThmWPinkImpliesChebotarevIntro_B}$ under the additional assumption that $T$ is abelian follows from Theorem~\ref{ThmAbelianQuasiTorus}.

\medskip\noindent
$\ref{ThmWPinkImpliesChebotarevIntro_B}\Rightarrow\ref{ThmWPinkImpliesChebotarevIntro_C}$. By Propositions~\ref{PropWeaklyPinkQuasiTorus}, \ref{PropWeaklyPinkPlusConservative}, \ref{direct_isoclinic}, $\CG\oplus \CI$ also satisfies condition \ref{ThmWPinkImpliesChebotarevIntro_B}. So we may assume $\CF=\CG\oplus\CI$. Let $S\subset|U|$ be a set of positive upper Dirichlet density. We replace $S$ by $S\cap|V|$ which has the same Dirichlet density as $S$ because $|U|\smallsetminus|V|$ has smaller dimension than $U$. By Theorem~\ref{ThmIsoclinic}, Conjecture~\ref{MainConj3} holds for $\CH$ on $V$. Let $\tilde t\in \wt T$ such that the connected component $\tilde t(H^\red)\open$ of $H^\red$ is contained in the closure of $\bigcup_{x\in S}\Frob_x(\CH)$. We claim that this is not changed if we remove from $S$ all points $x$ for which $\Frob_x(\CH)$ does not consist of semi-simple elements or does not meet $\tilde t(H^\red)\open$. Namely, by Theorem~\ref{ThmQuasiTorusReductive}\ref{ThmQuasiTorusReductive_D} there is an open set $O$ in $\tilde t(H^\red)\open$ consisting of semi-simple elements. Since the closure $\ol F=\tilde t(H^\red)\open$ of $F:=\tilde t(H^\red)\open\cap \bigcup_{x\in S}\Frob_x(\CH)$ is irreducible and contained in the union of $\ol{F\cap O}$ and $\ol{F\smallsetminus O}\subset (\tilde t(H^\red)\open\smallsetminus O)$, we conclude that $\tilde t(H^\red)\open$ equals the closure of $F\cap O$ which consists of semi-simple elements only.

Since $H^\red$ is reductive, the semi-simple conjugacy class $\Frob_x(\CH)$ is closed in $H^\red$ for every $x\in S$ by Theorem~\ref{ThmQuasiTorusReductive}\ref{ThmQuasiTorusReductive_B}. Therefore, $\Frob_x^{ss}(\CH,\tilde t\wt T\open)=\tilde t\wt T\open\cap\Frob_x(\CH)$ and Lemma~\ref{useful}\ref{useful_B} implies that $\tilde t\wt T\open$ is the closure of $\wt C:=\tilde t\wt T\open\cap\bigcup_{x\in S}\Frob_x(\CH)=\bigcup_{x\in S}\Frob_x^{ss}(\CH,\tilde t\wt T\open)$. We now lift the situation to $H$. Let $t:=(r_H|_{T})^{-1}(\tilde t)\in T$ and view it as an element of $G$ by Lemma~\ref{Lemma9.2}. By assumption \ref{ThmWPinkImpliesChebotarevIntro_B}, $T$ is also a maximal quasi-torus in $G$. By Lemma~\ref{useful}\ref{useful_B} and Proposition~\ref{reduction3->2} the implication $\ref{ThmWPinkImpliesChebotarevIntro_B}\Rightarrow\ref{ThmWPinkImpliesChebotarevIntro_C}$ is now a consequence of the following

\medskip\noindent
\emph{Claim.} The set $C:=\bigcup_{x\in S}\Frob_x^{ss}(\CF,tT\open)$ is dense in $tT\open$.

\medskip\noindent
To prove the claim, let $x\in S$ be arbitrary and pick a $g\in \Frob_x(f^*\CF)\cap tH\open$. Write $g=g_s\cdot g_u$, where $g_s, g_u\in H$ are the semi-simple and unipotent parts of $g$, respectively. Theorem~\ref{ThmQuasiTorusGeneral} shows that $g_s$ lies in a maximal quasi-torus of $H$ and can be conjugate by an element $h\in H\open$ such that $h^{-1}g_sh$ lies in $T\cap tH\open=tT\open$. Replacing $g$ with $h^{-1}gh\in\Frob_x(f^*\CF)$ we may assume that $g_s\in tT\open$. Since $g_s$ is also the semi-simple part of $g$ in the larger group $G$ which is reductive, and $\Frob_x(f^*\CF)\subset\Frob_x(\CF)$ by definition, we get in terms of Notation~\ref{InitialNotation} that $\Frob_x^{ss}(\CF,tT\open)=\ol{\Frob_x(\CF)}\cap tT\open={}^{G}\{g_s\}\cap tT\open$ by Lemma~\ref{useful}\ref{useful_A}. Therefore, also ${}^{G}\!\Frob_x^{ss}(\CF,tT\open)\cap tT\open=\Frob_x^{ss}(\CF,tT\open)$.

Since $r_H$ preserves Jordan decompositions and $r_H(g)\in\Frob_x(\CH)$ is semi-simple by our assumption on $S$, we get $r_H(g_s)=r_H(g)\in\Frob_x(\CH)$. Therefore, $\Frob_x(\CH)={}^{H^\red\!}\{r_H(g_s)\}$ and this means that ${}^{H^\red\!}r_H(\Frob_x^{ss}(\CF,tT\open))$ contains $\Frob_x(\CH)$. Note that this inclusion may be strict, because $\Frob_x^{ss}(\CF,tT\open)$ may contain elements that are conjugate under $G$ but not under $H$. Therefore, also ${}^{H^\red\!}r_H(\Frob_x^{ss}(\CF,tT\open))\cap \tilde t\wt T\open$ contains $\Frob_x(\CH)\cap \tilde t\wt T\open=\Frob_x^{ss}(\CH,\tilde t\wt T\open)$, and hence ${}^{H^\red\!}r_H(C)\cap \tilde t\wt T\open=\bigcup_{x\in S}{}^{H^\red\!}r_H(\Frob_x^{ss}(\CF,tT\open))\cap \tilde t\wt T\open$ contains $\bigcup_{x\in S}\Frob_x^{ss}(\CH,\tilde t\wt T\open)=\wt C$. Since $\wt C$ is dense in $\tilde t\wt T\open$, also ${}^{H^\red\!}r_H(C)\cap \tilde t\wt T\open$ is dense in $\tilde t\wt T\open$, and Proposition~\ref{finite_to_oneNEW} yields that ${}^{H\!}C\cap tT\open$ is dense in $tT\open$. Finally ${}^{H\!}C\cap tT\open\subset{}^{G\!}C\cap tT\open=\bigcup_{x\in S}{}^{G}\!\Frob_x^{ss}(\CF,tT\open)\cap tT\open=C$. So $C$ is dense in $tT\open$ as claimed and the implication $\ref{ThmWPinkImpliesChebotarevIntro_B}\Rightarrow\ref{ThmWPinkImpliesChebotarevIntro_C}$ is established.

\medskip\noindent
$\ref{ThmWPinkImpliesChebotarevIntro_C}\Rightarrow\ref{ThmWPinkImpliesChebotarevIntro_D}$ is obvious.

\medskip\noindent
$\ref{ThmWPinkImpliesChebotarevIntro_D}\Rightarrow\ref{ThmWPinkImpliesChebotarevIntro_A}$. Let $F=\bigcup_{x\in|V|}\Frob_x(f^*\CF)\subset H$ be the union of the Frobenius conjugacy classes (conjugacy under $H$), and let $F^{ss}=\{g_s\colon g\in F\}\subset H$ be the set of the semi-simple parts $g_s$ of the elements $g$ of $F$. In terms of Notation~\ref{InitialNotation}, the set ${}^{G\!}F=\bigcup_{x\in|V|}\Frob_x(\CF)\subset G$ is the union of the Frobenius conjugacy classes (conjugacy under $G$), and ${}^{G\!}(F^{ss})=({}^{G\!}F)^{ss}:=\{g_s\colon g\in {}^{G\!}F\}\subset G$. By our assumption $^{G\!}F$ is dense in $G$. By Corollary~\ref{Cor_ss-ize} we get that $({}^{G\!}F)^{ss}$ is dense in $G$, too. Since every element of $F^{ss}$ is conjugate to an element of the maximal quasi-torus $T\subset H$ by Theorem~\ref{ThmQuasiTorusReductive}\ref{ThmQuasiTorusReductive_A},\ref{ThmQuasiTorusReductive_B}, we get that ${}^{G}T$ is also dense in $G$, and hence that ${}^{G}T\cap G\open$ is dense in $G\open$. 

Since $T/T\open$ surjects onto $G/G\open$, we have ${}^{G}T={}^{G\open\!}T$, because every element $g\in G$ can be written as $g=g_0h$ with $g_0\in G\open$ and $h\in T$, and so $gTg^{-1}=g_0Tg_0^{-1}$. Therefore, ${}^{G\open\!}T\cap G\open={}^{G\open\!}(T\cap G\open)$ is dense in $G\open$. Now let $A\subset G\open$ be a maximal torus in $G\open$ containing $T\open$. By Steinberg's result \cite[\S\,III.3.4, Corollary~2]{Steinberg74}, see also Remark~\ref{RemSteinbergConjugacy}(c), any two elements of $A$ in the same $G\open$-conjugacy class actually lie in the same orbit under the action of the Weyl group $W(G\open,A)$. If $T\open$ was a proper closed subscheme of $A$ then the same would hold for the finite union of its images under the action of $W(G\open,A)$, because $A$ is irreducible. The ${}^{G\open\!}(T\cap G\open)\cap A$ would be a proper closed subscheme of $A$. By Lemma~\ref{useful}\ref{useful_B} this contradicts that ${}^{G\open\!}(T\cap G\open)$ is dense in $G\open$. We conclude that $T\open=A$ is a maximal torus in $G\open$.
\end{proof}

\begin{rem}\label{RemPinkForOverconv}
The arguments for the implication $\ref{ThmWPinkImpliesChebotarevIntro_B}\Rightarrow\ref{ThmWPinkImpliesChebotarevIntro_C}$ also give a second proof of Theorem~\ref{ThmOverconvIntro} for a semi-simple overconvergent $\CF^\dagger\in\FIsoc^\dagger_\olK(U)$ with associated convergent $\CF\in\FIsoc_\olK(U)$. Namely, keeping the notation of the proof of Theorem~\ref{ThmWPinkImpliesChebotarevIntro} but taking $G:=\Gr^\dagger(\CF^\dagger/U)$ instead, the maximal quasi-torus $T\subset H$ is also a maximal quasi-torus in $G$ by Theorem~\ref{ThmDAddezio} and Proposition~\ref{PropWeaklyPinkQuasiTorus}\ref{PropWeaklyPinkQuasiTorus_A}. Here we use Lemma~\ref{LemmaOverconvMonodr} instead of Lemma~\ref{Lemma9.2}. Literally the same arguments now prove Theorem~\ref{ThmOverconvIntro}.
\end{rem}

\begin{appendix}

\section{Results from Measure Theory}\label{AppendixB}

\begin{defn}\label{DefPushForwMeasure}
Let $f\colon X\to Y$ be a continuous map between topological spaces, each equipped with the Borel $\sigma$-algebra. For a measure $\mu$ on $X$ the \emph{push-forward measure} $f_*\mu$ is defined as $(f_*\mu)(V):=\mu(f^{-1}V)$ for every Borel-measurable subset $V\subset Y$. It satisfies $\int_Y h(y)\;d\,f_*\mu(y)=\int_X h(f(x))\;d\mu(x)$ for every measurable function $h$ on $Y$.
\end{defn}

Our next aim is to define the notion of a pull-back of measures under certain nice maps. 

\begin{lemma}\label{inj_borel} Let $f:X\to Y$ be an injective continuous map between compact Hausdorff spaces. Then $f$ maps Borel-measurable sets in $X$ to Borel-measurable sets in $Y$. 
\end{lemma}

\begin{proof}
Since $X$ is compact and $Y$ is Hausdorff, the image of every closed subset of $X$ is compact, and hence closed, and Borel-measurable. It thus suffices to show that $\CC=\{B\subset X\colon f(B)\subset Y\textrm{ is Borel-measurable}\}$ is a $\sigma$-algebra on $X$, because it will then contain the Borel $\sigma$-algebra. Since $\emptyset$ and $X$ are closed in $X$, we get that $\emptyset,X\in\CC$ by the above. Since $f$ is injective, $f(X\smallsetminus B)=f(X)\smallsetminus f(B)$ for any $B\in\CC$, and so $X\smallsetminus B\in\CC$. If $B_i\in\CC$ for $i\in\BN$, then $f(\bigcup_{i\in\BN}B_i)=\bigcup_{i\in\BN}f(B_i)\subset Y$ is Borel-measurable, and hence $\bigcup_{i\in\BN}B_i\in\CC$. 
\end{proof}

\begin{defn}\label{DefNiceMap}
Let $f:X\to Y$ be a continuous map between topological spaces. We say that the map $f$ is \emph{nice} if there is a countable pair-wise disjoint decomposition $X=\coprod_{i=1}^{\infty}X_i$ of $X$ into Borel-measurable subsets $X_i$ such that the restriction of $f$ onto the closure $\olX_i$ of each $X_i$ is injective.
\end{defn}

\begin{rem}\label{RemNiceMap}
A continuous map $f:X\to Y$ is nice if and only if there is a countable cover $\{V_i\}_{i\in\BN}$ of $X$ by closed subsets such that $f|_{V_i}$ is injective for every $i\in\BN$. Indeed, if $f$ is nice and $X=\coprod_{i=1}^{\infty}X_i$ is a decomposition of $X$ as in Definition~\ref{DefNiceMap}, then $\{\olX_i\}_{i\in\BN}$ of $X$ is such a cover. On the other hand, if $\{V_i\}_{i\in\BN}$ is a countable cover of $X$ by closed subsets such that $f|_{V_i}$ is injective for every $i\in\BN$ then $X_i=V_i\smallsetminus\bigcup_{j<i}V_j$ is Borel-measurable, and $X=\coprod_{i=1}^{\infty}X_i$ is a decomposition of $X$. Since $X_i\subset V_i$, we have $\olX_i\subset\olV_i=V_i$, and hence $f|_{\olX_i}$ is injective for every $i\in\BN$.
\end{rem}

\begin{lemma}\label{nice_borel} Let $f:X\to Y$ be a nice continuous map between compact Hausdorff spaces. Then $f$ maps Borel-measurable sets in $X$ to Borel-measurable sets in $Y$. 
\end{lemma}
\begin{proof} 
Fix a decomposition of $X$ as in Definition~\ref{DefNiceMap}. For every Borel-measurable subset $B\subset X$ we have $f(B)=\bigcup_{i\in\BN}f(B\cap X_i)$. Thus it will be sufficient to show that $f(B\cap X_i)$ is Borel-measurable for every $i\in\BN$. Since $\olX_i$ is closed, it is a compact Hausdorff space.  Also the restriction of $f$ onto $\olX_i$ is injective by assumption. As $B\cap X_i$ is Borel-measurable in $\olX_i$, the claim follows from Lemma~\ref{inj_borel}. 
\end{proof}

\begin{lemma} \label{LemmaCountingFcn}
Let $f:X\to Y$ be a nice continuous map between compact Hausdorff spaces. For every Borel-measurable subset $A\subset X$ the counting function $c_{A/Y}:Y\to\BR\cup\{\infty\}$ given by the rule $y\mapsto\#(f^{-1}(y)\cap A)$ is measurable. 
\end{lemma}

\begin{proof} Fix a decomposition of $X$ as in Definition~\ref{DefNiceMap}. Clearly $c_{A/Y}(y)=\sum_{i=1}^{\infty}c_i(y)$, where for every $i\in\BN$ the function $c_i:Y\to\BR\cup\{\infty\}$ is given by the rule
\[
y\mapsto\#(f^{-1}(y)\cap A\cap X_i).
\]
Since the latter are non-negative, $c_{A/Y}$ is the point-wise supremum of the sequence $\bigl(\sum_{i=1}^j c_i\bigr)_{j\in\BN}$. So it will be enough to show that each $c_i$ is measurable. However, the restriction of $f$ onto $X_i$ is injective, so $c_i$ is just the characteristic function of $f(A\cap X_i)$. By Lemma \ref{nice_borel} the set $f(A\cap X_i)$ is Borel-measurable, since $A\cap X_i$ is. Thus $c_i$ is also measurable. 
\end{proof}

\begin{defn}\label{DefPullbackMeasure}
Let $f:X\to Y$ be a nice continuous map between compact Hausdorff spaces. Let $\mu$ be a measure on the Borel $\sigma$-algebra $Y$. We define the \emph{pull-back measure} $f^*\mu$ on $X$ by the formula:
$$(f^*\mu)(A)\;:=\;\int_Y c_{A/Y}(y)\;d\mu(y)$$
for every Borel-measurable $A\subset X$. By Lemma~\ref{LemmaCountingFcn} the integrand is measurable, and it is also non-negative, so the integral is well-defined. When there is an upper bound on the cardinality of the fibers of $f$, we get that $f^*\mu$ is bounded, too, i.e.~it only takes finite values if the same is true for $\mu$. 
\end{defn}

\begin{prop}\label{PropWeakConvOfPullBack}
Let $f\colon X\to Y$ be a nice, continuous, surjective map between compact metrizable spaces and let $\mu_m$ for $m\in\BN$ be a sequence of measures on $Y$ which converges weakly to a measure $\nu$ on $Y$. Assume that there is a positive integer $M$ such that all fibers of $f$ have cardinality at most $M$. Assume further that there is a closed subset $V\subset Y$ of measure $\nu(V)=0$ whose complement is a finite disjoint union $Y\smallsetminus V=\coprod_{i=1}^n Y_i$ of open subsets $Y_i$ such that $f$ is trivial over $Y_i$ in the sense that there is a finite discrete set $F_i$ and a homeomorphism $g_i\colon F_i\times Y_i\isoto f^{-1}(Y_i)$ compatible with the projections onto $Y_i$. Then the pullback measures $f^*\mu_m$ converge weakly to $f^*\nu$.
\end{prop}

\begin{proof}
We first observe that
\begin{align*}
(f^*\nu)(f^{-1}V)\;=\enspace & \int_V\#f^{-1}(y) \;d\nu(y)\;\le\;M\cdot\nu(V)\;=\;0 \qquad \text{and likewise}\\[2mm]
(f^*\mu_m)(f^{-1}V)\;=\enspace & \int_V\#f^{-1}(y) \;d\mu_m(y)\;\le\;M\cdot\mu_m(V)\,.
\end{align*}
Since $\limsup_{m\to\infty}\mu_m(V)\le\nu(V)=0$ by the weak convergence and the Portemanteau theorem~\cite[Theorem~13.16]{Klenke}, we conclude that $\lim_{m\to\infty}(f^*\mu_m)(f^{-1}V)=0$.

Let $U\subset X$ be an open subset. Then 
\[
(f^*\nu)(U)\;=\;(f^*\nu)(U\cap f^{-1}V)+\sum_{i=1}^n(f^*\nu)(U\cap f^{-1}Y_i)\;=\;\sum_{i=1}^n\sum_{x\in F_i}(f^*\nu)\bigl(U\cap g_i(x\times Y_i)\bigr)\,,
\]
and likewise
\[
\liminf_{m\to\infty}(f^*\mu_m)(U)\;=\;\sum_{i=1}^n\sum_{x\in F_i}\liminf_{m\to\infty}(f^*\mu_m)\bigl(U\cap g_i(x\times Y_i)\bigr)\,,
\]
and similarly for $\limsup$. Since the projection map $x\times Y_i\to Y_i$ is a homeomorphism, the image set $f\bigl(U\cap g_i(x\times Y_i)\bigr)\subset Y_i$ is open. Since the counting function $c_{U\cap g_i(x\times Y_i)/Y}(y)$ takes the value $1$ if $y\in f(U\cap g_i(x\times Y_i))$ and the value $0$ if $y\notin f(U\cap g_i(x\times Y_i))$, we have $(f^*\nu)\bigl(U\cap g_i(x\times Y_i)\bigr)=\nu\bigl(f(U\cap g_i(x\times Y_i))\bigr)$ and $(f^*\mu_m)\bigl(U\cap g_i(x\times Y_i)\bigr)=\mu_m\bigl(f(U\cap g_i(x\times Y_i))\bigr)$. From this and the weak convergence of $\mu_m$ to $\nu$ we obtain $\liminf_{m\to\infty}(f^*\mu_m)(U)\ge (f^*\nu)(U)$.

If $U=X$ we also must show that $\limsup_{m\to\infty}(f^*\mu_m)(X)\le (f^*\nu)(X)$. Note that 
\[
\limsup_{m\to\infty}(f^*\mu_m)\bigl(g_i(x\times Y_i)\bigr)\;=\;\limsup_{m\to\infty}\mu_m(Y_i) \;=\;\limsup_{m\to\infty}\mu_m(V\cup Y_i)\;\le\;\nu(V\cup Y_i)
\]
because $V\cup Y_i$ is closed in $Y$ and $\mu_m$ converges weakly to $\nu$. We compute
\[
\nu(V\cup Y_i)\;=\;\nu(V)+\nu(Y_i)\;=\;\nu(Y_i)\;=\;(f^*\nu)\bigl(g_i(x\times Y_i)\bigr)\,.
\]
This shows that the sequence of measures $f^*\mu_m$ converges weakly to $f^*\nu$.
\end{proof}

\end{appendix}



\end{document}